\pdfoutput=1
\documentclass[11pt,letterpaper,]{thesis2}
\usepackage[nonumberlist]{glossaries}
\makeglossaries

\usepackage[ruled,vlined,linesnumbered,algochapter]{algorithm2e}
\usepackage{algorithm2e}
\usepackage{graphicx,color,psfrag}
\usepackage{pstool}
\usepackage{paralist}
\usepackage{multirow}
\usepackage[retainorgcmds]{IEEEtrantools}
\usepackage{url,xspace}
\usepackage{subfig}
\usepackage{array}
\usepackage{floatrow}
\newfloatcommand{capbtabbox}{table}[][\FBwidth]
\usepackage{booktabs}
\usepackage{url,doi}
\usepackage{tikz}
\usetikzlibrary{arrows,shapes}
\newlength\imagewidth
\newlength\imagescale
\newcommand{\ignore}[1]{}
\newcommand{\GEQRTWO}{\ensuremath{\mathit{GEQR2}}\xspace}

\newcommand{\GEQRT}{\ensuremath{\mathit{GEQRT}}\xspace}
\newcommand{\TSQRT}{\ensuremath{\mathit{TSQRT}}\xspace}

\newcommand{\UNMQR}{\ensuremath{\mathit{UNMQR}}\xspace}
\newcommand{\TSMQR}{\ensuremath{\mathit{TSMQR}}\xspace}
\newcommand{\TTQRT}{\ensuremath{\mathit{TTQRT}}\xspace}
\newcommand{\TTMQR}{\ensuremath{\mathit{TTMQR}}\xspace}
\newcommand{\coarse}{\ensuremath{\mathit{coarse}}\xspace}
\newcommand{\BS}{\ensuremath{\mathit{BS}}\xspace}
\newcommand{\s}{\ensuremath{\star}\xspace}

\newcommand{\SK}{{\sc Sameh-Kuck}\xspace}
\newcommand{\MC}{{\sc Fibonacci}\xspace}
\newcommand{\BT}{{\sc BinaryTree}\xspace}
\newcommand{\Greedy}{{\sc Greedy}\xspace}
\newcommand{\GA}{{\sc GrASAP}\xspace}
\newcommand{\FT}{{\sc FlatTree}\xspace}
\newcommand{\PT}{{\sc PlasmaTree}\xspace}
\newcommand{\ASAP}{{\sc Asap}\xspace}

\newcommand{\TiledQRURL}{\url{http://graal.ens-lyon.fr/~mjacquel/tiledQR.html}\xspace}
\newcommand{\elim}{\mathit{elim}}
\def\Section{\S}

\newtheorem{thm}{Theorem}[chapter]
\newtheorem{lemma}[thm]{Lemma}
\newtheorem{cor}[thm]{Corollary}
\newtheorem{prop}[thm]{Proposition}
\newtheorem{defin}[thm]{Definition}

\newenvironment{proof}{{\bf Proof:} }{\hfill\rule{2.1mm}{2.1mm}}

\newcounter{example}

\begin{document}

\pagenumbering{roman}
\pagestyle{empty}

\title{Tiled Algorithms for Matrix Computations on Multicore Architectures}
\author{Henricus M}{Bouwmeester}
\otherdegrees{B.S., Colorado Mesa University, 1998\\M.S., Colorado State University, 2000}
\degree{Doctor of Philosophy}{Ph.D., Applied Mathematics}
\degreeyear{2012} \dept{Applied Mathematics}{}
\titlepage

\pagestyle{fancy}

\advisor{Associate Professor}{Julien Langou}
\readerone{Gita Alaghband}
\readerthree{Stephen Billups}
\readertwo{Elizabeth R. Jessup}
\chair{Lynn Bennethum}
\approvalpage

\abstractpage{%
%
Current computer architecture has moved towards the multi/many-core
structure.  However, the algorithms in the current sequential dense numerical
linear algebra libraries (\emph{e.g.} LAPACK) do not parallelize well on
multi/many-core architectures. A new family of algorithms, the \emph{tile
algorithms}, has recently been introduced to circumvent this problem.  Previous
research has shown that it is possible to write efficient and scalable tile
algorithms for performing a Cholesky factorization, a (pseudo) LU factorization,
and a QR factorization.  The goal of this thesis is to study tiled algorithms in
a multi/many-core setting and to provide new algorithms that exploit the
current architecture to improve performance relative to current state-of-the-art
libraries while maintaining the stability and robustness of these libraries.

In Chapter~\ref{chp:cholinv}, we confront the problem of computing the inverse
of a symmetric positive definite matrix with tiled algorithms. We observe that,
using a dynamic task scheduler, it is relatively painless to translate existing
LAPACK code to obtain a ready-to-be-executed tile algorithm.  However we
demonstrate that nontrivial compiler techniques (array renaming, loop reversal
and pipelining) need to be applied to further increase the parallelism of
our application, both theoretically and experimentally.

Chapter~\ref{chp:tiledqr} revisits existing algorithms for the QR factorization
of rectangular matrices composed of $p \times q$ tiles, where $p \geq q$, for an
unlimited number of processors.  Within this framework, we study the critical
paths and performance of algorithms such as \SK, \MC, \Greedy, and those found
within PLASMA.  We also provide a monotonically increasing function to transform
the elimination list of a coarse-grain algorithm to a tiled algorithm.  Although
the optimality from the coarse-grain \Greedy algorithm does not translate
to the tiled algorithms, we propose a new algorithm and show that it is optimal
in the tiled context.

In Chapters~\ref{chp:cholinv} and~\ref{chp:tiledqr}, our context includes an unbounded
number of processors. The exercise was to find algorithmic variants with short
critical paths. Since the number of resources is unlimited, any task is
executed as soon as all its dependencies are satisfied.  In
Chapters~\ref{chp:cholfact} and~\ref{chp:qrfact}, we set ourselves in the more
realistic context of bounded number of processors.  In this context, at a given
time, the number of ready-to-go tasks can exceed the number of available
resources, and therefore a schedule which prescribes which tasks to execute when
needs to be defined. For the Cholesky factorization, we study standard
schedules and find that the critical path schedule is best. We also derive a
lower bound on the time to solution of the optimal schedule. We conclude that
the critical path schedule is nearly optimal for our study. For the QR
factorization problem, we study the problem of optimizing the reduction trees
(therefore the algorithm) and the schedule simultaneously. This is
considerably harder than the Cholesky factorization where the algorithm is
fixed and so, for Cholesky factorization, we are concerned only with
schedules. We provide a lower bound for the time to solution for any tiled QR
algorithm and any schedule. We also show that, in some cases, the optimal
algorithm for an unbounded number of processors (found in
Chapter~\ref{chp:tiledqr}) cannot be scheduled to solve optimally the combined
problem.  We compare our algorithms and schedules with our lower bound.


Finally, in Chapter~\ref{chp:strassen} we study a recursive tiled algorithm in
the context of matrix multiplication using the Winograd-Strassen algorithm using
a dynamic task scheduler.  Whereas most implementations obtain either one or two
levels of recursion, our implementation supports any level of recursion.  
}


\newpage

\tableofcontents
\listoffigures
\listoftables

\setcounter{page}{0}
\pagenumbering{arabic}

%
%

\newglossaryentry{BLAS}{
    name={BLAS},
    description={
        Basic Linear Algebra Subprograms are routines that provide standard
        building blocks for performing basic vector and matrix operations
    }
}

\newglossaryentry{LAPACK}{
    name={LAPACK},
    description={
        Linear Algebra PACKage provides routines for solving systems of
        simultaneous linear equations, least-squares solutions of linear systems
        of equations, eigenvalue problems, and singular value problems
    }
}

\newglossaryentry{ScaLAPACK}{
    name={ScaLAPACK},
    description={
        Scalale Linear Algebra PACKage is a library of high-performance linear
        algebra routines for parallel distributed memory machines which solves
        dense and banded linear systems, least squares problems, eigenvalue
        problems, and singular value problems
    }
}

\newglossaryentry{DAG}{
    name={DAG},
    description={
        Directed Acyclic Graphs where each node represents a task and each each
        represents the data dependencies between the tasks
    }
}

\newglossaryentry{PLASMA}{
    name={PLASMA},
    description={
        Parallel Linear Algebra Software for Multicore Architectures provides
        routines to solve dense general systems of linear equations, symmetric
        positive definite systems of linear equations and linear least squares
        problems, using LU, Cholesky, QR and LQ factorizations
    }
}

\newglossaryentry{Anti-dependency}{
    name={Anti-dependency},
    description={
        Dependency which occurs when an instruction requires a value that is
        later updated. Also known as a Write-After-Read dependency
    }
}

\newglossaryentry{WAR}{
    name={WAR},
    description={
        Write-After-Read. An anti-dependency
    }
}

\newglossaryentry{SYRK}{
    name={SYRK},
    description={
        Routine which is part of the BLAS that computes the rank–k update of the
        upper or lower triangular component of a symmetric matrix
    }
}

\newglossaryentry{GEMM}{
    name={GEMM},
    description={
        Routine which is part of the BLAS that computes the product of two general matrices
    }
}

\newglossaryentry{TRSM}{
    name={TRSM},
    description={
        Routine which is part of the BLAS that solves a triangular linear
        equation
    }
}

\newglossaryentry{TRMM}{
    name={TRMM},
    description={
        Routine which is part of the BLAS that computes the product of a general
        matrix with an upper or lower triangular matrix
    }
}

\newglossaryentry{POTRF}{
    name={POTRF},
    description={
        Routine which is part of the LAPACK that computes the Cholesky
        factorization of a symmetric positive definite matrix
    }
}

\newglossaryentry{TRTRI}{
    name={TRTRI},
    description={
        Routine which is part of the LAPACK that computes the inverse of a upper
        or lower triangular matrix
    }
}

\newglossaryentry{LAUUM}{
    name={LAUUM},
    description={
        Routine which is part of the LAPACK that computes the product of a upper
        or lower triangular matrix with its transpose
    }
}

\newglossaryentry{Strong Scalability}{
    name={Strong Scalability},
    description={
        Scalability which shows how the solution time varies with the number of
        processors for a fixed total problem size
    }
}

\newglossaryentry{Pipelining}{
    name={Pipelining},
    description={
        An implementation technique where multiple instructions are overlapped
        in execution
    }
}

\newglossaryentry{GEQRT}{
    name={GEQRT},
    description={
        Routine which is part of the LAPACK that constructs a compact WY
        representation to apply a sequence of $i_b$ Householder reflections
    }
}

\newglossaryentry{TTQRT}{
    name={TTQRT},
    description={
        Routine which constructs a compact WY representation of a matrix
        composed of two upper triangular matrices stacked one on top of the
        other
    }
}

\newglossaryentry{UNMQR}{
    name={UNMQR},
    description={
        Multiplies a complex matrix by the unitary matrix $Q$ of the QR
        factorization formed by GEQRT.
    }
}

\newglossaryentry{TTMQR}{
    name={TTMQR},
    description={
        Multiplies a complex matrix formed via stacking two square matrices one
        on top of the other by the unitary matrix $Q$ of the QR
        factorization formed by TTQRT
    }
}

\newglossaryentry{TSQRT}{
    name={TSQRT},
    description={
        Routine which constructs a compact WY representation of a matrix
        composed of a square matrix stacked on top of an upper triangular
        matrix
    }
}

\newglossaryentry{TSMQR}{
    name={TSMQR},
    description={
        Multiplies a complex matrix formed via stacking two square matrices one
        on top of the other by the unitary matrix $Q$ of the QR
        factorization formed by TSQRT
    }
}

\newglossaryentry{Elimination List}{
    name={Elimination List},
    description={
        Table which provides the ordered list of the transformations used to
        zero out all the tiles below the diagonal 
    }
}

\chapter{Introduction}\label{chp:intro}
%
%

%

The High-Performance Computing (HPC) landscape is trending more towards a
multi/many-core architecture~\cite{Asanovic:2009:VPC:1562764.1562783} as is
evidenced by the recent projects of major chip manufacturers and reports of
surveys conducted by consulting companies~\cite{intersect360}.  The
computational algorithms for dense linear algebra need to be re-examined to make
better use of these architectures and provide higher levels of efficiency.  Some
of these algorithms may have a straight forward translation from the current
state-of-the-art libraries while others require much more thought and effort to
gain performance increases.  In this thesis, we endeavor to achieve algorithms
that exploit
the architecture to improve performance relative to current state-of-the-art
computational libraries while maintaining the stability and robustness of these
libraries.

Our research will make use of the BLAS (Basic Linear Algebra
Subprograms)~\cite{blasurl} and the LAPACK (Linear Algebra
PACKage)~\cite{LAPACK_1999_guide} libraries.  The BLAS are a standard to perform
basic linear algebra operations involving vectors and matrices while LAPACK
performs the more complicated and higher level linear algebra operations.  

The BLAS divide numerical linear algebra operations into three distinct
groupings based upon the amount of input data and the computational cost.
Operations involving only vectors are considered Level 1; those involving both
vectors and matrices are Level 2; and those involving only matrices are Level 3.
For matrices of size $n\times n$, the Level 3 operations are most coveted since
they use $O(n^2)$ data for $O(n^3)$ computations and inherently reduce the
amount of memory traffic.  Since the BLAS provide fundamental linear algebra
operations, hardware and software vendors such as Intel, AMD, and IBM provide
optimized BLAS libraries for a variety of architectures.  The BLAS library can
be multithreaded to make use of multi/many-core architectures and most of the
vendor supplied libraries are multithreaded.

The LAPACK library is a collection of subroutines for solving most of the common
problems in numerical linear algebra and was developed to make use of Level 3
BLAS operations as much as possible.  Algorithms within LAPACK are written to
make use of panels which can be either a block of columns or a block of rows so that
updates are performed using matrix multiplications instead of vector operations.
LAPACK can make use of a multithreaded BLAS to exploit multi/many-core
architectures but this may not be enough to fully exploit the capability of the
architecture.

As an example, let us consider the Cholesky decomposition to factorize a symmetric
positive definite (SPD) matrix into its triangular factor.  There are three
variants for performing the Cholesky decomposition: bordered, left-looking, and
right-looking.  All three work on either the upper or lower triangular portion
of the matrix and produce the same triangular factor, but depending on the
usage, one may have an advantage over the others.  In Figure~\ref{fig:threechol}
we depict the steps involved in each variant using the lower triangular
formulations.  At the start of each variant, the matrix is subdivided into
blocks of rows and columns, or panels, which will take advantage of the Level 3 BLAS.
\linespread{1.0}
\begin{figure}[htbp]
    \centering
    \subfloat[Bordered variant]{%
        \label{fig:threechol_b}%
        \resizebox{.75\textwidth}{!}{
            \includegraphics[width=.25\textwidth]{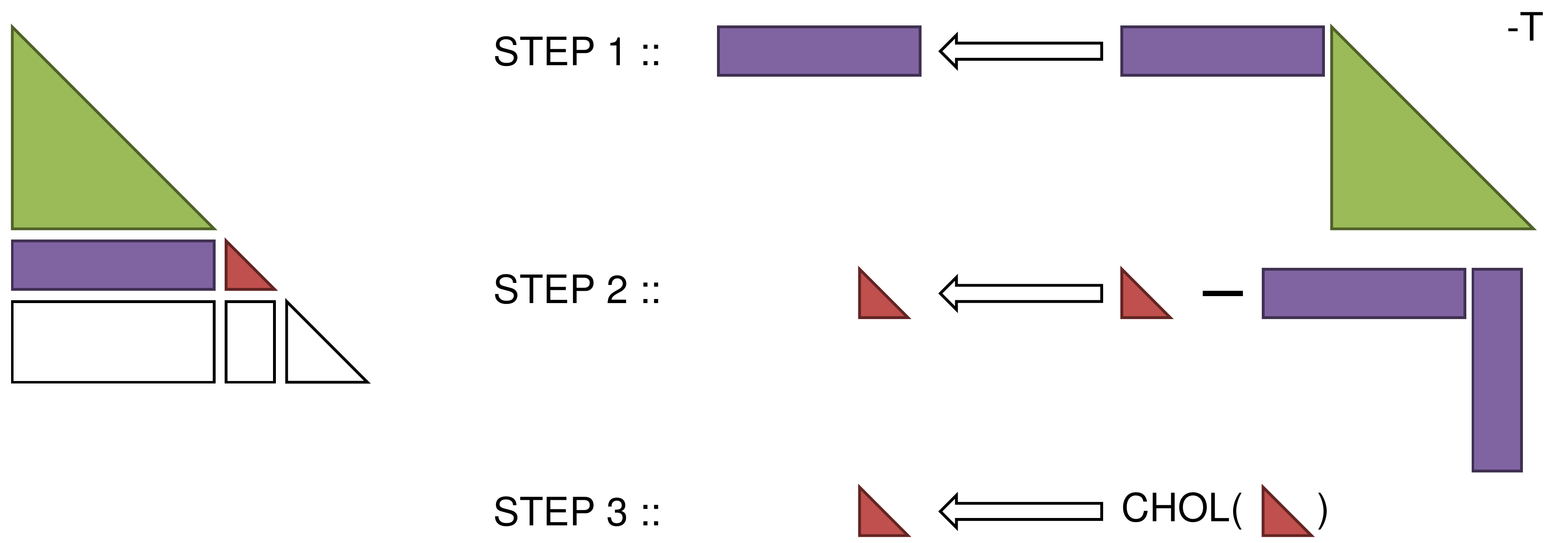}
        }%
    }\\
    \vspace{5mm}
    \subfloat[Left-looking variant]{%
        \label{fig:threechol_l}%
        \resizebox{.75\textwidth}{!}{
            \includegraphics[width=.25\textwidth]{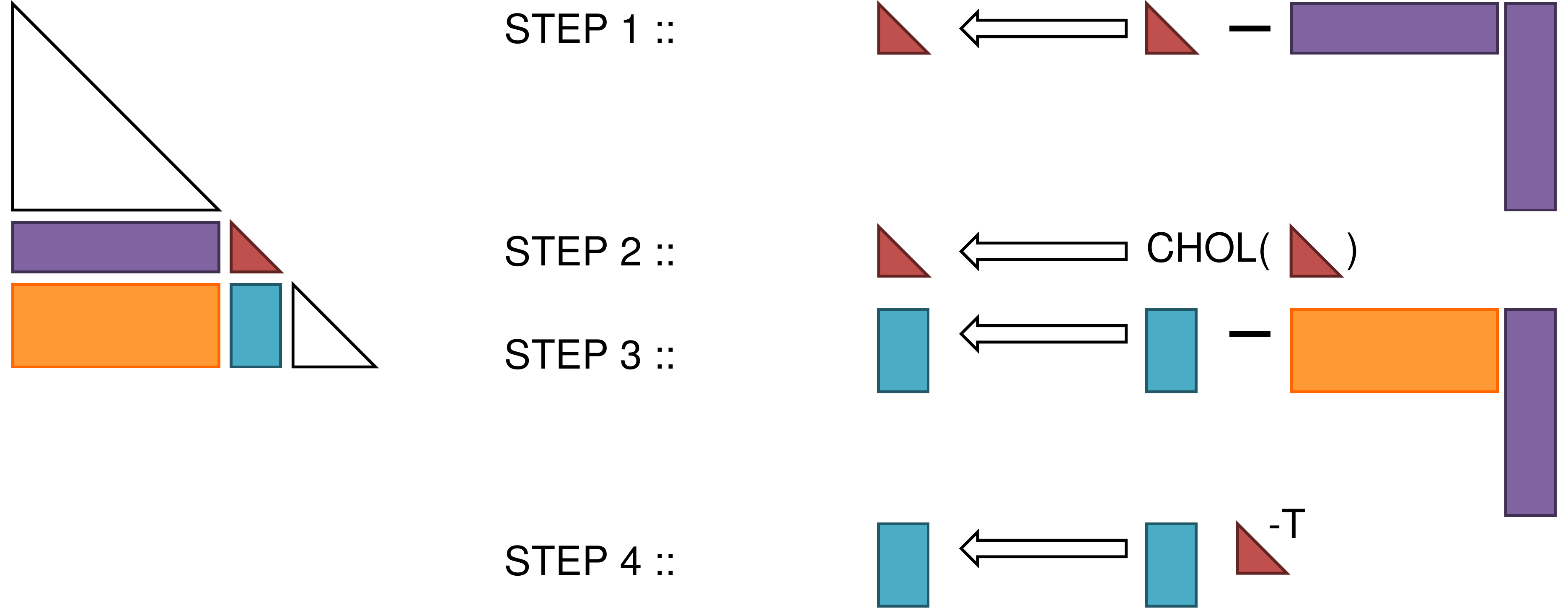}
        }%
    }\\
    \vspace{5mm}
    \subfloat[Right-looking variant]{%
        \label{fig:threechol_r}%
        \resizebox{.75\textwidth}{!}{
            \includegraphics[width=.25\textwidth]{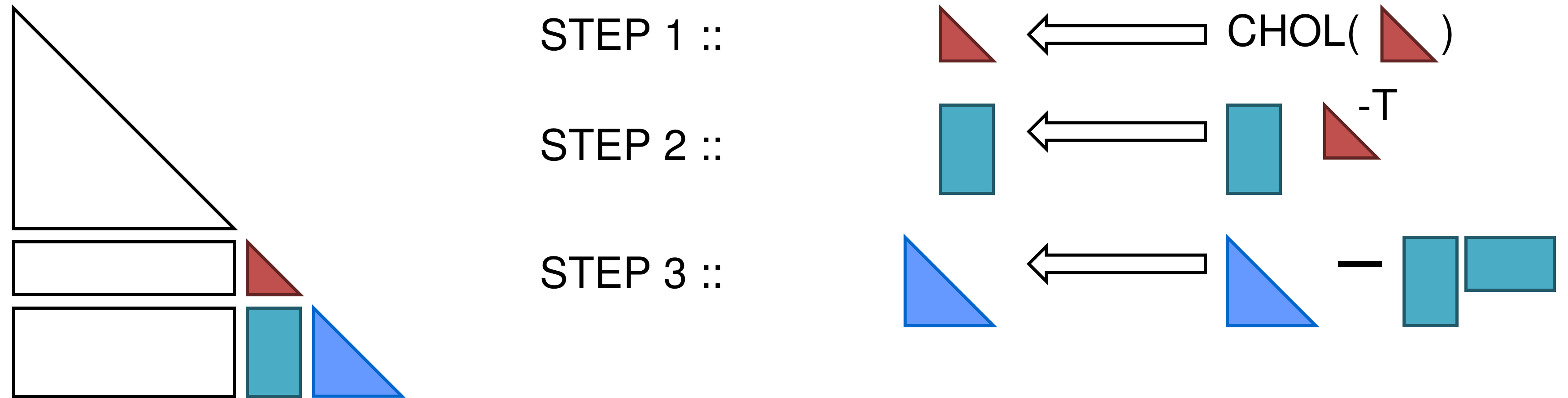}
        }%
    }
    \caption{
    \label{fig:threechol}
    Three variants for the Cholesky decomposition}
\end{figure}

The `bordered' variant, as depicted in Figure~\ref{fig:threechol_b}, involves a
loop over three steps.  The first step updates the purple row block using the already
factorized green portion, the second step updates the next triangular block to be
factorized (in red), and the third step performs the factorization of the triangular
block.  This is then repeated until the entire matrix is factorized.  Note that
the lower portion of the matrix is not touched by the preceding steps.  

The `left-looking' variant (see Figure~\ref{fig:threechol_l}) involves four
steps.  The first step updates the triangular block in red which is then factorized in the
second step, the third step updates the block column (in cyan) below the triangular block using
the previous columns, and the last step updates the column block using the
factorization of step 2.  It is called left-looking since the algorithm does not 
affect the portion to the right of the current block of the matrix and only
`looks' to the left for its updates.  The top most triangular portion of the
matrix is in its final form and will not change in the suceeding steps of the
algorithm.

The `right-looking' variant (see Figure~\ref{fig:threechol_r}) involves three
steps.  It performs the factorization of the red triangular block, updates the block
column (in cyna) and then updates the blue trailing matrix on the right.  This is called
right-looking since it does not require anything from the previous factorized
matrix and pushes its updates to the right part of the matrix.  The entire
matrix to the left of the column block in which the algorithm is currently
working is in its final form.

The advantage of the bordered variant is that it does the least amount of
operations to determine if a matrix is SPD.  The advantage of the right-looking
variant is that it provides the most parallelism.  A major disadvantage of the
left-looking variant is the added fork-join that it must perform between the
steps as compared to the other two variants which will negatively affect its
parallel performance.

The LAPACK scheme of using panels has three distinct disadvantages which limit
its performance.  The first can be seen in the third step of the right-looking
Cholesky decomposition (Figure~\ref{fig:threechol_r}) where potentially a large
symmetric rank $k$ operation is performed; the memory architecture will bound
the performance of the algorithm.  Secondly, there is some impact of the
synchronizations that must be performed between each step.  Third, the idea of
panels does not allow for fine-grained tasks.  We alleviate the last two of these
restrictions through the use of tiled algorithms whereas the first is overcome
through the use of Level 3 BLAS operations within the tiled algorithm.

We approach this via tiling a matrix which means reordering the data of the
matrix into smaller regions of contiguous memory as depicted in
Figure~\ref{fig:tiledlayout}.  By varying the tile size,
this data layout allows us to tune the algorithm such that the data needed for
the kernels is present in the cache of the processor core.  Moreover, we are
able to increase the amount of parallelism and minimize the synchronizations.

\begin{figure}
    \centering
    \includegraphics[width=.75\textwidth]{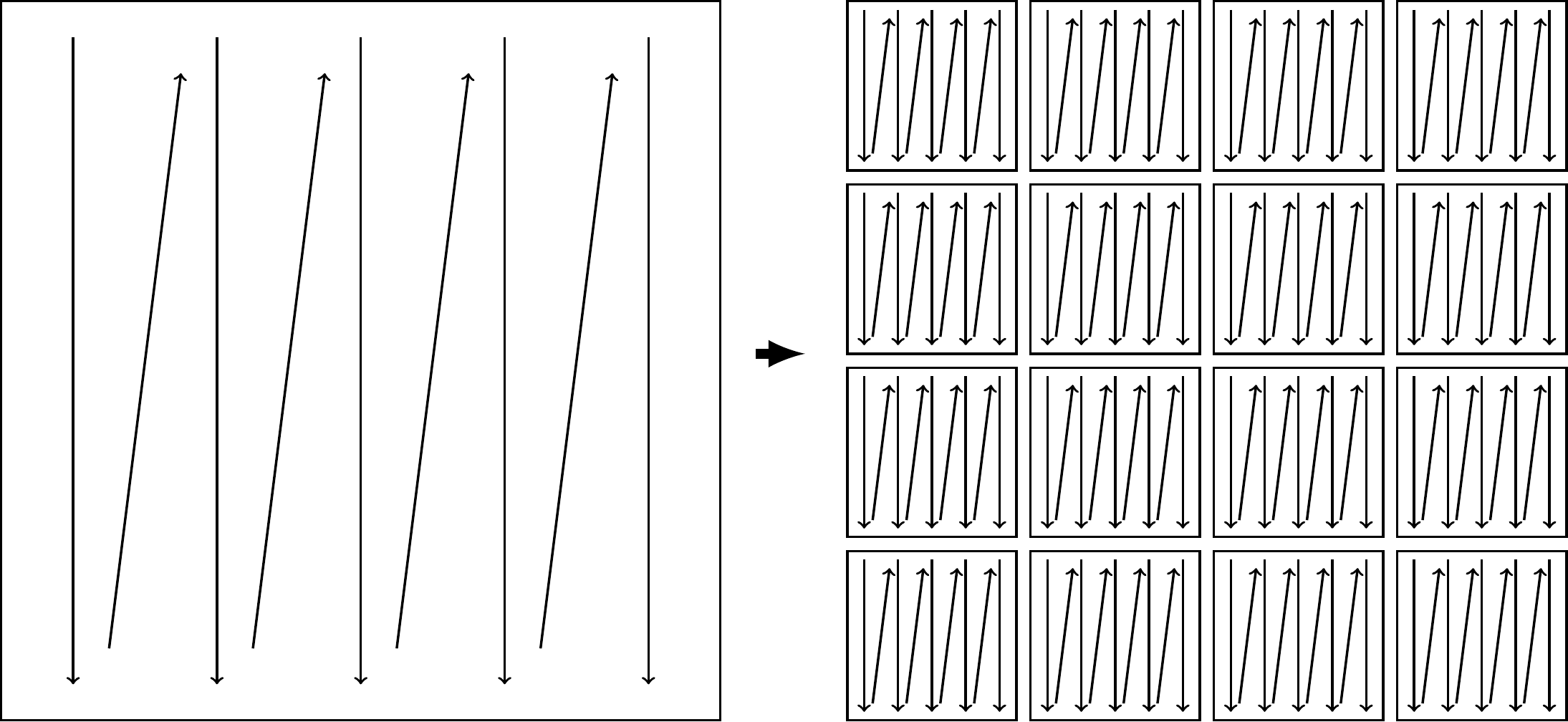}
    \caption{Data layout of tiled matrix}
    \label{fig:tiledlayout}
\end{figure}

Let us revisit the Cholesky decomposition as described earlier and apply each
step to the tiled matrix.  In Figure~\ref{fig:tiledlapack} we present the
directed acyclic graphs (DAG) for the three variants on a tiled matrix of $4
\times 4$ tiles.  In each of the DAGs, the tasks are represented as the nodes
and the data dependencies are the edges.  The dashed horizontal lines designate a
full sweep through all of the steps in each algorithm.  

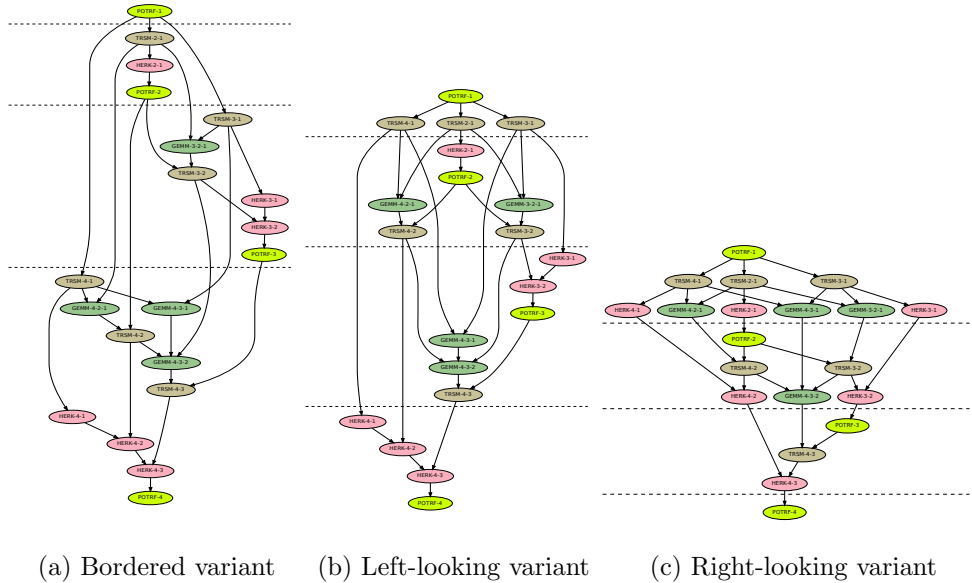
\begin{figure}[htpb]
    \centering
    \pgfmathsetlength{\imagewidth}{10cm} 
    \pgfmathsetlength{\imagescale}{\imagewidth/600} 
    \subfloat[Bordered variant]{%
        \resizebox{0.25\linewidth}{80mm}{%
            \begin{tikzpicture}[x=\imagescale,y=-\imagescale]
                \node[anchor=north west,inner sep=0pt,outer sep=0pt] at (-0.5,-0.5)
                {\includegraphics[trim=0mm 0mm 23mm 0mm, clip, width=\imagewidth]{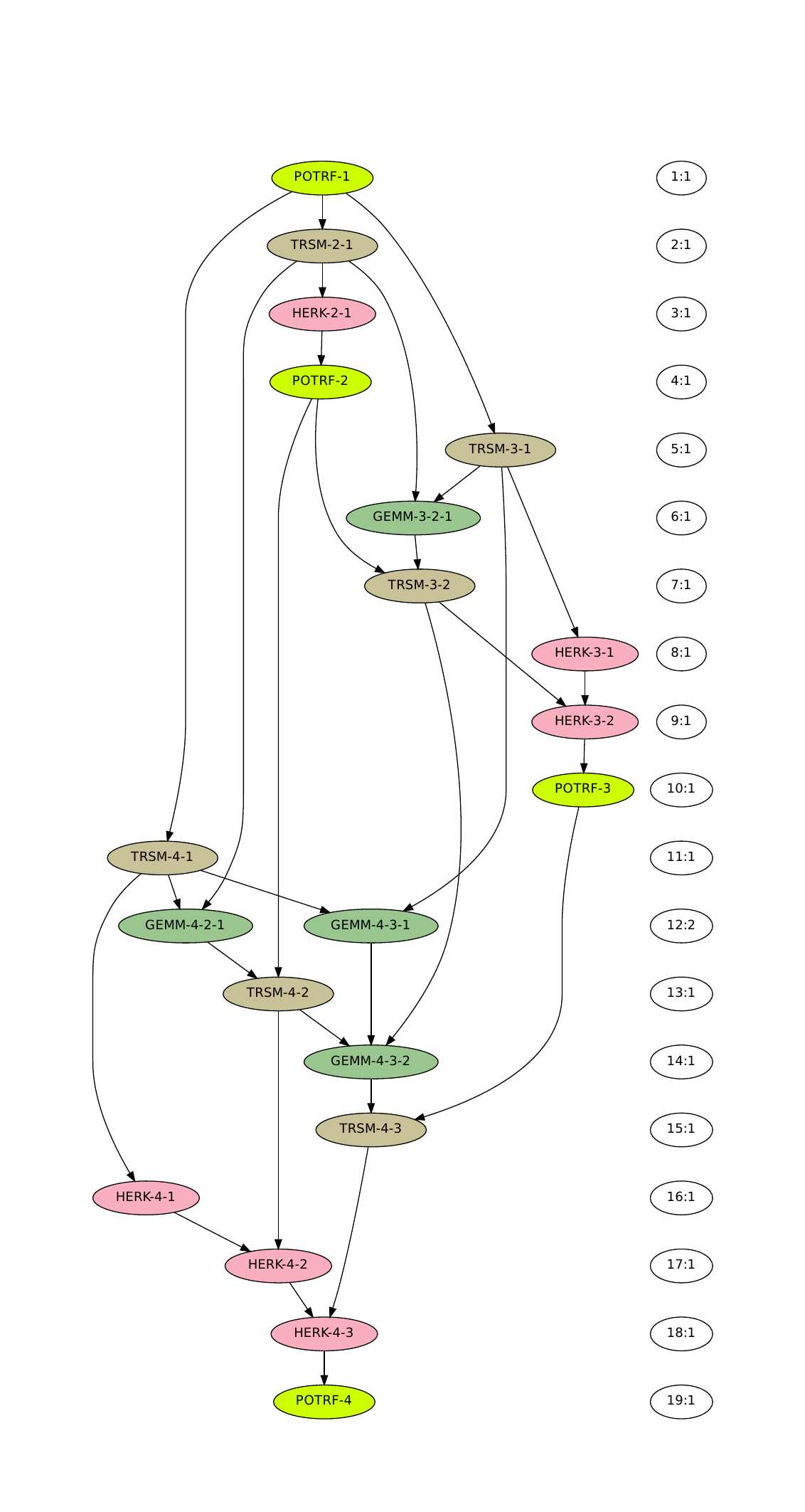}};
                \draw [dashed] (0,195) -- (600,195);
                \draw [dashed] (0,385) -- (600,385);
                \draw [dashed] (0,765) -- (600,765);
            \end{tikzpicture}
        }
    }
    \subfloat[Left-looking variant]{%
        \resizebox{0.25\linewidth}{68mm}{%
            \begin{tikzpicture}[x=\imagescale,y=-\imagescale]
                \node[anchor=north west,inner sep=0pt,outer sep=0pt] at (-0.5,-0.5)
                {\includegraphics[trim=0mm 0mm 25mm 0mm, clip, width=\imagewidth]{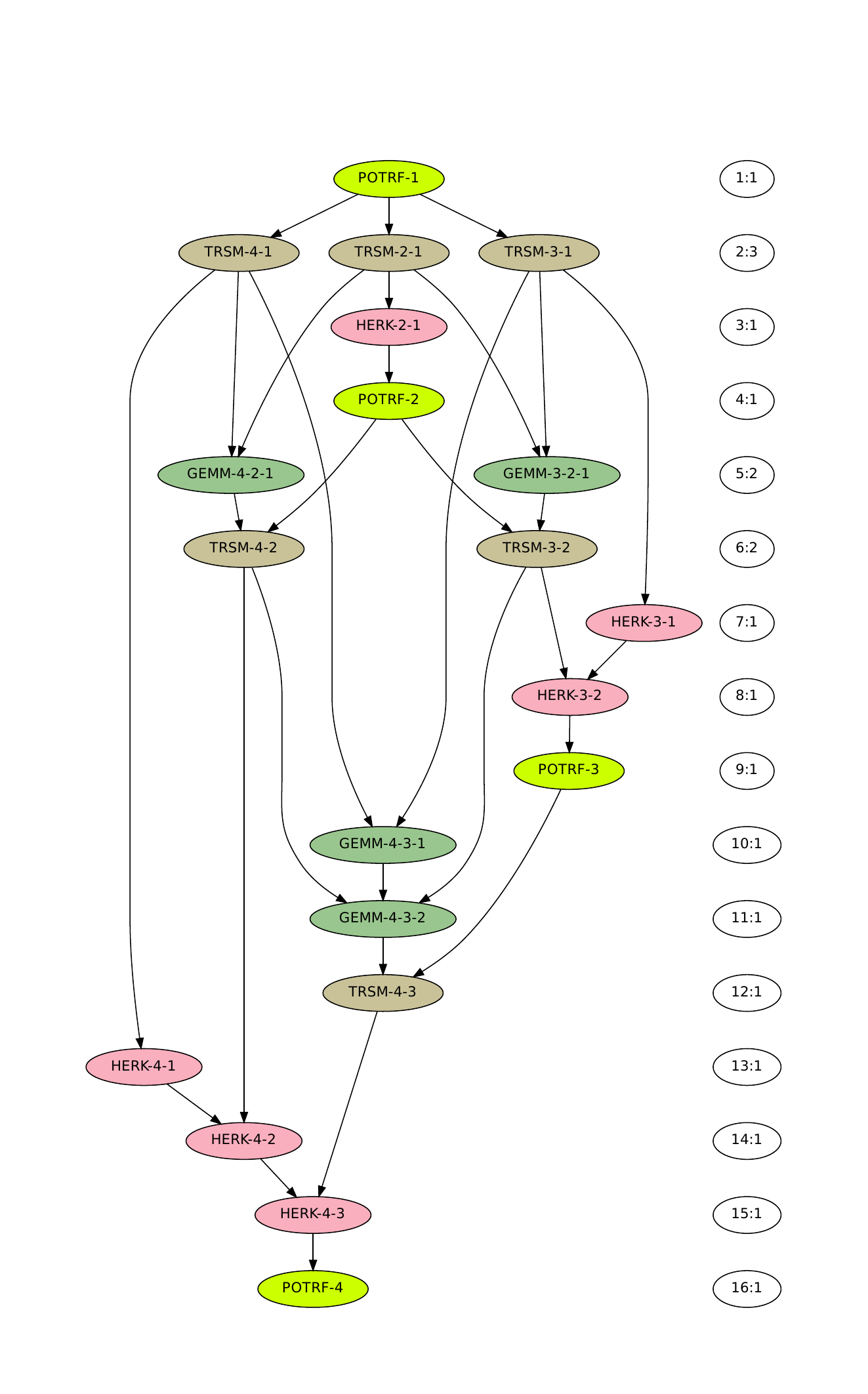}};
                \draw [dashed] (0,245) -- (600,245);
                \draw [dashed] (0,500) -- (600,500);
                \draw [dashed] (0,870) -- (600,870);
            \end{tikzpicture}
        }
    }
    \subfloat[Right-looking variant]{%
        \resizebox{0.325\linewidth}{!}{%
            \begin{tikzpicture}[x=\imagescale,y=-\imagescale]
                \node[anchor=north west,inner sep=0pt,outer sep=0pt] at (-0.5,-0.5)
                {\includegraphics[trim=31mm 0mm 0mm 0mm, clip, width=\imagewidth]{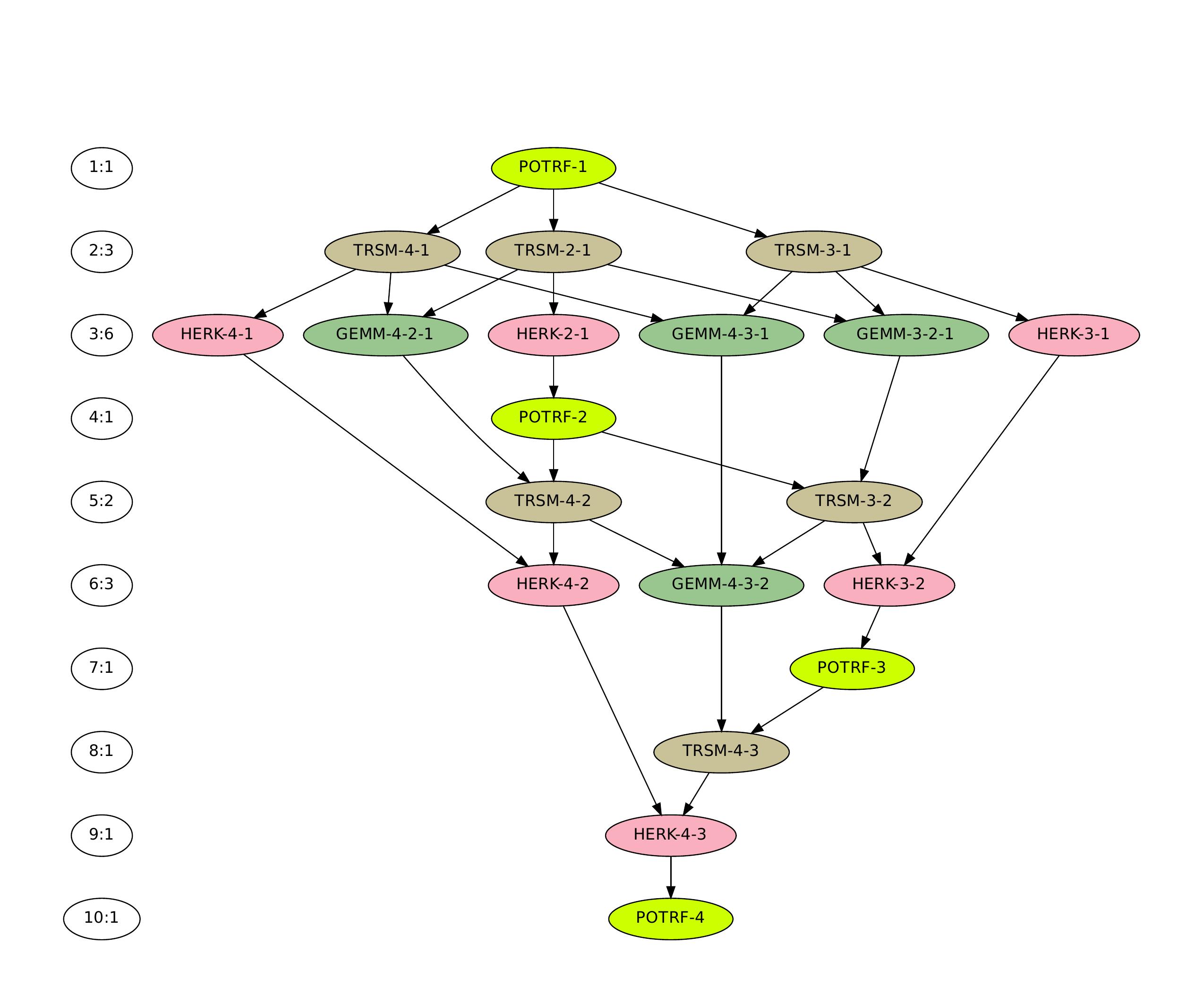}};
                \draw [dashed] (0,210) -- (600,210);
                \draw [dashed] (0,350) -- (600,350);
                \draw [dashed] (0,490) -- (600,490);
            \end{tikzpicture}
        }
    }
    \caption{Three variants of the Cholesky decomposition applied to a tiled
    matrix of $4 \times 4$ tiles.}
    \label{fig:tiledlapack}
\end{figure}

The first observation that one makes is that the height of each DAG varies
according to which variant is chosen.  The bordered variant is the tallest since
the tasks become almost sequential where the only portion that is not sequential
is that of the row block update.  The left-looking variant is almost of height
$t^2$ where $t$ is the number of tiles in a column of the tiled matrix.  It
gains parallelism from being able to update the column block of the final step
within the loop in parallel.  As before, the right-looking variant is the
shortest and provides the most parallelism.

However, in the tiled versions, the synchronizations between each step of the
LAPACK algorithms is superficial and can be removed.  By doing so, these three
variants reduce to only the right-looking variant. 

The main difference between the tiled version and the blocked version is the
amount of parallelism that is gained from updates of the trailing matrix.
Instead of performing a large symmetric rank $k$ update where $k$ is the number
of rows in a row block or the number of columns in a column block, the operation
is decomposed into smaller symmetric rank $n_b$ updates and associated matrix
multiplications where $n_b$ is the size of a tile such that $N = t\cdot n_b$.
In the right-looking variant, for an $N\times N$ matrix the size of the first
trailing matrix is $(N-k) \times (N-k)$ so that the update operation for this
first matrix has a computational cost of $O(kN^2)$.  The tiled update consists
of both symmetric updates and matrix multiplications of tiles of size $n_b
\times n_b$ so that the computational cost per task is $O(n_b^3)$. 

The size of the tiles will determine the granularity of the tasks for a tiled
algorithm.  For a matrix of size $N\times N$, the tile size of the tiled $t
\times t$ matrix can vary from the entire size of the matrix ($t=1$) down to a
scalar ($t=N$), but is held constant throughout the execution of the algorithm.
However, a balance must be kept between the efficiency of the kernel and the
amount of data movement.

Therefore, a tiled algorithm does overcome the restriction on the granularity
imposed by the panels concept of LAPACK as well as alleviate some of the
synchronizations and associated overhead.  The memory bound is still present due
to needing to move the updates of the trailing matrices.  

The Parallel Linear Algebra Software for Multicore Architectures
(PLASMA)~\cite{plasma_users_guide} library provides the framework for the tiled
algorithms.  For the experimental portions, our main assumption will be a shared
memory machine architecture wherein each processor has direct access to all portions of
the memory in which the matrices are stored.  

In Chapter~\ref{chp:cholinv} we describe more fully the implementation of the tiled Cholesky
decomposition and further the tiled Cholesky Inversion algorithm.  It shows that
translating from LAPACK to PLASMA can be straight forward, but that there are
caveats to be taken into account.  In Chapter~\ref{chp:tiledqr} we have
implemented a tiled QR decomposition showing that the implementation is not
straight-forward.  Moreover, results from previous work do not translate
directly to the tiled algorithm.  

Unlike Chapters~\ref{chp:cholinv} and~\ref{chp:tiledqr} where an unbounded number
of processors is assumed, Chapters~\ref{chp:cholfact} and~\ref{chp:qrfact}
restrict the number of processors and provide bounds on the performance of the
algorithms.  We observe the theoretical speed-up and efficiency and provide more
realistic bounds on the performance.

Finally, Chapter~\ref{chp:strassen} presents a study on a tiled implementation
of the Strassen-Winograd algorithm for matrix-matrix multiplication.  Unlike the
other algorithms presented, it is a recursive algorithm which becomes
interesting in the scope of tiled matrices.

\chapter{Cholesky Inversion}\label{chp:cholinv}
%
%

In this chapter, we present joint work with Emmanuel Agullo, Jack Dongarra,
Jakub Kurzak, Julien Langou, and Lee
Rosenberg~\cite{Agullo:2010:TET:1964238.1964254}.

The appropriate direct method to compute the solution of a symmetric positive
definite system of linear equations consists of computing the Cholesky
factorization of that matrix and then solving the underlying triangular
systems. It is not recommended to use the inverse of the matrix in this case.
However some applications need to explicitly form the inverse of the matrix.  A
canonical example is the computation of the variance-covariance matrix in
statistics. Higham~\cite[p.260,\S3]{higham} lists more such applications.

With their advent, multicore architectures~\cite{sutterlunch} induce the need
for algorithms and libraries that fully exploit their capacities.  A class of
such algorithms -- called tile algorithms~\cite{Buttari2008,tileplasma} -- has
been developed for one-sided dense factorizations (Cholesky, LU and QR) and made
available as part of the Parallel Linear Algebra Software for Multicore
Architectures (PLASMA) library~\cite{plasma_users_guide}. In this chapter, we
extend this class of algorithms to the case of the (symmetric positive definite)
matrix inversion. Besides constituting an important functionality for a library
such as PLASMA, the study of the matrix inversion on multicore architectures
represents a challenging algorithmic problem. Indeed, first, contrary to
standalone one-sided factorizations that have been studied so far, the matrix
inversion exhibits many anti-dependences~\cite{compil-AK} (Write after Read).
This is a false or artificial dependency which is reliant on the name of the
data and not the actual data flow. For example, given two operations where the
first only reads the data in the matrix $A$ and the second only writes to the
location of $A$, then in a parallel execution there may be a case where the data
being read by the first operation is wrong since the second may have already
written to the location.  By copying the data from $A$ beforehand, both
operations can be executed in parallel.  Those anti-dependences can be a
bottleneck for parallel processing, which is critical on multicore
architectures. It is thus essential to investigate (and adapt) well known
techniques used in compilation such as using temporary copies of data to remove
anti-dependences to enhance the degree of parallelism of the matrix inversion.
This technique is known as \emph{array renaming}~\cite{compil-AK} (or
\emph{array privatization}~\cite{compil-privatization}).  Second, \emph{loop
reversal}~\cite{compil-AK} is to be investigated. Third, the matrix inversion
consists of three successive steps (first of which is the Cholesky
decomposition). In terms of scheduling, it thus represents an opportunity to
study the effects of \emph{pipelining}~\cite{compil-AK} those steps on
performance.

The current version of PLASMA (version 2.1) is scheduled statically. Initially
developed for the IBM Cell processor~\cite{qr-static-scheduling-cell}, this
static scheduling relies on POSIX threads and simple synchronization
mechanisms. It has been designed to maximize data reuse and load balancing
between cores, allowing for very high performance~\cite{plasmaperf} on today's
multicore architectures. However, in the case of matrix inversion, the
design of an ad-hoc static scheduling is a time consuming task and raises load
balancing issues that are much more difficult to address than for a stand-alone
Cholesky decomposition, in particular when dealing with the pipelining of
multiple steps. Furthermore, the growth of the number of cores and the more
complex memory hierarchies make executions less deterministic.  In this chapter,
we rely on an experimental in-house dynamic scheduler~\cite{gust}.  This
scheduler is based on the idea of expressing an algorithm through its
sequential representation and unfolding it at runtime using data hazards (Read
after Write, Write after Read, Write after Write) as constraints for parallel
scheduling.  The concept is rather old and has been validated by a few
successful projects.  We could have as well used schedulers from the Jade
project from Stanford University \cite{jade_1993_computer} or from the SMPSs
project from the Barcelona Supercomputer Center \cite{cellss_2007_ibm_jrd}.

Our discussions are illustrated with experiments conducted on a dual-socket
quad-core machine based on an Intel Xeon EMT64 processor operating at $2.26$
GHz. The theoretical peak is equal to $9.0$ Gflop/s per core or $72.3$ Gflop/s
for the whole machine, composed of 8 cores.  The machine is running Mac OS X
10.6.2 and is shipped with the Apple vecLib v126.0 multithreaded
BLAS~\cite{blasurl} and LAPACK vendor library, as well as
LAPACK~\cite{LAPACK_1999_guide} v3.2.1 and
ScaLAPACK~\cite{ScaLAPACK_1997_guide} v1.8.0 references.


\section{Tile in-place matrix inversion}
\label{sec:inplace}

Tile algorithms are a class of Linear Algebra algorithms that allow
for fine granularity parallelism and asynchronous scheduling, enabling
high performance on multicore architectures~\cite{plasmaperf,Buttari2008,tileplasma,Quintana:2009}. The
matrix of order $n$ is split into $t\times t$ square submatrices of
order $b$ ($n=b\times t$). Such a submatrix is of small granularity
(we fixed $b = 200$ in this chapter) and is called a \emph{tile}. So
far, tile algorithms have been essentially used to implement one-sided
factorizations~\cite{plasmaperf,Buttari2008,tileplasma,Quintana:2009}.

\linespread{1.2}
\begin{algorithm}
  \KwIn{$A$, Symmetric Positive Definite matrix in tile storage ($t\times t$ tiles).}
  \KwResult{$A^{-1}$, stored in-place in $A$.}
  \emph{Step~1: Tile Cholesky Factorization (compute L such that $A=LL^T$)}\;
  \For{$j=0$ \KwTo $t-1$}{
    \For{$k=0$ \KwTo $j-1$}{
      $A_{j,j} \leftarrow A_{j,j} - A_{j,k} \ast A_{j,k}^T$ (SYRK(j,k)) \;
    }
    $A_{j,j} \leftarrow CHOL(A_{j,j})$ (POTRF(j)) \;
    \For{$i=j+1$ \KwTo $t-1$}{
      \For{$k=0$ \KwTo $j-1$}{
        $A_{i,j} \leftarrow A_{i,j} - A_{i,k} \ast A_{j,k}^T$ (GEMM(i,j,k)) \;
      }
    }
    \For{$i=j+1$ \KwTo $t-1$}{
      $A_{i,j} \leftarrow A_{i,j} / A_{j,j}^T$ (TRSM(i,j)) \;
    }
  }
  \emph{Step~2: Tile Triangular Inversion of $L$ (compute $L^{-1}$)}\;
  \For{$j=t-1$ \KwTo $0$}{
    $A_{j,j} \leftarrow TRINV(A_{j,j})$ (TRTRI(j)) \;
    \For{$i=t-1$ \KwTo $j+1$}{
      $A_{i,j} \leftarrow A_{i,i} \ast A_{i,j}$ (TRMM(i,j)) \;
      \For{$k=j+1$ \KwTo $i-1$}{
        $A_{i,j} \leftarrow A_{i,j} + A_{i,k} \ast A_{k,j} $ (GEMM(i,j,k)) \;
      }
      $A_{i,j} \leftarrow - A_{i,j} \ast A_{i,i}$ (TRMM(i,j)) \;
    }
  }
  \emph{Step~3: Tile Product of Lower Triangular Matrices (compute $A^{-1}={L^{-1}}^TL^{-1}$)}\;
  \For{$i=0$ \KwTo $t-1$}{
    \For{$j=0$ \KwTo $i-1$}{
      $A_{i,j} \leftarrow A_{i,i}^T \ast A_{i,j}$ (TRMM(i,j)) \;
    }
    $A_{i,i} \leftarrow A_{i,i}^T \ast A_{i,i}$ (LAUUM(i)) \;
    \For{$j=0$ \KwTo $i-1$}{
      \For{$k=i+1$ \KwTo $t-1$}{
        $A_{i,j} \leftarrow A_{i,j} + A_{k,i}^T \ast A_{k,j}$ (GEMM(i,j,k)) \;
      }
    }
    \For{$k=i+1$ \KwTo $t-1$}{
      $A_{i,i} \leftarrow A_{i,i} + A_{k,i}^T \ast A_{k,i}$ (SYRK(i,k)) \;
    }
  }
  \caption{Tile In-place Cholesky Inversion (lower format). Matrix
    $A$ is the on-going updated matrix (in-place algorithm).}
  \label{alg:InPlace}
\end{algorithm}
\renewcommand{\baselinestretch}{\normalspace}

Algorithm~\ref{alg:InPlace} extends this class of algorithms to the
case of the matrix inversion. As in state-of-the-art libraries
(LAPACK, ScaLAPACK), the matrix inversion is performed \emph{in-place},
\emph{i.e.}, the data structure initially containing matrix~$A$ is
directly updated as the algorithm is progressing, without using any
significant temporary extra-storage; eventually, $A^{-1}$ replaces
$A$. Algorithm~\ref{alg:InPlace} is composed of three steps. Step~1 is
a Tile Cholesky Factorization computing the Cholesky factor $L$ (lower
triangular matrix satisfying $A=LL^T$). This step was studied
in~\cite{tileplasma}. Step~2 computes $L^{-1}$ by inverting
$L$. Step~3 finally computes the inverse matrix
$A^{-1}={L^{-1}}^TL^{-1}$. Each step is composed of multiple fine
granularity tasks (since operating on tiles). These tasks are part of
the BLAS (SYRK, GEMM, TRSM, TRMM) and LAPACK (POTRF, TRTRI, LAUUM)
standards. A more detailed description is beyond the scope of this
extended chapter and is not essential to the understanding of the
rest of the chapter. Indeed, from a high level point of view, an
operation based on tile algorithms can be represented as a Directed
Acyclic Graphs (DAG)~\cite{graph} where nodes represent the fine
granularity tasks in which the operation can be decomposed and the
edges represent the dependences among them. For instance,
Figure~~\ref{fig:dag-inplace} represents the DAG of Step~3 of
Algorithm~\ref{alg:InPlace}. 
\linespread{1.0}
\begin{figure}[htbp]
   \centering
    \subfloat[In-place (Algorithm~\ref{alg:InPlace})]{
      \label{fig:dag-inplace}
      \hspace{0mm}\resizebox{.25\textwidth}{!}{\includegraphics[trim=0mm 0cm 0mm 0cm, clip=true]{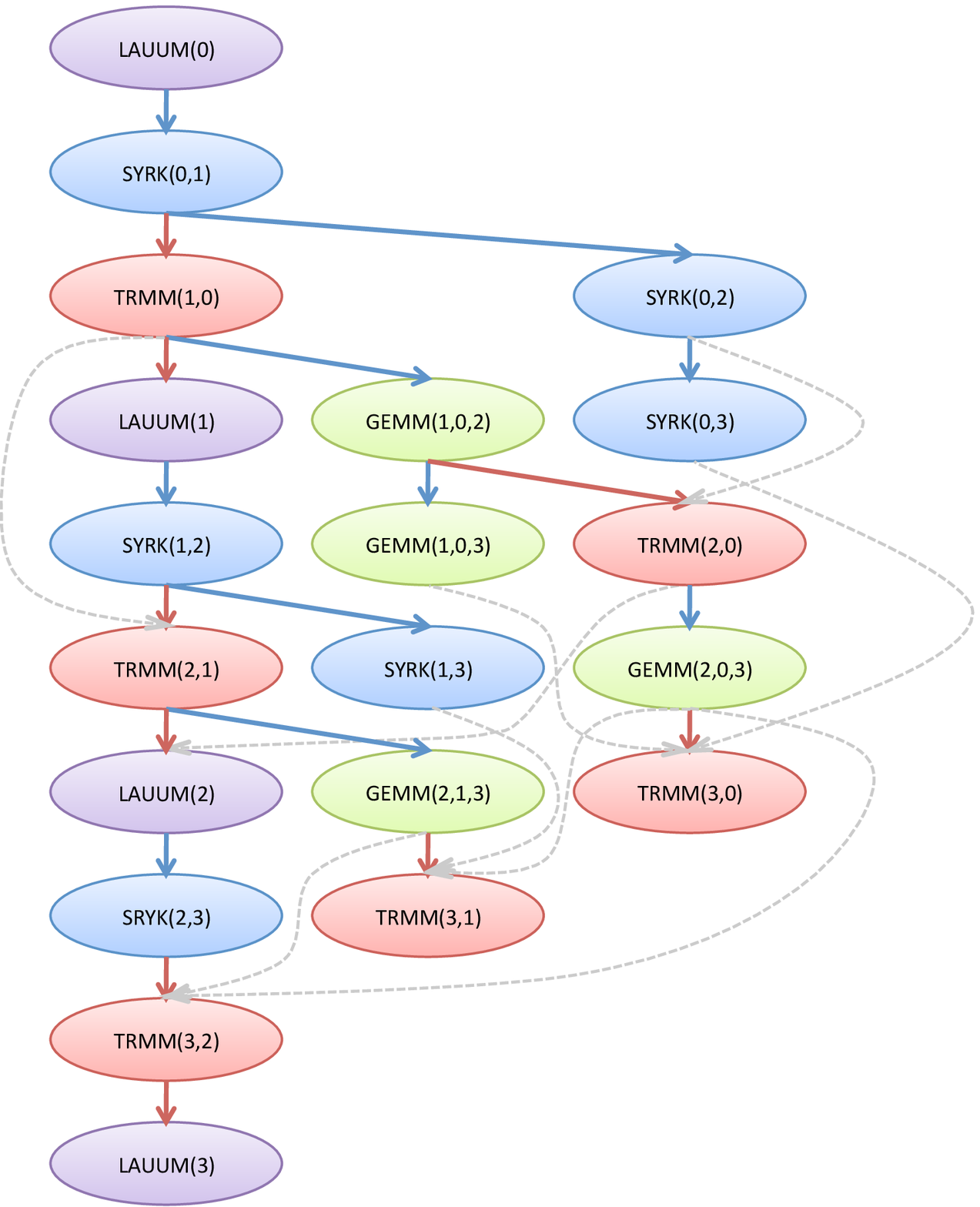}}
      \hspace{4mm}
    }
    \subfloat[Out-of-place (variant introduced in \Section~\ref{sec:study})]{
      \label{fig:dag-outofplace}
      \hspace{0mm}\resizebox{.6\textwidth}{!}{\includegraphics[trim=0mm 0cm 0mm 0cm, clip=true]{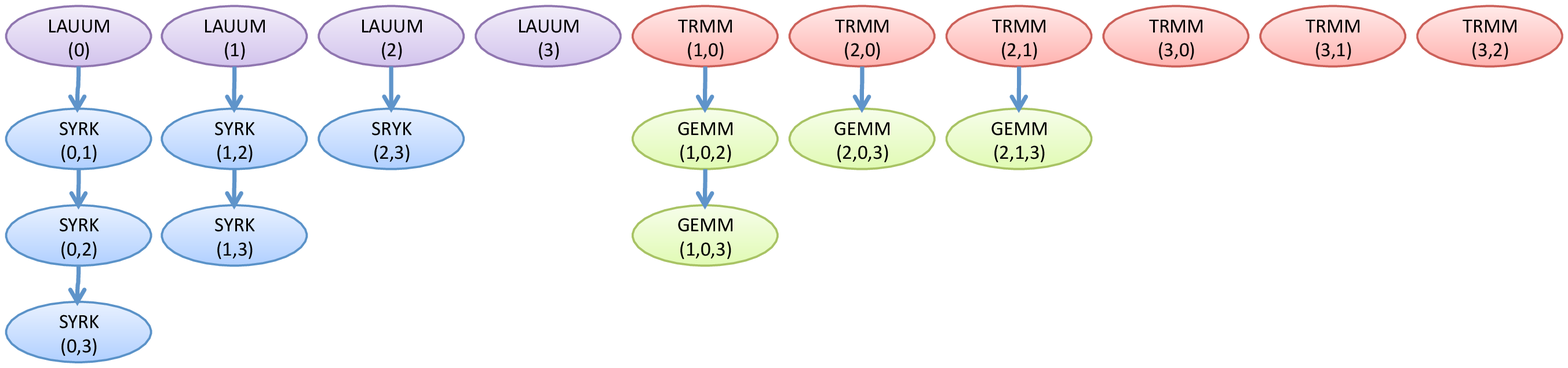}}
    }
  \caption{DAGs of Step 3 of the Tile Cholesky Inversion ($t=4$).}
  \label{fig:dags}
\end{figure}
\renewcommand{\baselinestretch}{\normalspace}


Algorithm~\ref{alg:InPlace} is based on the variants used in LAPACK 3.2.1.
Bientinesi, Gunter and van de Geijn~\cite{BientinesiGunterVanDeGeijn:08}
discuss the merits of algorithmic variations in the case of the computation
of the inverse of a symmetric positive definite matrix.  Although of definite
interest, this is not the focus of this extended chapter.

We have implemented Algorithm~\ref{alg:InPlace} using our dynamic
scheduler introduced in the beginning of the chapter.
Figure~\ref{fig:perf} shows its performance against state-of-the-art
libraries and the vendor library on the machine described in
the beginning of the chapter. For a matrix of small size, it is
difficult to extract parallelism and have a full use of all the
cores~\cite{plasmaperf,Buttari2008,tileplasma,Quintana:2009}. We indeed observe a limited scalability
($N=1000$, Figure~\ref{fig:perf-1000}). However, tile algorithms
(Algorithm~\ref{alg:InPlace}) still benefit from a higher degree of
parallelism than blocked algorithms~\cite{plasmaperf,Buttari2008,tileplasma,Quintana:2009}. Therefore
Algorithm~\ref{alg:InPlace} (in place) consistently achieves a
significantly better performance than vecLib, ScaLAPACK and LAPACK libraries.
A larger matrix size ($N=4000$, Figure~\ref{fig:perf-4000}) allows for
a better use of parallelism. In this case, an optimized implementation
of a blocked algorithm (vecLib) competes well against tile algorithms
(in place) on few cores (left part of
Figure~\ref{fig:perf-1000}). However, only tile algorithms scale to a
larger number of cores (rightmost part of Figure~\ref{fig:perf-4000})
thanks to a higher degree of parallelism. In other words, the tile
Cholesky inversion achieves a better \emph{strong scalability} than the
blocked versions, similarly to what had been observed for the
factorization step~\cite{plasmaperf,Buttari2008,tileplasma,Quintana:2009}.
\linespread{1.0}
\begin{figure}[htbp]
   \centering
   \subfloat[$n=1000$]{
      \label{fig:perf-1000}
      \hspace{0mm}\resizebox{.475\textwidth}{!}{\includegraphics[trim=4mm 0cm 4mm 0cm, clip=true]{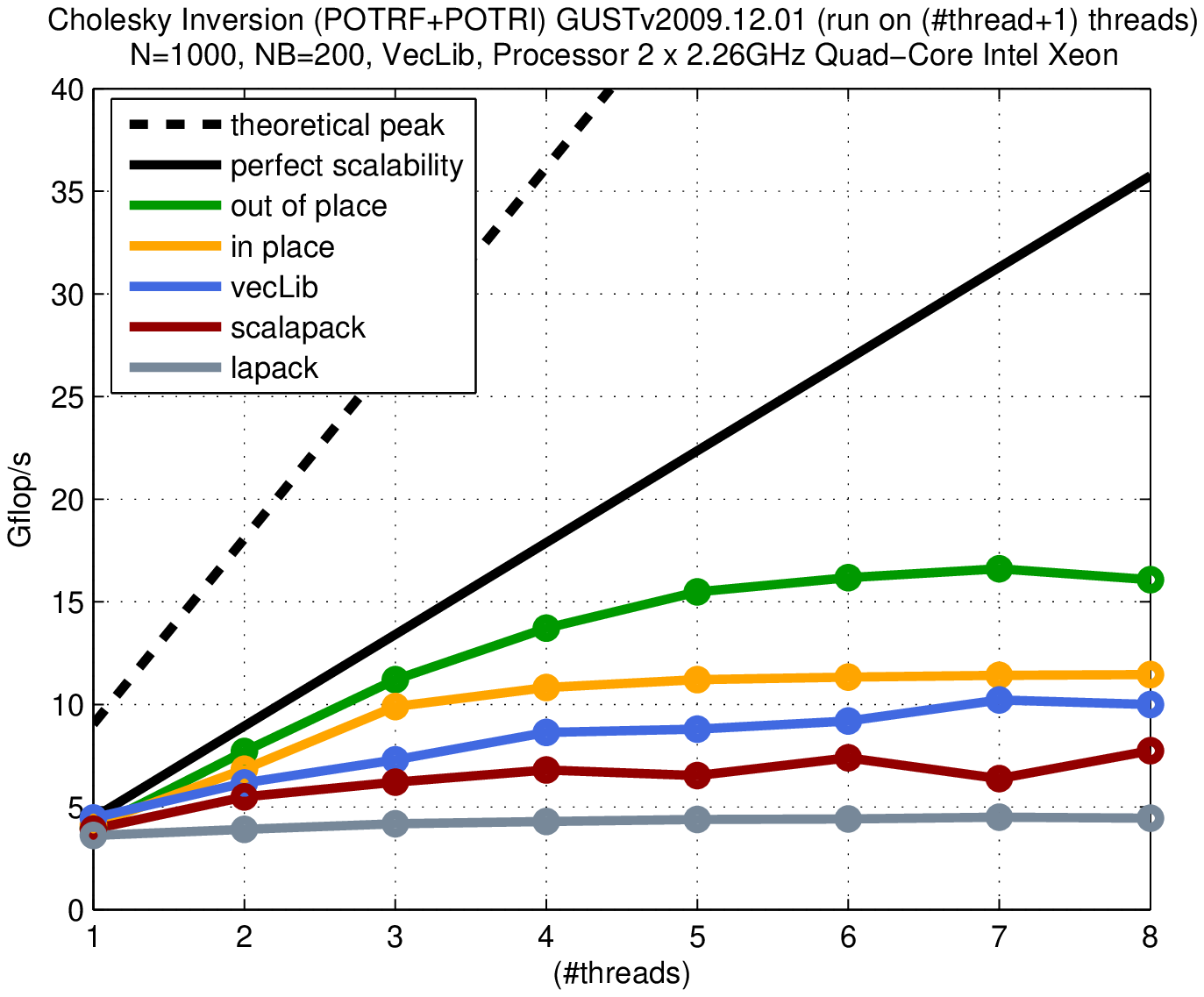}}
    }
    \subfloat[$n = 4000$]{
      \label{fig:perf-4000}
      \hspace{0mm}\resizebox{.475\textwidth}{!}{\includegraphics[trim=4mm 0cm 4mm 0cm, clip=true]{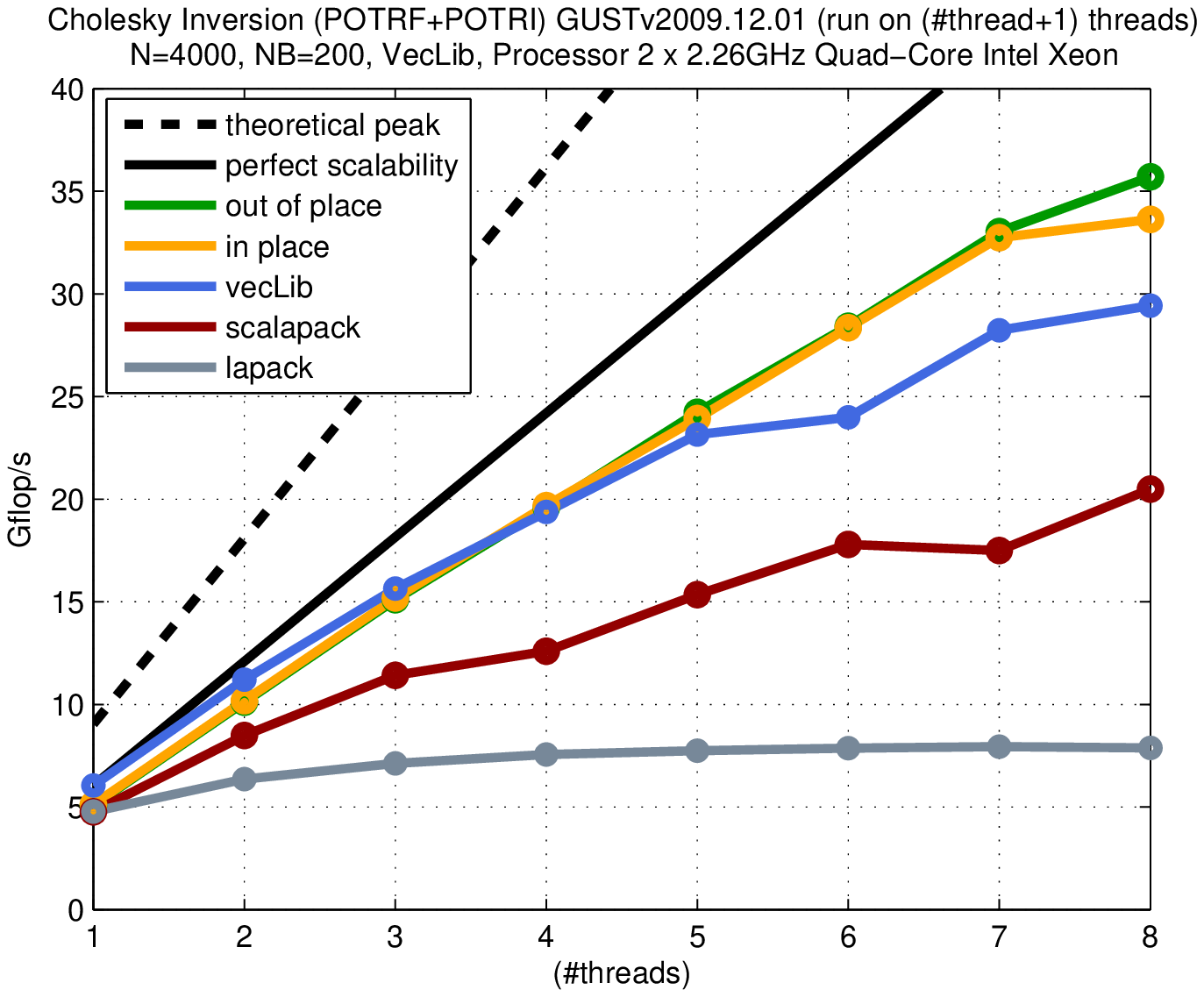}}
    }
  \caption{Scalability of Algorithm~\ref{alg:InPlace} (in place) and
    its out-of-place variant introduced in \Section~\ref{sec:study},
    using our dynamic scheduler against vecLib, ScaLAPACK and
    LAPACK libraries.}
  \label{fig:perf}
\end{figure}
\renewcommand{\baselinestretch}{\normalspace}

%

\section{Algorithmic study}
\label{sec:study}

In the \Section~\ref{sec:inplace}, we compared the performance of the tile
Cholesky inversion against state-the-art libraries. In this section,
we focus on tile Cholesky inversion and we discuss the impact of several
variants of Algorithm~\ref{alg:InPlace} on performance.


\linespread{1.0}
\begin{table}[tb]
\centering
\begin{tabular}{@{}rcc@{}}
   \toprule
           & In-place case & Out-of-place case\\%
   \cmidrule(lr){2-2}
   \cmidrule(l ){3-3}
      Step 1   & $3t-2$       & $3t-2$          \\%
      Step 2   & $3t-3$       & $2t-1$       \\%
      Step 3   & $3t-2$       & $t$ \\%
   \bottomrule
\end{tabular}
\caption{Length of the critical path as a function of the number of
  tiles $t$.}
\label{table:criticalpath}
\end{table}
\renewcommand{\baselinestretch}{\normalspace}

{\bf Array renaming (removing anti-dependences).} The dependence
between SYRK(0,1) and TRMM(1,0) in the DAG of Step~3 of
Algorithm~\ref{alg:InPlace} (Figure~\ref{fig:dag-inplace}) represents
the constraint that the SYRK operation (l.~28 of
Algorithm~\ref{alg:InPlace}) needs to read $A_{k,i}=A_{1,0}$ before
TRMM (l.~22) can overwrite $A_{i,j}=A_{1,0}$. This anti-dependence
(Write after Read) can be removed thanks to a temporary copy of $A_{1,0}$.
Similarly, all the SYRK-TRMM anti-dependences, as well as TRMM-LAUMM
and GEMM-TRMM anti-dependences can be removed. 
We have designed a variant of Algorithm~\ref{alg:InPlace}
that removes all the anti-dependences thanks to the use of a large working
array (this technique is called \emph{array renaming}~\cite{compil-AK} in 
compilation~\cite{compil-AK}). 
The subsequent DAG (Figure~\ref{fig:dag-outofplace})
is split in multiple pieces (Figure~\ref{fig:dag-outofplace}), leading
to a shorter critical path (Table~\ref{table:criticalpath}).
We implemented the out-of-place algorithm, based on our dynamic scheduler
too. Figure~\ref{fig:perf-1000} shows that our dynamic scheduler exploits its higher
degree of parallelism to achieve a much higher strong scalability even
on small matrices ($N=1000$). For a larger matrix
(Figure~\ref{fig:perf-4000}), the in-place algorithm already achieved
very good scalability. Therefore, using up to $7$ cores, their performance
are similar. However, there is not enough parallelism with a
$4000\times 4000$ matrix to use efficiently all $8$ cores with the
in-place algorithm; thus the higher performance of the out-of-place
version in this case (leftmost part of Figure~\ref{fig:perf-4000}).

\ignore{TODO:8 performance graphs (4*2): 
  (step 1, step 2, step 3, pipeline) * (inplace , outofplace); 
  we can arrange in two plots: 
  - 1 plot for 6 graphs (separate steps), 
  - 1 plot for 2 graphs (pipelines)}


{\bf Loop reversal (exploiting commutativity).} The most internal loop
of each step of Algorithm~\ref{alg:InPlace} (l.~8, l.~17 and l.~26)
consists in successive commutative GEMM operations. Therefore they can
be performed in any order, among which increasing order and decreasing
order of the loop index. Their ordering impacts the length of the
critical path. Algorithm~\ref{alg:InPlace} orders those three loops in
increasing (U)
and decreasing (D)
order, respectively. We had manually chosen these respective 
orders (UDU) because they minimize 
the critical path of each step (values reported in
Table~\ref{table:criticalpath}). A
naive approach would have, for example, been comprised of consistently ordering the loops
in increasing order (UUU). In this case (UUU), the critical path of TRTRI
would have been equal to $t^2-2t+3$ (in-place) or
($\frac{1}{2}t^2-\frac{1}{2}t+2$) (out-of-place) instead of $3t-3$
(in-place) or $2t-1$ (out-of-place) for (UDU). Figure~\ref{fig:perf-loops} shows
how loop reversal impacts performance.
\linespread{1.0}
\begin{figure}[htbp]
   \centering
   \subfloat[$n=1000$]{
      \label{fig:perf-loops-1000}
      \hspace{0mm}\resizebox{.475\textwidth}{!}{\includegraphics[trim=4mm 0cm 4mm 0cm, clip=true]{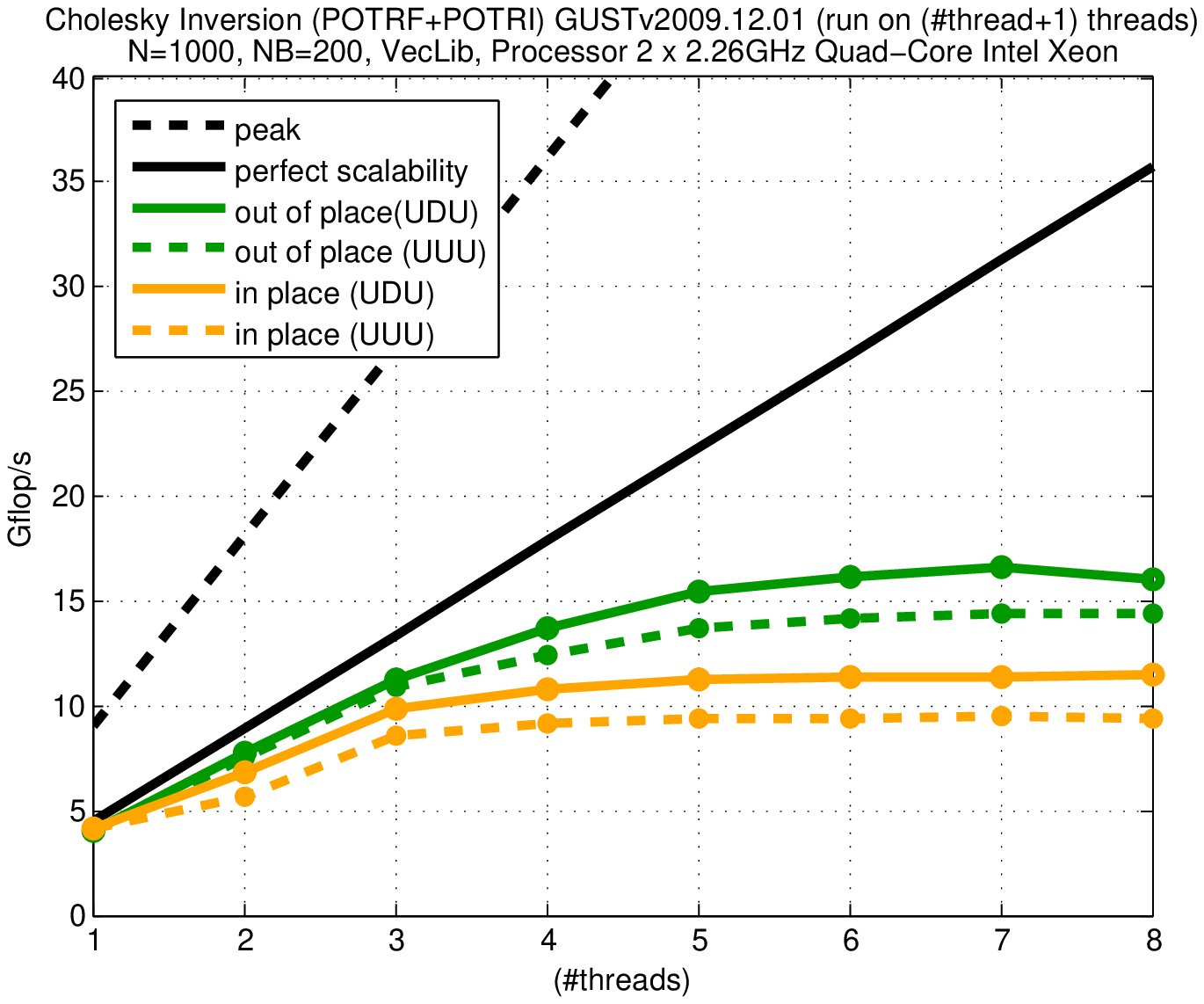}}
    }
    \subfloat[$n = 4000$]{
      \label{fig:perf-loop-4000}
      \hspace{0mm}\resizebox{.475\textwidth}{!}{\includegraphics[trim=4mm 0cm 4mm 0cm, clip=true]{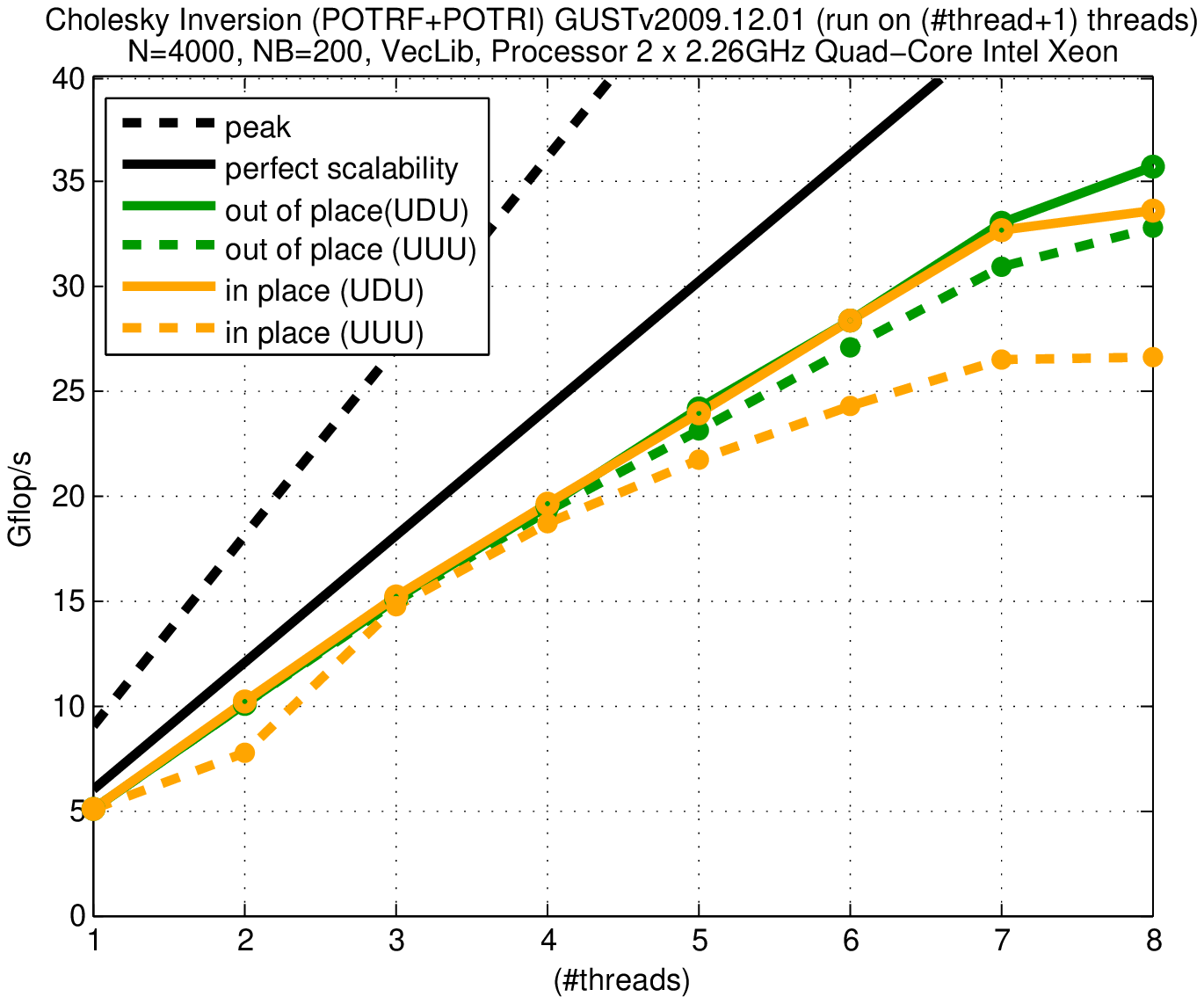}}
    }
  \caption{Impact of loop reversal on performance.}
  \label{fig:perf-loops}
\end{figure}
\renewcommand{\baselinestretch}{\normalspace}



{\bf Pipelining.} Pipelining the multiple steps of the inversion
reduces the length of its critical path. For the in-place case, the
critical path is reduced from $9t-7$ tasks ($t$ is the number of
tiles) to $9t-9$ tasks (negligible). For the out-of-place case, it is reduced from
$6t-3$ to $5t-2$ tasks. We studied the effect of pipelining on the
performance of the inversion on a $8000\times 8000$ matrix with an
artificially large tile size ($b=2000$ and $t=4$). As expected, we observed almost
no effect on performance of the in-place case (about $36.4$ seconds 
with or without pipelining).
For the out-of-place case, the elapsed time grows from
$25.1$ to $29.2$ seconds (16\% overhead) when
pipelining is prevented.

\section{Conclusion and future work}
\label{sec:conclusion}

We have proposed a new algorithm to compute the inverse of a symmetric
positive definite matrix on multicore architectures. An experimental
study has shown both an excellent scalability of our algorithm and a
significant performance improvement compared to state-of-the-art
libraries. Beyond extending the class of so-called tile algorithms,
this study brought back to the fore well known issues in the domain of
compilation. Indeed, we have shown the importance of loop reversal,
array renaming and pipelining.

The use of a dynamic scheduler allowed an out-of-the-box pipeline of
the different steps whereas loop reversal and array renaming required
a manual change to the algorithm. The future work directions consist
in enabling the scheduler to perform itself loop reversal and array
renaming. We exploited the commutativity of GEMM operations to perform
array renaming. Their associativity would furthermore allow to process them
in parallel (following a binary tree); the subsequent impact on
performance is to be studied. Array renaming requires extra-memory.
It will be interesting to address the problem of the maximization of
performance under memory constraint. This work aims to be incorporated
into PLASMA.

\chapter{QR Factorization}\label{chp:tiledqr}
%
%


In this chapter we present joint work with Mathias Jacquelin, Julien Langou, and
Yves Robert~\cite{rr-inria-7601}.

Given an $m$-by-$n$ matrix $A$ with $n \leq m$, we consider the computation of
its QR factorization, which is the factorization $A = QR$, where $Q$ is an
$m$-by-$n$ unitary matrix ($Q^HQ = I_n$), and $R$ is upper triangular.

The QR factorization of an $m$-by-$n$ matrix with $n \leq m$ is the time
consuming stage of some important numerical computations.  It is needed for
solving a linear least squares problem with $m$ equations (observations) and $n$
unknowns and is used to compute an orthonormal basis (the $Q$-factor) of the
column span of the initial matrix $A$.  For example, all block iterative
methods (used to solve large sparse linear systems of equations or computing
some relevant eigenvalues of such systems) require orthogonalizing a set of
vectors at each step of the process. For these two usage examples, while $n\leq
m$, $n$ can range from $n \ll m$ to $ n = m$. We note that the
extreme case $n=m$ is also relevant:  the QR factorization of a matrix can be
used to solve (square) linear systems of equations. While this requires twice
as many flops as an LU factorization, using a QR factorization (a) is
unconditionally stable (Gaussian elimination with partial pivoting or pairwise
pivoting is not) and (b) avoids pivoting so it may well be faster in some
cases.

To obtain a QR factorization, we consider algorithms which apply a sequence of
$m$-by-$m$ unitary transformations, $U_i$, ($U_i^HU_i=I$,), $i=1,\dots,\ell$,
on the left of the matrix $A$, such that after $\ell$ transformations the
resulting matrix $ R = U_\ell \ldots U_1 A $ is upper triangular, in which
case, $R$ is indeed the $R$-factor of the QR factorization.  The $Q$-factor (if
needed) can then be obtained by computing $ Q = U_1^H \ldots U_\ell^H $.  These
types of algorithms are in regular use, e.g., in the LAPACK and ScaLAPACK
libraries, and are favored over others algorithms (Cholesky QR or Gram-Schmidt)
for their stability.

The unitary transformation $U_i$ is chosen so as to introduce some zeros in the
current update matrix $ U_{i-1} \ldots U_1 A $. The two basic transformations
are Givens rotations and Householder reflections.  One Givens rotation
introduces one additional zero; the whole triangularization requires $mn -
n(n+1)/2$ Givens rotations for $n<m$. One elementary Householder reflection
simultaneously introduces $m-i$ zeros in position $i+1$ to $m$ in column $i$;
the whole triangularization requires $n$ Householder reflections for $n<m$.
(See LAPACK subroutine \GEQRTWO.) The LAPACK \GEQRT subroutine constructs a
compact WY representation to apply a sequence of $i_b$ Householder reflections,
this enables one to introduce the appropriate zeros in $i_b$ consecutive columns
and thus leverage optimized Level 3 BLAS subroutines during the update. The
blocking of Givens rotations is also possible but is more costly in terms of
flops.

The main interest of Givens rotations over Householder transformations is that
one can concurrently introduce zeros using disjoint pairs of rows, in other
words, two transformations $U_i$ and $U_{i+1}$ may be applicable concurrently.
This is not possible using the original Householder reflection algorithm since
the transformations work on whole columns and thus do not exhibit this type
of intrinsic parallelism, forcing this kind of Householder reflections to be
applied sequentially. The advantages of Householder reflections over Givens
rotations are that, first, Householder reflections perform fewer flops, and second,
the compact WY transformation enables high sequential performance of the
algorithm. In a multicore setting, where data locality and parallelism are
crucial algorithmic characteristics for enabling performance, the tiled QR
factorization algorithm combines both ideas: use of Householder reflections for
high sequential performance and use of a scheme such as Givens rotations to enable
parallelism within cores. In essence, one can think either (i) of the tiled QR
factorization as a Givens rotation scheme but on tiles ($m_b$-by-$n_b$
submatrices) instead of on scalars ($1$-by-$1$ submatrices) as in the original
scheme, or (ii) of it as a blocked Householder reflection scheme
where each reflection is confined to an extent much less than the full column
span, which enables concurrency with other reflections.

Tiled QR factorization in the context of multicore architectures has been
introduced in~\cite{Buttari2008,tileplasma,Quintana:2009}. Initially the focus
was on square matrices and the sequence of unitary transformations presented
was analogous to \SK~\cite{SamehKuck78}, which corresponds to reducing the
panels with flat trees.  The possibility of using any tree in order to either
maximize parallelism or minimize communication is explained in~\cite{CAQR}. The
focus of this chapter is on maximizing parallelism.
We reduce the communication (data movement between memory hierarchy) within the algorithm to acceptable levels by tiling the operations.
Stemming from the observation that a binary tree is best for tall
and skinny matrices and a flat tree is best for square matrices, Hadri et
al.~\cite{Hadri_ipdps_2010}, propose to use trees which combine flat trees at
the bottom level with a binary tree at the top level in order to exhibit more
parallelism.  Our theoretical and experimental work explains that we can adapt
\MC~\cite{ModiClarke84} and \Greedy~\cite{j12,j14} to tiles, resulting in yet
better algorithms in terms of parallelism. Moreover our new algorithms do not
have any tuning parameter such as the domain size in the case of~\cite{Hadri_ipdps_2010}.

The sequential kernels of the Tiled QR factorization (executed on a core) are
made of standard blocked algorithms such as LAPACK encoded in kernels; the
development of these kernels is well understood. The focus of this chapter is on
improving the overall degree of parallelism of the algorithm.  Given a
$p$-by-$q$ tiled matrix, we seek to find an appropriate sequence of unitary
transformations on the tiled matrix so as to maximize parallelism (minimize
critical path length). We will get our inspiration from previous work from the
1970s/80s on Givens rotations where the question was somewhat related: given an
$m$-by-$n$ matrix, find an appropriate sequence of Givens rotations as to
maximize parallelism. This question is essentially answered
in~\cite{j12,j14,ModiClarke84,SamehKuck78}; we call this class of
algorithms ``{\em coarse-grain algorithms}.''

Working with tiles instead of scalars, we introduce four essential differences
between the analysis and the reality of the tiled algorithms versus the
coarse-grain algorithms.  First, while there are only two states for a scalar
(nonzero or zero), a tile can be in three states (zero, triangle or full).
Second, there are more operations available on tiles to introduce zeros; we
have a total of three different tasks which can introduce zeros in a matrix.
Third, in the tiled algorithms, the factorization and the update are dissociated
to enable factorization stages to overlap with update stages; whereas, in the
coarse-grain algorithm, the factorization and the associated update are
considered as a single stage. Lastly, while coarse-grain algorithms have only
one task, we end up with six different tasks: three from the factorizations
(zeroing of tiles) and three for each of the associated updates (since these
have been unlinked from the factorization). Each of these six tasks have
different computational weights; this dramatically complicates the critical
path analysis of the tiled algorithms.

While the \Greedy algorithm is optimal for coarse-grain algorithms, we show
that it is not in the case of tiled algorithms. However, we have devised and
proved that there does exist an optimal tiled algorithm. 



\section{The QR factorization algorithm}
\label{sec.QR}

Tiled algorithms are expressed in terms of tile operations rather than
elementary operations.  Each tile is of size $n_b \times n_b$, where $n_b$ is a
parameter tuned to squeeze the most out of arithmetic units and memory
hierarchy. Typically, $n_b$ ranges from $80$ to $200$ on state-of-the-art
machines~\cite{sc09-agullo}. Algorithm~\ref{alg.QR} outlines a naive tiled QR
algorithm, where loop indices represent tiles:

\linespread{1.2}
\begin{algorithm}[htb]
  \DontPrintSemicolon
  \For{$\textnormal{k} = 1$ to $\min(p,q)$}{
      \ForAll{$\textnormal{i} \in \left\{ k+1,\ldots, p\right\}$ using any
  ordering, }{
     $\elim(i, piv(i,k), k)$
    }
  }
\caption{Generic QR algorithm for a tiled $p \times q$ matrix.}
\label{alg.QR}
\end{algorithm}
\renewcommand{\baselinestretch}{\normalspace}

In Algorithm~\ref{alg.QR}, $k$ is the panel index, and $\elim(i, piv(i,k), k)$
is an orthogonal transformation that combines rows $i$ and $piv(i,k)$ to zero
out the tile in position $(i,k)$. However, this formulation is somewhat
misleading, as there is much more freedom for QR factorization algorithms than,
say, for Cholesky algorithms (and contrarily to LU elimination algorithms,
there are no numerical stability issues).  For instance in column $1$, the
algorithm must eliminate all tiles $(i,1)$ where $i>1$, but it can do so in
several ways. Take $p=6$. Algorithm~\ref{alg.QR} uses the transformations
\[\elim(2, 1, 1), \elim(3, 1, 1), \elim(4, 1, 1), \elim(5, 1, 1), \elim(6, 1, 1)\]
But the following scheme is also valid:
\[\elim(3, 1, 1), \elim(6, 4, 1), \elim(2, 1, 1), \elim(5, 4, 1), \elim(4, 1, 1)\]
In this latter scheme, the first two transformations  $\elim(3, 1, 1)$ and $\elim(6, 4, 1)$
use distinct pairs of rows, and they can execute in parallel. On the contrary, $\elim(3, 1, 1)$ and
$\elim(2, 1, 1)$ use the same pivot row and must be sequentialized.
To complicate matters, it is possible to have two orthogonal transformations that execute in parallel
but involve zeroing a tile in two different columns. For instance we can add $\elim(6, 5, 2)$
to the previous transformations and run it concurrently with, say, $\elim(2, 1, 1)$.
Any tiled QR algorithm will be characterized by an \emph{elimination list}, which provides
the ordered list of the transformations used to zero out all the tiles below the diagonal.
This elimination list must obey certain conditions so that the factorization is valid.
For instance, $\elim(6, 5, 2)$ must follow $\elim(6, 4, 1)$ and $\elim(5, 4, 1)$
in the previous list, because there is a flow dependence between these transformations.
Note that, although the elimination list is given as a totally ordered sequence, some transformations can execute
in parallel, provided that they are not linked by a dependence: in the example,
$\elim(6, 4, 1)$ and $\elim(2, 1, 1)$  could have been swapped, and the elimination list
would still be valid.

In order to describe more fully the dependencies inherent in the eliminations we
shall observe a snippet of an example.  In Figure~\ref{fig:coarsedepexample},
to the left we have the row identifications, the empty circles represent zeroed
elements, and the filled circles represent the pivots used to zero out the
elements.  The first column's eliminations are shown in green and the second in
red.  From the elimination list, we define $I_{s,k}$ as the set of rows in
column $k$ that are zeroed out at time step $s$. 
\linespread{1.2}
\begin{figure}[htbp]
    \subfloat[Diagram of elimination list]{
        \label{fig:coarsedepexample}
        \begin{minipage}[c][2\width]{0.225\textwidth}
            \centering%
            \includegraphics[width=0.8\textwidth]{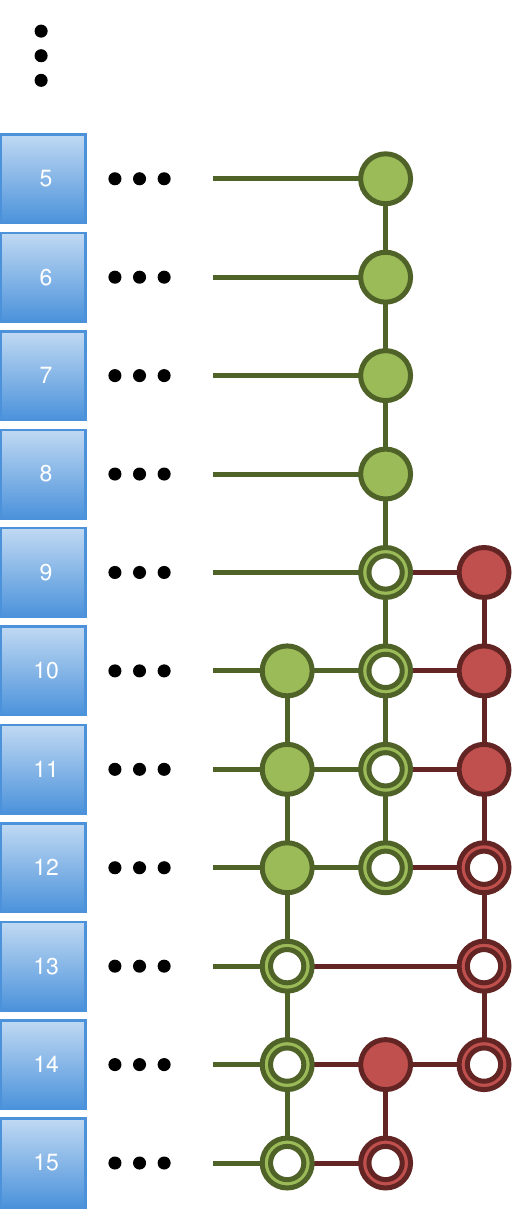}
        \end{minipage}
    }
    \subfloat[Elimination list]{
        \begin{minipage}[c][0.65\width]{0.7\textwidth}
            \centering%
            \begin{tabular}{lccl}
                $\displaystyle \left. \begin{array}{l}
                    \elim(13,10,1)\\
                    \elim(14,11,1)\\
                    \elim(15,12,1)\\
                \end{array} \right.$ & $\displaystyle \left. \begin{array}{l}
                    \quad\\\quad\\\quad\end{array}\right\} I_{s_{1},1}$ & $\Rightarrow$ &
                    $\elim(I_{s_{1},1},1)$\\
                $\displaystyle \left. \begin{array}{l}
                    \elim(9,5,1)\\
                    \elim(10,6,1)\\
                    \elim(11,7,1)\\
                    \elim(12,8,1)\\
                \end{array} \right.$ & $\displaystyle \left. \begin{array}{l}
                    \quad\\\quad\\\quad\\\quad\end{array}\right\} I_{s_{2},1}$ & $\Rightarrow$ &
                    $\elim(I_{s_{2},1},1)$\\
                $\displaystyle \left. \begin{array}{l}
                    \elim(15,14,2)\\
                \end{array} \right.$ & $\displaystyle \left. \begin{array}{l}
                    \quad\end{array}\right\} I_{s_{1},2}$ & $\Rightarrow$ &
                    $\elim(I_{s_{1},2},2)$\\
                $\displaystyle \left. \begin{array}{l}
                    \elim(12,9,2)\\
                    \elim(13,10,2)\\
                    \elim(14,11,2)\\
                \end{array} \right.$ & $\displaystyle \left. \begin{array}{l}
                    \quad\\\quad\\\quad\end{array}\right\} I_{s_{2},2}$ & $\Rightarrow$ &
                    $\elim(I_{s_{2},2},2)$\\
            \end{tabular}
        \end{minipage}
    }
\end{figure}
\renewcommand{\baselinestretch}{\normalspace}

What may not be so evident from the elimination list but is more apparent in the
diagram of the elimination list are the following dependency relationships (note
that $\prec$ indicates that the operation on the left must finish prior to the
operation on the right starting):
\begin{subequations}
    \label{eqn:coarsedeps}
    \begin{flalign}
        \quad & \elim(piv(I_{s,k},k), k-1) \prec \elim(I_{s,k}, k)
        \label{eqn:coarsedeps_piv}\\%
        \quad & \elim(I_{s,k}, k-1) \prec \elim(I_{s,k}, k)
        \label{eqn:coarsedeps_cprev}\\%
        \quad & \elim(I_{s-1,k}, k) \prec \elim(I_{s,k}, k)
        \label{eqn:coarsedeps_sprev}
    \end{flalign}%
\end{subequations}
However, not all of these dependencies may cause an elimination to be locked to
a particular time step.  In fact, some dependencies may not be needed for a
particular instance, but the addition of these will not create an artificial
lock.  For example, $\elim(I_{s_2,2},2)$ is dependent upon
$\elim(piv(I_{s_2,2}),1)$, $\elim(I_{s_2,2},1)$, and $\elim(I_{s_1,2},2)$ but
$\elim(I_{s_1,2},2)$ only depends upon $\elim(piv(I_{s_1,2}),1)$ and
$\elim(I_{s_1,2},1)$.

Before formally stating the conditions that guarantee the validity of (the
elimination list of) an algorithm, we explain how orthogonal transformations
can be implemented.

\subsection{Kernels}

To implement a given orthogonal transformation $\elim(i, piv(i,k),
k)$, one can use six different kernels, whose costs are given in
Table~\ref{tab.kernels}.  In this table, the unit of time is the time to
perform $\frac{n_b^3}{3}$ floating-point operations.


\linespread{1.2}
\begin{table*}
\centering
\scalebox{0.9}{
\begin{tabular}{@{}lcccc@{}}
  \toprule
  & \multicolumn{2}{c@{}}{Panel} & \multicolumn{2}{c@{}}{Update} \\
  \cmidrule(l){2-3}
  \cmidrule(l){4-5}
  Operation & Name & Cost & Name & Cost \\ 
  \cmidrule( r){1-1}
  \cmidrule(lr){2-2}
  \cmidrule(lr){3-3}
  \cmidrule(lr){4-4}
  \cmidrule(l ){5-5}
  Factor square into triangle        & \GEQRT & 4 & \UNMQR & 6  \\
  Zero square with triangle on top   & \TSQRT & 6 & \TSMQR & 12 \\
  Zero triangle with triangle on top & \TTQRT & 2 & \TTMQR & 6  \\
  \bottomrule
\end{tabular}
}
\caption{Kernels for tiled QR. The unit of time is $\frac{n_b^3}{3}$, where $n_b$ is the blocksize.}
\label{tab.kernels}
\end{table*}
\renewcommand{\baselinestretch}{\normalspace}

There are two main possibilities to implement an orthogonal transformation $\elim(i,
piv(i,k), k)$: The first version eliminates tile $(i,k)$ with the \emph{TS (Triangle
on top of square)} kernels, as shown in Algorithm~\ref{alg.elimSQ}:

\linespread{1.2}
\begin{algorithm}[htb]
  \DontPrintSemicolon
  $\GEQRT(piv(i,k), k)$\\
  $\TSQRT(i,piv(i,k), k)$\\
  \For{$\textnormal{j} = k+1$ to $q$}{
     $\UNMQR(piv(i,k), k, j)$\\
     $\TSMQR(i, piv(i,k), k, j)$
  }
\caption{Elimination $\elim(i, piv(i,k), k)$ via \emph{TS (Triangle on top of square)} kernels.}
\label{alg.elimSQ}
\end{algorithm}
\renewcommand{\baselinestretch}{\normalspace}

\vspace*{-0.2cm}
Here the tile panel $(piv(i,k), k)$ is factored into a triangle (with \GEQRT).
The transformation is applied to subsequent tiles $(piv(i,k),j)$, $j>k$, in row
$piv(i,k)$ (with \UNMQR). Tile $(i,k)$ is zeroed out (with \TSQRT), and
subsequent tiles $(i,j)$, $j>k$, in row $i$ are updated (with \TSMQR). The flop
count is $4+6+(6+12)(q-k)=10+18(q-k)$ (expressed in same time unit as in
Table~\ref{tab.kernels}).  Dependencies are the following:
\[
\begin{array}{ll}%
\GEQRT(piv(i,k), k) \prec \TSQRT(i, piv(i,k), k)\\%
\GEQRT(piv(i,k), k) \prec \UNMQR(piv(i,k), k, j) & \text{ for } j>k\\%
\UNMQR(piv(i,k) ,k, j) \prec \TSMQR(i, piv(i,k), k, j) & \text{ for } j>k\\%
\TSQRT(i,piv(i,k),k) \prec \TSMQR(i, piv(i,k), k, j) & \text{ for } j>k%
\end{array}%
\]
$\TSQRT(i, piv(i,k), k)$ and $\UNMQR(piv(i,k), k, j)$ can be executed in
parallel, as well as \UNMQR operations on different columns $j, j' >k$. With an
unbounded number of processors, the parallel time is thus $4+6+12=22$
time-units.

\linespread{1.2}
\begin{algorithm}[htb]
    \DontPrintSemicolon
            \If{$k>0$}{
                $\TTQRT(i, piv(i,k), k)$
            }
            \If{$k<q$}{
                \For{$\textnormal{j} = k+1$ to $q$}{
                    \If{$k>0$}{
                        $\TTMQR(i, piv(i,k), k, j)$
                    }
                }
                $\GEQRT(i, k+1)$\\
                \For{$\textnormal{j} = k+2$ to $q$}{
                    $\UNMQR(i, k, j)$
                }
            }
\caption{Elimination $\elim(i, piv(i,k), k)$ via \emph{TT (Triangle on top of triangle)} kernels.}
\label{alg.elimTR}
\end{algorithm}
\renewcommand{\baselinestretch}{\normalspace}

The second approach to implement the orthogonal transformation
\linebreak$\elim(i, piv(i,k), k)$ is with the \emph{TT (Triangle on top of
triangle)} kernels, as shown in Algorithm~\ref{alg.elimTR}. Here tile $(i,k)$
is zeroed out (with \TTQRT) and subsequent tiles $(i,j)$ and $(piv(i,k),j)$,
$j>k$, in rows $i$ and $piv(i,k)$ are updated (with \TTMQR).  Immediately
following,  tile $(i,k+1)$ is factored into a triangle and the corresponding
transformations are applied to the remaining columns in row $i$.  Necessarily,
\TTQRT must have the triangularization of tile $(i,k)$ and $(piv(i,k),k)$
completed in order to proceed.  Hence for the first column there are no updates
to be applied from previous columns such that the triangularization of these
tiles (with \GEQRT) is completed and can be considered a preprocessing step.
The flop count is $2(4+6(q-k))+2+6(q-k)=10+18(q-k)$, just as before.
Dependencies are the following:
%
%
\begin{subequations}
    \label{eqn:ttdeps}
    \begin{flalign}
        \quad & \GEQRT(piv(i,k), k) \prec \UNMQR(piv(i,k), k, j)       & \text{
        for } j>k \quad\label{eqn:ttdeps_gpivu}\\%
        \quad & \GEQRT(i, k) \prec \UNMQR(i, k, j)                     & \text{
        for } j>k \quad\label{eqn:ttdeps_gu}\\%
        \quad & \GEQRT(piv(i,k), k) \prec \TTQRT(i, piv(i,k), k) &
        \label{eqn:ttdeps_gpivt}\\%
        \quad & \GEQRT(i, k) \prec \TTQRT(i, piv(i,k), k) &
        \label{eqn:ttdeps_gt}\\%
        \quad & \TTQRT(i, piv(i,k), k) \prec \TTMQR(i, piv(i,k), k, j) & \text{
        for } j>k \quad\label{eqn:ttdeps_tm}\\%
        \quad & \UNMQR(piv(i,k), k, j) \prec \TTMQR(i, piv(i,k), k, j) & \text{
        for } j>k \quad\label{eqn:ttdeps_upivm}\\%
        \quad & \UNMQR(i, k, j) \prec \TTMQR(i, piv(i,k), k, j)        & \text{
        for } j>k \quad\label{eqn:ttdeps_um}
    \end{flalign}%
\end{subequations}
Now the factor operations in row $piv(i,k)$ and $i$ can be executed in parallel.
Moreover, the \UNMQR updates can be run in parallel with the \TTQRT
factorization.  Thus, with an unbounded number of processors, the parallel time
is $4+6+6=16$ time-units.

Recall our definition of the set $I_{s,k}$ to be the set of rows in column $k$ that
will be zeroed out at time step $s$ in the coarse-grain algorithm.  Thus the
following dependencies are a direct consequence of~\ref{eqn:coarsedeps_sprev} as
applied to the zeroing of a tile and the corresponding updates.
\begin{subequations}
    \label{eqn:ttdeps2}
    \begin{flalign}
        \quad & \TTQRT(I_{s-1}, piv(I_{s-1},k), k) \prec \TTQRT(I_s, piv(I_s,k), k) &
        \label{eqn:ttdeps_tt}\\%
        \quad & \TTQRT(I_{s-1}, piv(I_{s-1},k), k, j) \prec \TTMQR(I_s, piv(I_s,k), k, j) & \text{
        for } j>k \quad\label{eqn:ttdeps_mm}%
    \end{flalign}%
\end{subequations}

In Algorithm~\ref{alg.elimSQ} and~\ref{alg.elimTR}, it is understood that if a
tile is already in triangular form, then the associated \GEQRT and update kernels
do not need to be applied.

All the new algorithms introduced in this chapter are based on \emph{TT} kernels.
From an algorithmic perspective, \emph{TT} kernels are more appealing than
\emph{TS} kernels, as they offer more parallelism.  More precisely, we can
always break a \emph{TS} kernel into two \emph{TT} kernels: We can replace a
$\TSQRT(i,piv(i,k), k)$ (following a $\GEQRT(piv(i,k), k)$) by a $\GEQRT(i, k)$
and a $\TTQRT(i, piv(i,k), k)$. A similar transformation can be made for the
updates. Hence a \emph{TS}-based tiled algorithm can always be executed with
\emph{TT} kernels, while the converse is not true.  However, the \emph{TS}
kernels provide more data locality, they benefit from a very efficient
implementation (see \Section\ref{sec.experiments}), and several existing
algorithms use these kernels. For all these reasons, and for comprehensiveness,
our experiments will compare approaches based on both kernel types.

At this point, the PLASMA library only contains \emph{TS} kernels. We have
mapped the PLASMA algorithm to \emph{TT} kernel algorithm using this
conversion. Going from a \emph{TS} kernel algorithm to a \emph{TT} kernel
algorithm is implicitly done by Hadri et al.~\cite{Hadri_enhancingparallelism} when going
from their ``Semi-Parallel'' to their ``Fully-Parallel'' algorithms.

\subsection{Elimination lists}

As stated above, any algorithm factorizing a tiled matrix
of size $p \times q$ is characterized by its elimination list.
Obviously, the algorithm must zero out all tiles below the diagonal: for each tile $(i,k)$, $i>k$,
$1 \leq k \leq \min(p,q)$, the list must contain exactly one entry $\elim(i, \s, k)$, where \s denotes some
row index $piv(i,k)$ . There are two conditions for a transformation $\elim(i, piv(i,k), k)$ to be valid:
\begin{compactitem}
  \item both rows $i$ and $piv(i,k)$ must be ready, meaning that all their
     tiles left of the panel (of indices $(i,k')$ and $(piv(i,k),k')$ for $1
     \leq k' < k$) must have already been zeroed out: all transformations
     $\elim(i, piv(i,k'), k')$ and $\elim(piv(i,k), piv(piv(i,k),k'), k')$ must
     precede $\elim(i, piv(i,k), k)$ in the elimination list
  \item row $piv(i,k)$ must be a potential annihilator, meaning that tile
     $(piv(i,k),k)$ has not been zeroed out yet: the transformation
     $\elim(piv(i,k), piv(piv(i,k),k), k)$ must follow $\elim(i, piv(i,k), k)$
     in the elimination list
\end{compactitem}
Any algorithm that factorizes the tiled
matrix obeying these conditions is called a \emph{generic tiled algorithm} in
the following.

\begin{thm}
    \label{thm.flops}
    No matter what elimination list (any combination of TT, TS) is used the
    total weight of the tasks for performing a tiled QR factorization algorithm
    is constant and equal to $6pq^2- 2q^3$.
\end{thm}

\begin{proof}
    We have that the quantity of each kernel is given by the following 
    \[\begin{array}{lcl}
        L_1 & :: & GEQRT = TTQRT + q\\
        L_2 & :: & UNMQR = TTMQR + (1/2)q(q-1)\\
        L_3 & :: & TTQRT + TSQRT = pq - (1/2)q(q+1)\\
        L_4 & :: & TTMQR + TSMQR = (1/2)pq(q-1) - (1/6)q(q-1)(q+1)
    \end{array}\]
    The quantity of \TTQRT provides the number of tiles zeroed out via a
    triangle on top of a triangle kernel.  Thus equation $L_1$ is composed of
    two parts:  necessarily, the diagonal tiles must be triangularized and each
    \TTQRT must admit one more triangularization in order to provide a pairing.
    The number of updates of these triangularizations, given by $L_2$,  is
    simply the sum of the updates from the triangularization of the tiles on the diagonal
    and the updates from the zeroed tiles via \TTQRT.  The combination of \TTQRT
    and \TSQRT, equation $L_3$, is exactly the total number of tiles that are
    zeroed, namely every tile below the diagonal.  Hence, the total number of
    updates, provided by $L_4$, is the number of tiles below the diagonal beyond
    the first column minus the sum of the progression through the columns. 
    Now we define 
    \[ L_5 = 4L_1 + 6L_2 + 6L_3 + 12L_4 \]
    then
    \begin{flalign*}
        L_5 &= 4GEQRT + 6TSQRT + 2TTQRT + 6UNMQR + 6TTMQR + 12TSMQR.
    \end{flalign*}
    As can be noted in $L_5$, the coefficients of each term correspond precisely
    to the weight of the kernels as derived from the number of flops each kernel
    incurs.  Simplifying $L_5$, we have our result
    \[ L_5 = 6pq^2 - 2q^3. \]
\end{proof}

A critical result of Theorem~\ref{thm.flops} is that no matter what elimination
list is used, the total weight of the tasks for performing a tiled QR
factorization algorithm is constant and by using our unit task weight of
$n_b^3/3$, with $m = p n_b$, and $n=q n_b$, we obtain $2 mn^2 -2/3 n^3 $ flops
which is the exact same number as for a standard Householder reflection
algorithm as found in LAPACK (e.g.,~\cite{lawn41}).

\subsection{Execution schemes}

In essence, the execution of a generic tiled algorithm is fully determined by
its elimination list.  This list is statically given as input to the scheduler,
and the execution progresses dynamically, with the scheduler executing all
required transformations as soon as possible.  More precisely, each
transformation involves several kernels, whose execution starts as soon as they
are ready, i.e., as soon as all dependencies have been enforced.
Recall that a tile $(i,k)$ can be zeroed out only after all tiles $(i,k')$, with
$k' < k$, have been zeroed out. Execution progresses as follows:
\begin{compactitem}
\item Before being ready for elimination, tile $(i,k)$, $i>k$, must be updated
    $k-1$ times, in order to zero out the $k-1$ tiles to its left (of index
    $(i,k')$, $k' < k$). The last update is a transformation $\TTMQR(i, piv(i,k-1),
    k-1, k)$ for some row index $piv(i,k-1)$ such that  $\elim(i, piv(i,k-1), k-1)$
    belongs to the elimination list. When completed, this transformation
    triggers the transformation $\GEQRT(i,k)$, which can be executed
    immediately after the completion of the \TTMQR. In turn, $\GEQRT(i,k)$
    triggers all updates $\UNMQR(i, k, j)$ for all $j > k$.  These updates
    are executed as soon as they are ready for execution.

\item The elimination $\elim(i, piv(i,k), k)$ is performed as  soon as possible
    when both rows $i$ and $piv(i,k)$ are ready. Just after the completion of
    $\GEQRT(i, k)$ and  $\GEQRT(piv(i,k), k)$, kernel
    $\TTQRT(i,$ $piv(i,k), k)$ is launched. When finished, it
    triggers the updates  $\TTMQR(i,piv(i,k), k, j)$ for all $j > k$.
\end{compactitem}

Obviously, the degree of parallelism that can be achieved depends upon the
eliminations that are chosen.  For instance, if all eliminations in a given
column use the same factor tile, they will be sequentialized.  This corresponds
to the flat tree elimination scheme described below: in each column $k$, it
uses $\elim(i, k, k)$ for all $i>k$. On the contrary, two eliminations
$\elim(i, piv(i,k), k)$  and $\elim(i', piv(i',k), k)$ in the same column can
be fully parallelized provided that they involve four different rows. Finally,
note that several eliminations can be initiated in different columns
simultaneously, provided that they involve different pairs of rows, and that
all these rows are ready (i.e., they have the desired number of leftmost
zeros).

The following lemma will prove very useful; it states that we can assume
w.l.o.g.\ that each tile is zeroed out by a tile above it, closer to the
diagonal.

\begin{lemma}
\label{th.above}
Any generic tiled algorithm can be modified, without changing its
execution time, so that all eliminations $\elim(i, piv(i,k), k)$ satisfy $i >
piv(i,k)$.
\end{lemma}

\begin{proof}
   Define a \emph{reverse} elimination as an elimination $\elim(i, piv(i,k), k)$
   where $i < piv(i,k)$.  Consider a generic tiled algorithm whose
   elimination list contains some reverse eliminations. Let $k_0$ be the first
   column to contain one of them. Let $i_{0}$ be the largest row index involved in
   a reverse elimination in column $k_0$.  The elimination list in column $k_0$
   may contain several reverse eliminations  $\elim(i_1, i_0, k_0)$, $\elim(i_2,
   i_0, k_0)$, \dots, $\elim(i_r, i_0, k_0)$, in that order, before row $i_0$ is
   eventually zeroed out by the transformation $\elim(i_0,$ $ piv(i_0,k_0), k_0)$.
   Note that $piv(i_0,k_0) < i_0$ by definition of $i_0$.  We modify
   the algorithm by exchanging the roles of rows $i_0$ and $i_1$ in column $k_0$:
   the elimination list now includes $\elim(i_0, i_1, k_0)$, $\elim(i_2, i_1,
   k_0)$, \dots, $\elim(i_r, i_1, k_0)$, and $\elim(i_1, piv(i_0,k_0), k_0)$. All
   dependencies are preserved, and the execution time is unchanged. Now the largest
   row index involved in a reverse elimination in column $k_0$ is strictly smaller
   than $i_0$, and we repeat the procedure until there does not remain any reverse
   elimination in column $k_0$. We proceed inductively to the following columns,
   until all reverse eliminations have been suppressed.
\end{proof}

\section{Critical paths}
\label{sec.CP}

In this section we describe several generic tiled algorithms, and we provide
their critical paths, as well as optimality results. These algorithms are
inspired by algorithms that have been introduced twenty to thirty years
ago~\cite{SamehKuck78,ModiClarke84,j14,j12}, albeit for a much simpler,
\emph{coarse-grain} model.  In this ``old'' model, the time-unit is the time needed to
execute an orthogonal transformation across two matrix rows, regardless of the
position of the zero to be created, hence regardless of the length of these rows.
Although the granularity is much coarser in this model, any existing algorithm
for the old model can be transformed into a generic tiled algorithm, just by
enforcing the very same elimination list provided by the algorithm.
Critical paths are obtained using a discrete event based simulator specially developed
to this end, based on the Simgrid framework~\cite{simgrid}. It carefully
handles dependencies across tiles, and allows for the analysis of both static and dynamic
algorithms.\footnote{The discrete event based simulator, together with the code for all tiled algorithms,
is publicly available at \TiledQRURL}

\subsection{Coarse-grain algorithms}\label{sec:Coarsegrain}

We start with a short description of three algorithms for the coarse-grain model.
These algorithms are illustrated in Table~\ref{tab.coarse} for a $15 \times 6$ matrix.\\

\subsubsection{\SK algorithm}
The \SK algorithm~\cite{SamehKuck78} uses the panel row for all eliminations in
each column, starting from below the diagonal and proceeding downwards.
Time-steps indicate the time-unit at which the elimination can be done,
assuming unbounded resources. Formally, the elimination list is
\[ \left\{ \left( \elim(i, k, k), i=k+1, k+2, \dots, p \right), k =1, 2, \dots, \min(p,q) \right\} \]
This algorithm is also referred as \FT.

\subsubsection{\MC algorithm}
The \MC algorithm is the Fibonacci scheme of order $1$ in~\cite{ModiClarke84}.
Let $\coarse(i,k)$ be the time-step at which tile $(i,k)$, $i>k$, is zeroed
out.  These values are computed as follows. In the first column, there are one
$5$, two $4$'s, three $3$'s, four $2$'s and four $1$'s (we would have had five
$1$'s with $p=16$).  Given $x$ as the least integer such that $x(x+1)/2 \geq
p-1$, we have $\coarse(i,1)= x-y+1$ where $y$ is the least integer such that $
i \leq y(y+1)/2 +1$.  Let the row indices of the $z$ tiles that are zeroed out at
step $s$, $1 \leq s \leq x$, range from $i$ to $i+z-1$. The elimination list
for these tiles is $\elim(i+j, piv(i+j,1), 1)$, with $piv(i+j) = i+j-z$ for $0
\leq j \leq z-1$. In other words, to eliminate a bunch of $z$ consecutive tiles
at the same time-step, the algorithm uses the $z$ rows above them, pairing them
in the natural order.  Now the elimination scheme of the next column is the
same as that of the previous column, shifted down by one row, and adding two
time-units: $\coarse(i,k) = \coarse(i-1,k-1)+2$, while the pairing obeys the
same rule.

\subsubsection{\Greedy algorithm}
At each step, the \Greedy algorithm~\cite{j12,j14} eliminates as many tiles as
possible in each column, starting with bottom rows. The pairing for the
eliminations is done exactly as for \MC.  There is no closed-form formula to
compute $\coarse(i,k)$, the time-step at which tile $(i,k)$ is eliminated, but
it is possible to provide recursive expressions (see~\cite{j12,j14}).

\linespread{1.0}
\begin{table}[htb]
    \centering
    \resizebox{0.9\linewidth}{!}{%
        \begin{tabular}{|rrrrrr|rrrrrr|rrrrrr|}%
            \hline%
            \multicolumn{6}{|c|}{ (a) \SK } & \multicolumn{6}{c|}{ (b) \MC }  &\multicolumn{6}{c|}{ (c) \Greedy } \\%
            \hline%
            \s &    &    &    &    &          &\s  &      &     &     &      &        &    \s  &     &    &     &     &    \\%
              1 & \s &    &    &    &         & 5  & \s   &     &     &      &        &     4  & \s  &    &     &     &    \\%
              2 &   3 & \s &    &    &        & 4  &  7   & \s  &     &      &        &     3  &  6  & \s &     &     &    \\%
               3 &   4 &   5 & \s &    &      & 4  &  6   & 9  &  \s &      &         &     3  &  5  &  8 &  \s &     &    \\%
               4 &   5 &   6 &   7 & \s &     & 3  &  6   & 8  &  11 &   \s &         &     2  &  5  &  7 &  10 &  \s &    \\%
               5 &   6 &   7 &   8 &   9 & \s & 3  &  5   & 8  &  10 &   13 &   \s    &     2  &  4  &  7 &   9 &  12 &  \s\\%
               6 &   7 &   8 &   9 &  10 &  11& 3  &  5   & 7  &  10 &   12 &   15    &     2  &  4  &  6 &   9 &  11 &  14\\%
               7 &   8 &   9 &  10 &  11 &  12& 2  &  5   & 7  &   9 &   12 &   14    &     2  &  4  &  6 &   8 &  10 &  13\\%
               8 &   9 &  10 &  11 &  12 &  13& 2  &  4   & 7  &   9 &   11 &   14    &     1  &  3  &  5 &   8 &  10 &  12\\%
               9 &  10 &  11 &  12 &  13 &  14& 2  &  4   & 6  &   9 &   11 &   13    &     1  &  3  &  5 &   7 &   9 &  11\\%
              10 &  11 &  12 &  13 &  14 &  15& 2  &  4   & 6  &   8 &   11 &   13    &     1  &  3  &  5 &   7 &   9 &  11\\%
              11 &  12 &  13 &  14 &  15 &  16& 1  &  4   & 6  &   8 &   10 &   13    &     1  &  3  &  4 &   6 &   8 &  10\\%
              12 &  13 &  14 &  15 &  16 &  17& 1  &  3   & 6  &   8 &   10 &   12    &     1  &  2  &  4 &   6 &   8 &  10\\%
              13 &  14 &  15 &  16 &  17 &  18& 1  &  3   & 5  &   8 &   10 &   12    &     1  &  2  &  4 &   5 &   7 &   9\\%
              14 &  15 &  16 &  17 &  18 &  19& 1  &  3   & 5  &   7 &   10 &   12    &     1  &  2  &  3 &   5 &   6 &   8\\%
            \hline%
        \end{tabular}%
    }
    \caption{Time-steps for coarse-grain algorithms.}
    \label{tab.coarse}
\end{table}
\renewcommand{\baselinestretch}{\normalspace}

Consider a rectangular $p \times q$ matrix, with $p>q$.  With the coarse-grain
model, the critical path of \SK is $p+q-2$, and that of \MC is $x + 2q - 2$,
where $x$ is the least integer such that $x(x+1)/2 \geq p-1$.  The critical path
of \Greedy is unknown, but the critical path of \Greedy is optimal.  For square
$q \times q$ matrices, critical paths are slightly different ($2q-3$ for \SK,
$x+2q-4$ for \MC).

\linespread{1.0}
\begin{algorithm}[ht]
   \DontPrintSemicolon
   \For{ $j = 1$ to $q$}{
        \tcc{$nz(j)$ is the number of tiles which have been eliminated in column $j$}
        $nZ(j) = 0$\;
        \tcc{$nT(j)$ is the number of tiles which have been triangularized in column $j$}
        $nT(j) = 0$\;
    }

\While{ column $q$ is not finished}{

   \For{ $j = q$ down to $1$}{

        \If{$j==1$}{
            \tcc{Triangularize the first column if not yet done}
            $nT_{\textnormal{new}} = nT(j) + ( p - nT(j) )$\;
            \If{ $p-nT(j) > 0$ }{
                \For{$k = p$ down to $1$ }{
                    $\GEQRT(k, j)$\;
                    \For{$\textnormal{jj} = j+1$ to $q$}{
                        $\UNMQR(k, j, jj)$\;
                     }
                }
            }
        }
        \Else{
            \tcc{Triangularize every tile having a zero in the previous column}
            $nT_{\textnormal{new}} = nZ(j-1)$\;
            \For{$k = nT(j)$ to $nT_{\textnormal{new}}-1$}{
                $\GEQRT(p-k, j)$\;
                \For{$\textnormal{jj} = j+1$ to $q$}{
                  $\UNMQR(p-k, j, jj)$\;
                }
            }
        }

        \tcc{Eliminate every tile triangularized in the previous step}
        $nZ_{\textnormal{new}} = nZ(j) + \lfloor \dfrac{ nT(j) - nZ(j) }{2} \rfloor$\;
        \For{ $kk = nZ(j)$ to $nZ_{\textnormal{new}}-1$}{
            $piv(p-kk) = p-kk-nZ_{\textnormal{new}}+nZ(j)$\;
            $\TTQRT(p-kk, piv(p-kk), j)$\;
            \For{$\textnormal{jj} = j+1$ to $q$}{
                $\TTMQR(p-kk, piv(p-kk), j, jj)$\\
            }
        }

        \tcc{Update the number of triangularized and eliminated tiles at the next step}
        $nT(j) = nT_{\textnormal{new}}$\;
        $nZ(j) = nZ_{\textnormal{new}}$\;
    }
}
\caption{\Greedy algorithm via \emph{TT} kernels.}
\label{alg.tiled-greedy}
\end{algorithm}
\renewcommand{\baselinestretch}{\normalspace}

\linespread{1.0}
\begin{table*}[tb]
\centering
\resizebox{\linewidth}{!}{%
\begin{tabular}{|rrrrrr|rrrrrr|rrrrrr|rrrrrr|rrrrrr|}%
\hline
\multicolumn{6}{|c|}{ (a) \SK }       & \multicolumn{6}{c|}{ (b) \MC }        &\multicolumn{6}{c|}{ (c) \Greedy }& \multicolumn{6}{c|}{ (d) \BT }  & \multicolumn{6}{c|}{ (e) \PT ($\BS = 5$) }  \\%
\hline%
 \s  &     &     &     &      &     & \s  &     &     &     &     &        &  \s&    &    &    &    &        &  \s &     &     &     &     &    & \s &    &    &    &    &   \\%
  6  & \s  &     &     &      &     &  14 & \s  &     &     &     &        &  12&  \s&    &    &    &        &   6 & \s  &     &     &     &    &  6 &  \s &    &    &    &   \\%
  8  &  28 & \s  &     &      &     &  12 &  48 & \s  &     &     &        &  10&  42&  \s&    &    &        &   8 &  28 & \s  &     &     &    &  8 &  28 &  \s &    &    &   \\%
 10  &  34 &  50 &  \s &      &     &  12 &  46 &  70 & \s  &     &        &  10&  40&  64&  \s&    &        &   6 &  36 &  56 &  \s &     &    & 10 &  34 &  50 &  \s &    &   \\%
 12  &  40 &  56 &  72 &  \s  &     &  10 &  42 &  68 &  92 & \s  &        &   8&  36&  62&  86& \s &        &  10 &  34 &  70 &  90 & \s  &    & 12 &  40 &  56 &  72 &  \s &   \\%
 14  &  46 &  62 &  78 &  94  &   \s&  10 &  40 &  64 &  90 & 114 & \s     &   8&  34&  56&  84& 106& \s     &   6 &  44 &  68 & 104 & 124 & \s & 14 &  46 &  62 &  78 &  94 &  \s\\%
 16  &  52 &  68 &  84 & 100  &  116&  10 &  40 &  62 &  86 & 112 & 136    &   8&  34&  56&  78& 102& 128    &   8 &  28 &  78 & 102 & 138 & 158&  6 &  54 &  74 &  90 & 106 & 122\\%
 18  &  58 &  74 &  90 & 106  &  122&   8 &  36 &  62 &  84 & 108 & 134    &   8&  30&  52&  78& 100& 122    &   6 &  42 &  62 & 112 & 136 & 172&  8 &  28 &  82 & 102 & 118 & 134\\%
 20  &  64 &  80 &  96 & 112  &  128&   8 &  34 &  58 &  84 & 106 & 130    &   6&  28&  50&  72& 100& 118    &  12 &  40 &  76 &  96 & 146 & 170& 10 &  34 &  50 & 110 & 130 & 146\\%
 22  &  70 &  86 & 102 & 118  &  134&   8 &  34 &  56 &  80 & 106 & 128    &   6&  28&  50&  72&  94& 116    &   6 &  46 &  74 & 110 & 130 & 180& 12 &  40 &  56 &  72 & 138 & 158\\%
 24  &  76 &  92 & 108 & 124  &  140&   8 &  34 &  56 &  78 & 102 & 128    &   6&  28&  50&  68&  94& 116    &   8 &  28 &  80 & 108 & 144 & 164& 16 &  52 &  68 &  84 & 100 & 166\\%
 26  &  82 &  98 & 114 & 130  &  146&   6 &  28 &  56 &  78 & 100 & 122    &   6&  28&  44&  66&  88& 110    &   6 &  36 &  56 & 114 & 142 & 178&  6 &  56 &  80 &  96 & 112 & 128\\%
 28  &  88 & 104 & 120 & 136  &  152&   6 &  28 &  50 &  78 & 100 & 122    &   6&  22&  44&  66&  88& 110    &  10 &  34 &  64 &  84 & 148 & 176&  8 &  28 &  84 & 108 & 124 & 140\\%
 30  &  94 & 110 & 126 & 142  &  158&   6 &  28 &  44 &  72 & 100 & 122    &   6&  22&  44&  60&  82& 104    &   6 &  38 &  62 &  92 & 112 & 182& 10 &  34 &  50 & 112 & 136 & 152\\%
 32  & 100 & 116 & 132 & 148  &  164&   6 &  22 &  44 &  60 &  94 & 116    &   6&  22&  38&  60&  76&  98    &   8 &  28 &  66 &  90 & 114 & 134& 12 &  40 &  56 &  72 & 140 & 164\\%
\hline
\end{tabular}%
}
  \caption{Time-steps for tiled algorithms.}
  \label{tab.tiled}
\end{table*}
\renewcommand{\baselinestretch}{\normalspace}

\subsection{Tiled algorithms}\label{sec:TiledAlgorithms}

As stated above, each coarse-grain algorithm can be transformed into a tiled
algorithm, simply by keeping the same elimination list, and triggering the
execution of each kernel as soon as possible.  However, because the weights of
the factor and update kernels are not the same, it is much more difficult to
compute the critical paths of the transformed (tiled) algorithms.
Table~\ref{tab.tiled} is the counterpart of Table~\ref{tab.coarse}, and depicts
the time-steps at which tiles are actually zeroed out. Note that the tiled
version of \SK is indeed the \FT algorithm in
PLASMA~\cite{Buttari2008,tileplasma}, and we have renamed it accordingly.  As an
example, Algorithm~\ref{alg.tiled-greedy} shows the \Greedy algorithm for the
tiled model.

A first (and quite unexpected) result is that \Greedy is no longer optimal, as
shown in the first two columns of Table~\ref{table.15x2x3} for a $15\times2$
matrix.  In each column and at each step, ``{\em the \ASAP algorithm}'' starts
the elimination of a tile as soon as there are at least two rows ready for the
transformation.  When $s \geq 2$ eliminations can start simultaneously, \ASAP
pairs the $2s$ rows just as \MC and \Greedy, the first row (closest to the
diagonal) with row $s+1$, the second row with row $s+2$, and so on.  As a matter
of a fact, when processing the second column, both \ASAP and \Greedy begin with
the elimination of lines 10 to 15 (at time step 20).  However, once tiles
$(13,2)$, $(14,2)$ and $(15,2)$ are zeroed out (i.e.  at time step 22), \ASAP
eliminates $4$ zeros, in rows $9$ through $12$. On the contrary, \Greedy waits
until time step $26$ to eliminate 6 zeros in rows $6$ through $12$.  In a sense,
\ASAP is the counterpart of \Greedy at the tile level.  However, \ASAP is not
optimal either, as shown in Table~\ref{table.15x2x3} for a $15\times3$ matrix.
On larger examples, the critical path of \Greedy is better than that of \ASAP,
as shown in Table~\ref{table.biggermat}.

We can however use the optimality of the coarse-grain \Greedy to devise an
optimal tiled algorithm.  Let us define the following algorithm:
\begin{defin}
    \label{def.grasap}
    Given a matrix of $p \times q$ tiles, with $p > q$, the \GA($i$) algorithm
    \begin{enumerate}
        \item uses Algorithm~\ref{alg.elimTR} to execute \Greedy on the first
            $q-i$ columns and propagate the updates through column $q$.
        \item and for column(s) $q-i+1$ through $q$, apply the \ASAP algorithm.
    \end{enumerate}
\end{defin}

Clearly, if we let $i = q$ we obtain the \ASAP algorithm. We define \GA to be
\GA(1), i.e., only the elimination of the last column will differ from \Greedy,
and we will show that \GA is an optimal tiled algorithm.

\linespread{1.0}
\begin{table}[htb]
\centering
\resizebox{\linewidth}{!}{%
\subfloat[\Greedy nor \ASAP are optimal.\label{table.15x2x3}]{%
\begin{tabular}{c|rrr|rrr|c}%
\cline{2-7}%
~~ & \multicolumn{3}{|c|}{ (a)  \Greedy  }  &\multicolumn{3}{c|}{ (b) \ASAP } & ~~ \\%
\cline{2-7}%
~~ &                \s&    &                &         \s  &     &             & ~~ \\%
~~ &                12&  \s&                &         12  & \s  &             & ~~ \\%
~~ &                10&  42&  \s            &         10  & 40  & \s          & ~~ \\%
~~ &                10&  40&  64            &         10  & 36  & 86          & ~~ \\%
~~ &                 8&  36&  62            &          8  & 34  & 80          & ~~ \\%
~~ &                 8&  34&  56            &          8  & 32  & 74          & ~~ \\%
~~ &                 8&  34&  56            &          8  & 30  & 68          & ~~ \\%
~~ &                 8&  30&  52            &          8  & 28  & 62          & ~~ \\%
~~ &                 6&  28&  50            &          6  & 28  & 56          & ~~ \\%
~~ &                 6&  28&  50            &          6  & 26  & 50          & ~~ \\%
~~ &                 6&  28&  50            &          6  & 24  & 46          & ~~ \\%
~~ &                 6&  28&  44            &          6  & 24  & 44          & ~~ \\%
~~ &                 6&  22&  44            &          6  & 22  & 44          & ~~ \\%
~~ &                 6&  22&  44            &          6  & 22  & 40          & ~~ \\%
~~ &                 6&  22&  38            &          6  & 22  & 38          & ~~ \\%
\cline{2-7}%
\end{tabular}%
}

~~
\subfloat[\Greedy generally outperforms \ASAP.]{%
\label{table.biggermat}%
\begin{tabular}{|c|l|r|r|r|r|}%
\cline{3-6}%
\multicolumn{2}{c}{}             & \multicolumn{4}{|c|}{$q$}\\%
\hline%
$p$                  & \multicolumn{1}{c|}{Algorithm} & \multicolumn{1}{c|}{16}  & \multicolumn{1}{c|}{32} & \multicolumn{1}{c|}{64} & \multicolumn{1}{c|}{128}    \\%
\hline%
\multirow{2}{*}{16}  & \Greedy   & 310 & \multirow{2}{*}{~}   &  \multirow{4}{*}{~}  &  \multirow{6}{*}{~}      \\%
                     & \ASAP     & 310 &    &    &        \\%
\cline{1-4}%
\multirow{2}{*}{32}  & \Greedy   & 360 &650 &    &        \\%
                     & \ASAP     & 402 &656 &    &        \\%
\cline{1-5}%
\multirow{2}{*}{64}  & \Greedy   & 374 &726 & 1342&        \\%
                     & \ASAP     & 588 &844 & 1354&        \\%
\hline%
\multirow{2}{*}{128} & \Greedy   & 396 &748 & 1452 & 2732    \\%
                     & \ASAP     & 966 &1222 & 1748& 2756    \\%
\hline%
\end{tabular}%
}
}
\caption{Neither \Greedy nor \ASAP are optimal.\label{table.greedyvsasap}}
\end{table}
\renewcommand{\baselinestretch}{\normalspace}

Although we cannot provide an elimination list for the entire tiled matrix of
size $p \times q$, we do provide an elimination list for the first $q-1$
columns.  This tiled elimination list describes the time-steps at which
Algorithm~\ref{alg.elimTR} is complete, i.e., all of the factorization kernels
are complete for $k\leq q-1$ and corresponding update kernels are complete for
all columns $k<j\leq q$.

We must make note of one consequence from the coarse-grain elimination list
before proceeding.  We will use this repeatedly within the proof of translating
a coarse-grain elimination list to a tiled elimination list.
\begin{lemma}
    \label{lem:scoarse}
    Given an elimination list from any coarse-grain algorithm, let\\
    $s = coarse(i,k)$ be the time step at which element $(i,k)$ is eliminated and
    let 
    \[ I_{s,k} = \left\{ i \vert s = coarse(I_{s,k},k) \right\}.\] 
    Then for any $s$ we have
    \[ s - 1 = coarse(I_{s,k},k) - 1 \geq \max \left( \begin{array}{l}
            coarse(I_{s,k},k-1)\\coarse(piv(I_{s,k},k),k-1)
\end{array}\right) \]
and in particular
\[ s_1 - 1 = \max \left( \begin{array}{l}
        coarse(I_{s_1,k},k-1)\\coarse(piv(I_{s_1,k},k),k-1)
\end{array}\right) \]
where $s_1 = \min_{k+1\leq i \leq p}(coarse(i,k))$.
\end{lemma}
\begin{proof}
    This follows directly from the dependencies given
    in~\eqref{eqn:coarsedeps_piv}-\eqref{eqn:coarsedeps_sprev}.  
\end{proof}

\begin{thm}
    \label{thm.lowerbound}
    Given the elimination list of a coarse-grain algorithm for a matrix of size
    $p \times q$, using Algorithm~\ref{alg.elimTR}, the tiled elimination list
    for all but the last column is given by
    \[ tiled(i,k) = 10k + 6 \cdot coarse(i,k), \qquad 1\leq i \leq p, 1 \leq
        k < q-1\]
    where $coarse(i,k)$ is the elimination list of the coarse-grain algorithm.
\end{thm}
\begin{proof}
    In this analysis, when $k$ is clear, we will use $I_{s}$ instead of
    $I_{s,k}$.  By abuse of notation, we will write $\GEQRT(i,k)$ to denote the
    time at which the task $\GEQRT(i,k)$ is complete and this will be the same
    for all of the kernels.  Thus we will prove that
    \[ tiled(i,k) = \TTMQR(i,piv(i,k),k,j) \quad \textnormal{for $j>k$}.\] 
    Note that $j$ represents the column in which the updates are applied and all
    columns $j$ for $j>k$ have the same update history.  In
    Algorithm~\ref{alg.elimTR}, the two $j$-loops spawn mutually independent
    tasks.  Since we have an unbounded number of processors, these tasks can all
    run simultaneously.  So $j$ represents any one of these columns.

    We will proceed by induction on $k$.  For the first column, $k=1$, we do not
    have any dependencies which concern the $GEQRT$ operations.  Thus from 
    \eqref{eqn:ttdeps_gu} we have for $1\leq i \leq p$,
    \begin{eqnarray}
        \GEQRT(i,1) &=& 4 \label{eqn:geqrt1}\\
        \UNMQR(i,1,j) &=& 4 + 6 = 10 \label{eqn:unmqr1}
    \end{eqnarray}
    Since each column in the coarse-grain elimination list is composed of one
    or more time steps, we must also proceed with induction on the time steps.
    Let
    \begin{equation}
        s_1 = \min_{2 \leq i \leq p}\left( coarse(i,1) \right).
        \label{eqn:s}
    \end{equation}
    In the case $k=1$, we have that
    \[ s_1 = 1. \]
    In other words, the first tasks finish at time step 1 for the coarse-grain
    algorithm.  This is a complicated manner in which to state that $s_1 = 1$,
    but it will be needed in the general setting.

    So for $s_1$, from \eqref{eqn:ttdeps_gpivt} and \eqref{eqn:ttdeps_gt} we have
    \begin{flalign*}
        \TTQRT(I_{s_1},piv(I_{s_1},1),1) &= \max \left(
        \begin{array}{c}
            \GEQRT(piv(I_{s_1},1),1)\\\GEQRT(I_{s_1},1)
        \end{array}\right) + 2\\
        &= 4 + 2\\
    \end{flalign*}
    thus
    \begin{equation}
        \TTQRT(I_{s_1},piv(I_{s_1},1),1) = 4 + 2s_1.
        \label{eqn:ttqrts1k1}
    \end{equation}
    Now from \eqref{eqn:ttdeps_tm}, \eqref{eqn:ttdeps_upivm}, and \eqref{eqn:ttdeps_um} we have
    \begin{flalign*} 
        \TTMQR(I_{s_1},piv(I_{s_1},1),1,j) &= \max\left( \begin{array}{c}TTQRT\left(
            I_{s_1}, piv(I_{s_1},1), 1 \right)\\ 
            \UNMQR\left( piv(I_{s_1},1), 1, j \right)\\
            \UNMQR\left( I_{s_1}, 1, j \right)
        \end{array}\right) + 6\\
        &= 10\cdot 1 + 6s_1\\ 
        &= 10\cdot 1 + 6\cdot coarse(I_{s_1},1)\\ 
    \end{flalign*}
    Therefore,
    \begin{equation}
        \TTMQR(I_{s_1},piv(I_{s_1},1),1,j) = tiled(I_{s_1},1).
        \label{eqn:ttmqr_s1k1}
    \end{equation}
    Assume that for $1 \leq t \leq s-1$ we have
    \begin{eqnarray}
        \TTQRT(I_t,piv(I_t,1),1) &=& 4 + 2t
        \label{eqn:ttqrt_tk1}\\
        \TTMQR(I_t,piv(I_t,1),1,j) &=& 10 + 6t
        \label{eqn:ttmqr_tk1}
    \end{eqnarray}
    then from \eqref{eqn:ttdeps_gpivt}, \eqref{eqn:ttdeps_gt}, and \eqref{eqn:ttdeps_tt} we have
    \begin{flalign*}
        \TTQRT(I_s,piv(I_s,1),1) &= \max\left( \begin{array}{c}
            \GEQRT\left( piv(I_s,1), 1 \right)\\ 
            \GEQRT\left( I_s,1 \right) \\ 
            \TTQRT\left( I_{s-1}, piv(I_{s-1},1),1 \right)
        \end{array}\right) + 2\\
        &= \max\left( \begin{array}{c}
            4\\4\\4+2(s-1)
        \end{array}\right) + 2\\
    \end{flalign*}
    Thus
    \begin{equation}
        TTQRT(I_s,piv(I_s,1),1) = 4 + 2s.
        \label{eqn:ttqrt_sk1}
    \end{equation}
    From \eqref{eqn:ttdeps_tm},  \eqref{eqn:ttdeps_upivm}, \eqref{eqn:ttdeps_um},
    and \eqref{eqn:ttdeps_mm} we have
    \begin{flalign*} 
        \TTMQR(I_s,piv(I_s,1),1,j) &= \max\left( \begin{array}{c}
            \TTQRT\left( I_s, piv(I_s,1), 1 \right)\\ 
            \UNMQR\left( piv(I_s,1), 1, j \right)\\
            \UNMQR\left( I_s, 1, j \right)\\
            \TTMQR\left( I_{s-1}, piv(I_{s-1},1),1,j \right)
        \end{array}\right) + 6\\
        &= \max\left( \begin{array}{c}
            4 + 2s\\10\\10\\10+6(s-1)
        \end{array}\right) + 6\\
        &= 10 + 6(s-1) + 6\\
        &= 10 + 6s\\
        &= 10\cdot 1 + 6 \cdot coarse(I_s,1)
    \end{flalign*}
    Thus
    \begin{equation}
        \TTMQR(I_s,piv(I_s,1),1,j) = tiled(I_s,1)
        \label{eqn:ttmqr_sk1}
    \end{equation}
    establishing our base case for the induction on $k$.

    Now assume that for $1 \leq h \leq k-1$ we have, for any $s$ in column $h$,
    \[ TTMQR(I_s,piv(I_s,h),h,j) = tiled(I_s,h). \]
    In order to start the elimination of the next column, we must have that all
    updates from the elimination of the previous column are complete.  Thus
    using the induction assumption, we have 
    \[ \GEQRT(i,k) = \TTMQR(i,piv(i,k-1),k-1,k) + 4 \]
    so that
    \begin{equation}
        \GEQRT(i,k) = 10(k-1) + 6\cdot coarse(i,k-1) + 4
        \label{eqn:geqrt_k}
    \end{equation}
    and
    \[ \UNMQR(i,k,j) = \max \left( \begin{array}{c}
            \GEQRT(i,k)\\\TTMQR(i,piv(i,k-1),k-1,j) \end{array} \right) + 6 \]
    so that
    \begin{equation}
        \UNMQR(i,k,j) = 10k + 6 \cdot coarse(i,k-1).
        \label{eqn:unmqr_k}
    \end{equation}
    Again, we must proceed with an induction on the time steps in column $k$.
    Let
    \begin{equation}
        s_1 = \min_{k+1 \leq i \leq p}\left( coarse(i,k) \right).
        \label{eqn:sk}
    \end{equation}
    From \eqref{eqn:ttdeps_gpivt} and \eqref{eqn:ttdeps_gt} we have

    \begin{flalign*}
        \TTQRT(I_{s_1},piv(I_{s_1},k),k) &= \max\left( \begin{array}{c}
                \GEQRT\left( piv(I_{s_1},k), k \right)\\ 
                \GEQRT\left( I_{s_1},k \right)
            \end{array}\right) + 2\\
        &= \max\left( \begin{array}{c}
            \scalebox{.95}{
                $10(k-1) + 6 \cdot coarse(piv(I_{s_1},k),k-1) + 4$}\\
            \scalebox{.95}{
                $10(k-1) + 6 \cdot coarse(I_{s_1},k-1) + 4$}
            \end{array}\right) + 2\\
        &= 10(k-1) + 6 \max\left( \begin{array}{c}
            \scalebox{.95}{ $coarse(piv(I_{s_1},k),k-1)$}\\
            \scalebox{.95}{ $coarse(I_{s_1},k-1)$}\\
            \end{array}\right) + 4 + 2\\
    \end{flalign*}
    From the application of Lemma~\ref{lem:scoarse}, we have
    \[ coarse(I_{s_1},k) - 1 = \max \left( \begin{array}{l}
        coarse(I_{s_1,k},k-1)\\coarse(piv(I_{s_1,k},k),k-1)
    \end{array}\right) \]
    such that
    \[ \TTQRT(I_{s_1},piv(I_{s_1},k),k) = 10(k-1) + 6 \left[ coarse(I_{s_1},k) -
        1 \right] + 4 + 2. \] 
    Therefore,
    \begin{equation}
        \TTQRT(I_{s_1},piv(I_{s_1},k),k) = 10(k-1) + 6s_1.
        \label{eqn:ttqrt_s1k}
    \end{equation}

    For the updates, we must again examine the three dependencies which result
    from \eqref{eqn:ttdeps_tm}, \eqref{eqn:ttdeps_upivm}, and
    \eqref{eqn:ttdeps_um} such that we have
    \begin{flalign*} 
        TTMQR(I_{s_1},piv(I_{s_1},k),k,j) &= \max\left( \begin{array}{l}
            UNMQR\left( piv(I_{s_1},k), k, j \right)\\
            UNMQR\left( I_{s_1}, k, j \right)\\
            TTQRT\left( I_{s_1}, piv(I_{s_1},k), k \right)
        \end{array}\right) + 6\\
        &= \max\left( \begin{array}{l}
            10k+6\cdot coarse(I_{s_1},k-1)\\
            10k+6\cdot coarse(piv(I_{s_1},k),k-1)\\
            10(k-1)+6s_1
        \end{array}\right) + 6
    \end{flalign*}
    Using Lemma~\ref{lem:scoarse}, we have
    \begin{flalign*} 
        TTMQR(I_{s_1},piv(I_{s_1},k),k,j) &= \max\left( \begin{array}{l}
            10k+6(s_1 - 1)\\
            10(k-1)+6s_1
        \end{array}\right) + 6\\
        &= 10k + \max\left( \begin{array}{l}
            6s_1 - 6\\
            6s_1 - 10
        \end{array}\right) + 6\\
        &= 10k + 6s_1
    \end{flalign*}
    Therefore
    \begin{equation}
        TTMQR(I_{s_1},piv(I_{s_1},k),k,j) = tiled(I_{s_1},k).
        \label{eqn:ttmqr_s1k}
    \end{equation}
    Now assume that for $s_1 \leq t \leq s - 1$ we have
    \begin{eqnarray}
        \TTQRT(I_t,piv(I_t,k),k) &\leq& 10(k-1) + 6t
        \label{eqn:ttqrt_tk}\\
        \TTMQR(I_t,piv(I_t,k),k,j) &=& 10k + 6t.
        \label{eqn:ttmqr_tk}
    \end{eqnarray}
    and note that we do not have equality for $s>s_1$ via
    Lemma~\ref{lem:scoarse}.  From \eqref{eqn:ttdeps_gpivt},
    \eqref{eqn:ttdeps_gt}, and \eqref{eqn:ttdeps_tt} we have
    \begin{flalign*}
        \TTQRT(I_s,piv(I_s,k),k) &= \max\left( \begin{array}{c}
            \GEQRT\left( piv(I_s,k), k \right)\\ 
            \GEQRT\left( I_s,k \right) \\ 
            \TTQRT\left( I_{s-1}, piv(I_{s-1},k),k \right)
        \end{array}\right) + 2\\
        &\leq \max\left( \begin{array}{c}
            10(k-1) + 6 \cdot coarse(piv(I_s,k),k-1) + 4\\
            10(k-1) + 6 \cdot coarse(I_s,k-1) + 4\\
            10(k-1) + 6(s-1)
        \end{array}\right) + 2
    \end{flalign*}
    Note that from Lemma~\ref{lem:scoarse}
    \[ s - 1 \geq \max \left( \begin{array}{c} coarse(piv(I_s,k),k-1)\\
            coarse(I_s,k-1) \end{array} \right)\]
    such that
    \[ \TTQRT(I_s,piv(I_s,k),k) \leq 10(k-1) + 6(s-1) + 4 + 2. \]
    Thus
    \begin{equation}
        \TTQRT(I_s,piv(I_s,k),k) \leq 10(k-1) + 6s.
        \label{eqn:ttqrt_sk}
    \end{equation}

    For the updates, we must examine the four dependencies which result
    from \eqref{eqn:ttdeps_tm}, \eqref{eqn:ttdeps_upivm}, \eqref{eqn:ttdeps_um}, 
    and \eqref{eqn:ttdeps_mm} such that we have
    \begin{flalign*} 
        TTMQR(I_s,piv(I_s,k),k,j) &= \max\left( \begin{array}{l}
            UNMQR\left( piv(I_s,k), k, j \right)\\
            UNMQR\left( I_s, k, j \right)\\
            TTQRT\left( I_s, piv(I_s,k), k \right)\\ 
            TTMQR\left( I_{s-1}, piv(I_{s-1},k), k, j \right)
        \end{array}\right) + 6\\
        &= \max\left( \begin{array}{l}
            10k+6\cdot coarse(I_s,k-1)\\
            10k+6\cdot coarse(piv(I_s,k),k-1)\\
            10(k-1)+6s\\
            10k + 6(s-1)
        \end{array}\right) + 6.
    \end{flalign*}
    As before, Lemma~\ref{lem:scoarse} allows us to write
    \begin{flalign*} 
        TTMQR(I_s,piv(I_s,k),k,j) &= \max\left( \begin{array}{l}
            10k+6(s-1)\\
            10(k-1)+6s
        \end{array}\right) + 6\\
        &= 10k + \max\left( \begin{array}{l}
            6s-6\\
            6s-10
        \end{array}\right) + 6\\
        &= 10k + 6s
    \end{flalign*}
    Therefore
    \begin{equation}
        TTMQR(I_s,piv(I_s,k),k,j) = tiled(i,k).
        \label{eqn:ttmqr_sk}
    \end{equation}
\end{proof}

\begin{cor}
    \label{cor.cplowup}
    Given an elimination list for a coarse-grain algorithm on a matrix of size
    $p \times q$ where $p>q$, the critical path length of the corresponding
    tiled algorithm is bounded by
    \[ tiled(p,q-1) + 4 + 2 \leq CP(p,q) < tiled(p,q). \]
\end{cor}
\begin{proof}
    For any tiled matrix, the last column will necessarily need to be factorized
    which explains the addition of four time steps and since $p>q$ at least one
    \TTQRT will be present which accounts for the two time steps thereby
    establishing the lower bound.  By including one more column, the upper bound
    not only includes the factorization of column $q$, but also the respective
    updates onto column $q+1$ such that the critical path of the $p \times q$
    tiled matrix must be smaller.
\end{proof}

\begin{cor}
    \label{cor.cpsquare}
    Given an elimination list for a coarse-grain algorithm on a matrix of size
    $p \times q$ where $p=q$, the critical path length of the corresponding
    tiled algorithm is
    \[ CP(p,q) = tiled(p,q-1) + 4.\]
\end{cor}
\begin{proof}
    In the last column, we need only to factorize the diagonal tile which
    explains the additional four time steps.  Moreover, there are no further
    columns to apply any updates to nor any tiles below the diagonal that need
    to be eliminated.  Thus the result is obtained.
\end{proof}
%
%

In the remainder of this chapter, we will make use of diagrams to clarify
certain aspects of the proofs and provide examples to further illustrate the
points being made.  These diagrams make use of the kernel representations as
shown in Figure~\ref{fig:treekernels}.
\linespread{1.0}
\begin{figure*}[h]
    \centering
    \begin{minipage}{0.45\textwidth}
        \begin{tabular}{ m{15mm} m{25mm} m{12mm}}
            kernel & & weight\\
            \hline 
            \hline\\
            GEQRT &
            \includegraphics[width=0.30\textwidth]{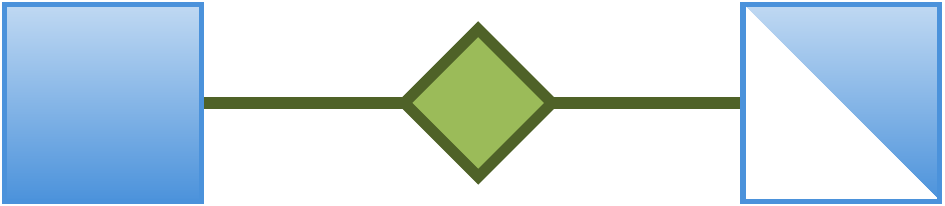}
            & \quad 4\\
            TTQRT &
            \includegraphics[width=0.30\textwidth]{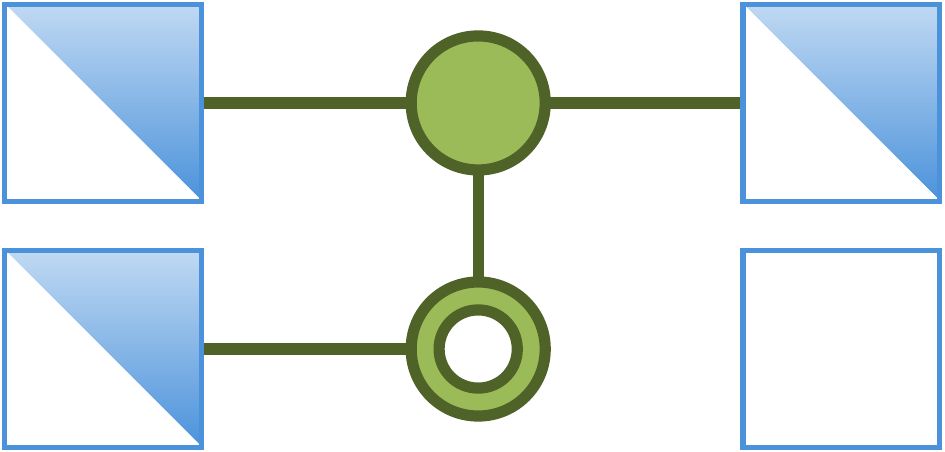}
            & \quad 2\\
        \end{tabular}
    \end{minipage}
    \begin{minipage}{0.45\textwidth}
        \begin{tabular}{ m{18mm} m{25mm} m{12mm}}
            kernel & & weight\\
            \hline 
            \hline\\
            UNMQR &
            \includegraphics[width=0.30\textwidth]{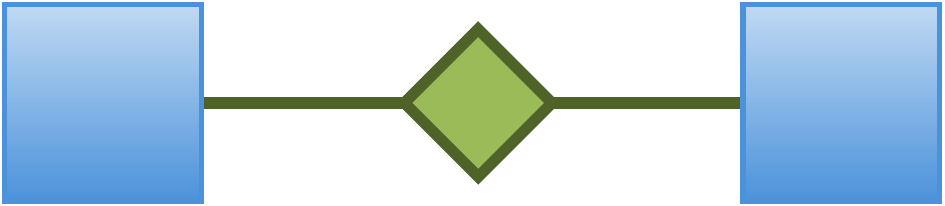}
            & \quad 6\\
            TTMQR &
            \includegraphics[width=0.30\textwidth]{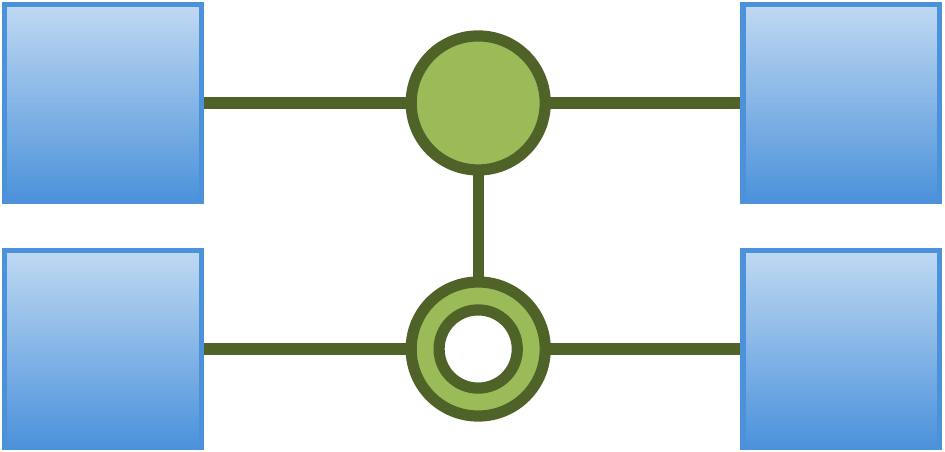}
            & \quad 6\\
        \end{tabular}
    \end{minipage}
    \caption{\label{fig:treekernels} Icon representations of the kernels}
\end{figure*}
\renewcommand{\baselinestretch}{\normalspace}

We have a closed-form expression for the critical path of tiled \FT for all
three cases: single tiled column, square tiled matrix, and rectangular tiled
matrix of more than one column. 
\begin{prop}
\label{th.ft}
Consider a tiled matrix of size $p \times q$, where $p \geq q \geq 1$. The critical
path length of \FT is
\[ CP_{ft}(p,q) = 
    \left\{ 
    \begin{array}{ll}
        2p+2,       &\mbox{if $q=1$;}\\
        22p-24,     &\mbox{if $p=q>1$;}\\
        6p+16q-22,\quad  &\mbox{if $p>q>1$.}\\
    \end{array}
    \right. \]
\end{prop}

\begin{proof}
    Consider first the case $q = 1$. We shall proceed by
    induction on $p$ to show that the critical path of \FT is of length $2p+2$.
    If $p=1$, then from Table~\ref{tab.kernels} the result
    is obtained since only $\GEQRT(1,1)$ is required.  With the base case
    established, now assume that this holds for all $p-1 > q = 1$.  Thus at
    time $t=2(p-1) + 2 = 2p$, we have that for all $p-1 \geq i \geq 1$ tile
    $(i,1)$ has been factorized into a triangle and for all $p-1 \geq i > 1$,
    tile $(i,1)$ has been zeroed out.  Therefore, tile $(p,1)$ will be zeroed
    out with $\TTQRT(p,1)$ at time $t+2 = 2(p-1) + 2 + 2 = 2p + 2$.

    Considering the second case $p=q>1$, we will be using
    Figure~\ref{fig:wflattree-4x4} to illustrate.  We initialize with a
    triangularization of the first column and send the update to the remaining
    column(s), $10$ time units.  The we fill the pipeline with the updates onto
    the remaining column(s) from the zeroing operations of the first column,
    $6(p-1)$ time units.  Then for each column after the first, except the last
    one, we fill the pipeline with the triangularization, update of
    triangularization, and update of zeroing for the bottom most tile, $(4+6+
    6)(p-2)$ time units.  In the last column, we then triangularize the bottom
    most tile, $4$ time units.  Thus
    \[ 10 + 6(p-1) + (4+6+6)(q-2) + 4 = 6p + 16q - 24 = 22p - 24\]

    The third case is analogous to the second case but we still need to zero out
    the bottom most tile in the last column which explains the difference of 2
    in the formula from the square case.
\end{proof}

\linespread{1.0}
\begin{figure*}[htb]
\centering
    \resizebox{.95\textwidth}{!}{\includegraphics{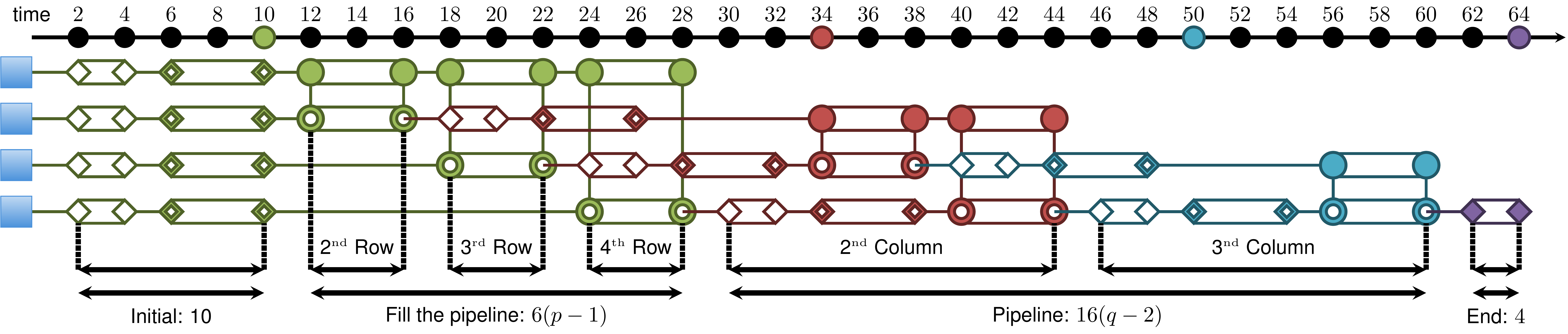}}%
\caption{ Critical Path length for the weighted \FT on a matrix of $4 \times 4$
tiles.}
\label{fig:wflattree-4x4}
\end{figure*}
\renewcommand{\baselinestretch}{\normalspace}

We remind that for the coarse algorithm,
\[ coarse(p,q) = 
    \left\{ 
    \begin{array}{ll}
        0,       &\mbox{if $q<1$;}\\
        0,       &\mbox{if $p=q=1$;}\\
        p+q-2,\quad  &\mbox{if $p>q\geq1$;}\\
        2p-3,     &\mbox{if $p=q>1$.}\\
    \end{array}
    \right. 
\]
So we find that considering a tiled matrix of size $p \times q$, where $p \geq q
\geq 1$. The critical path length of \FT is given as
\[ \resizebox{\textwidth}{!}{$CP(p,q) = 10(q-1) + 6\left( coarse(p,q-1) \right)  +
    4 + 2 \left( coarse(p,q) - coarse(p,q-1) \right).$} \]

%

\begin{prop}
    \label{thm.bt}
    The critical path length of \MC is greater than $22q - 30$ and less than $22q + 6
    \lceil \sqrt{2p} \rceil.$
\end{prop}

\begin{proof}
    The critical path length of the coarse-grain \MC algorithm for a $p\times
    q$ matrix is 
    \[ coarse(p,q) = x + 2q - 2. \]
    Thus from Proposition~\ref{cor.cplowup} we have 
    \[ 10(q-1) + 6(x+2(q-1)-2) + 4 \leq CP(p,q) < 10q + 6(x+2q-2).\]
    Recall that $x$ is the least integer such that
    \[ \frac{x(x+1)}{2} \geq p-1 \]
    whereby
    \[ x = -\frac{1}{2} + \frac{\sqrt{8p-7}}{2}.\]
    Thus $x \leq \left\lceil \sqrt{2p} \right\rceil$ and therefore
    \[ 22q - 30 < CP(p,q) < 22q - 12 + 6  \left\lceil \sqrt{2p} \right\rceil. \]
\end{proof}


Similarly to~\cite{j14} in which an iterate of a column is defined for the
coarse-grain algorithms, we define a weighted iterate and our notation will
follow in the same manner.

A column of length $n$ is a sequence of $n$ integers:
\[ a = a_1^{n_1} \cdots a_q^{n_q} \]
where power means concatenation with the following restrictions:
\begin{alignat*}{3}
    a_1 \geq 0 & \quad & a_{i+1} > a_i, & \quad & 1 \leq i \leq q-1;\\
    n_i > 0, && 1 \leq i \leq q; && n_1 + \cdots + n_q = n. 
\end{alignat*}
We define on the set of columns of length $n$ the classical partial ordering of
$\mathbb{R}^n$:
\[ x \leq y \iff (x_i \leq y_i, i\leq i \leq n)\]
and the $s$-truncate ($1\leq s \leq n$) of $a$ is a column of length $s$
composed of the $s$ first elements of $a$ and is denoted $a^s$.
\begin{defin}
    \label{def.iter}
    Given a task weight of $w$ and column $a = a_1^{n_1} \cdots a_q^{n_q}$, the column $c = c_1^{m_1} \cdots c_p^{m_p}$ is called an iterate of $a$, or
    $c= iter(a)$, if
    \begin{enumerate}[(i)]
        \item $c$ is a column of length $n-1$
        \item $a_1 + w \leq c_1$
            \begin{enumerate}
                \item if $a_1 + w \leq c_1 \leq a_2$ then $m_1 \leq \left\lfloor
                    n_1 / 2 \right\rfloor$
                \item if there exists an $h$ such that $a_{k-1} + w \leq c_h
                    \leq a_k$ then 
                    \[m_h \leq \left\lfloor \left( n_1 + \cdots + n_{k-1} - m_1
                        - \cdots - m_{h-1} \right) / 2 \right\rfloor\]
                        for $2 \leq k \leq q$ and $1 \leq
                    h \leq p$ with $m_0 = 0$.
                \item else $a_j \leq c_h$ and 
                    \[m_h \leq \left\lfloor \left( n_1
                    + \cdots + n_j - m_1 - \cdots - m_{h-1} \right) / 2
                    \right\rfloor\] 
                    where $j = \min (p+1,q)$.
            \end{enumerate}
    \end{enumerate}
\end{defin}
\begin{defin}
    \label{def.optiter}
    Given a task weight of $w$, let $a = a_1^{n_1}\cdots a_q^{n_q}$ be a column
    iterate of length $n$ then the sequence $b = b_1^{m_1} \cdots b_p^{m_p}$, or
    $b = opiter(a)$, is defined as
    \begin{enumerate}[(i)]
        \item for $b_1$ and $m_1$
            \begin{enumerate}
                \item if $n_1 = 1$, then $b_1 = a_2 + w$ and $m_1 = \left\lfloor \left(
                    n_1 + n_2 \right) / 2 \right\rfloor$.
                \item if $n_1 > 1$, then $b_1 = a_1 + w$ and $m_1 = \left\lfloor \left(
                    n_1 \right) / 2 \right\rfloor$.
            \end{enumerate}
        \item if there exists $k$ such that $a_{k-1} + w \leq b_{i-1} \leq a_k$,
            then
            \[ r_{i-1} = n_1 + \cdots + n_{k-1} - m_1 - \cdots - m_{i-1} \geq 1.\]
            \begin{enumerate}
                \item if $b_{i-1} < a_k$ and $r_{i-1} > 1$, then $b_i =
                    b_{i-1} + w$, and $m_i = \left\lfloor r_{i-1} / 2
                    \right\rfloor$.
                \item else $b_i = a_k + w$, $m_i = \left\lfloor (n_k +
                    r_{i-1}) / 2 \right\rfloor$.
            \end{enumerate}
        \item if $b_{i-1} > a_j$ where $j = \min (i,q)$, then $b_i = b_{i-1} +
            w$, and 
            \[m_i \leq \left\lfloor \left( n_1 + \cdots + n_j - m_1 - \cdots - m_{i-1} \right) / 2
                    \right\rfloor.\] 
    \end{enumerate}
\end{defin}

\begin{prop}
    \label{prop.optiter}
    Given an iterated column $a$ of length $n$, the sequence 
    \[b = b_1^{m_1}\cdots b_p^{m_p},\] 
    or $b = optiter(a)$ is an iterate of $a$.
\end{prop}
\begin{proof}
    The proof follows directly from the definition.
\end{proof}
\begin{prop}
    \quad
    \begin{enumerate}[(i)]
        \item Let $a_n$ be a column of length $n$ and $c_{n-1} = iter(a_n)$ an
            iterated column of $a_n$.  Then
            \[ b_{n-1} = optiter(a_n) \leq iter(a_n) = c_{n-1}.\]
        \item Let $a_n$ and $c_n$ be two columns of length $n$ such that $a_n
            \leq c_n$.  Then
            \[ optiter(a_n) \leq optiter(c_n). \]
    \end{enumerate}
\end{prop}
\begin{proof}
    \quad
    \begin{enumerate}[(i)]
        \item Clearly, $b_1 \leq c_1$ by definition since $b_1$ is chosen to be
            as small as possible.  Moreover, by definition $b_i \leq c_j$ for $i
            \leq j$ since (i) if $b_{i-1} < a_k$ and $r_{i-1} > 1$, meaning there
            are enough elements available to perform a pairing, then $b_i$ is
            again chosen as small as possible, (ii) otherwise $b_i = a_k + w$ which
            is the smallest again, (iii) else $b_{i-1} > a_j$ and $b_i$ is
            chosen as the next smallest element.  Thus 
            \[ b_{n-1}^{m_1 + \cdots + m_i} \leq c_{n-1}^{m_1 + \cdots + m_i},
                \quad 1 \leq i \leq p \]
            so that $b_{n-1} \leq c_{n-1}$.
        \item This is another direct application of the definition and follows
            along the same argument.
    \end{enumerate}
\end{proof}

Clearly, letting $w = 1$ gives the definitions of $iter$ and $optiter$ of the coarse-grain
algorithms as presented in~\cite{j14}.  Definition~\ref{def.optiter} is the
\ASAP algorithm on a single tiled column and can be viewed as the counter part of the
coarse-grain \Greedy algorithm in the tiled case and follows a bottom to top
elimination of the tiles.  In order to preserve the bottom to top elimination,
the weight of the updates must be an integer multiple of the iterated column
weight.  

\begin{thm}
    \label{thm.grasapopt}
    Given a matrix of $p \times q$ tiles, a factorization kernel weight of
    $\gamma$, an elimination kernel weight of $\alpha$, and an update kernel
    weight of $\beta = n \alpha$ for some $n \in \mathbb{N}$, the \GA algorithm
    is optimal in the context of the class of algorithms that progress left to
    right, bottom to top.
\end{thm}
\begin{proof}
    From Theorem~\ref{thm.lowerbound} we have a direct translation from any
    coarse-grain algorithm in this class to the tiled algorithm for the first
    $q-1$ columns.  Thus we are given the time steps at which rows in column $q$
    are available for elimination.  Now we fix the time steps for the
    elimination of the last column and follow whatever tree the algorithm
    provides for this last column.  We can replace the elimination of the first
    $q-1$ columns and updates from these eliminations onto the remaining columns
    with the tiled \Greedy algorithm.  This is possible since the translation
    function is monotonically increasing and we know that \Greedy is optimal for
    the coarse-grain algorithms and therefore optimal for the first $q-1$
    columns in the tiled algorithms.  In another manner of speaking, we slow
    down the eliminations and updates on the first $q-1$ columns when not using
    the tiled \Greedy algorithm.  (An illustrative example is shown in
    Figure~\ref{fig:grfibonacci15x2}.)

    Let $c$ be the next to last column of the coarse-grain elimination
    table which is of length $p-(q-1)+1$. Now letting 
    \[ a = (\gamma + \beta)(q-1) + \beta\cdot coarse(p,q-1) + \gamma\] 
    provides an iterated column of length $p-(q-1)+1$ for the tiled algorithm.
    With $w=\alpha$ we have that $b = optiter(a)$ is an optimal iterated column
    of length $p-q+1$ with the elimination progressing from bottom to top.  This
    can be applied to any tiled algorithm in this class since we only concern
    ourselves with the time steps at which the last column's elements are
    available for elimination.  In other words, this is a speeding up of the
    elimination of the last column while adhering to any restrictions incurred
    from the previous columns. (An illustrative example is shown in
    Figure~\ref{fig:fibonacciasap15x2}.)

    Combining these two ideas, let \Greedy be performed on the first $q-1$
    columns and then \ASAP on the remaining column $q$.  This provides an
    optimal algorithm in this class of algorithms.
%
%
\end{proof}

\linespread{1.0}
\begin{figure*}[htb]
\centering
\subfloat[\MC]{%
    \hspace{-0mm}\resizebox{.475\textwidth}{!}{\includegraphics{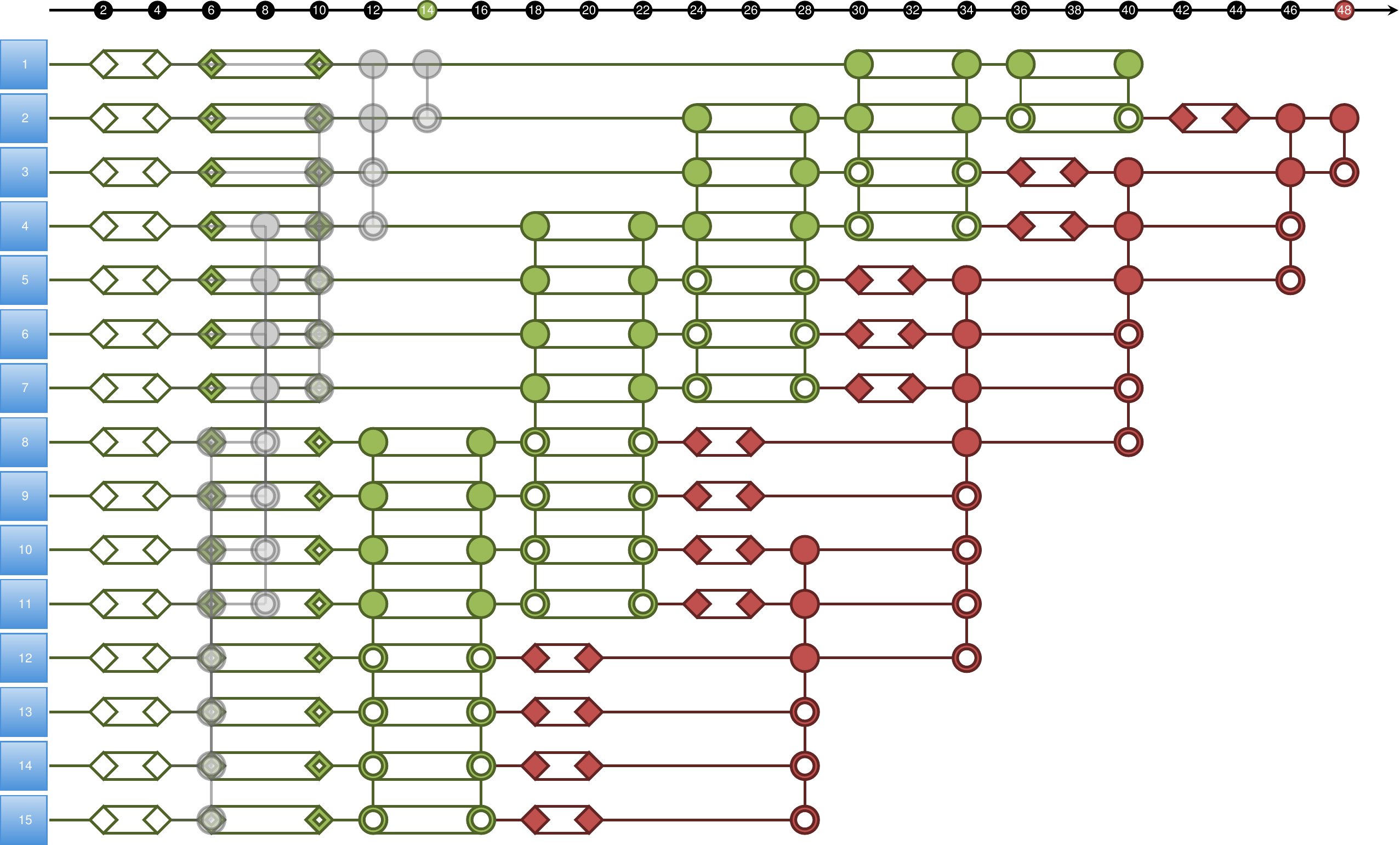}}
}
\subfloat[\Greedy on first $q-1$ columns]{%
    \label{fig:grfibonacci15x2}
    \hspace{0mm}\resizebox{.475\textwidth}{!}{\includegraphics{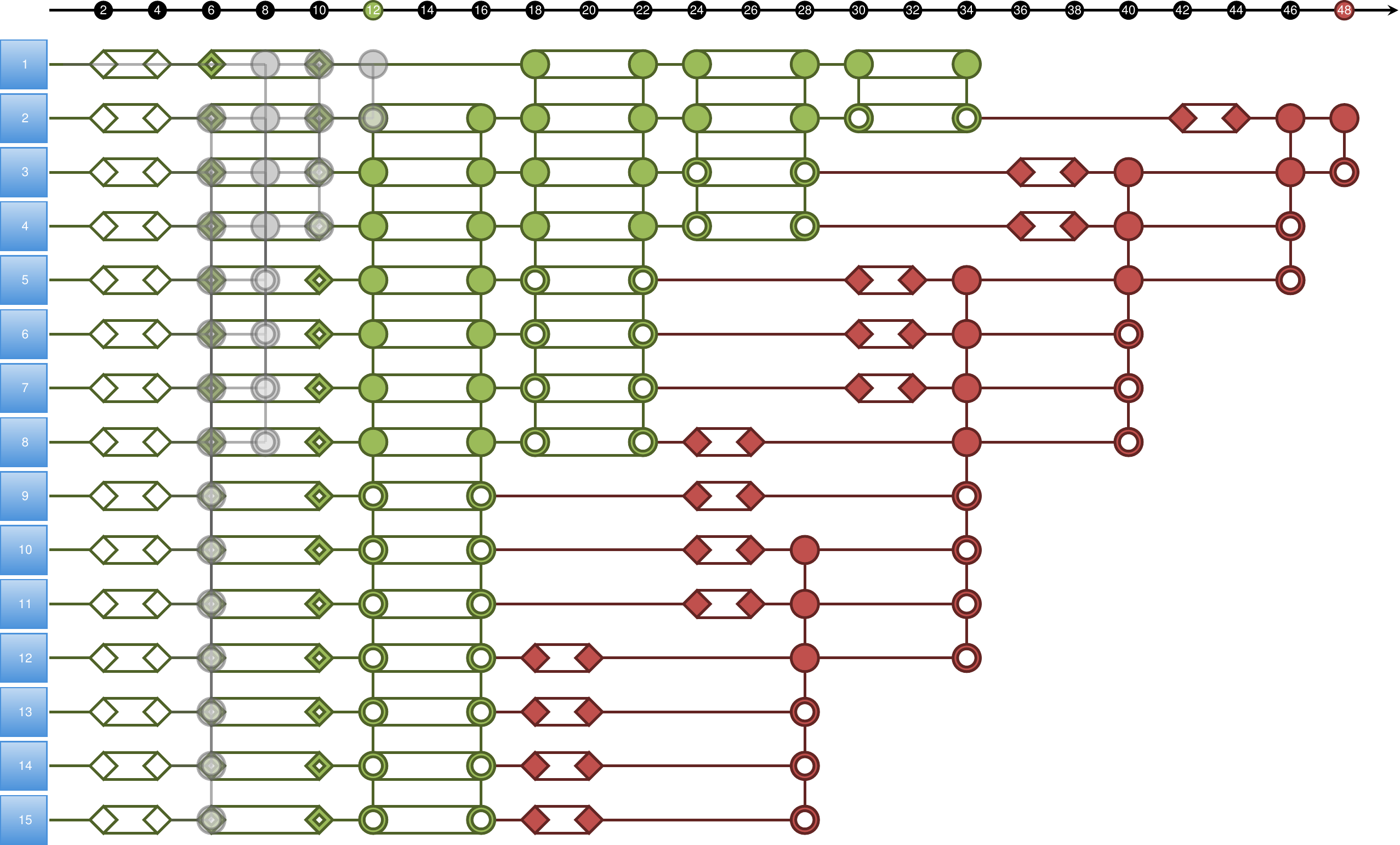}}
}\\
\subfloat[\ASAP on column $q$]{%
    \label{fig:fibonacciasap15x2}
    \hspace{0mm}\resizebox{.475\textwidth}{!}{\includegraphics{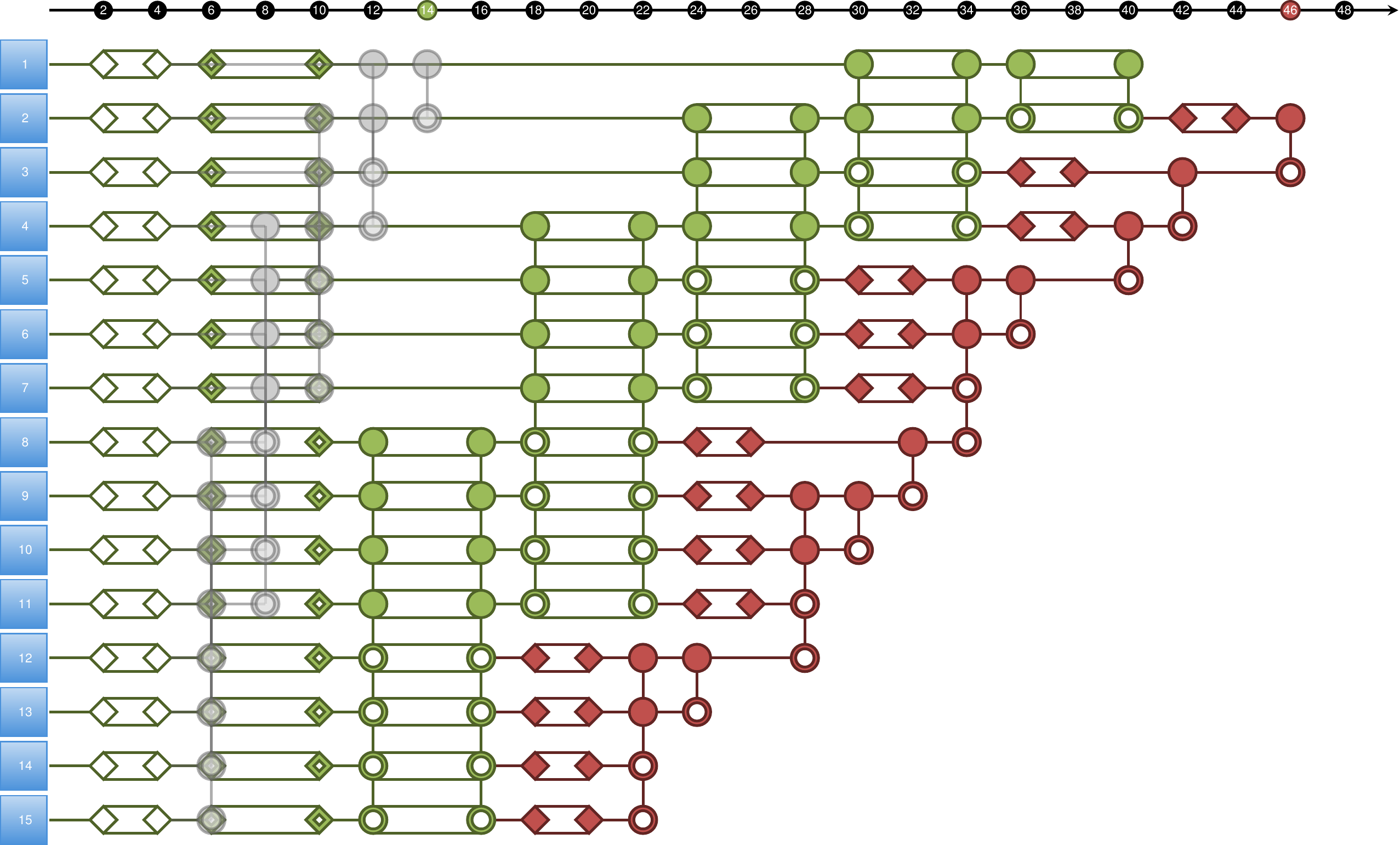}}
}
\caption{
\label{fig:grfibonacci}
Illustration of first and second parts of the proof of Theorem~\ref{thm.grasapopt} using the \MC
algorithm on a matrix of $15\times2$ tiles.}
\end{figure*}
\renewcommand{\baselinestretch}{\normalspace}

\linespread{1.0}
\begin{figure*}[htb]
\centering
\subfloat[\Greedy]{%
    \label{fig:greedy15x2}
    \hspace{-0mm}\resizebox{.475\textwidth}{!}{\includegraphics{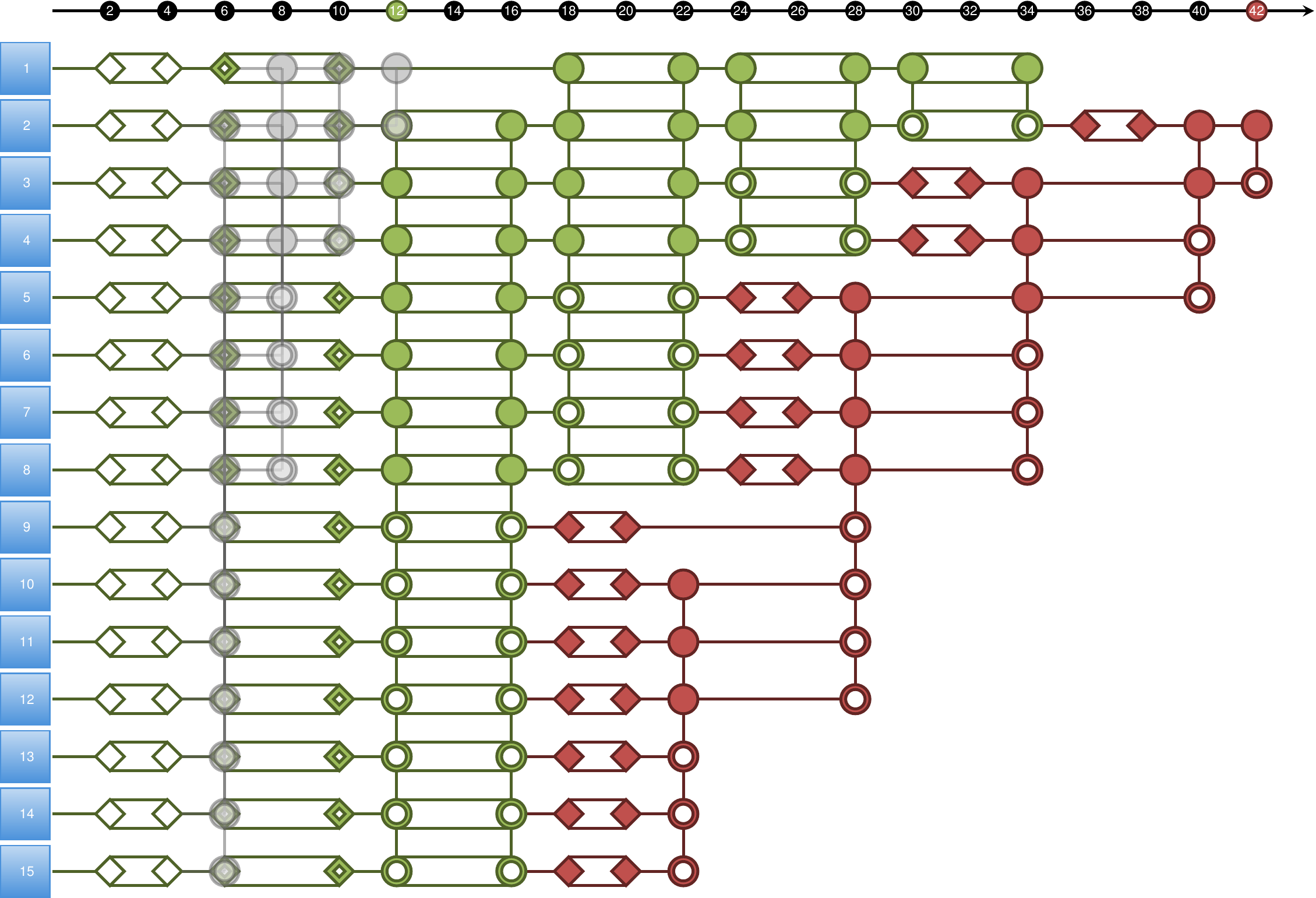}}
}
\subfloat[\GA]{%
    \label{fig:grasap15x2}
    \hspace{0mm}\resizebox{.475\textwidth}{!}{\includegraphics{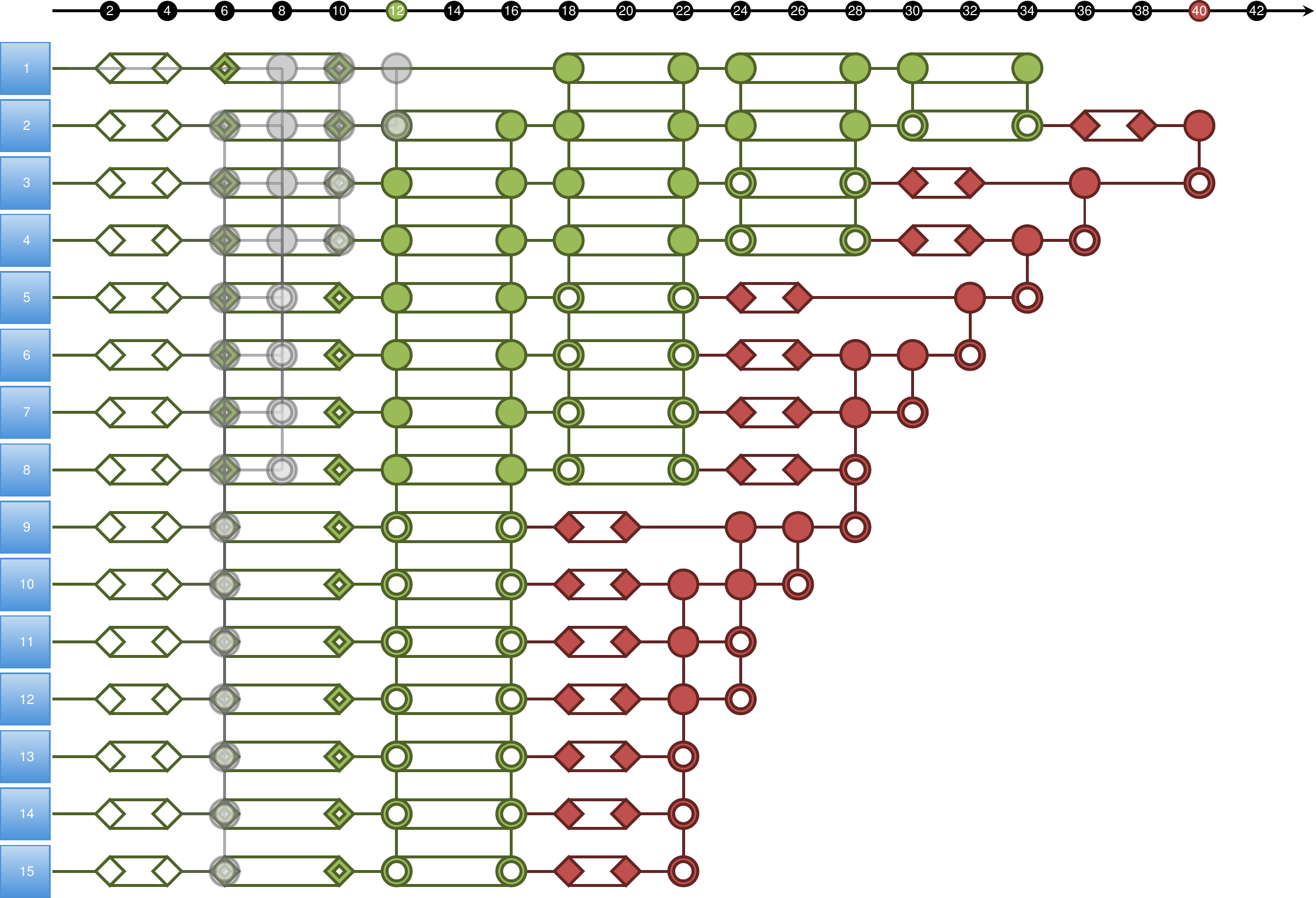}}
}
\caption{
\label{fig:greedyvsgrasap}
\Greedy versus \GA on matrix of $15\times2$ tiles.}
\end{figure*}
\renewcommand{\baselinestretch}{\normalspace}

In Figure~\ref{fig:greedyvsgrasap} we provide an illustrated example of the
\Greedy and \GA algorithms on a matrix of $15 \times 2$ tiles where the
operations are given by Figure~\ref{fig:treekernels}.

It can be seen that \GA finishes before \Greedy since \Greedy must wait and
progress with the same elimination scheme as the coarse-grain algorithm while
\GA can begin eliminating in the last column as soon as a pair of tiles becomes
available.  (The elimination of the first column is shown in light gray.)  

\linespread{1.0}
\begin{figure*}[htb]
\centering
    \resizebox{.45\textwidth}{!}{\includegraphics{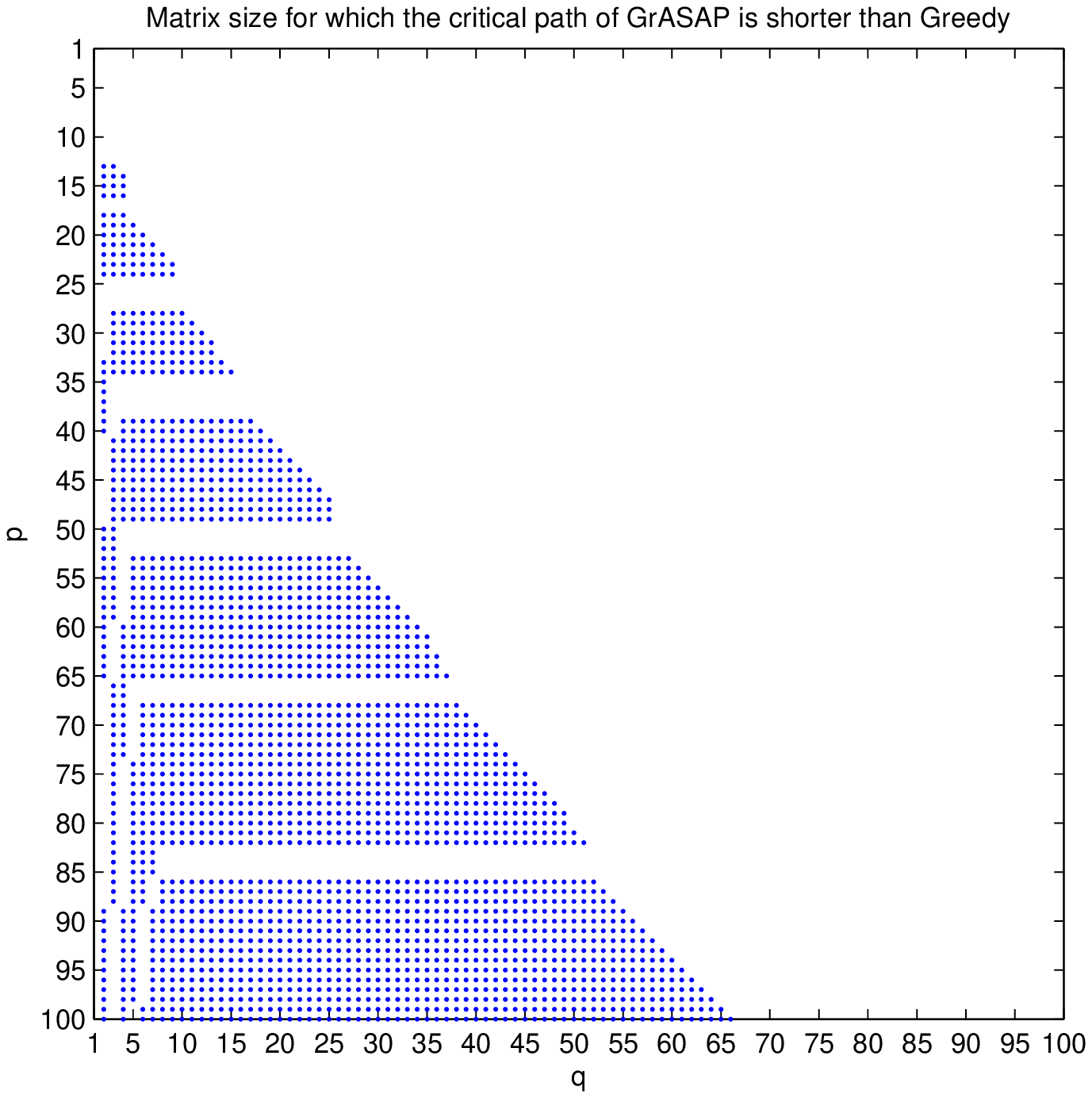}}
\caption{ Tiled matrices of $p \times q$ where the critical path length of \GA
is shorter than that of \Greedy for $1 \leq p \leq 100$ and $1 \leq q \leq p$.}
\label{fig:spyGrASAPvsGreedy}
\end{figure*}
\renewcommand{\baselinestretch}{\normalspace}

We have analyzed the critical path length of \GA versus that of \Greedy for
tiled matrices $p \times q$ where $1 \leq p \leq 100$ and $1 \leq q \leq p$ (see
Figure~\ref{fig:spyGrASAPvsGreedy}).  In all cases where there is a difference
(which is just over 44\% of the cases), the difference is always two time steps.  

We now show that without having the update kernel weight an integer multiple of
the elimination kernel weight, the bottom to top progression is nullified and we
cannot provide optimality of the algorithm.  

\linespread{1.0}
\begin{table}[htbp]
    \centering
    \begin{tabular}{|c|c|c|c|}
        \hline
        $a_{11}$ & (1) & (2) & (3)\\
        \hline
               6 &     &     &    \\
               6 & 13  & 14  & 12 \\
               6 & 11  &  8  & 10 \\
               6 &  9  &  8  & 10 \\
               6 &  9  & 12  &  8 \\
               6 &  9  &  9  &  8 \\
               6 &  7  &  7  &  8 \\
               6 &  7  &  7  &  8 \\
               6 &  5  &  5  &  5 \\
               6 &  5  &  5  &  5 \\
               6 &  5  &  5  &  5 \\
        \hline
    \end{tabular}
    \caption{Three schemes applied to a column whose update kernel weight is not
    an integer multiple of the elimination kernel weight.}
    \label{tab:nonintelim}
\end{table}
\renewcommand{\baselinestretch}{\normalspace}

Assume that the update kernel weight is 3 and the elimination kernel weight is
2.  Let $a_{11} = 3^{7}6^{4}$ be  column from some elimination scheme.  We shall
apply three iterated schemes to this column: (1) an \ASAP scheme that progresses
from bottom to top, (2) an \ASAP scheme that can progress in any manner, and (3)
an \ASAP scheme which may provide a lag.  

In Table~\ref{tab:nonintelim} we clearly see that elimination scheme (3)
provides the best time for the algorithm.  The reason is that enough of a lag
was provided such that a binomial tree could progress without hindrance.
Therefore without integer multiple weights on the update kernel, we cannot know
which scheme will be optimal.

The PLASMA library provides more algorithms, that can be informally described as
trade-offs between \FT and \BT.  (We remind that \FT is the same as algorithm as
\SK.) These algorithms are referred to as \PT in all the following, and differ
by the value of an input parameter called the \emph{domain size} $\BS$.  This
domain size can be any value between $1$ and $p$, inclusive.  Within a domain,
that includes $\BS$ consecutive rows, the algorithm works just as \FT: the first
row of each domain acts as a local panel and is used to zero out the tiles in
all the other rows of the domain. Then the domains are merged: the panel rows
are zeroed out by a binary tree reduction, just as in \BT.  As the algorithm
progresses through the columns, the domain on the very bottom is reduced
accordingly, until such time that there is one less domain.  For the case that
$\BS=1$, \PT follows a binary tree on the entire column, and for $\BS = p$, the
algorithm executes a flat tree on the entire column.  It seems very difficult
for a user to select the domain size $\BS$ leading to best performance, but it
is known that $\BS$ should increase as $q$ increases. Table~\ref{tab.tiled}
shows the time-steps of \PT with a domain size of $\BS =5$.  In the experiments
of \Section\ref{sec.experiments}, we use all possible values of $\BS$ and retain
the one leading to the best value.

\section{Experimental results}
\label{sec.experiments}

All experiments were performed on a 48-core machine composed of eight
hexa-core AMD Opteron 8439 SE (codename Istanbul) processors running at 2.8
GHz. Each core has a theoretical peak of 11.2 Gflop/s with a peak of 537.6
Gflop/s for the whole machine. The Istanbul micro-architecture is a NUMA
architecture where each socket has 6 MB of level-3 cache and each processor has
a 512 KB level-2 cache and a 128 KB level-1 cache.  After having benchmarked
the AMD ACML and Intel MKL BLAS libraries, we selected MKL (10.2) since it
appeared to be slightly faster in our experimental context.  Linux 2.6.32 and
Intel Compilers 11.1 were also used in conjunction with PLASMA 2.3.1.


\linespread{1.0}
\begin{figure*}[htb]
\centering
\subfloat[Upper bound (double complex)]{%
    \label{fig:fig_tt_pic1_p40_z.th}%
    \hspace{-2mm}\resizebox{.475\textwidth}{!}{\includegraphics{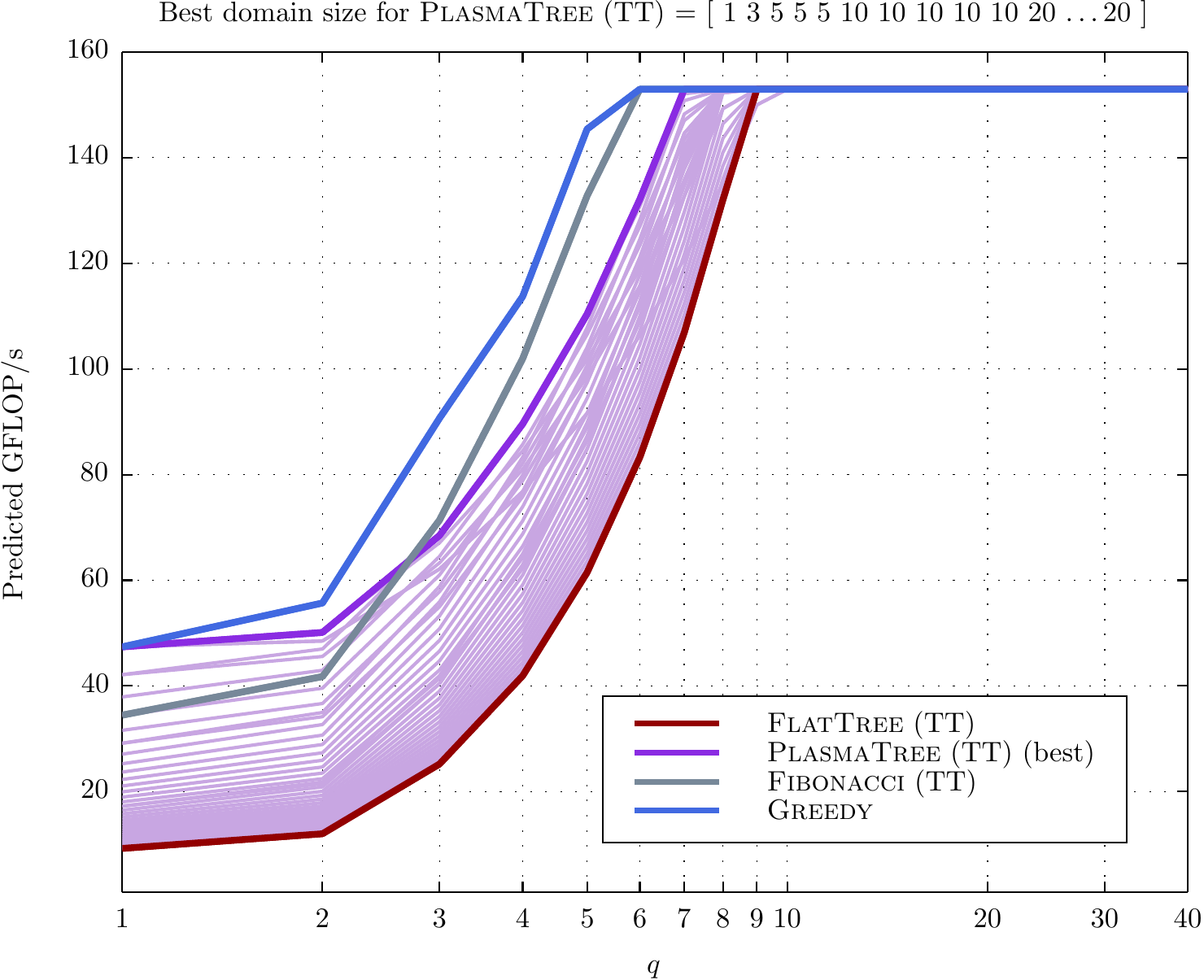}}%
}
\subfloat[Experimental (double complex)]{%
    \label{fig:fig_tt_pic1_p40_z.exp}%
    \hspace{7mm}\resizebox{.475\textwidth}{!}{\includegraphics{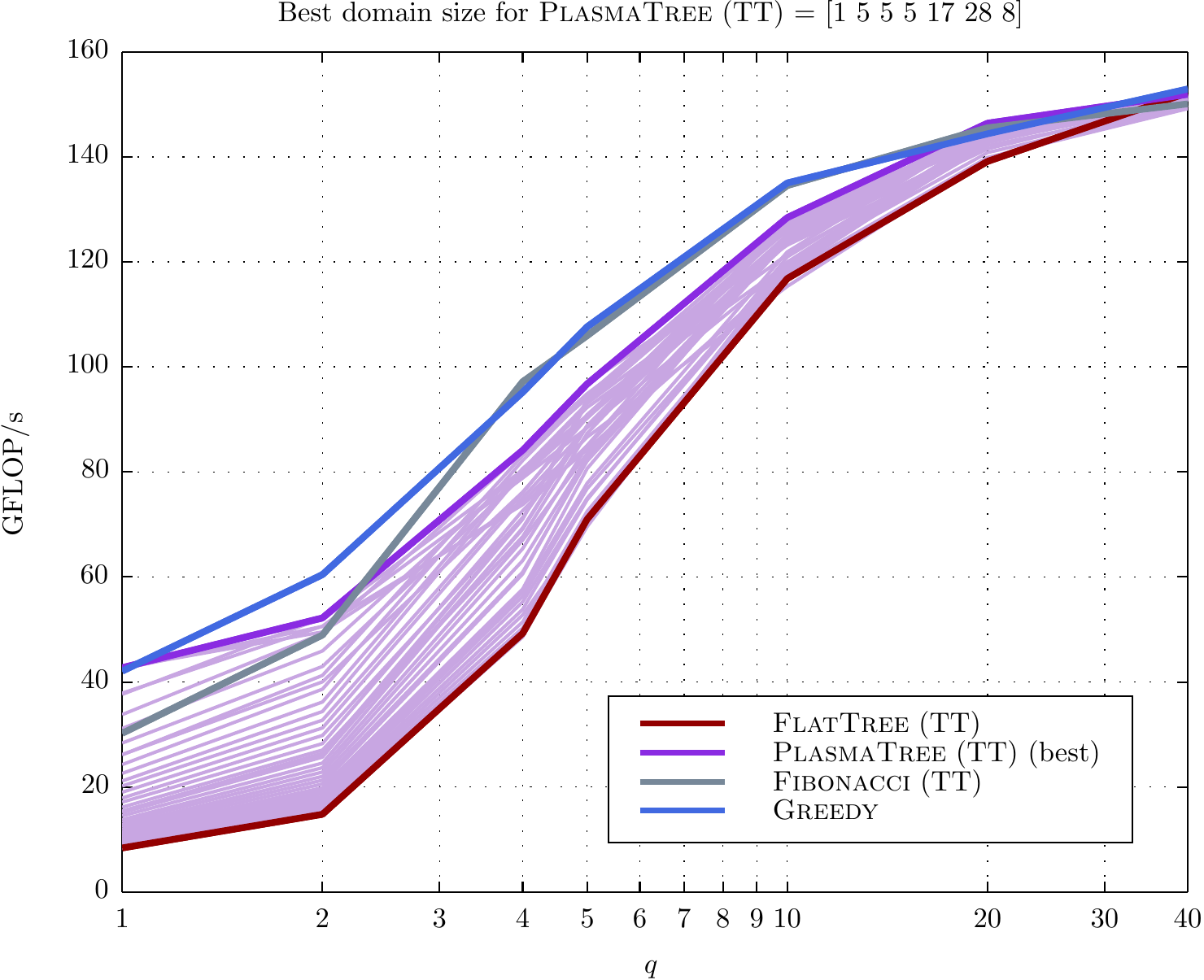}}%
}
\\
\subfloat[Upper bound (double)]{%
    \label{fig:fig_tt_pic1_p40_d.th}%
    \hspace{-2mm}\resizebox{.475\textwidth}{!}{\includegraphics{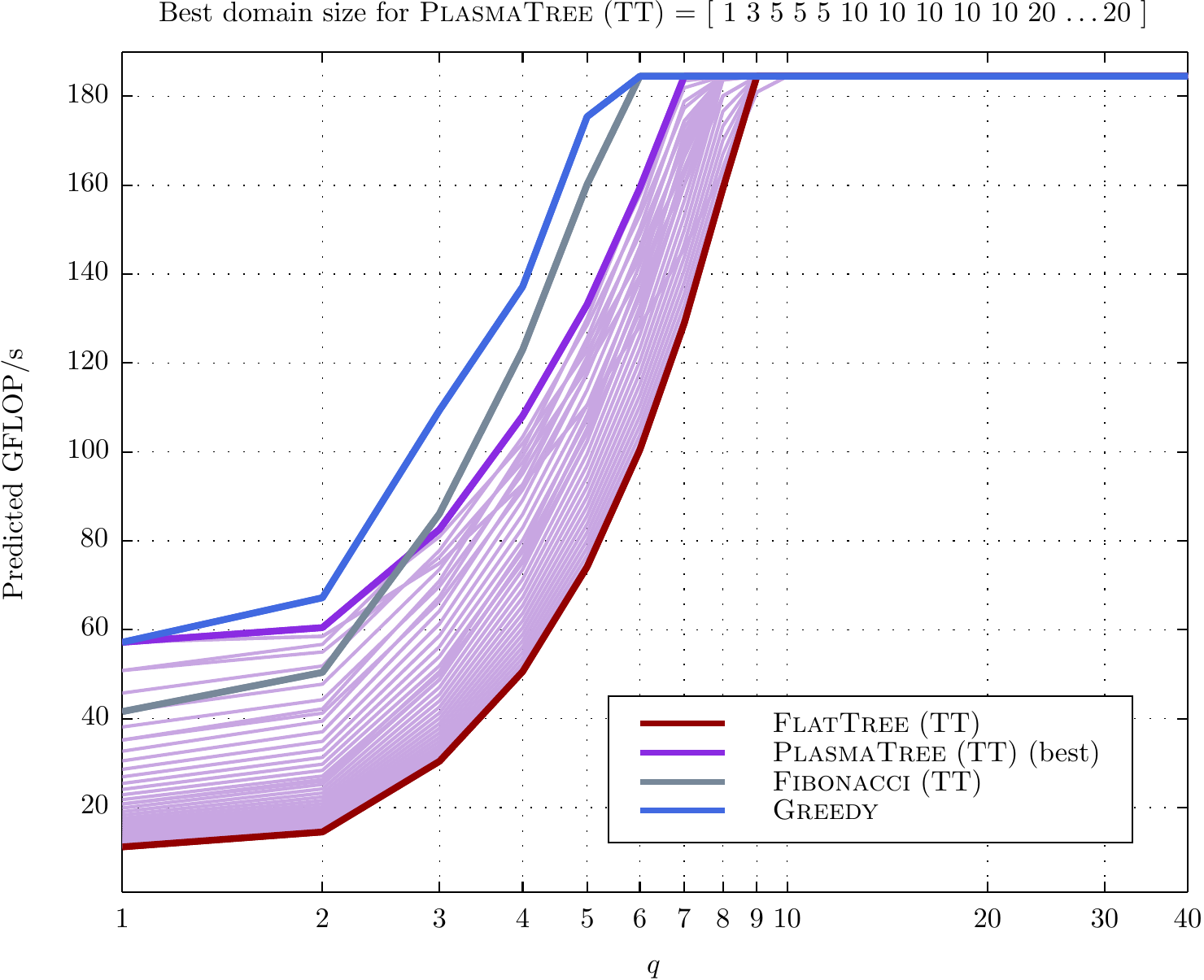}}%
}
\subfloat[Experimental (double)]{%
    \label{fig:fig_tt_pic1_p40_d.exp}%
    \hspace{7mm}\resizebox{.475\textwidth}{!}{\includegraphics{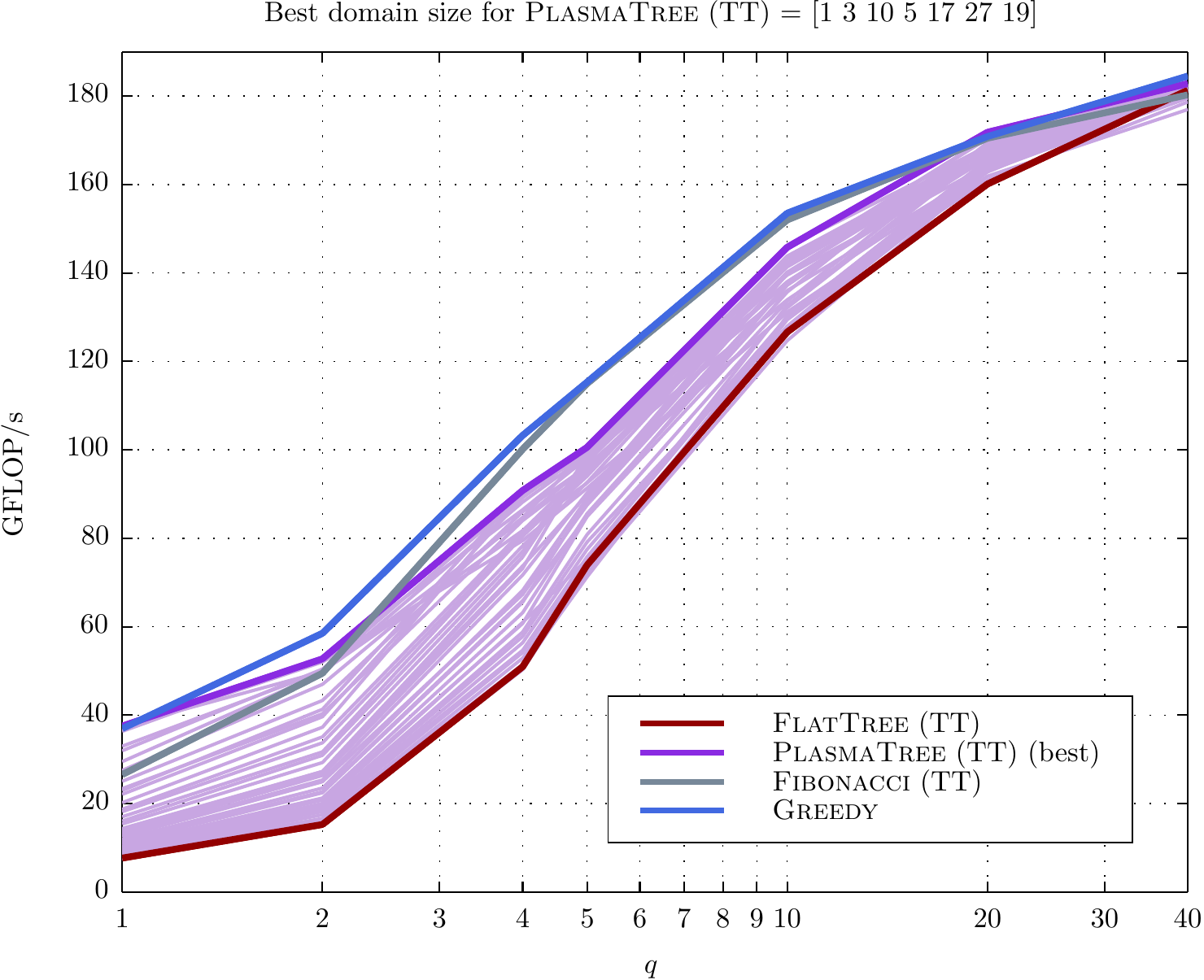}}%
}
\caption{\label{fig:fig_tt_pic1_p40}
Upper bound and experimental performance of QR factorization - \emph{TT} kernels}
\end{figure*}
\renewcommand{\baselinestretch}{\normalspace}

\linespread{1.0}
\begin{figure*}[htb]
\centering
\subfloat[Theoretical CP length]{%
   \hspace{-0mm}\resizebox{.475\textwidth}{!}{\includegraphics{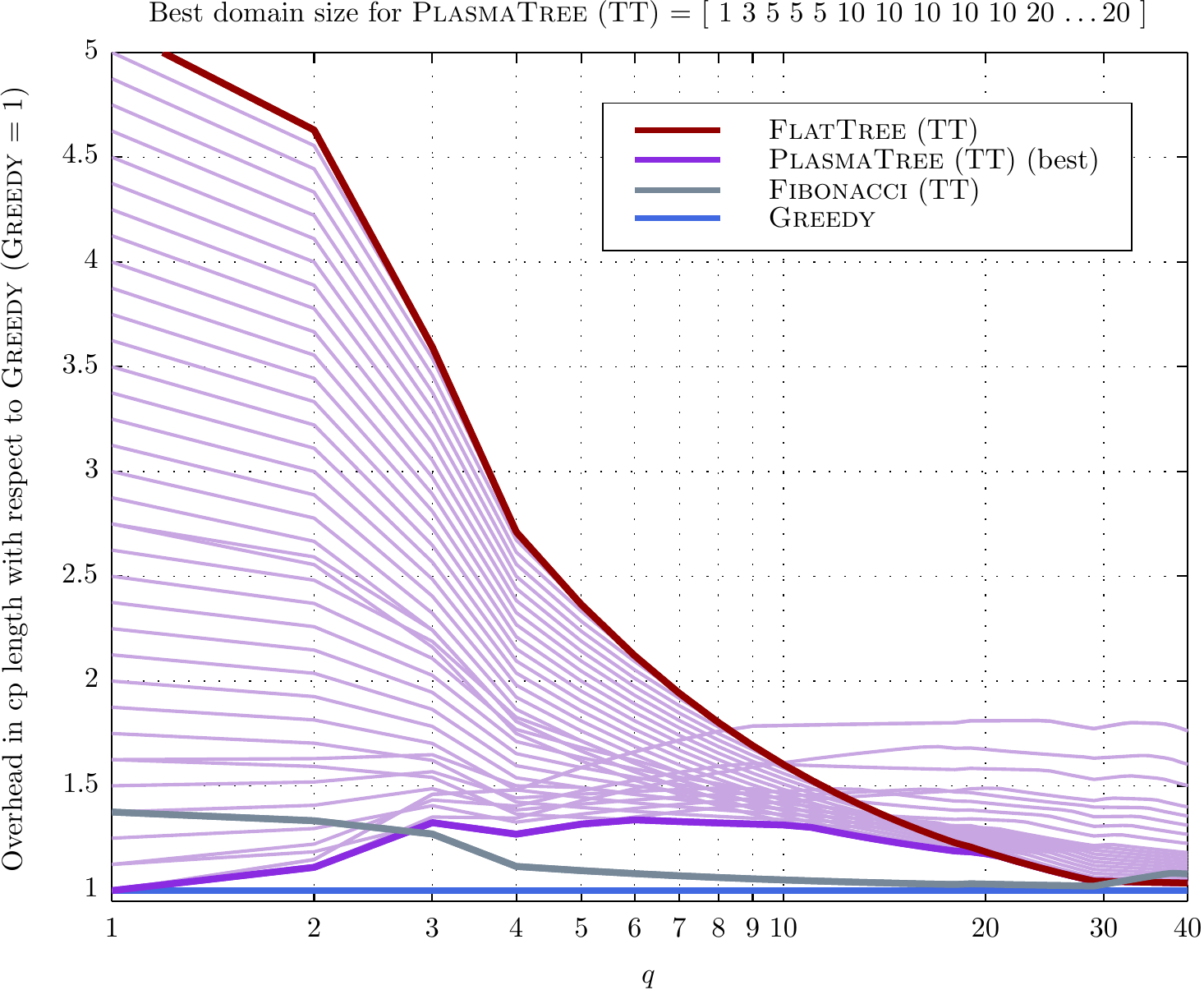}}%
}\\
\subfloat[Experimental (double complex)]{%
    \hspace{-1mm}\resizebox{.475\textwidth}{!}{\includegraphics{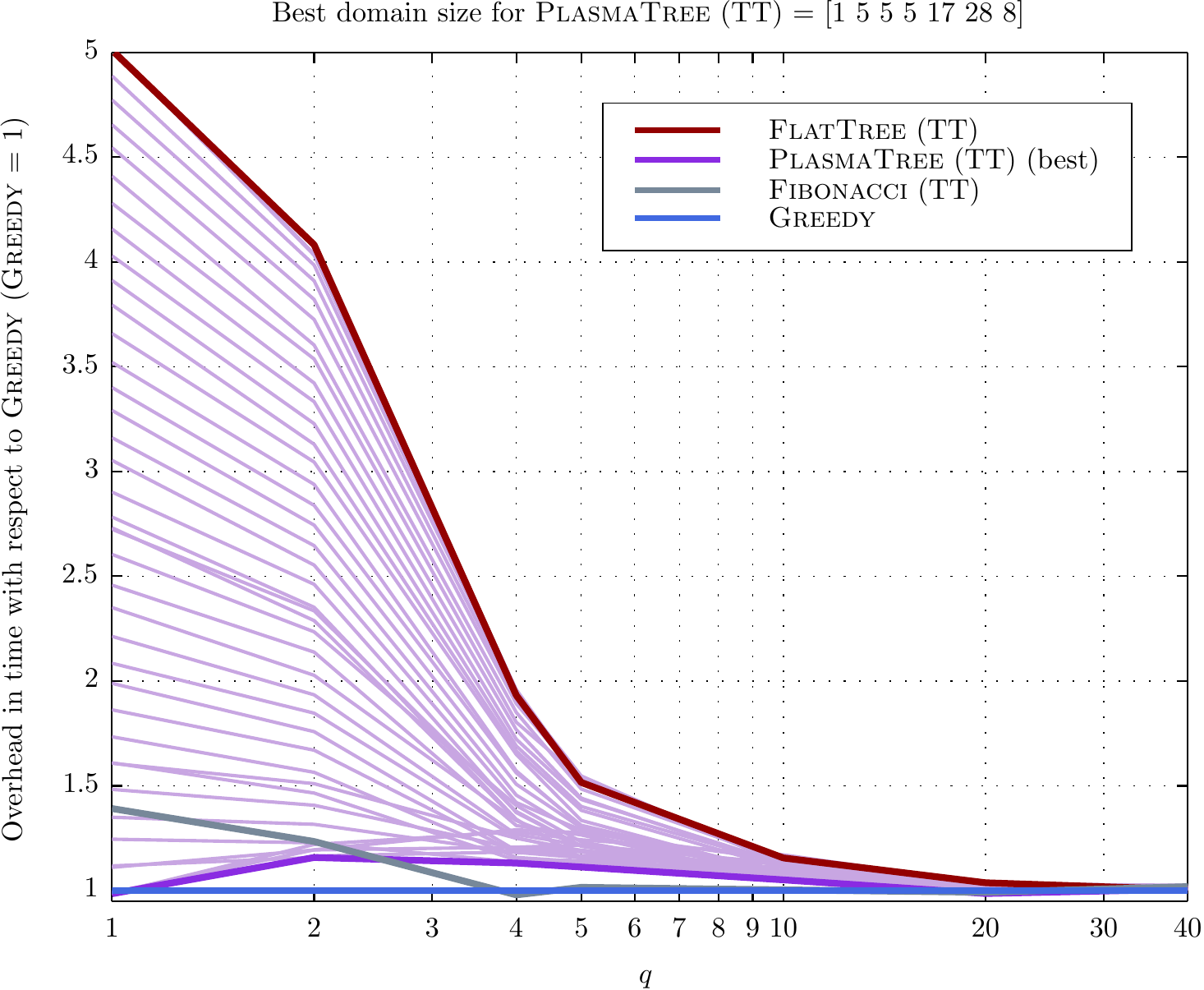}}%
}
\subfloat[Experimental (double)]{%
    \hspace{7mm}\resizebox{.475\textwidth}{!}{\includegraphics{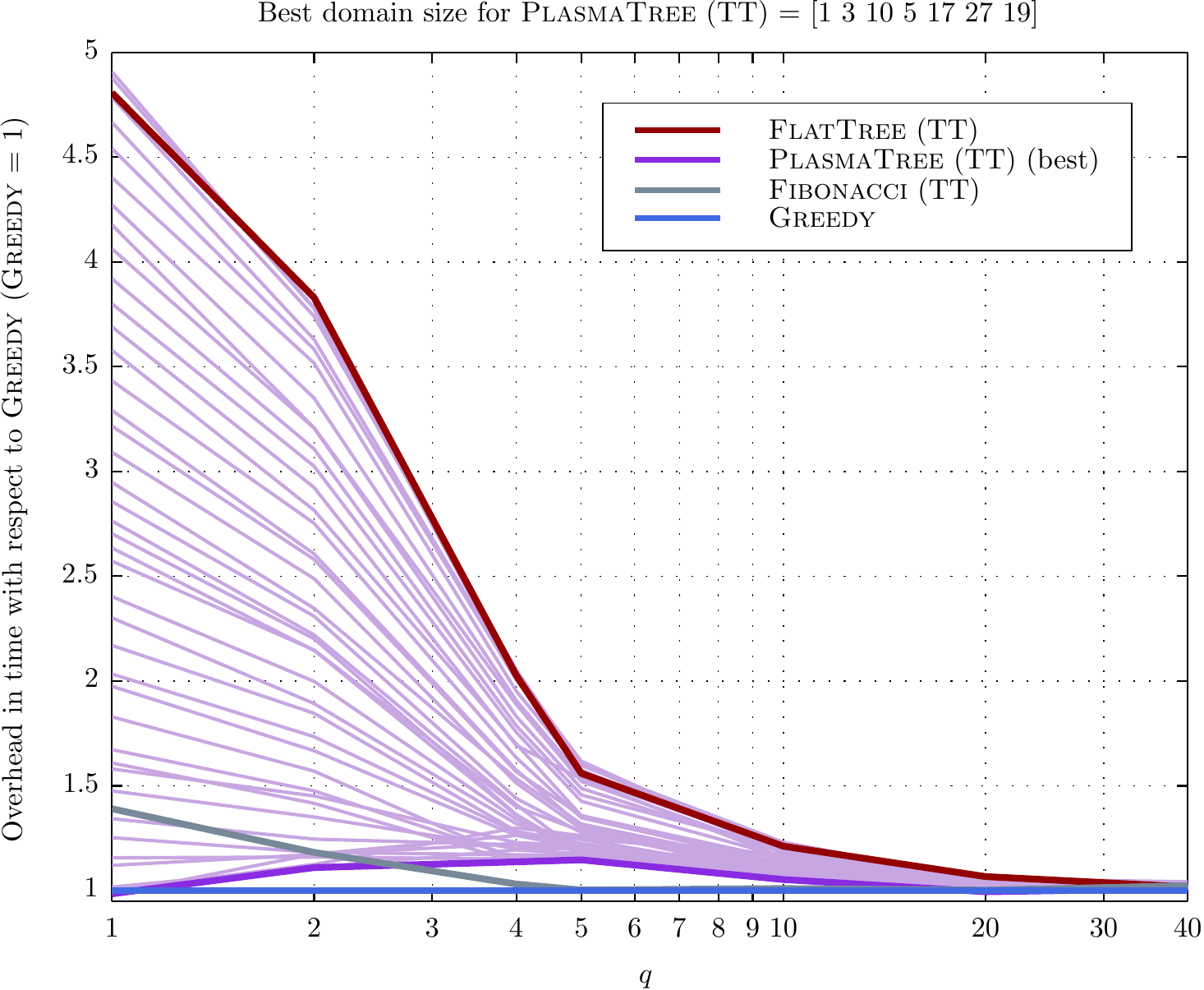}}%
}
\caption{\label{fig:fig_tt_pic2_p40}
Overhead in terms of critical path length and time with respect to \Greedy (\Greedy = 1)
}
\end{figure*}
\renewcommand{\baselinestretch}{\normalspace}

For all results, we show both double and double complex precision, using all
48 cores of the machine.   The matrices are of size $m=8000$ and $200 \leq n \leq
8000$.  The tile size is kept constant at $n_b=200$, so that the
matrices can also be viewed as $p \times q$ tiled matrices where $p=40$ and $1
\leq q \leq 40$. All kernels use an inner blocking parameter of $i_b=32$.

In double precision, an FMA (``{\em fused multiply-add}'', $ y \leftarrow
\alpha x + y$) involves three double precision numbers for two flops, but these
two flops can be combined into one FMA and thus completed in one cycle.  In double
complex precision, the operation $ y \leftarrow \alpha x + y$ involves six
double precision numbers for eight flops; there is no FMA.  The ratio of
computation/communication is therefore, potentially, four times higher in double
complex precision than in double precision. Communication aware algorithms are
much more critical in double precision than in double complex precision.

For each experiment, we provide a comparison of the theoretical performance to
the actual performance.  The theoretical performance is obtained by modeling
the limiting factor of the execution time as either the critical path, or the
sequential time divided by the number of processors.  This is similar in approach to the Roofline
model~\cite{Williams:2009:RIV:1498765.1498785}.  Taking $\gamma_{seq}$ as the
sequential performance, $T$ as the total number of flops, $cp$ as the length of
the critical path, and $P$ as the number of processors, the upper bound on
performance, $\gamma_{ub}$, is
\[ \gamma_{ub} = \frac{\gamma_{seq} \cdot T}{\max \left( \frac{T}{P}, cp \right)} \]
Figures~\ref{fig:fig_tt_pic1_p40_z.th} and~\ref{fig:fig_tt_pic1_p40_d.th} depict
the upper bound on performance of all algorithms which use the \emph{Triangle on top
of triangle} kernels.  Since \PT provides an additional tuning parameter of the domain size,
we show the results for each value of this parameter as well as the composition
of the best of these domain sizes. Again, it is not evident what the domain size
should be for the best performance, hence our exhaustive search.

Part of our comprehensive study also involved comparisons made to the
Semi-Parallel Tile and Fully-Parallel Tile CAQR algorithms found
in~\cite{Hadri_enhancingparallelism} which are much the same as those found in
PLASMA.  As with PLASMA, the tuning parameter $\BS$ controls the domain size
upon which a flat tree is used to zero out tiles below the root tile within the
domain and a binary tree is used to merge these domains.  Unlike PLASMA, it is
not the bottom domain whose size decreases as the algorithm progresses through
the columns, but instead is the top domain.  In this study, we found that the
PLASMA algorithms performed identically or better than these algorithms and
therefore we do not report these comparisons.

\linespread{1.0}
\begin{figure*}[htb]
\centering
\subfloat[Theoretical CP length]{%
    \label{fig:fig_tt_pic3_p40.th}
    \hspace{0mm}\resizebox{.475\textwidth}{!}{\includegraphics{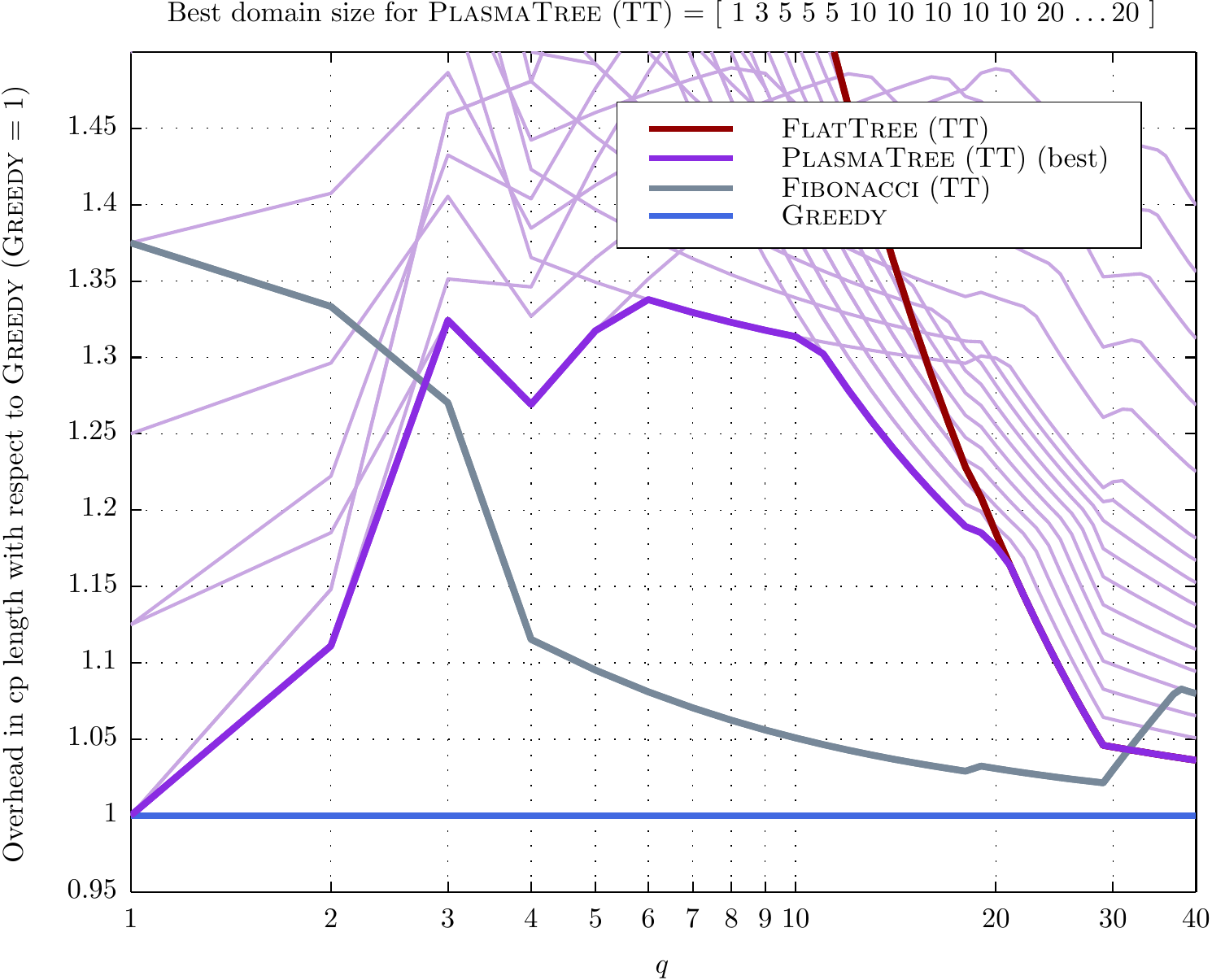}}%
}\\
\subfloat[Experimental (double complex)]{%
    \label{fig:fig_tt_pic3_p40_z.exp}
    \hspace{-2mm}\resizebox{.475\textwidth}{!}{\includegraphics{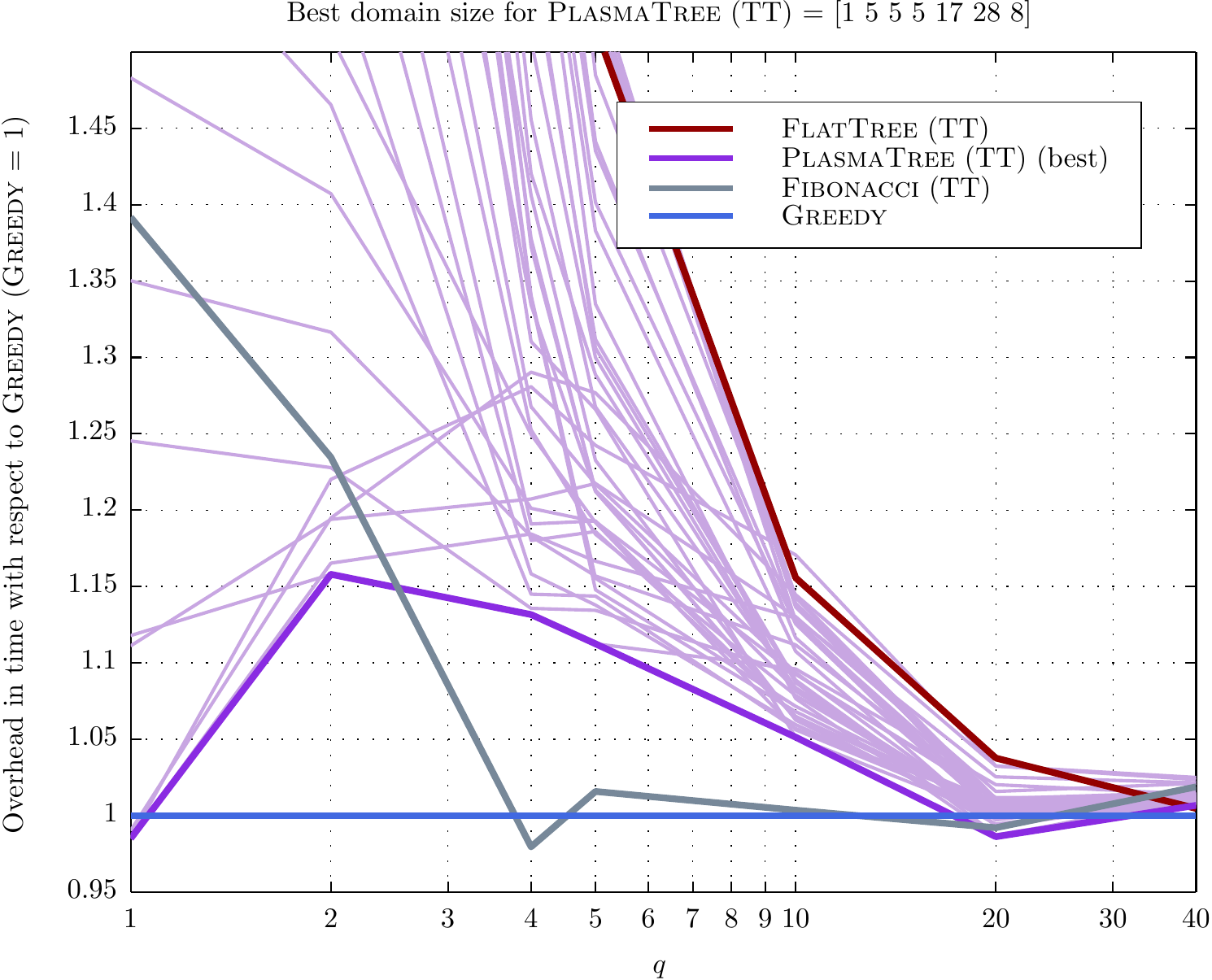}}%
}
\subfloat[Experimental (double)]{%
    \label{fig:fig_tt_pic3_p40_d.exp}
    \hspace{7mm}\resizebox{.475\textwidth}{!}{\includegraphics{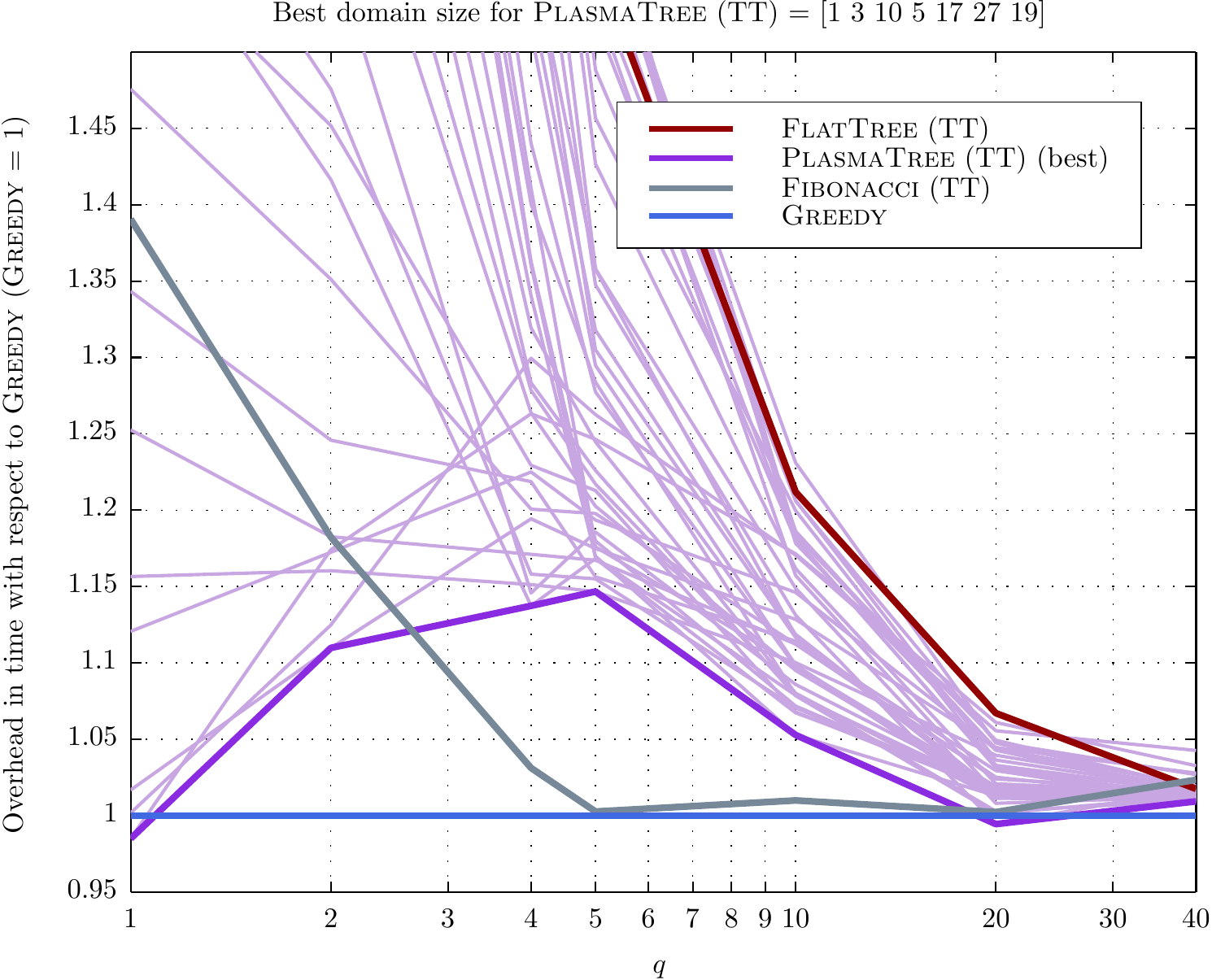}}%
}
\caption{\label{fig:fig_tt_pic3_p40}
Overhead in terms of critical path length and time with respect to \Greedy (\Greedy = 1)
}
\end{figure*}
\renewcommand{\baselinestretch}{\normalspace}

Figure~\ref{fig:fig_tt_pic1_p40_z.exp} and~\ref{fig:fig_tt_pic1_p40_d.exp}
illustrate the experimental performance reached by \Greedy, \MC and \PT
algorithms using the \emph{TT (Triangle on top of triangle)} kernels. In both
cases, double or double complex precision, the performance of \Greedy is better
than \PT even for the best choice of domain size.  Moreover, as expected from
the analysis in \Section\ref{sec:TiledAlgorithms}, \Greedy outperforms \MC the
majority of the time.  Furthermore, we see that, for rectangular matrices, the
experimental performance in double complex precision matches the upper bound on
performance.  This is not the case for double precision because communications
have higher impact on performance.

While it is apparent that \Greedy does achieve higher levels of performance,
the percentage may not be as obvious.  To that end, taking \Greedy as the
baseline, we present in Figure~\ref{fig:fig_tt_pic2_p40} the theoretical,
double, and double complex precision overhead for each algorithm that uses the
\emph{Triangle on top of triangle} kernel as compared to \Greedy. These
overheads are respectively computed in terms of critical path length and time.
At a smaller scale (Figure~\ref{fig:fig_aa_pic3_p40}), it can be seen that
\Greedy can perform up to 13.6\% better than \PT.

For all matrix sizes considered, $p=40$ and $1 \leq q \leq 40$, in the
theoretical model, the critical path length for \Greedy is either the same as
that of \PT ($q=1$) or is up to 25\% shorter than \PT ($q=6$).  Analogously, the
critical path length for \Greedy is at least 2\% to 27\% shorter than that of
\MC.  In the experiments, the matrix sizes considered were $p=40$ and $q \in \{
1, 2, 4, 5, 10, 20, 40\}$.  In double precision, \Greedy has a decrease of at
most 1.5\% than the best \PT ($q=1$) and a gain of at most 12.8\% than the best
\PT ($q=5$).  In double complex precision, \Greedy has a decrease of at most
1.5\% than the best \PT ($q=1$) and a gain of at most 13.6\% than the best \PT
($q=2$).  Similarly, in double precision, \Greedy provides a gain of 2.6\% to
28.1\% over \MC and in double complex precision, \Greedy has a decrease of at
most 2.1\% and a gain of at most 28.2\% over \MC.

Although it is evidenced that \PT does not vary too far from \Greedy or \MC,
one must keep in mind that there is a tuning parameter involved and we choose
the best of these domain sizes for \PT to create the composite result, whereas
with \Greedy, there is no such parameter to consider.  Of particular interest
is the fact that \Greedy always performs better than any other
algorithm\footnote{When $q=1$, \Greedy and \FT exhibit close performance. They
both perform a binary tree reduction, albeit with different row pairings.} for
$ p \gg q$. In the scope of \PT, a domain size $\BS=1$ will force the use of a
binary tree so that both \Greedy and \PT behave the same. However, as the
matrix tends more to a square, i.e., $q$ tends toward $p$, we observe that the
performance of all of the algorithms, including \FT, are on par with \Greedy.
As more columns are added, the parallelism of the algorithm is increased and
the critical path becomes less of a limiting factor, so that the performance of
the kernels is brought to the forefront.  Therefore, all of the algorithms are
performing similarly since they all share the same kernels.

\linespread{1.0}
\begin{figure*}[htb]
\centering
\subfloat[Factorization kernels]{
    \hspace{-2mm}\resizebox{.475\textwidth}{!}{\includegraphics{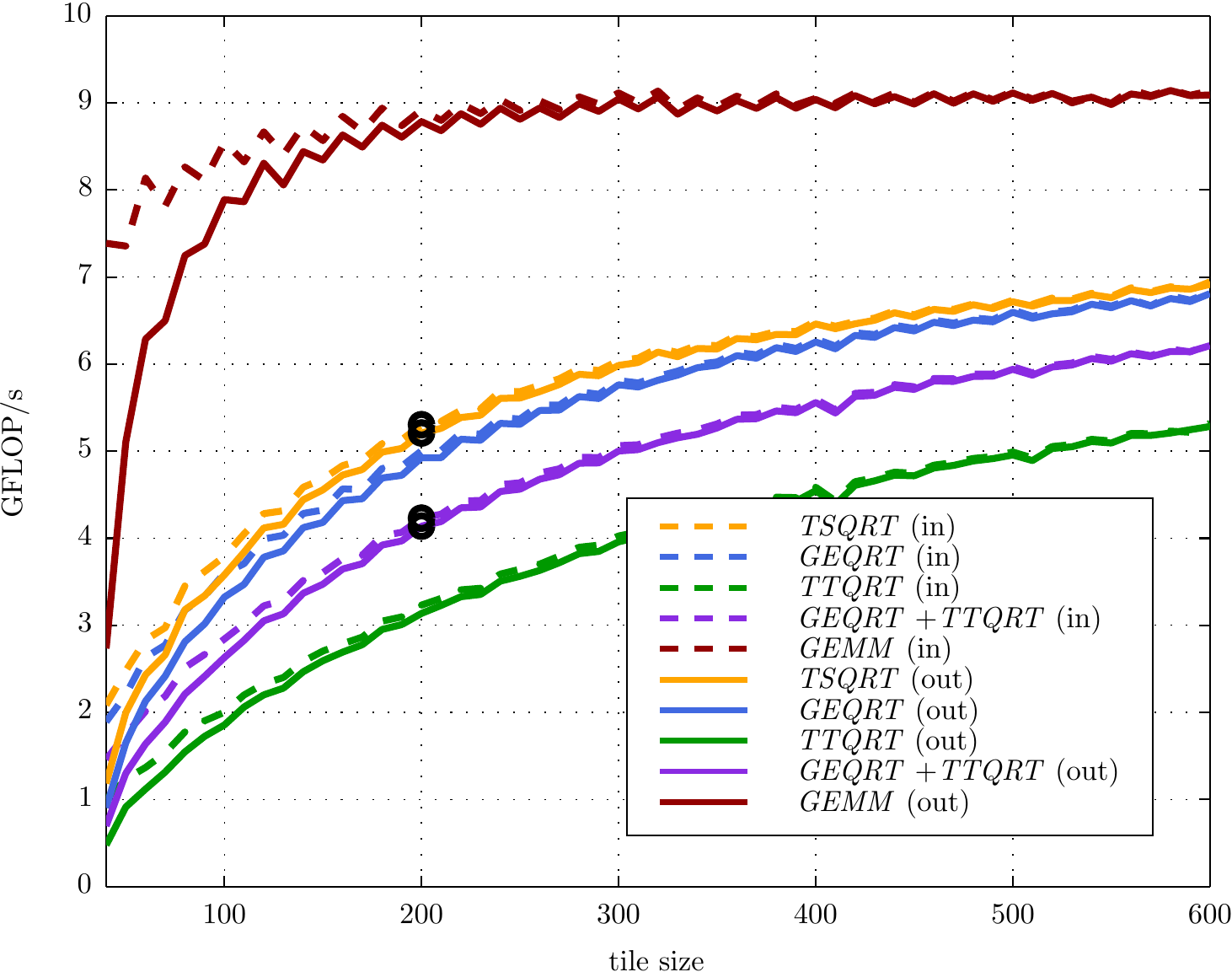}}%
}
\subfloat[Update kernels]{
    \hspace{7mm}\resizebox{.475\textwidth}{!}{\includegraphics{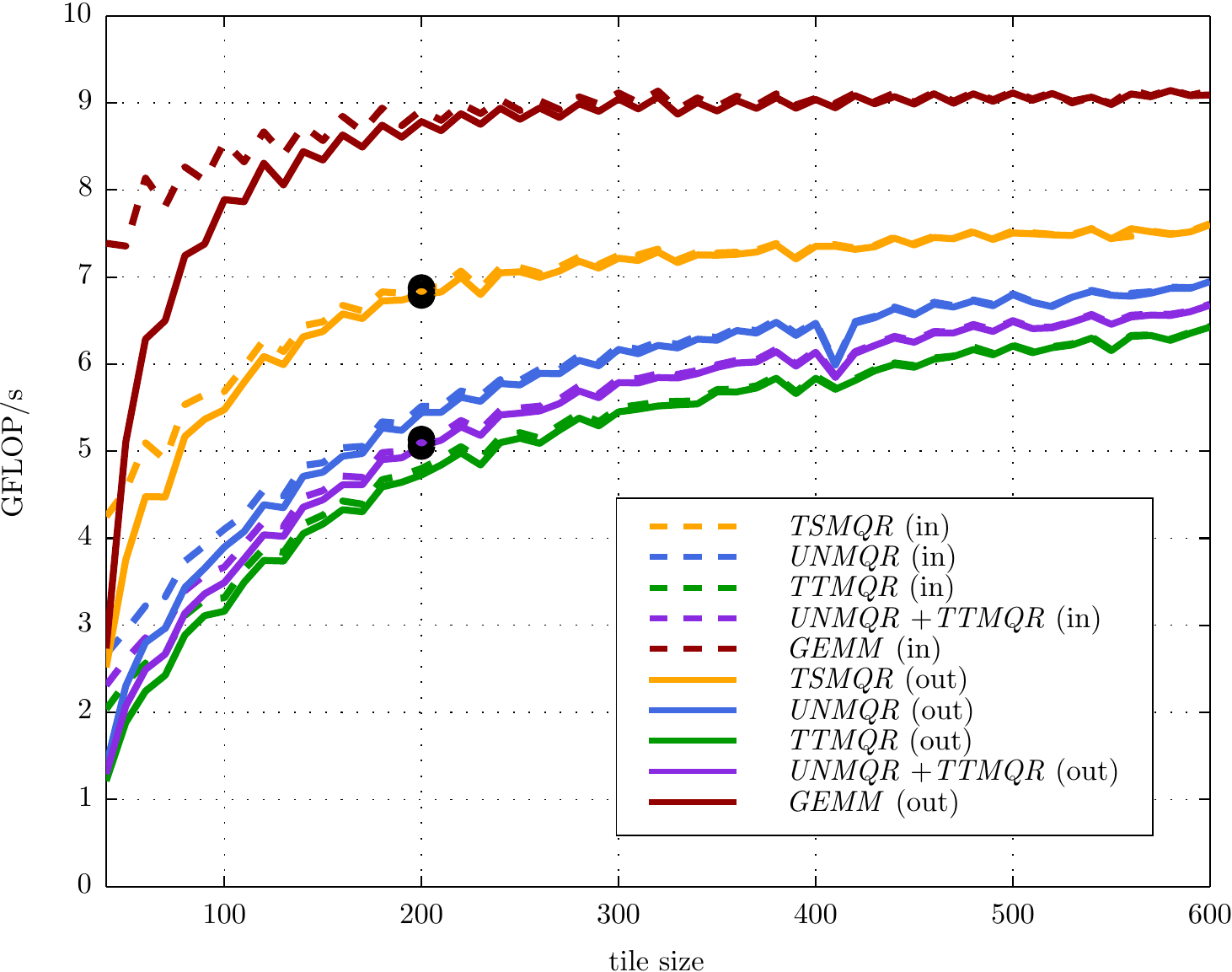}}%
}
    \caption{\label{fig:perf_kernels_z}
    Kernel performance for double complex precision}
\end{figure*}

\begin{figure*}[htb]
\centering
\subfloat[Factorization kernels]{
    \hspace{-2mm}\resizebox{.475\textwidth}{!}{\includegraphics{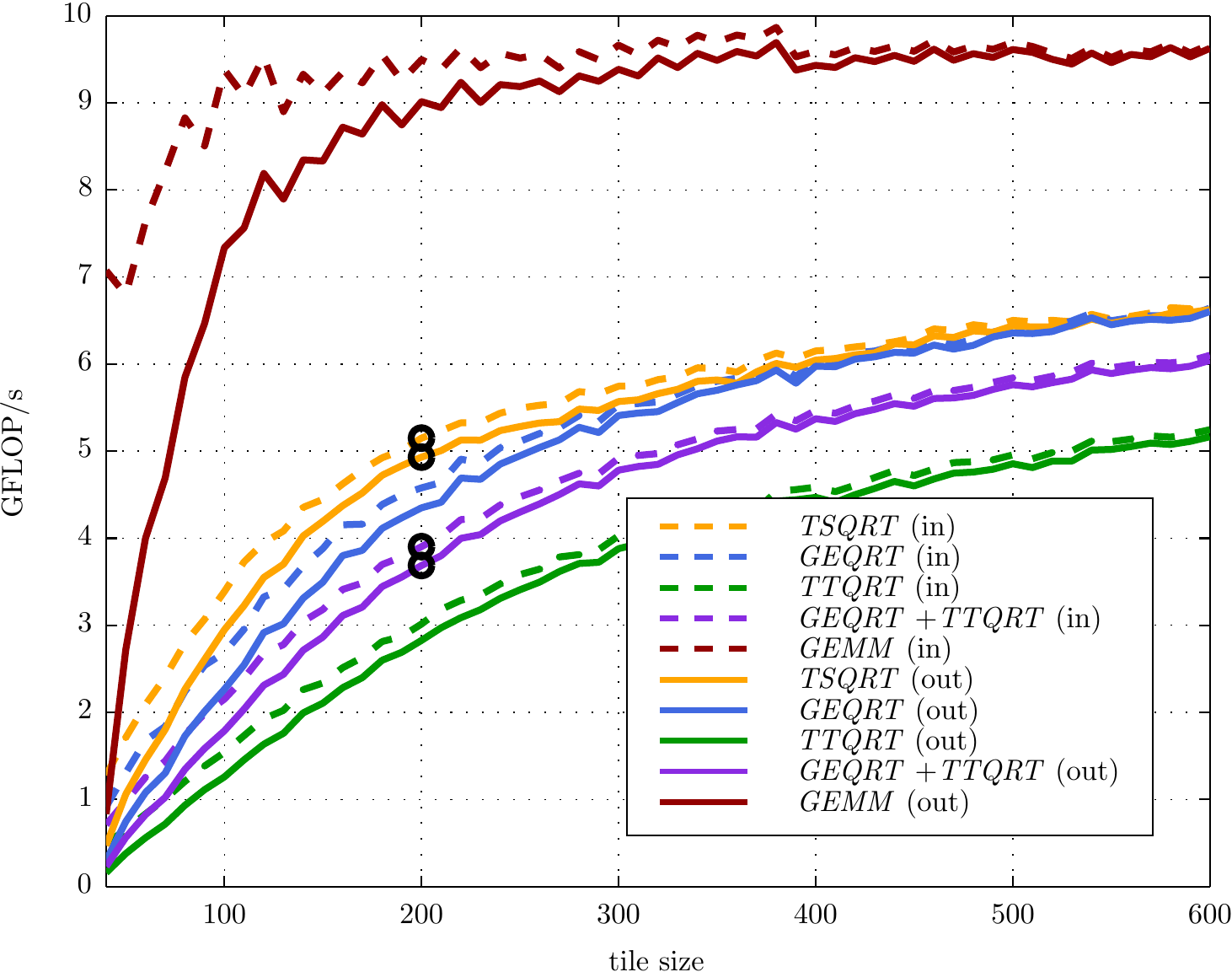}}%
}
\subfloat[Update kernels]{
    \hspace{7mm}\resizebox{.475\textwidth}{!}{\includegraphics{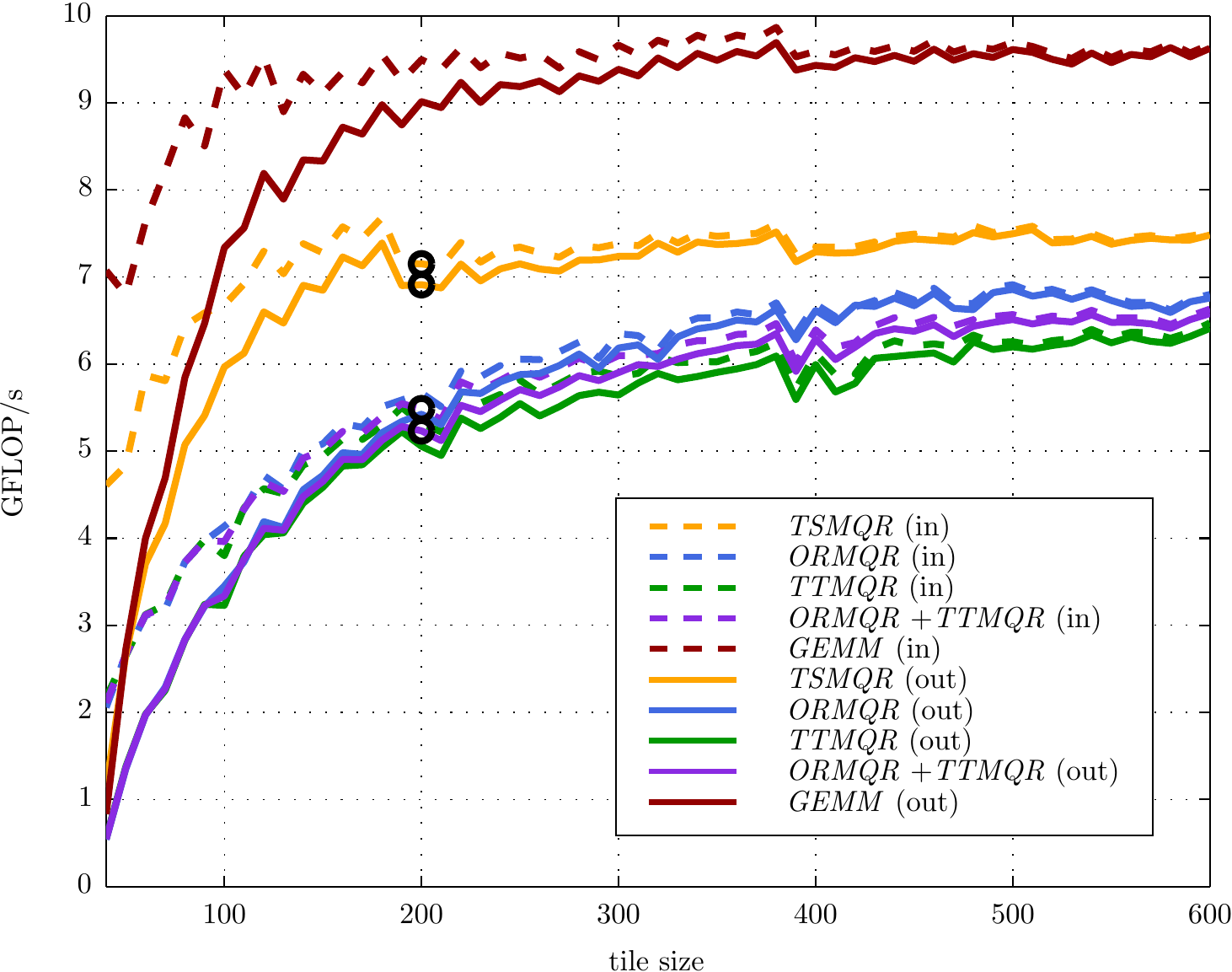}}%
}
    \caption{\label{fig:perf_kernels_d}
    Kernel performance for double precision}
\end{figure*}
\renewcommand{\baselinestretch}{\normalspace}

In order to accurately assess the impact of the kernel selection towards the
performance of the algorithms, Figures~\ref{fig:perf_kernels_z}
and~\ref{fig:perf_kernels_d} show both the in cache and out of cache
performance using the \emph{No Flush} and \emph{MultCallFlushLRU} strategies as
presented in~\cite{lawn242,Whaley:2008:AAC:1462062.1462065}.  Since an
algorithm using \emph{TT} kernels will need to call \GEQRT as well as \TTQRT to
achieve the same as the \emph{TS} kernel \TSQRT, the comparison is made between
\GEQRT + \TTQRT and \TSQRT (and similarly for the updates).  For $n_b=200$, the
observed ratio for in cache kernel speed for \TSQRT to \GEQRT + \TTQRT is
1.3374, and for \TSMQR to \UNMQR + \TTMQR is 1.3207. For out of cache, the
ratio for \TSQRT to \GEQRT + \TTQRT is 1.3193 and for \TSMQR to \UNMQR + \TTMQR
it is 1.3032.  Thus, we can expect about a 30\% difference between the
selection of the kernels, since we will have instances of using in cache and
out of cache throughout the run.  Most of this difference is due to the higher
efficiency and data locality within the \emph{TT} kernels as compared to the
\emph{TS} kernels.


\linespread{1.0}
\begin{figure*}[htb]
\centering
\subfloat[Upper bound (double complex)]{%
    \label{fig:fig_aa_pic1_p40_z.th}%
    \hspace{0mm}\resizebox{.475\textwidth}{!}{\includegraphics{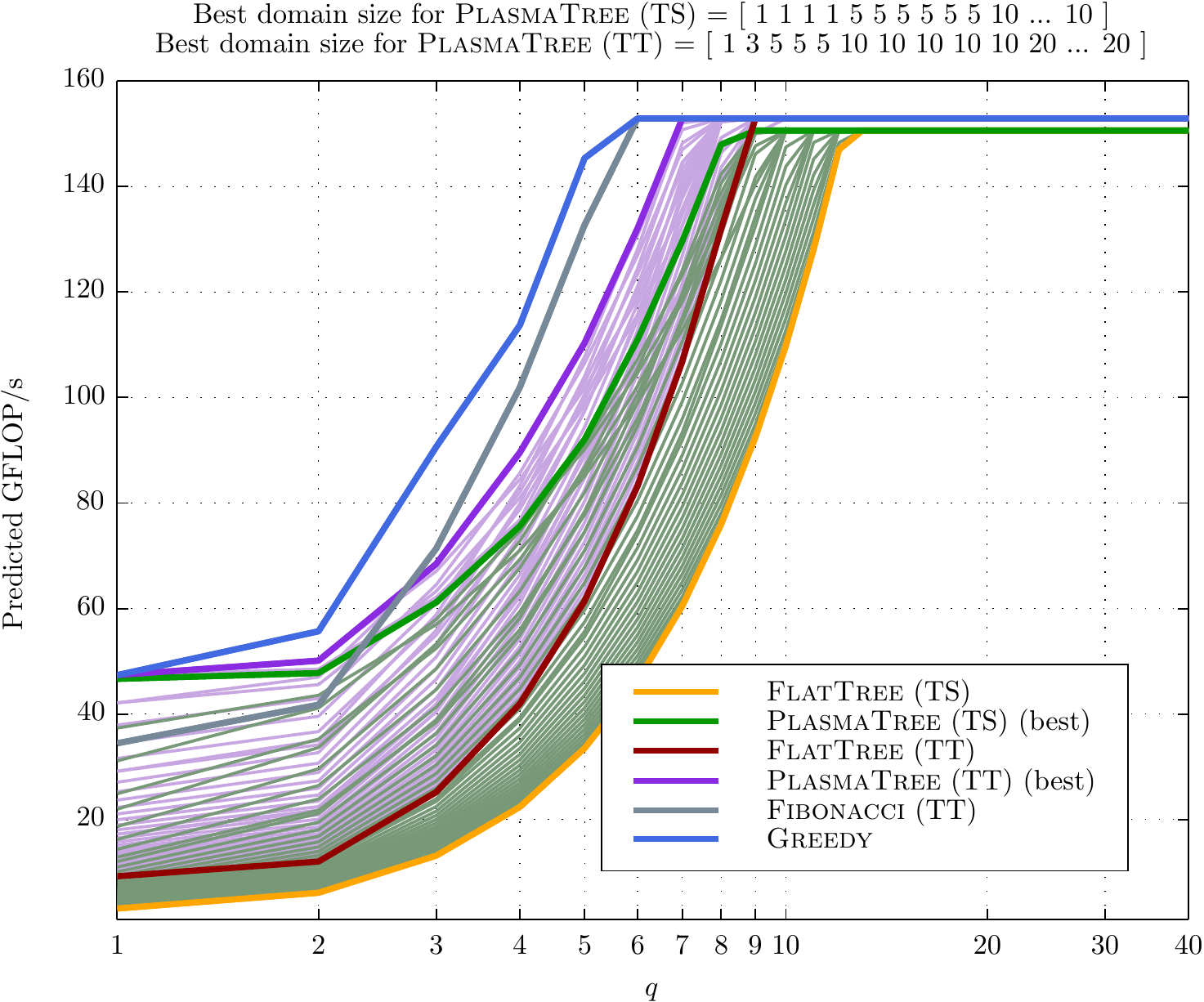}}%
}
\subfloat[Experimental (double complex)]{%
    \label{fig:fig_aa_pic1_p40_z.exp}%
    \hspace{7mm}\resizebox{.475\textwidth}{!}{\includegraphics{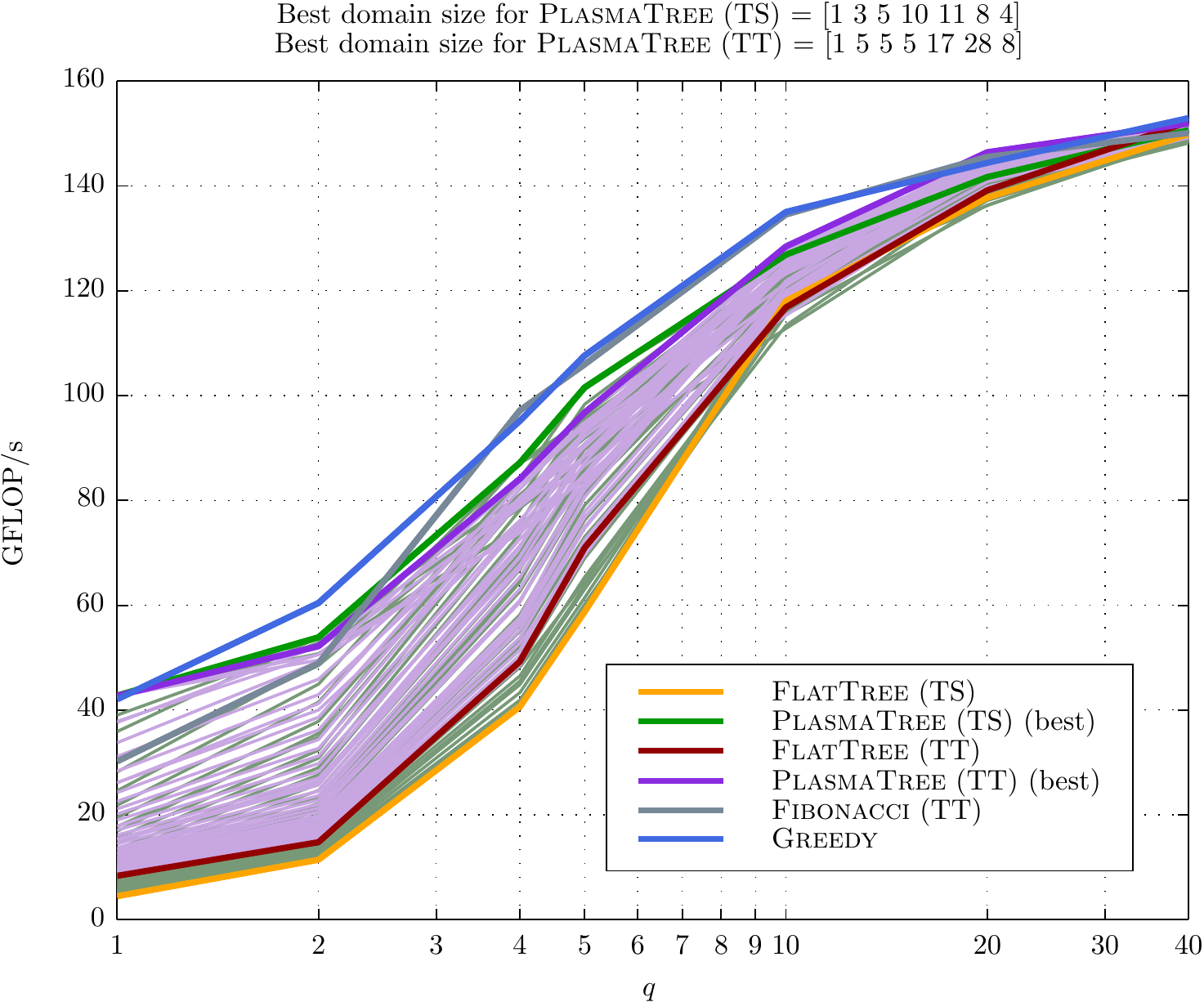}}%
}
\\
\subfloat[Upper bound (double)]{%
    \label{fig:fig_aa_pic1_p40_d.th}%
    \hspace{0mm}\resizebox{.475\textwidth}{!}{\includegraphics{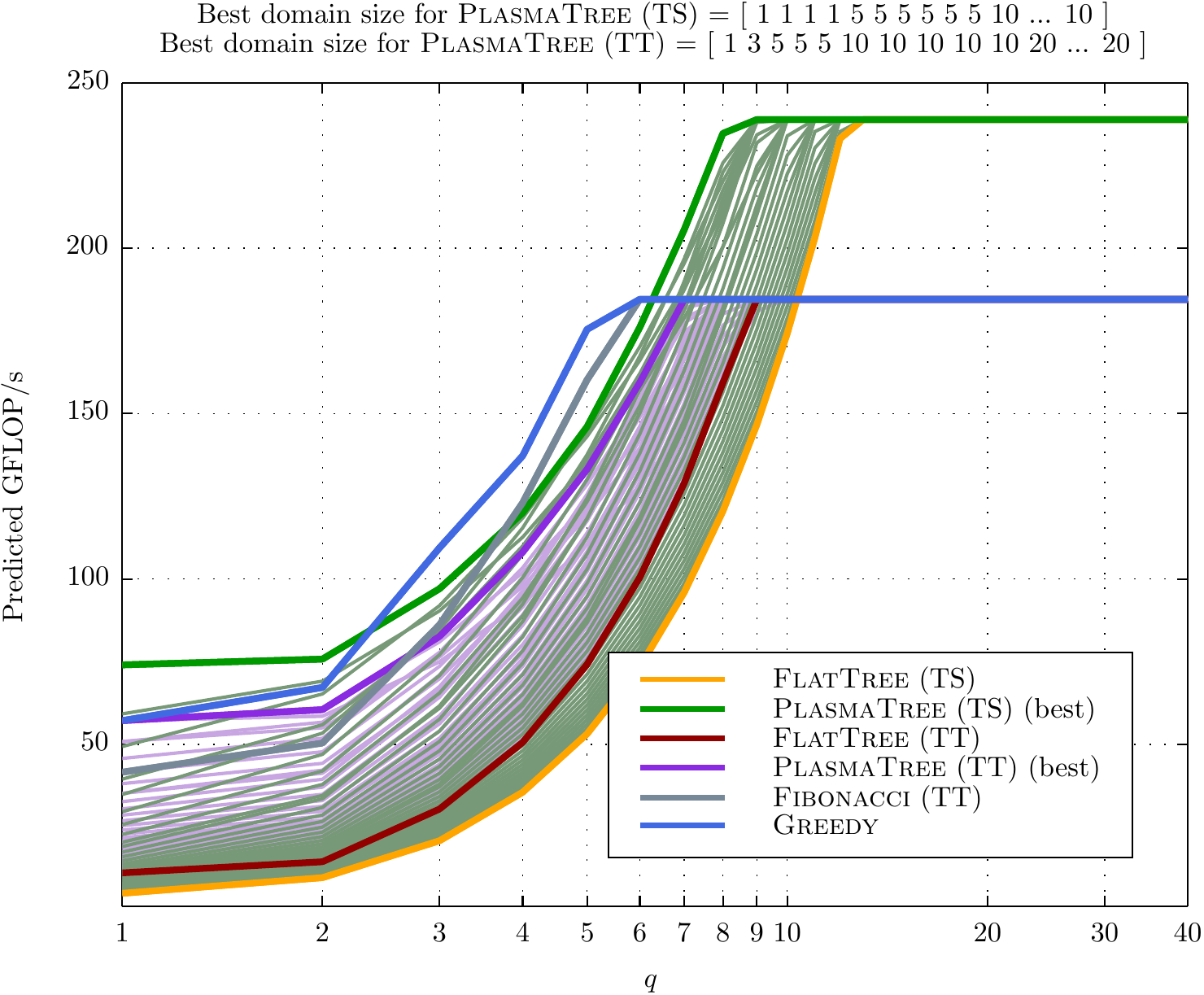}}%
}
\subfloat[Experimental (double)]{%
    \label{fig:fig_aa_pic1_p40_d.exp}%
    \hspace{7mm}\resizebox{.475\textwidth}{!}{\includegraphics{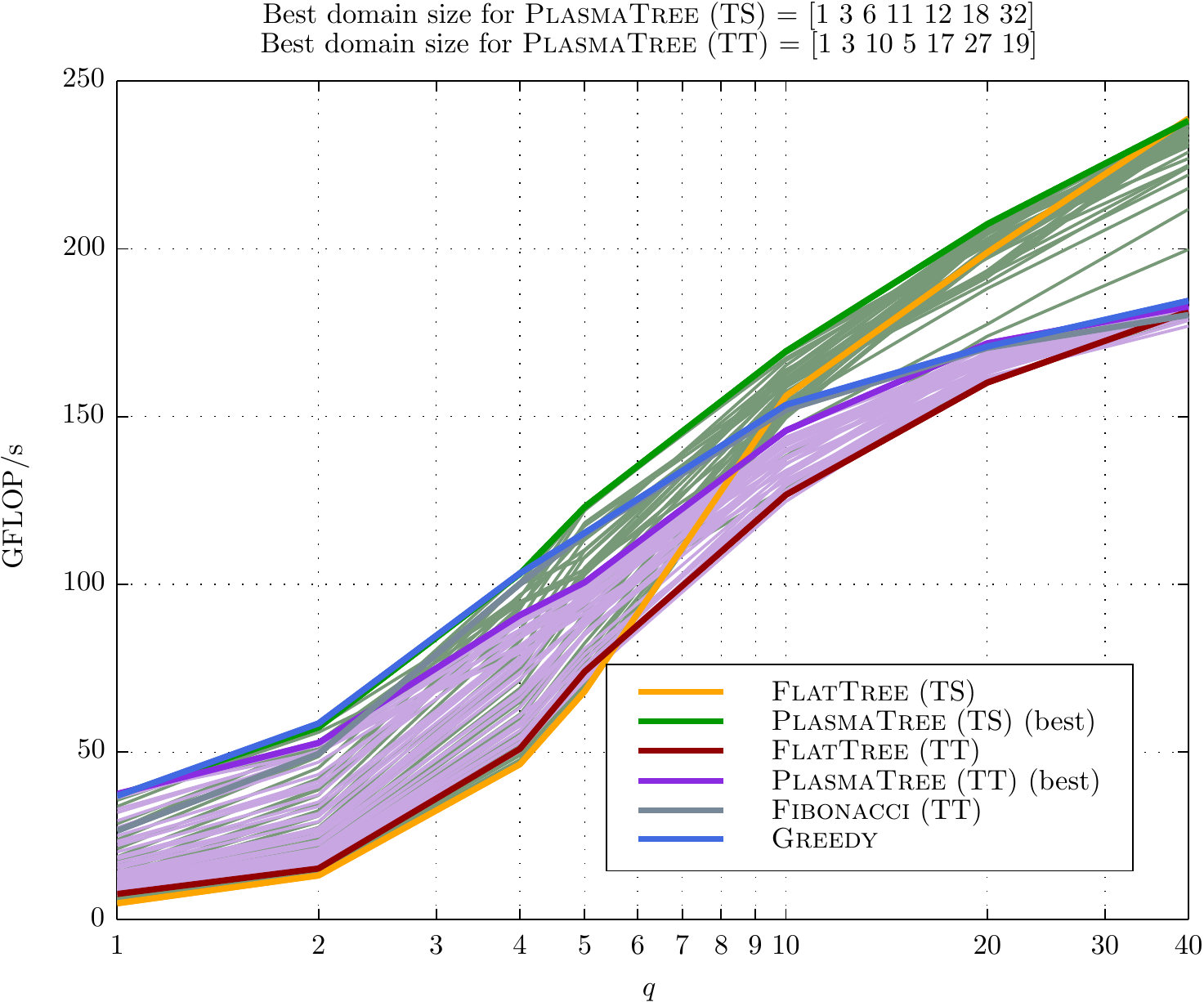}}%
}
\caption{\label{fig:fig_aa_pic1_p40}
Upper bound and experimental performance of QR factorization - All kernels}
\end{figure*}
\renewcommand{\baselinestretch}{\normalspace}

\linespread{1.0}
\begin{figure*}[htb]
\centering
\subfloat[Theoretical CP length]{%
    \hspace{0mm}\resizebox{.475\textwidth}{!}{\includegraphics{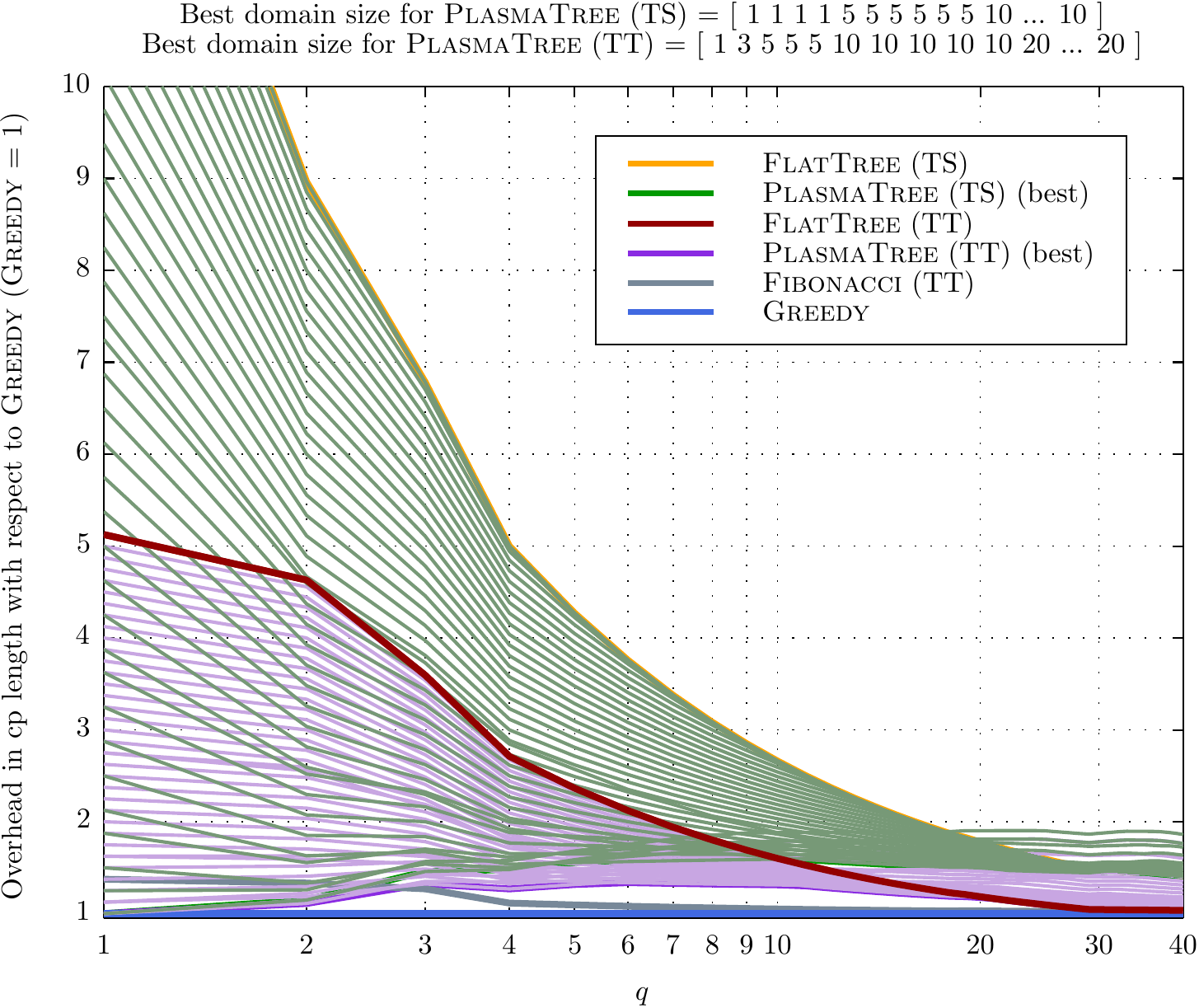}}%
}\\
\subfloat[Experimental (double complex)]{%
    \hspace{0mm}\resizebox{.475\textwidth}{!}{\includegraphics{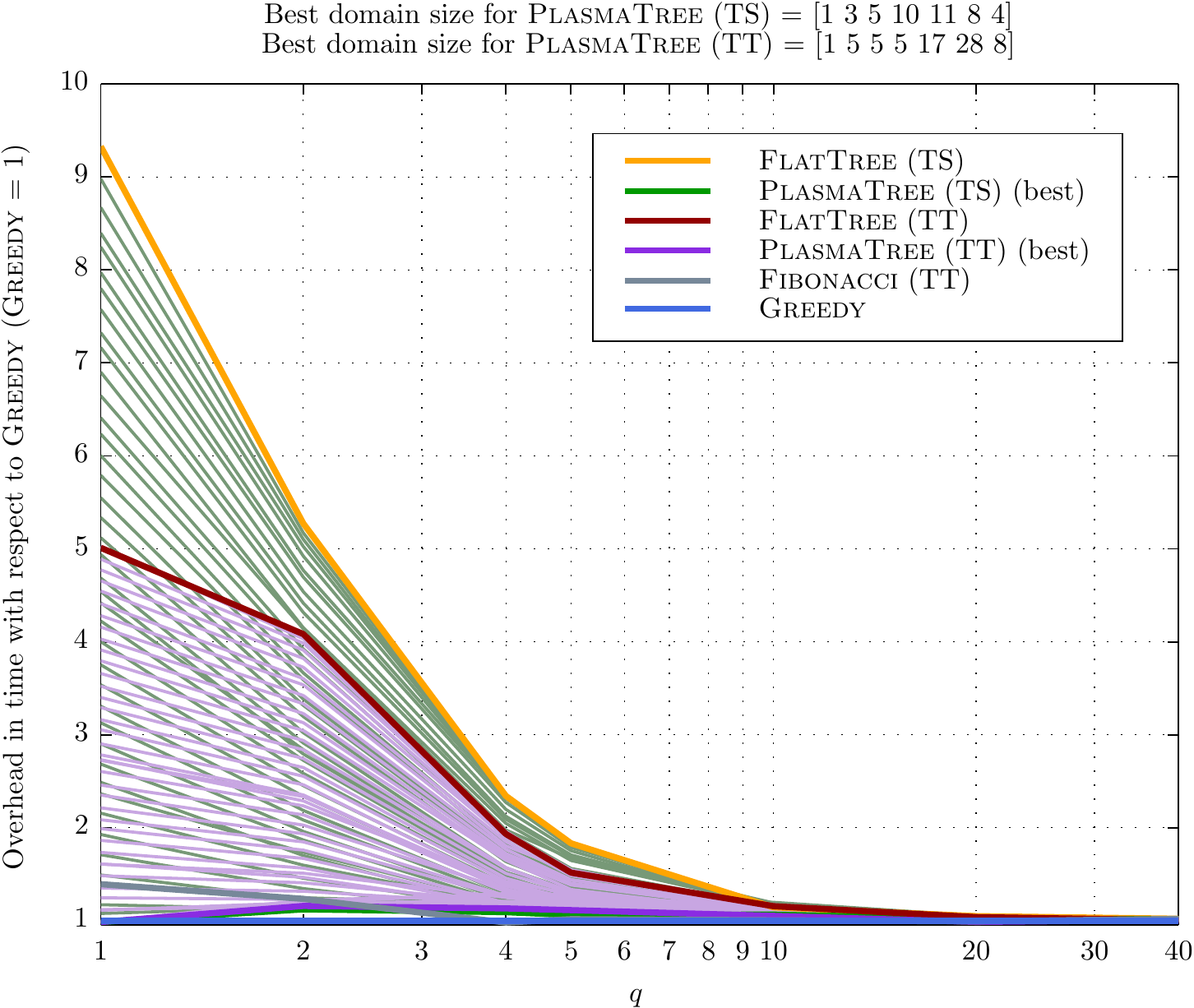}}%
}
\subfloat[Experimental (double)]{%
    \hspace{7mm}\resizebox{.475\textwidth}{!}{\includegraphics{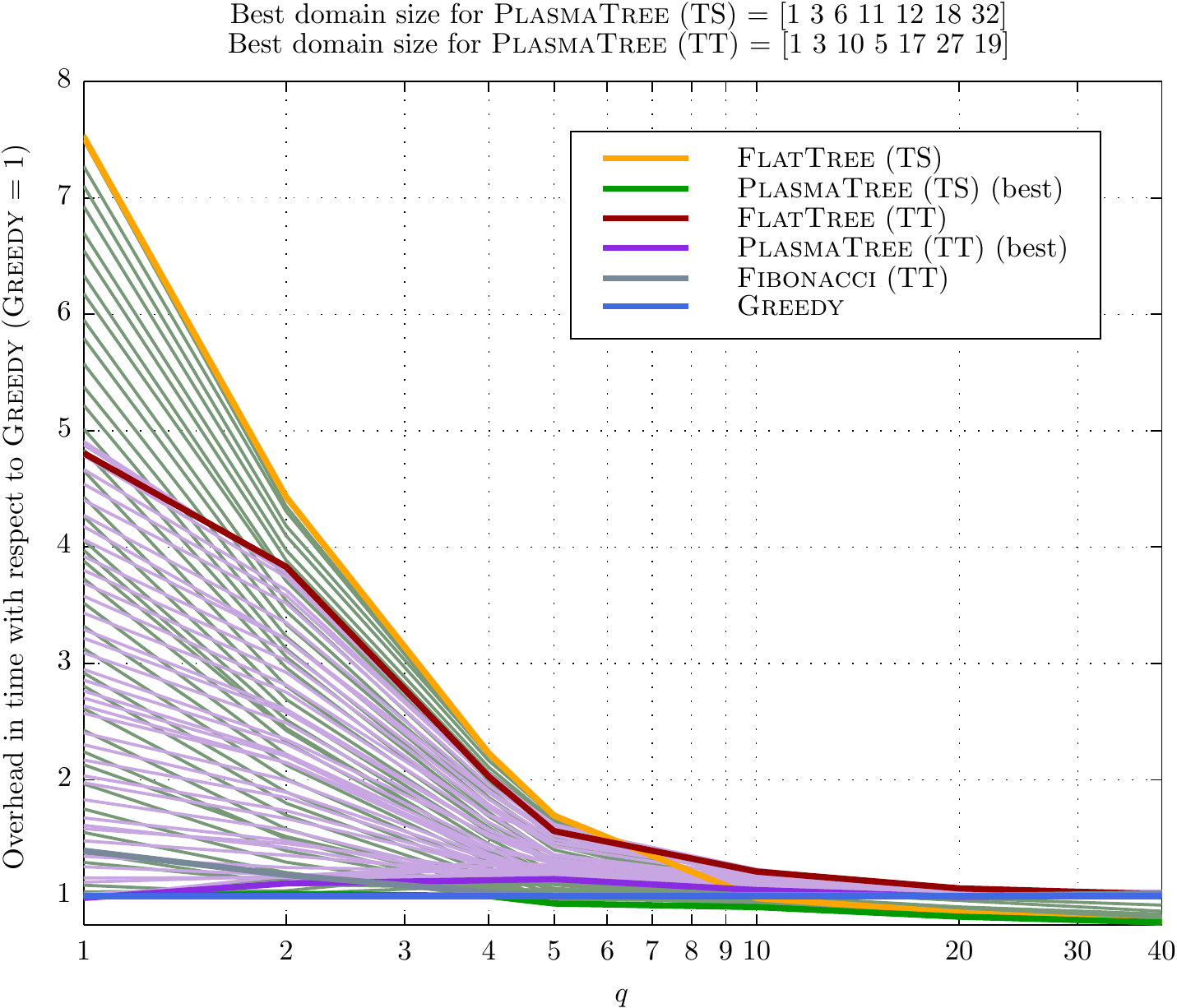}}%
}
\caption{\label{fig:fig_aa_pic2_p40}
Overhead in terms of critical path length and time with respect to \Greedy (\Greedy = 1)
}
\end{figure*}
\renewcommand{\baselinestretch}{\normalspace}

\linespread{1.0}
\begin{figure*}[htb]
\centering
\subfloat[Theoretical CP length]{%
    \hspace{0mm}\resizebox{.475\textwidth}{!}{\includegraphics{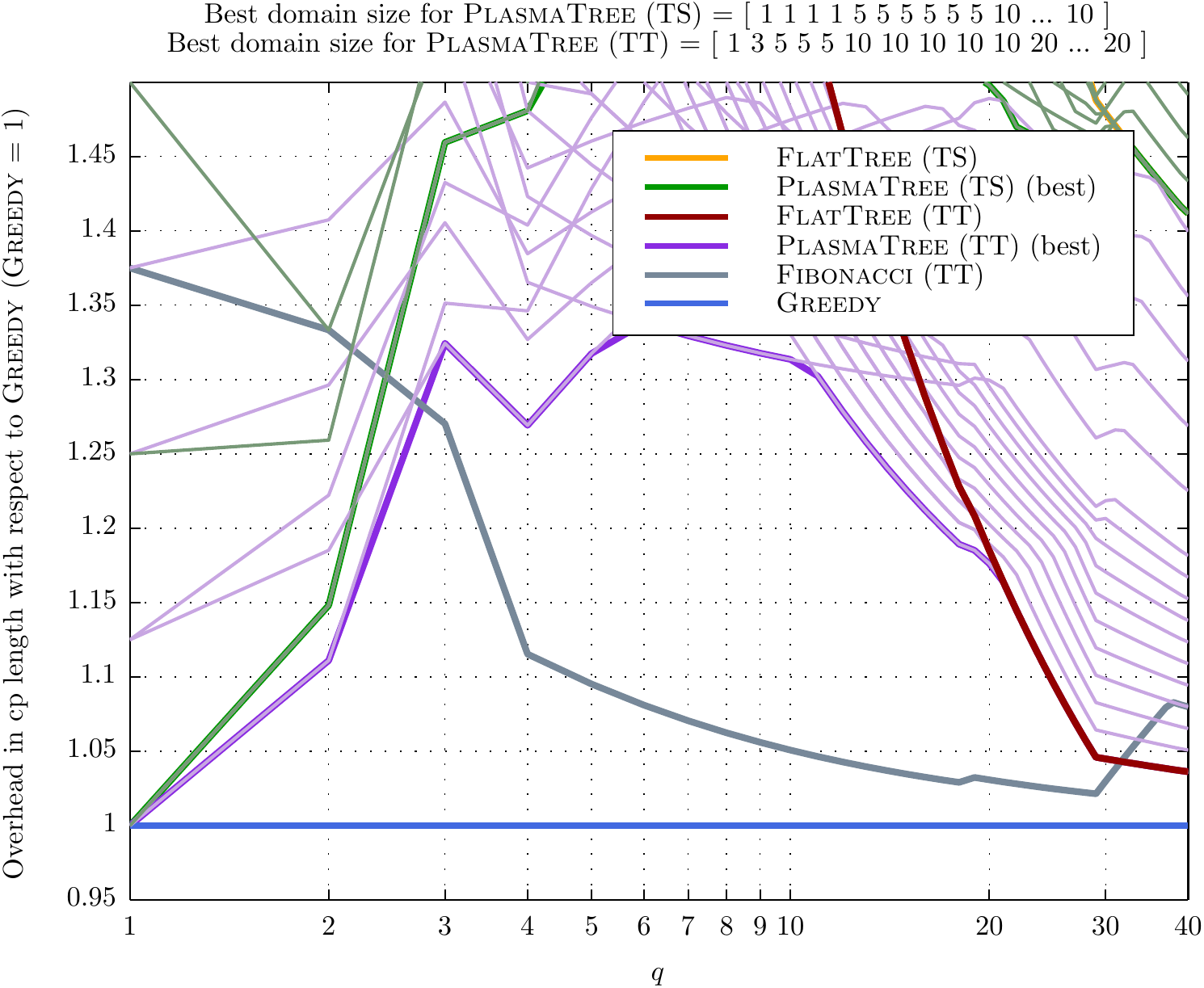}}%
}\\
\subfloat[Experimental (double complex)]{%
    \hspace{0mm}\resizebox{.475\textwidth}{!}{\includegraphics{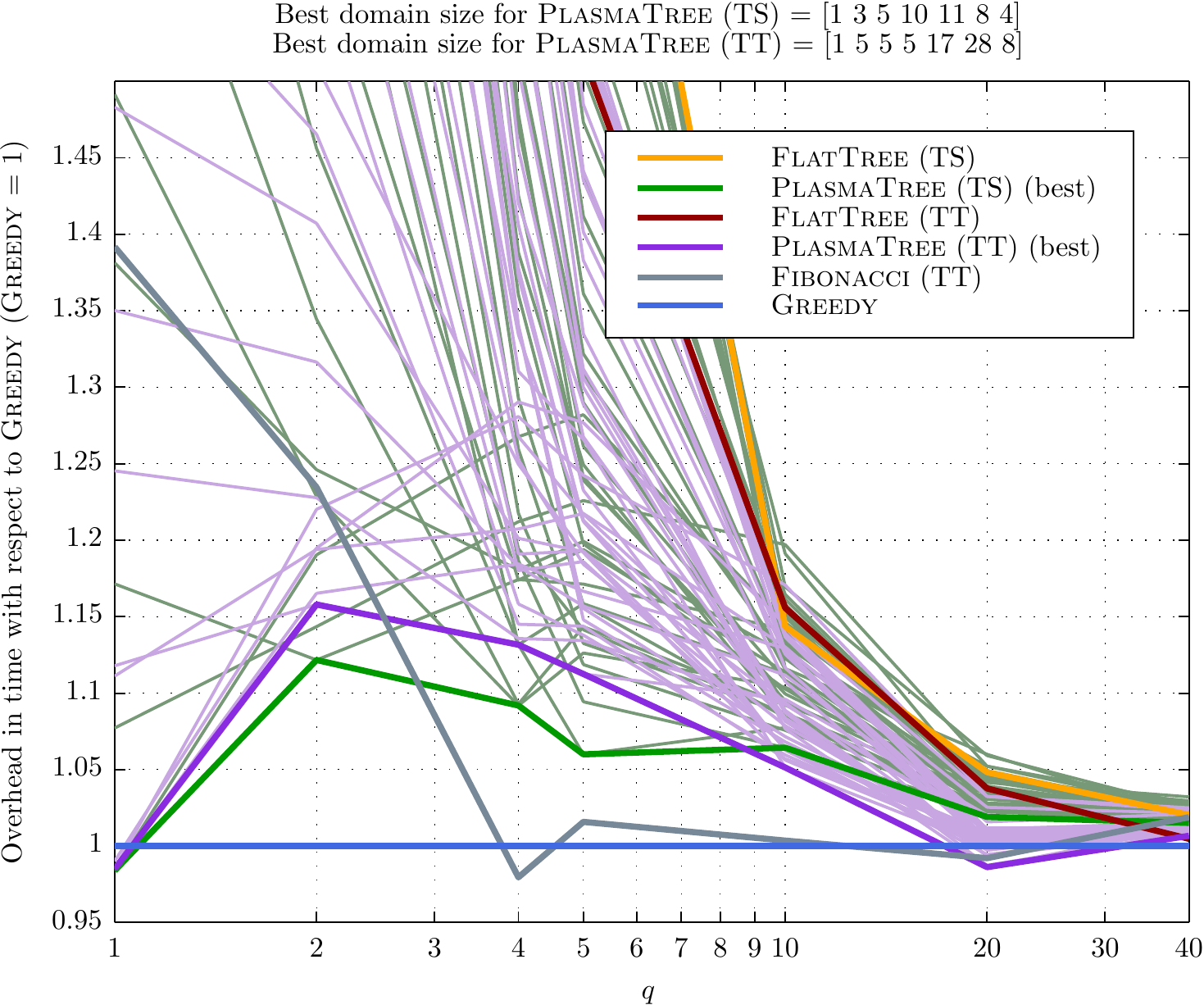}}%
}
\subfloat[Experimental (double)]{%
    \hspace{7mm}\resizebox{.475\textwidth}{!}{\includegraphics{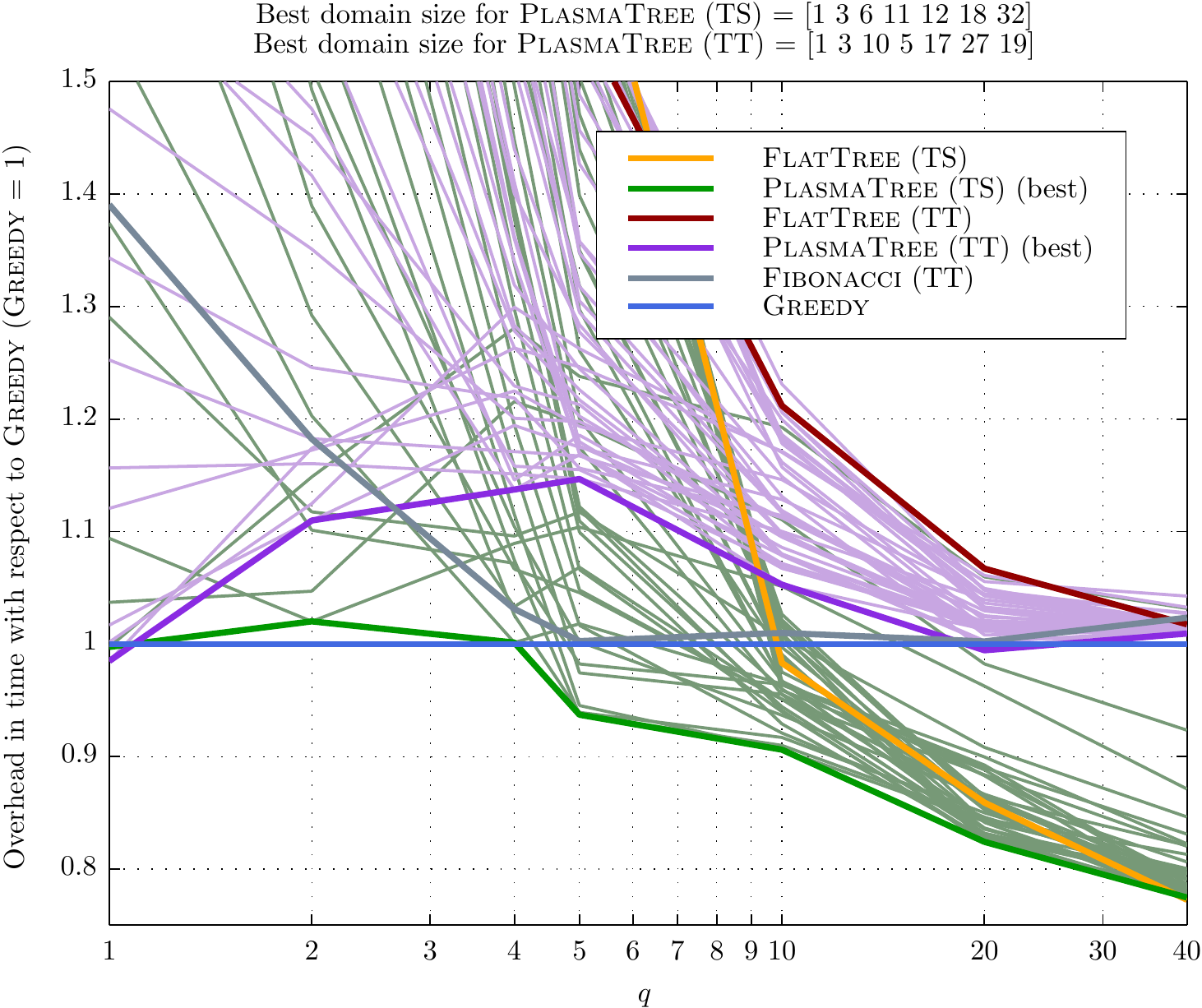}}%
}
\caption{\label{fig:fig_aa_pic3_p40}
Overhead in terms of critical path length and time with respect to \Greedy (\Greedy = 1)
}
\end{figure*}
\renewcommand{\baselinestretch}{\normalspace}

Having seen that kernel performance can have a significant impact,
we also compare the \emph{TT} based
algorithms to those using the \emph{TS} kernels. The goal is to provide
a complete assessment of all currently available
algorithms,  as shown
in Figure~\ref{fig:fig_aa_pic1_p40}.  For double precision, the observed
difference in kernel speed is 4.976 GFLOP/sec for the \emph{TS} kernels
versus 3.844 GFLOP/sec for the \emph{TT} kernels which provides a ratio of 1.2945 and is in accordance with
our previous analysis.  It can be seen that as the number of columns increases,
whereby the amount of parallelism increases, the effect of the kernel
performance outweighs the benefit provided by the extra parallelism afforded
through the \emph{TT} algorithms.  Comparatively, in
double complex precision, \Greedy does perform better, even against the
algorithms using the \emph{TS} kernels.  As before, one
must keep in mind that \Greedy does not require the tuning parameter of the
domain size to achieve this better performance.

From these experiments, we showed that in double complex precision, \Greedy
demonstrated better performance than any of the other algorithms and moreover,
it does so without the need to specify a domain size as opposed to the
algorithms in PLASMA. In addition, in double precision, for matrices where $p
\gg q$, \Greedy continues to excel over any other algorithm using
the \emph{TT} kernels, and continues to do
so as the matrices become more square.

\linespread{1.0}
\begin{table*}
    \centering
    \begin{tabular}{rrrrrrrrrr}
       \toprule
        p & q & GREEDY & PT\_TT &    BS & Overhead & Gain    &  Fib & Overhead &    Gain\\
        \midrule
       40 & 1 &     16 &     16 &     1 &   1.0000 &  0.0000 &   22 &   1.3750 &  0.2727\\
       40 & 2 &     54 &     60 &     3 &   1.1111 &  0.1000 &   72 &   1.3333 &  0.2500\\
       40 & 3 &     74 &     98 &     5 &   1.3243 &  0.2449 &   94 &   1.2703 &  0.2128\\
       40 & 4 &    104 &    132 &     5 &   1.2692 &  0.2121 &  116 &   1.1154 &  0.1034\\
       40 & 5 &    126 &    166 &     5 &   1.3175 &  0.2410 &  138 &   1.0952 &  0.0870\\
       40 & 6 &    148 &    198 &    10 &   1.3378 &  0.2525 &  160 &   1.0811 &  0.0750\\
       40 & 7 &    170 &    226 &    10 &   1.3294 &  0.2478 &  182 &   1.0706 &  0.0659\\
       40 & 8 &    192 &    254 &    10 &   1.3229 &  0.2441 &  204 &   1.0625 &  0.0588\\
       40 & 9 &    214 &    282 &    10 &   1.3178 &  0.2411 &  226 &   1.0561 &  0.0531\\
       40 &10 &    236 &    310 &    10 &   1.3136 &  0.2387 &  248 &   1.0508 &  0.0484\\
       40 &11 &    258 &    336 &    20 &   1.3023 &  0.2321 &  270 &   1.0465 &  0.0444\\
       40 &12 &    280 &    358 &    20 &   1.2786 &  0.2179 &  292 &   1.0429 &  0.0411\\
       40 &13 &    302 &    380 &    20 &   1.2583 &  0.2053 &  314 &   1.0397 &  0.0382\\
       40 &14 &    324 &    402 &    20 &   1.2407 &  0.1940 &  336 &   1.0370 &  0.0357\\
       40 &15 &    346 &    424 &    20 &   1.2254 &  0.1840 &  358 &   1.0347 &  0.0335\\
       40 &16 &    368 &    446 &    20 &   1.2120 &  0.1749 &  380 &   1.0326 &  0.0316\\
       40 &17 &    390 &    468 &    20 &   1.2000 &  0.1667 &  402 &   1.0308 &  0.0299\\
       40 &18 &    412 &    490 &    20 &   1.1893 &  0.1592 &  424 &   1.0291 &  0.0283\\
       40 &19 &    432 &    512 &    20 &   1.1852 &  0.1562 &  446 &   1.0324 &  0.0314\\
       40 &20 &    454 &    534 &    20 &   1.1762 &  0.1498 &  468 &   1.0308 &  0.0299\\
       40 &21 &    476 &    554 &    20 &   1.1639 &  0.1408 &  490 &   1.0294 &  0.0286\\
       40 &22 &    498 &    570 &    20 &   1.1446 &  0.1263 &  512 &   1.0281 &  0.0273\\
       40 &23 &    520 &    586 &    20 &   1.1269 &  0.1126 &  534 &   1.0269 &  0.0262\\
       40 &24 &    542 &    602 &    20 &   1.1107 &  0.0997 &  556 &   1.0258 &  0.0252\\
       40 &25 &    564 &    618 &    20 &   1.0957 &  0.0874 &  578 &   1.0248 &  0.0242\\
       40 &26 &    586 &    634 &    20 &   1.0819 &  0.0757 &  600 &   1.0239 &  0.0233\\
       40 &27 &    608 &    650 &    20 &   1.0691 &  0.0646 &  622 &   1.0230 &  0.0225\\
       40 &28 &    630 &    666 &    20 &   1.0571 &  0.0541 &  644 &   1.0222 &  0.0217\\
       40 &29 &    652 &    682 &    20 &   1.0460 &  0.0440 &  666 &   1.0215 &  0.0210\\
       40 &30 &    668 &    698 &    20 &   1.0449 &  0.0430 &  688 &   1.0299 &  0.0291\\
       40 &31 &    684 &    714 &    20 &   1.0439 &  0.0420 &  710 &   1.0380 &  0.0366\\
       40 &32 &    700 &    730 &    20 &   1.0429 &  0.0411 &  732 &   1.0457 &  0.0437\\
       40 &33 &    716 &    746 &    20 &   1.0419 &  0.0402 &  754 &   1.0531 &  0.0504\\
       40 &34 &    732 &    762 &    20 &   1.0410 &  0.0394 &  776 &   1.0601 &  0.0567\\
       40 &35 &    748 &    778 &    20 &   1.0401 &  0.0386 &  798 &   1.0668 &  0.0627\\
       40 &36 &    764 &    794 &    20 &   1.0393 &  0.0378 &  820 &   1.0733 &  0.0683\\
       40 &37 &    780 &    810 &    20 &   1.0385 &  0.0370 &  842 &   1.0795 &  0.0736\\
       40 &38 &    796 &    826 &    20 &   1.0377 &  0.0363 &  862 &   1.0829 &  0.0766\\
       40 &39 &    812 &    842 &    20 &   1.0369 &  0.0356 &  878 &   1.0813 &  0.0752\\
       40 &40 &    826 &    856 &    20 &   1.0363 &  0.0350 &  892 &   1.0799 &  0.0740\\
       \bottomrule
    \end{tabular}
    \caption{Greedy versus PT\_TT and Fibonacci (Theoretical)}
\end{table*}
\renewcommand{\baselinestretch}{\normalspace}

\linespread{1.0}
\begin{table*}
    \centering
    \begin{tabular}{rrrrrrr}
       \toprule
         p &  q & GREEDY(d) & PT\_TT(d) & BS & Overhead &   Gain\\
         \midrule
        40 &  1 &   36.9360 &   37.5020 &  1 &   1.0153 &-0.0153\\
        40 &  2 &   58.5090 &   52.7180 &  3 &   0.9010 & 0.0990\\
        40 &  4 &  103.2670 &   90.7940 & 10 &   0.8792 & 0.1208\\
        40 &  5 &  115.3060 &  100.5540 &  5 &   0.8721 & 0.1279\\
        40 & 10 &  153.5180 &  145.8200 & 17 &   0.9499 & 0.0501\\
        40 & 20 &  170.8730 &  171.8270 & 27 &   1.0056 &-0.0056\\
        40 & 40 &  184.5220 &  182.8160 & 19 &   0.9908 & 0.0092\\
        \bottomrule
    \end{tabular}
    \caption{Greedy versus PT\_TT (Experimental Double)}
\end{table*}
\renewcommand{\baselinestretch}{\normalspace}

\linespread{1.0}
\begin{table*}
    \centering
    \begin{tabular}{rrrrrrr}
       \toprule
         p &  q & GREEDY(z) & PT\_TT(z) & BS & Overhead &   Gain\\
         \midrule
        40 &  1 &   42.0710 &   42.7120 &  1 &   1.0152 &-0.0152\\
        40 &  2 &   60.4420 &   52.1970 &  5 &   0.8636 & 0.1364\\
        40 &  4 &   95.1820 &   84.1120 &  5 &   0.8837 & 0.1163\\
        40 &  5 &  107.6370 &   96.7530 &  5 &   0.8989 & 0.1011\\
        40 & 10 &  135.0270 &  128.4320 & 17 &   0.9512 & 0.0488\\
        40 & 20 &  144.4010 &  146.4220 & 28 &   1.0140 &-0.0140\\
        40 & 40 &  152.9280 &  151.9090 &  8 &   0.9933 & 0.0067\\
        \bottomrule
    \end{tabular}
    \caption{Greedy versus PT\_TT (Experimental Double Complex)}
\end{table*}
\renewcommand{\baselinestretch}{\normalspace}

\linespread{1.0}
\begin{table*}
    \centering
    \begin{tabular}{rrrrrr}
       \toprule
         p &  q & GREEDY(d) &    FIB(d) & Overhead &   Gain\\
         \midrule
        40 &  1 &   36.9360 &  26.5610  &   0.7191 & 0.2809\\
        40 &  2 &   58.5090 &  49.4870  &   0.8458 & 0.1542\\
        40 &  4 &  103.2670 & 100.1440  &   0.9698 & 0.0302\\
        40 &  5 &  115.3060 & 115.0020  &   0.9974 & 0.0026\\
        40 & 10 &  153.5180 & 152.0090  &   0.9902 & 0.0098\\
        40 & 20 &  170.8730 & 170.4780  &   0.9977 & 0.0023\\
        40 & 40 &  184.5220 & 180.2990  &   0.9771 & 0.0229\\
        \bottomrule
    \end{tabular}
    \caption{Greedy versus Fibonacci (Experimental Double)}
\end{table*}
\renewcommand{\baselinestretch}{\normalspace}

\linespread{1.0}
\begin{table*}
    \centering
    \begin{tabular}{rrrrrr}
         \toprule
         p &  q & GREEDY(z) &    FIB(z) & Overhead &   Gain\\
         \midrule
        40 &  1 &   42.0710 &  30.2280  &   0.7185 & 0.2815\\
        40 &  2 &   60.4420 &  48.9570  &   0.8100 & 0.1900\\
        40 &  4 &   95.1820 &  97.1650  &   1.0208 &-0.0208\\
        40 &  5 &  107.6370 & 105.9610  &   0.9844 & 0.0156\\
        40 & 10 &  135.0270 & 134.5500  &   0.9965 & 0.0035\\
        40 & 20 &  144.4010 & 145.5530  &   1.0080 &-0.0080\\
        40 & 40 &  152.9280 & 150.0980  &   0.9815 & 0.0185\\
        \bottomrule
    \end{tabular}
    \caption{Greedy versus Fibonacci (Experimental Double Complex)}
\end{table*}
\renewcommand{\baselinestretch}{\normalspace}

\section{Conclusion}
\label{sec.conclusion}

In this chapter, we have presented \MC, and \Greedy, two new algorithms for tiled
QR factorization.  These algorithms exhibit more parallelism than
state-of-the-art implementations based on reduction trees. We have provided
accurate estimations for the length of their critical path.

Comprehensive experiments on multicore platforms confirm the superiority of the
new algorithms for $p \times q$ matrices, as soon as, say, $p \geq 2q$.  This
holds true when comparing not only with previous algorithms using TT
(\emph{Triangle on top of triangle}) kernels, but also with all known
algorithms based on TS (\emph{Triangle on top of square}) kernels. Given that
TS kernels offer more locality, and benefit from better elementary arithmetic
performance, than TT kernels, the better performance of the new algorithms is
even more striking, and further demonstrates that a large degree of a
parallelism was not exploited in previously published solutions.

Future work will investigate several promising directions. First, using
rectangular tiles instead of square tiles could lead to efficient algorithms,
with more locality and still the same potential for parallelism. Second,
refining the model to account for communications, and extending it to fully
distributed architectures, would lay the ground work for the design of MPI
implementations of the new algorithms, unleashing their high level of
performance on larger platforms.  Finally, the design of robust algorithms,
capable of achieving efficient performance despite variations in processor
speeds, or even resource failures, is a challenging but crucial task to fully
benefit from future platforms with a huge number of cores.

\chapter{Scheduling of Cholesky Factorization}\label{chp:cholfact}

In Chapter~\ref{chp:cholinv} we studied the Cholesky Inversion algorithm which
consists of the three steps: Cholesky factorization, inversion of the factor,
and the multiplication of two triangular matrices.  In this chapter, we will
focus on the Cholesky factorization but unlike the previous work where the
number of processors was unbounded, we will consider the factorization in the
context of a finite number of processors.  By limiting the number of processors,
the scheduling of the tasks becomes an issue.  Moreover, the weight (processing
time) of each task must be taken into consideration when creating the schedule.  

As before, we will be considering the critical path length for the algorithm but
not as a function of the number of tiles rather as a function of the weights of
the tasks.  The weights are based upon the total computational cost for each
kernel and are provided in Table~\ref{tab:taskweights}.  A more in-depth
analysis of the length of the critical path with weighted tasks for the Cholesky
Inversion algorithm can be found in~\cite{DBLP:journals/corr/abs-1010-2000}
which also provides $9t-10$ as the weighted critical path length for the
Cholesky factorization of a matrix of $t \times t$ tiles.

\linespread{1.0}
\begin{table}[htbp]
        \centering
        \begin{tabular}[ht]{ccc}
            \toprule
                  &           & Weights\\ 
                  & \# flops  & (in $\frac{1}{3}n_b^3$ flops)\\
            \cmidrule(lr){2-2}
            \cmidrule( r){3-3}
            POTRF & $\frac{1}{3}n_b^3 + O(n_b^2)$   & 1 \\
            TRSM  & $n_b^3$                         & 3 \\
            SYRK  & $n_b^3 + O(n_b^2)$              & 3 \\
            GEMM  & $2n_b^3 + O(n_b^2)$             & 6 \\
            \bottomrule
        \end{tabular}
    \caption{Task Weights}
    \label{tab:taskweights}
\end{table}
\renewcommand{\baselinestretch}{\normalspace}

The upper bound on performance of perfect speedup and critical path introduced
in Chapter~\ref{chp:cholinv} remains too optimistic and does not take into
account any information which can be garnered from the DAG of the algorithm.
This work makes progress towards providing a more representative bound on the
performance of the Cholesky factorization in the tiled setting. 

We also provide gains toward a bound on the minimum number of processors
required to obtain the shortest possible weighted critical path (minimum
make span) for the Cholesky factorization for a matrix of $t \times t$ tiles.

\section{ALAP Derived Performance Bound}
\label{sec:perfbound}

To obtain our bounds, we calculate the latest possible start time for each
task (ALAP) and consider an unbounded number of processors without any costs
for communication.  If we did account for communication, we might see the
critical path length increase which would in turn decrease our upper bound.
We start at the final tasks and consider how many processors are needed to
execute these tasks without increasing the length of the critical path.  We
step backwards in time until such a point that there are more processors
needed to keep the critical path length constant.  Thus we must add enough
processors to execute the tasks and in turn create more idle time for the
execution of tasks which are successors.  At a certain point, there is no more
need to add processors and this is then the number of processors needed to
obtain the constant length critical path.  

By forcing as late as possible (ALAP) start times, any schedule will keep as
many or fewer processors active as the ALAP execution on an unbounded number of
processors.  Thus by evaluating the Lost Area ($LA$), or idle time, for a given
number of processors, $p$, at the end of the ALAP execution on an unbounded
number of processors, we can increase the sequential time by the amount of $LA$
and divide this result by the $p$ to obtain the best possible execution time,
i.e.,
\begin{equation}
    T_p = \frac{T_{seq} + LA_p}{p}
    \label{eqn:alapbound}
\end{equation}
and we define this to be the \emph{ALAP Derived Performance Bound}.  Hence the maximum
speedup that we can expect is given by

    \[ T_{seq} \cdot \frac{p}{T_{seq} + LA_p}.\]

An example will help to further illustrate this technique.  In
Figure~\ref{fig:ALAP5x5} we are given the ALAP execution of a $5 \times 5$ tiled
matrix which has $T_{seq} = 125$.  The ordered pairs indicated provide the
number of processors and idle time, respectively, and in
Table~\ref{tab:bound5x5} are given the values for $T_p$, speedup, and
efficiency.  For more than four processors, there are enough processors to
obtain the critical path length which becomes our limiting factor.

\linespread{1.0}
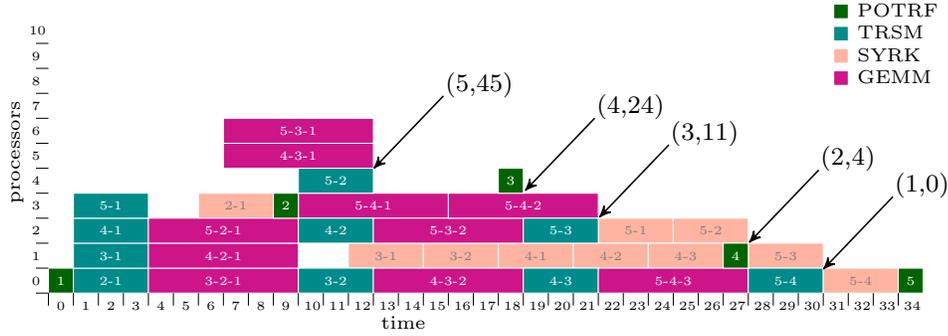
\begin{figure}[htbp]
    \centering
    \pgfmathsetlength{\imagewidth}{10cm} 
    \pgfmathsetlength{\imagescale}{\imagewidth/600} 
    \resizebox{0.85\linewidth}{!}{%
        \begin{tikzpicture}[x=\imagescale,y=-\imagescale]
            \node[anchor=north west,inner sep=0pt,outer sep=0pt] at (-0.5,-0.5)
            {\includegraphics[trim = 2mm 1mm 5mm 1mm, clip = true, width=\imagewidth]{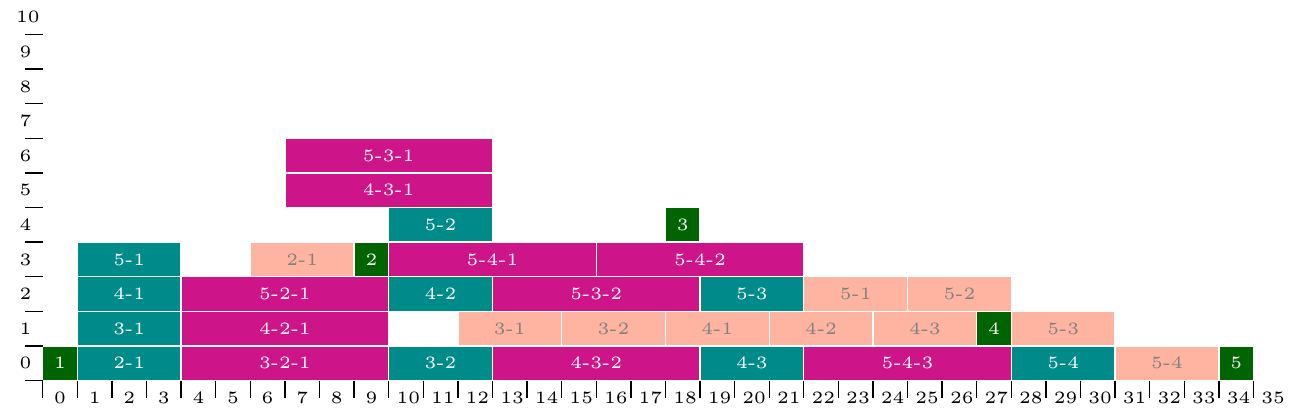}};
            \draw[<-, line width = 0.5, black, >=stealth'] (533,164) -- (580,115)
            node[above=0.125cm, right=-0.125cm] {$\mathsmaller{(1,0)}$};
            \draw[<-, line width = 0.5, black, >=stealth'] (483,146) -- (530,97)
            node[above=0.125cm, right=-0.125cm] {$\mathsmaller{(2,4)}$};
            \draw[<-, line width = 0.5, black, >=stealth'] (382,129) -- (429,80)
            node[above=0.125cm, right=-0.125cm] {$\mathsmaller{(3,11)}$};
            \draw[<-, line width = 0.5, black, >=stealth'] (332,111) -- (379,62)
            node[above=0.125cm, right=-0.125cm] {$\mathsmaller{(4,24)}$};
            \draw[<-, line width = 0.5, black, >=stealth'] (230,96) -- (277,47)
            node[above=0.125cm, right=-0.125cm] {$\mathsmaller{(5,45)}$};
            \node at (250,200) {\tiny{time}};
            \node[rotate=90] at (-10,85) {\tiny{processors}};
            \definecolor{tempcolor}{RGB}{0,100,0}
            \draw[color=white,fill=tempcolor] (540,-15) -- (550,-15) -- (550,-5)
            -- (540,-5) -- cycle node[color=black,above=0.1cm, right=0.125cm] {\tiny{POTRF}};
            \definecolor{tempcolor}{RGB}{0,139,139}
            \draw[color=white,fill=tempcolor] (540,0) -- (550,0) -- (550,10) --
            (540,10) -- cycle node[color=black,above=0.1cm, right=0.125cm] {\tiny{TRSM}};
            \definecolor{tempcolor}{RGB}{255,180,162}
            \draw[color=white,fill=tempcolor] (540,15) -- (550,15) -- (550,25)
            -- (540,25) -- cycle node[color=black,above=0.1cm, right=0.125cm] {\tiny{SYRK}};
            \definecolor{tempcolor}{RGB}{205,20,137}
            \draw[color=white,fill=tempcolor] (540,30) -- (550,30) -- (550,40)
            -- (540,40) -- cycle node[color=black,above=0.1cm, right=0.125cm] {\tiny{GEMM}};
        \end{tikzpicture}
    }
    \caption{ALAP execution for $5 \times 5$ tiles}
    \label{fig:ALAP5x5}
\end{figure}
\renewcommand{\baselinestretch}{\normalspace}
\linespread{1.2}
\begin{table}[htbp]
    \centering
    \begin{tabular}{rrrr}
        \toprule
        $p$ & $T_p$ & $S_p$ & $E_p$ \\
        \cmidrule(l ){1-1}
        \cmidrule(lr){2-2}
        \cmidrule(lr){3-3}
        \cmidrule( r){4-4}
         1 &  125.00 &  1.00 &  1.00\\
         2 &   64.50 &  1.94 &  0.97\\
         3 &   45.33 &  2.76 &  0.92\\
         4 &   37.25 &  3.36 &  0.84\\
         5 &   35.00 &  3.57 &  0.71\\
         6 &   35.00 &  3.57 &  0.60\\
         7 &   35.00 &  3.57 &  0.51\\
         8 &   35.00 &  3.57 &  0.45\\
         9 &   35.00 &  3.57 &  0.40\\
        10 &   35.00 &  3.57 &  0.36\\
        \bottomrule
    \end{tabular}
    \caption{Upper bound on speedup and efficiency for $5 \times 5$ tiles}
    \label{tab:bound5x5}
\end{table}
\renewcommand{\baselinestretch}{\normalspace}

\section{Critical Path Scheduling}
\label{sec:backflow}

In order to provide a critical path schedule, we use the Backflow algorithm to
assign priorities to tasks of a DAG such that each task's priority adheres to
its dependencies.  The algorithm is described in four steps: 

\begin{center}
    \begin{tabular}{@{}rp{10cm}@{}}
        STEP 1 : & Beginning at the final task in the DAG, set its priority to
        its processing time.\\
        STEP 2 : & Moving in the reverse direction, set each incidental task's
        priority to the sum of its the processing time and the final task's priority.\\
        STEP 3 : & For each task in STEP 2, moving in the reverse direction, set
        each incidental task's priority to the sum of its processing time and
        the maximum priority of any incidental successor task.\\
        STEP 4 : & Repeat the procedure until all tasks have been assigned a
        priority.
    \end{tabular}
\end{center}

An example is given in Figure~\ref{fig:exbackflow}.  The processing times are
given in parenthesis and the assigned priorities (cp) are designated in square
brackets. Tasks $A$ and $B$ will be assigned a priority of $16$ since
$\mathop{\mathrm{cp}}(A) = 3 + \max \left( \mathop{\mathrm{cp}}(C),
\mathop{\mathrm{cp}}(D) \right)$ and $\mathop{\mathrm{cp}}(B) = 3 + \max \left(
\mathop{\mathrm{cp}}(D), \mathop{\mathrm{cp}}(E) \right)$.
\linespread{1.0}
\begin{figure}[htb]
\centering
\subfloat[State at start]{%
    \label{fig:backflow_start}%
    \hspace{-2mm}\resizebox{.475\textwidth}{100mm}{\includegraphics{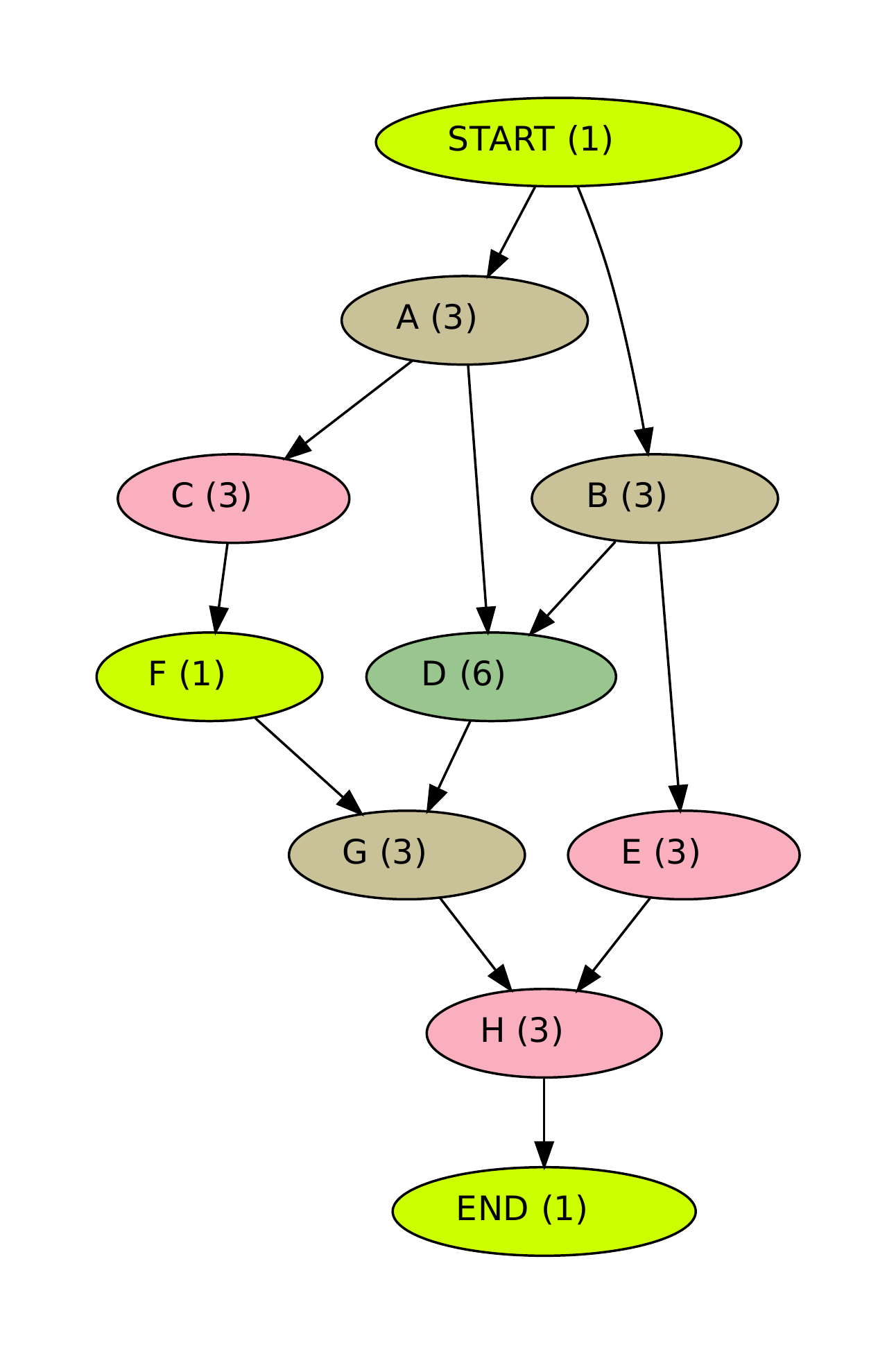}}
}
\subfloat[State at finish]{%
    \label{fig:backflow_finish}%
    \hspace{7mm}\resizebox{.475\textwidth}{100mm}{\includegraphics{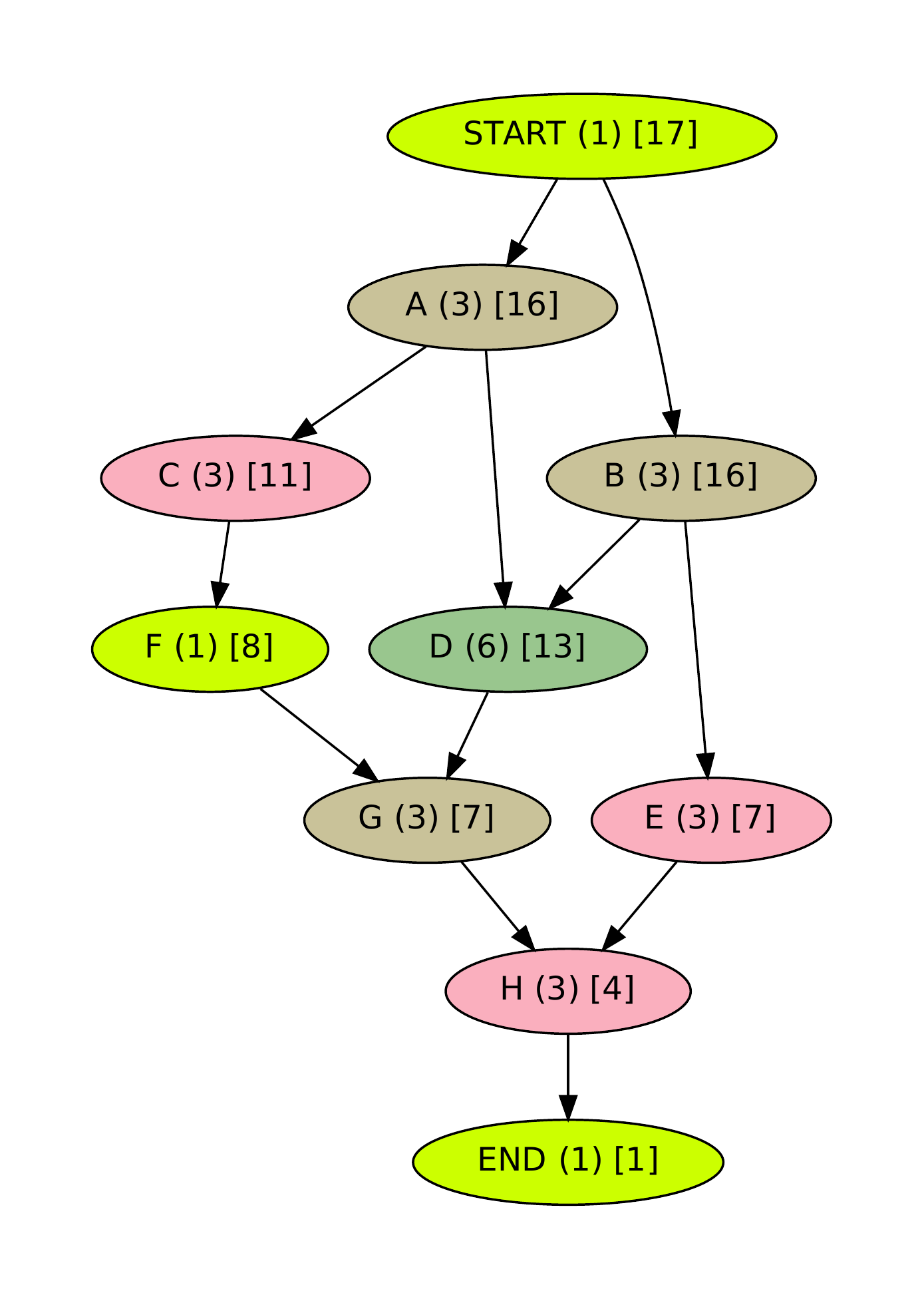}}
}
\caption{Example derivation of task priorities via the Backflow algorithm}
\label{fig:exbackflow}
\end{figure}
\renewcommand{\baselinestretch}{\normalspace}

By following the path with the highest priorities, a critical path can be
discerned from the weighted DAG.  Thus any schedule which then chooses from the
available tasks those with the highest priorities to execute first inherently
follows the critical path.  It is well known that a critical path scheduling is
not always optimal.  As an example~\cite{robert1991impact}, take two processors and four tasks.  Let
tasks A, B, C, and D have weights of 3, 3, 1, and 1, respectively and let the
only relationship between tasks be that C is a predecessor of D.  Then
$\mathop{\mathrm{cp}}(A)=\mathop{\mathrm{cp}}(B)=3$,
$\mathop{\mathrm{cp}}(C)=2$, and $\mathop{\mathrm{cp}}(D)=1$.  A critical path
schedule would choose to schedule tasks A and B simultaneously and follow up
with C and then D, resulting in a schedule of length five.  However, if A and C
are scheduled simultaneously and then D follows A on the same processor and B
follows C on the other processor, the length of the schedule is four. 

We will use this critical path information to analyze three schedules by
choosing available tasks via $\max (cp)$, $\mathop{\mathrm{rand}}(cp)$, or $\min
(cp)$.  The $\max(cp)$ will naturally follow the critical path by scheduling
tasks with the highest $cp$ first and, vice versa, the $\min(cp)$ will schedule
from the available tasks those with the minimum $cp$.  Between
these two, we also choose randomly amongst the available tasks with
$\mathop{\mathrm{rand}}(cp)$.

\section{Scheduling with synchronizations}
\label{sec:schedsync}

The right-looking version of the LAPACK Cholesky factorization as depicted in
Figure~\ref{fig:threechol_r} provides an alternative schedule which can be
easier to analyze and understand.  We will apply the three steps of the
algorithm to a matrix of $t\times t$ tiles. In the tiled setting, we can provide
synchronization points between the varying tasks of each step and simply
schedule any of the available tasks since there are no dependencies between the
tasks in each grouping.  By adding these synchronizations, this schedule is not
able to obtain the critical path no matter how many processors are available.
Algorithm~\ref{alg:cholceil1} is the right-looking variant of the Cholesky
factorization with added synchronization points.
\linespread{1.2}
\begin{algorithm}
  \emph{Tile Cholesky Factorization (compute L such that $A=LL^T$)}\;
  \For{$j=0$ \KwTo $t-1$}{
      schedule POTRF(i) \;
      synchronize\;
      \For{$j=i+1$ \KwTo $t-1$}{
          schedule TRSM(j,i) \;
      }
      synchronize\;
      \For{$j=i+1$ \KwTo $t-1$}{
          \For{$k=i+1$ \KwTo $j-1$}{
              schedule GEMM(j,i,k) \;
          }
      }
      synchronize\;
      \For{$j=i+1$ \KwTo $t-1$}{
          schedule SYRK(j,i) \;
      }
      synchronize\;
  }
  \caption{Schedule for tiled right-looking Cholesky factorization with added
  synchronizations to allow for grouping.}
  \label{alg:cholceil1}
\end{algorithm}
\renewcommand{\baselinestretch}{\normalspace}

Naturally, we can improve upon the above schedule by removing the
synchronization between some of the groupings (Algorithm~\ref{alg:cholceil2}).
The update of the trailing matrix is composed of two groupings, namely the GEMMs
and the SYRKs, which can be executed in parallel if enough processors are
available.  Moreover, the added synchronization point between the update of the
trailing matrix and the factorization of the next diagonal tile may also be
removed.  This schedule does become more complex, but given enough processors,
the schedule is able to obtain the critical path as the limiting factor to
performance.  The minimum number of processors, $p$, needed to obtain the
critical path is $p = \left\lceil\frac{1}{2}(t-1)^2 \right\rceil$ for a matrix
of $t\times t$ tiles since the highest degree of parallelism is realized for the
update of the first trailing matrix which is of size $(t-1) \times (t-1)$.

Both of these schedules differ from the critical path scheduling due to the
added synchronization points and will show lower theoretical performance.  In
the theoretical results, we only show Algorithm~\ref{alg:cholceil2}.
\linespread{1.2}
\begin{algorithm}
  \emph{Tile Cholesky Factorization (compute L such that $A=LL^T$)}\;
  \For{$j=0$ \KwTo $t-1$}{
      schedule POTRF(i) \;
      synchronize\;
      \For{$j=i+1$ \KwTo $t-1$}{
          schedule TRSM(j,i) \;
      }
      synchronize\;
      \For{$j=i+1$ \KwTo $t-1$}{
          \For{$k=i+1$ \KwTo $j-1$}{
              schedule GEMM(j,i,k) \;
          }
          schedule SYRK(j,i) \;
      }
  }
  \caption{Improvement upon Algorithm~\ref{alg:cholceil1}}
  \label{alg:cholceil2}
\end{algorithm}
\renewcommand{\baselinestretch}{\normalspace}

\section{Theoretical Results}
\label{sec:theoryres}


In the following figures, our \emph{Rooftop} bound will be that which uses the
perfect speedup until the weighted critical path is the limiting factor, i.e.,
\begin{equation}
    \mbox{\emph{Rooftop Bound}} = \max \left( \mbox{Critical Path},
    \frac{T_{seq}}{p} \right)
    \label{eqn:rooftop}
\end{equation}
We will compare this to our \emph{ALAP Derived} bound which was derived using
the ALAP execution, our various scheduling strategies, and the following lower bound.
From~\cite[p.221,\Section7.4.2]{casanova2009parallel}, given our DAG, we know
that the make span, $MS$, of any list schedule, $\sigma$, for a given number of
processors, $p$, is 
\[ MS(\sigma,p) \leq \left( 2 - \frac{1}{p} \right) MS_{opt}(p) \] 
where $MS_{opt}$ is the make span of the optimal list schedule without
communication costs.  However, we do not know $MS_{opt}$ and must therefore
substitute the make span of the Critical Path Scheduling using the maximum
strategy to compute our lower bound.  

The ALAP Derived bound improves upon the Rooftop bound precisely in the area that is of
most concern, namely where there is enough parallelism but not enough processors
to fully exploit that parallelism.  
\linespread{1.0}
\begin{figure}[htb]
\centering
\subfloat[Speedup]{%
    \hspace{0mm}\resizebox{.475\textwidth}{!}{\includegraphics{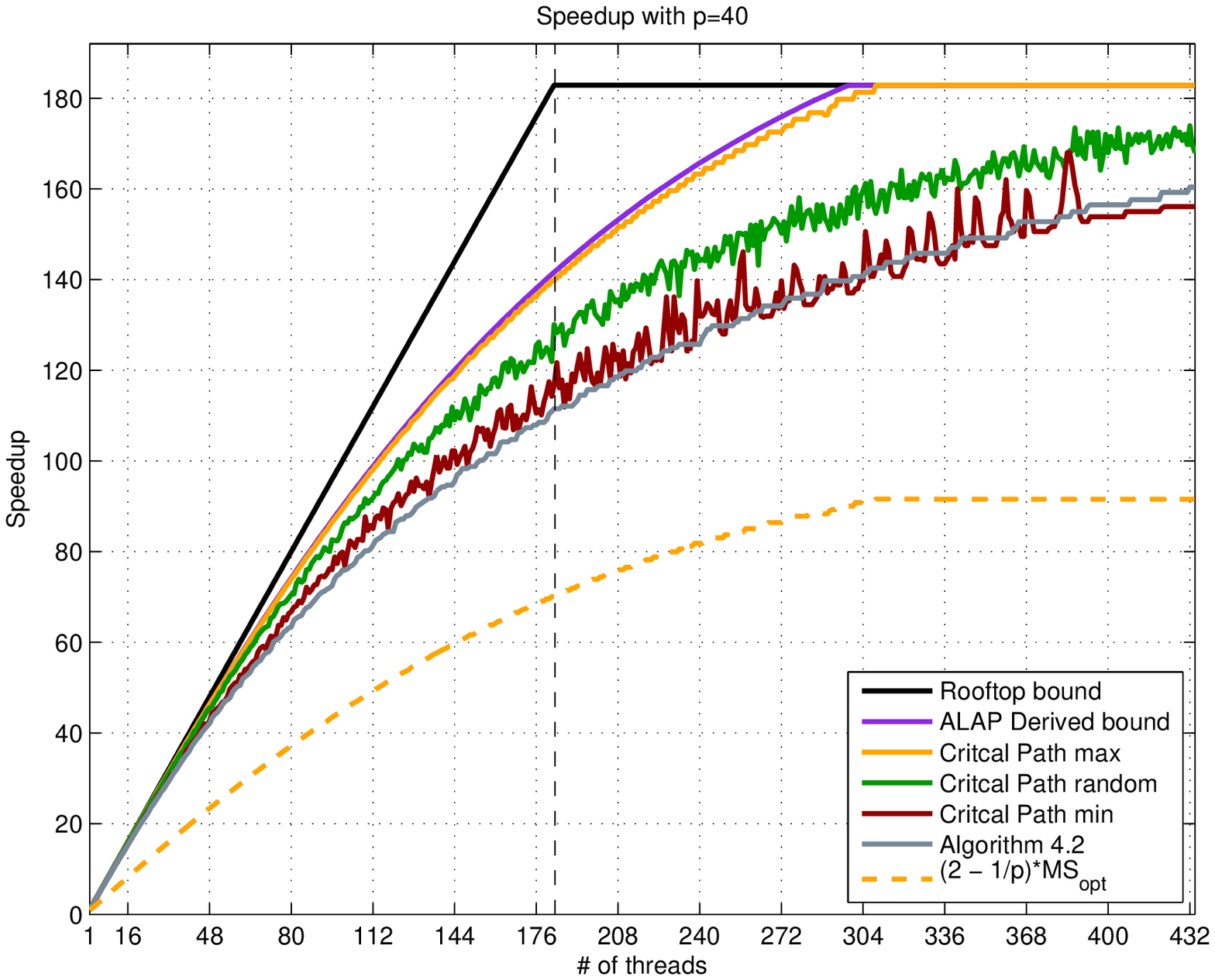}}
}
\subfloat[Efficiency]{%
    \hspace{2mm}\resizebox{.475\textwidth}{!}{\includegraphics{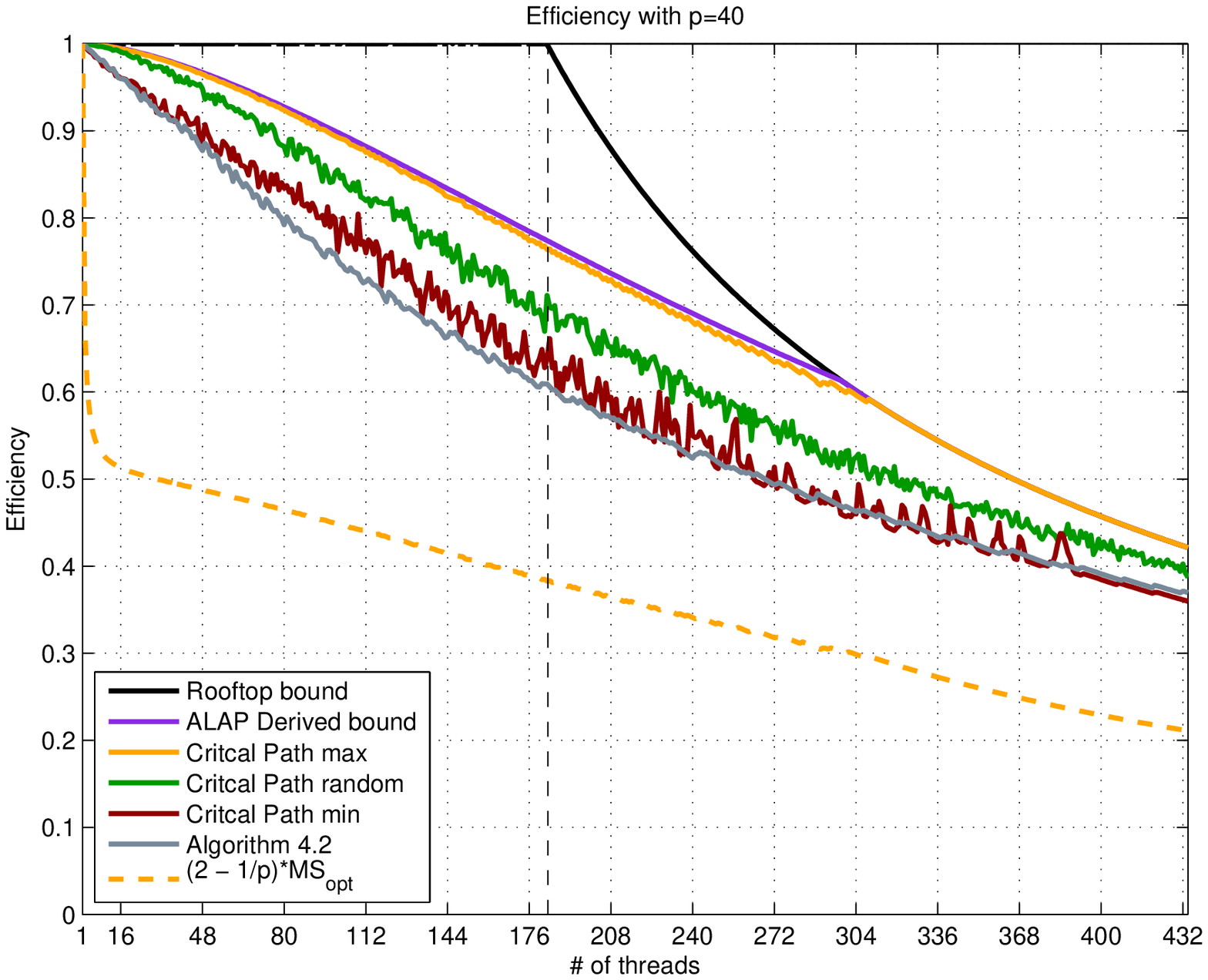}}
}\\
\subfloat[Comparison to new bound]{%
    \hspace{0mm}\resizebox{.475\textwidth}{!}{\includegraphics{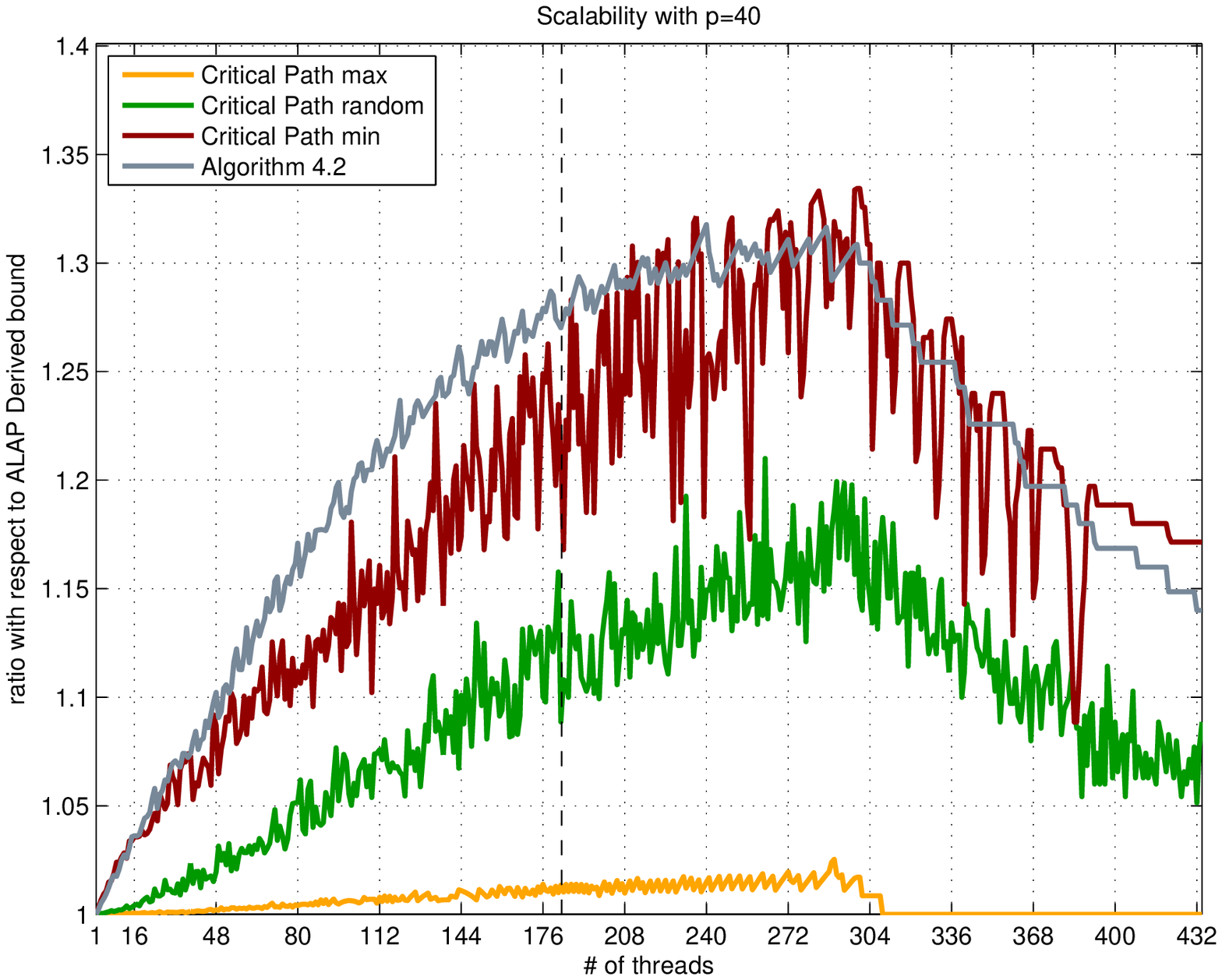}}
}
\caption{Theoretical results for matrix of $40
\times 40$ tiles.
}
\label{fig:theoryresults40} 
\end{figure}
\renewcommand{\baselinestretch}{\normalspace}

Figure~\ref{fig:theoryresults40}(a) shows that the Critical Path schedule is
actually quite descent and comes to within two percent of the ALAP Derived
bound.  Moreover, the ALAP Derived bound has significantly reduced the gap
between the Rooftop bound and any of our list schedules.  

\section{Toward an $\alpha_{opt}$}
\label{sec:alpha0}

It is interesting to know how many processors one needs to be able to schedule
all of the tasks and maintain the weighted critical path.  We will view this
problem in terms of tiles and state the problem as follows:

\begin{quote}
    \emph{Given a matrix of $t \times t$ tiles, determine the minimum number of
    processors, $p_{opt}$, needed to schedule all tasks and achieve an execution time
    equal to the weighted critical path.}
\end{quote}

Toward that end, we will let $p = \alpha t^2$ where $0 < \alpha \leq 1$.  Our
analysis will be asymptotic in which we let $t$ tend to infinity.  From the
analysis of Algorithm~\ref{alg:cholceil2}, we already know that $\alpha_{opt}
\leq \frac{1}{2}$.  Using MATLAB, we have calculated the $\alpha$ value for
matrices of $t \times t$ tiles as shown in Figure~\ref{fig:alpha3to40}.  It is
our conjecture that $\alpha_{opt} \approx \frac{1}{5}$.

\linespread{1.0}
\begin{figure}[htbp]
    \centering
    \resizebox{0.475\linewidth}{!}{%
        \includegraphics[trim = 0mm 0mm 0mm 0mm, clip=true]{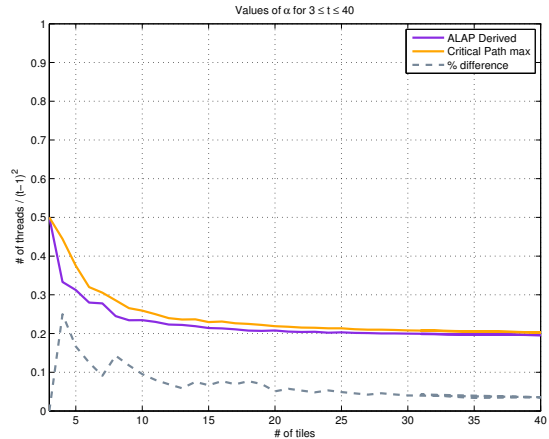}}
    \caption{Values of $\alpha$ for matrices of $t \times t$ tiles where $3
    \leq t \leq 40$}
    \label{fig:alpha3to40}
\end{figure}
\renewcommand{\baselinestretch}{\normalspace}

\section{Related Work}
\label{sec:cholrelated}

For the LU decomposition with partial pivoting, much work has been accomplished
to discern asymptotically optimal algorithms for all values of
$\alpha$~\cite{Marrakchi1989183,Lord:1983:SLA:322358.322366,robert1991impact}.  They
consider a problem of size $n$ and assume $p = \alpha n$ processors on which to
schedule the LU decomposition.  The critical path length of the optimal schedule
in this case is $n^2$ and $\alpha_{opt} \approx 0.347$ where $\alpha_{opt}$ is
a solution to the equation $3\alpha - \alpha^3 = 1$.   

\linespread{1.0}
\begin{figure}[htb]
\centering
\subfloat[LU Decomposition]{%
    \hspace{0mm}\resizebox{.475\textwidth}{!}{\includegraphics{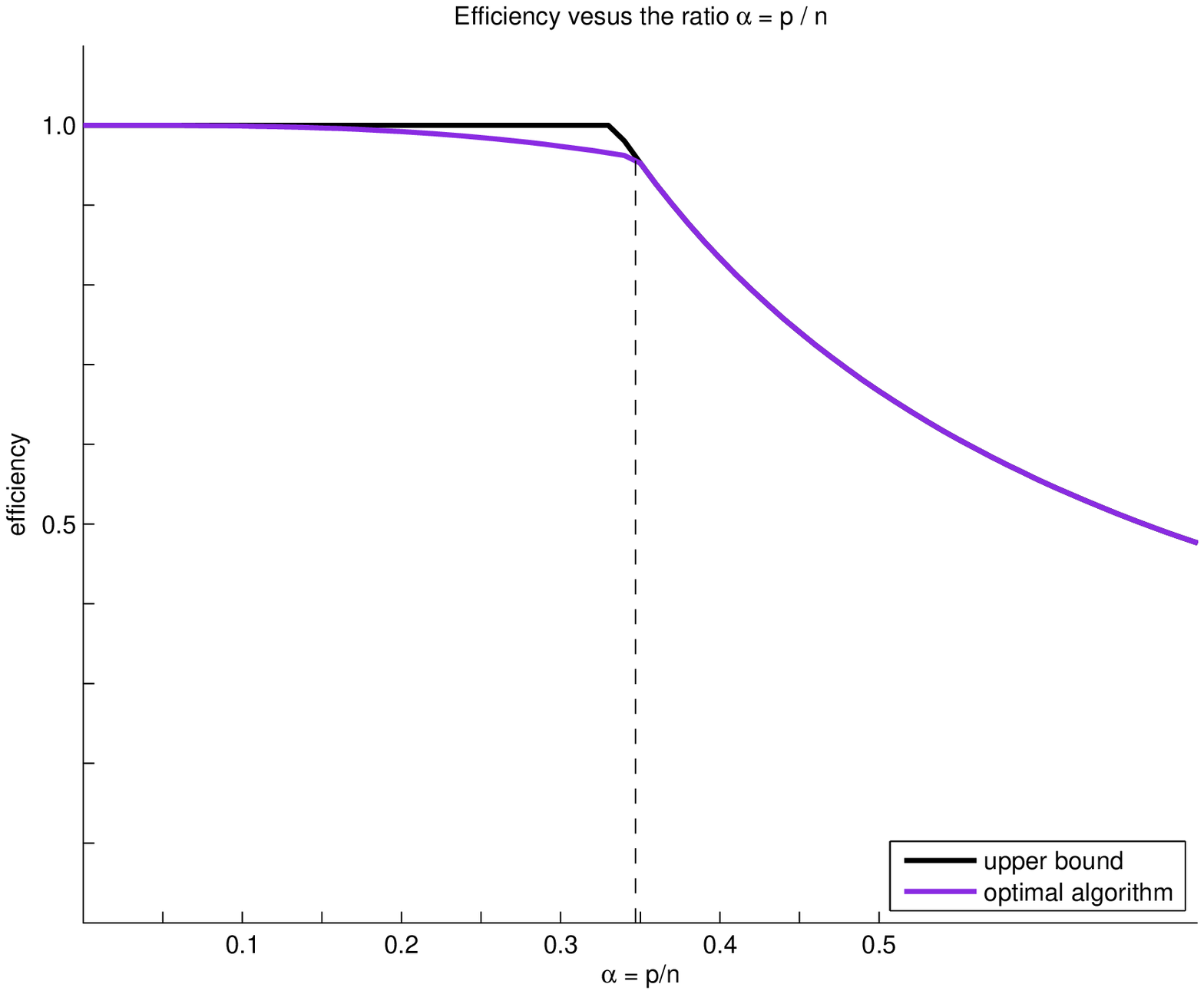}}
}
\subfloat[Tiled Choelsky factorization]{%
    \hspace{2mm}\resizebox{.475\textwidth}{!}{\includegraphics{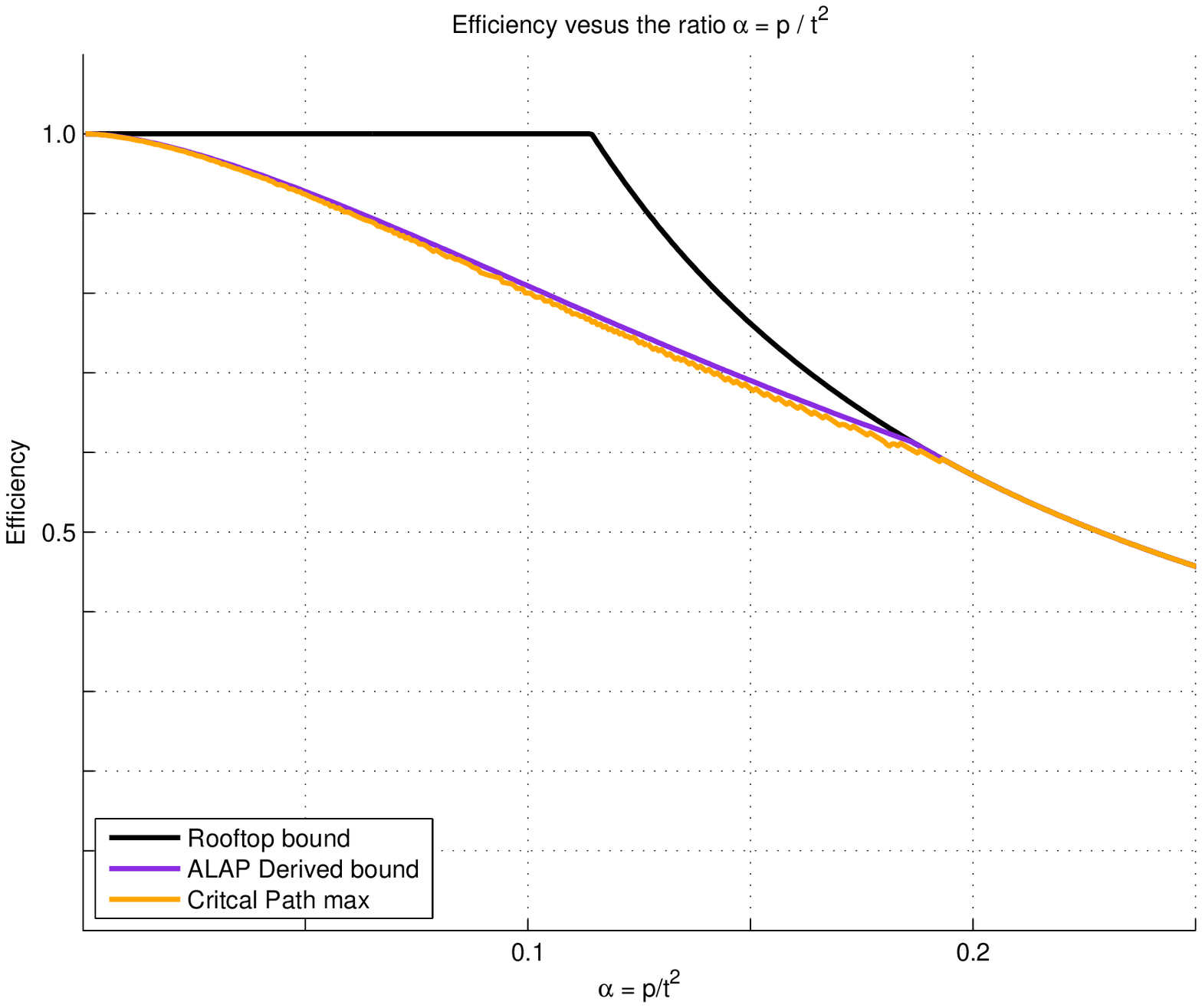}}
}
\caption{Asymptotic efficiency versus $\alpha = p / n$ for LU decomposition and
    versus $\alpha = p / t^2$ for Tiled Cholesky factorization.}
\label{fig:alpharesults40} 
\end{figure}
\renewcommand{\baselinestretch}{\normalspace}

In Figure~\ref{fig:alpharesults40}, we make a comparison between the algorithmic
efficiency of the LU decomposition and the tiled Cholesky factorization.  In the
case of the LU decomposition, the attainable upper bound on efficiency closely
resembles our previous bound of perfect speedup which is then limited by the
critical path.  On the other hand, the tiled Cholesky factorization does not
exhibit this type of efficiency which can be seen from the gap between our
Rooftop bound and the ALAP Derived bound.  Unlike the work
in~\cite{Marrakchi1989183}, we do not provide an algorithm which attains the
ALAP Derived bound.  


\section{Conclusion and future work}
\label{sec:cholschedconclusion}

In many research papers, the performance of an algorithm is usually compared to
either the performance of the GEMM kernel or against perfect scalability
resulting in large discrepancies between the peak performance of the algorithm
and these upper bounds.  If an algorithm displays a DAG such as that of the
tiled Cholesky factorization, it is unrealistic to expect perfect scalability or
even make comparisons to the performance of the GEMM kernel.  Thus one needs to
consider a new bound which is more representative of the algorithm and accounts
for the structure of the DAG.  Without such a bound it is difficult to access
whether there are any performance gains to be achieved.  Although we do not have
a closed-form expression for this new bound, we have shown that such a bound
exists.  Moreover, we have also shown that any algorithm which schedules the
tiled Cholesky factorization while maintaining the weighted critical path will
require $O(t^2)$ processors for a matrix of $t \times t$ tiles and the
coefficient is somewhere around $0.2$.

In this chapter, we have applied a combination of existing techniques to a tiled
Cholesky factorization algorithm to discover a more realistic bound on the
performance and efficiency.  We did so by considering an ALAP execution on an
unbounded number of processors and used this information to calculate the idle
time for any list schedule on a bounded number of processors.  This is then used
to calculate the maximum speedup and efficiency that may be observed.

Further work is necessary to provide a closed-form expression of the new bound
dependent upon the number of processors. In addition, we need to include
communication costs in the bound to make it more reflective of the actual
scheduling problem on parallel machines.  As can be seen in
Figure~\ref{fig:theoryresults40}(c), the Critical Path Schedule is within 2\% of
our ALAP Derived bound.  Although scheduling a DAG on a bounded number of processors is
an NP-complete problem, it may be not be the case for the DAG of the tiled
Cholesky factorization.  Pursuant investigation might show that the Critical
Path Scheduling is the optimal schedule.

\chapter{Scheduling of QR Factorization}\label{chp:qrfact}
%
%

%
%
%

In this chapter, we present collaborative work with Jeffrey Larson.  We revisit
the tiled QR factorization as presented in Chapter~\ref{chp:tiledqr} but do so
in the context of a bounded number of resources. (Chapter~\ref{chp:tiledqr} was
concerned with finding the optimal elimination tree on an unlimited number of
processors.)  We will be using the same analytical tools as in
Chapter~\ref{chp:cholfact} to derive good schedules and to improve upon the
Rooftop bound.  The Cholesky factorization has just one DAG, therefore
Chapter~\ref{chp:cholfact} is a standard scheduling problem, i.e., how to
schedule a DAG on a finite number of processor.  Unlike the previous chapter, we
will need to consider all of the various algorithms (i.e., elimination trees)
and cannot distill the analysis down to a single DAG. 

\section{Performance Bounds}
\label{sec:QRperfbounds}

Each of the algorithms studied in Chapter~\ref{chp:tiledqr}, namely \FT, \MC,
\Greedy, and \GA, produces a unique DAG for a matrix of $p \times q$ tiles.  In
turn, the ALAP Derived bounds~\eqref{eqn:alapbound} for each elimination tree
will also be unique.  In Figure~\ref{fig:newbounds}, we give the computed upper
bounds and make comparisons to the scheduling strategies of maximum, random, and
minimum via the Critical Path Method for a matrix of $20 \times 6$ tiles.  The
matrix size was chosen such that the critical path length of \Greedy is 136 and
the critical path length of \GA is 134 (see Figure~\ref{fig:spyGrASAPvsGreedy} in \Section~\ref{sec:TiledAlgorithms}).

\linespread{1.0}
\begin{figure*}[htb]
\centering
\subfloat[\FT]{%
    \label{fig:newboundFT}%
    \hspace{0mm}\resizebox{.475\textwidth}{!}{\includegraphics{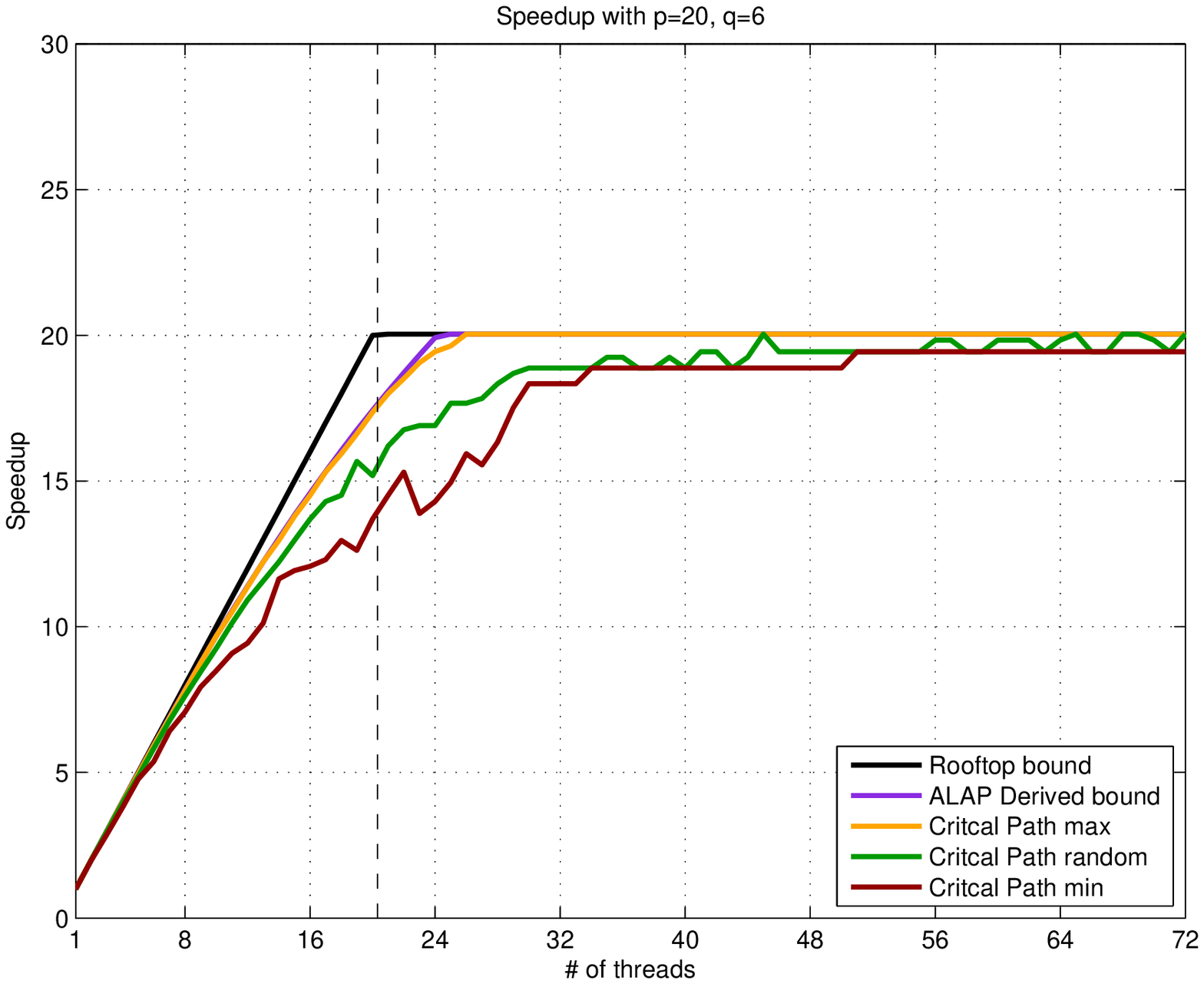}}%
}
\subfloat[\MC]{%
    \label{fig:newboundMC}%
    \hspace{7mm}\resizebox{.475\textwidth}{!}{\includegraphics{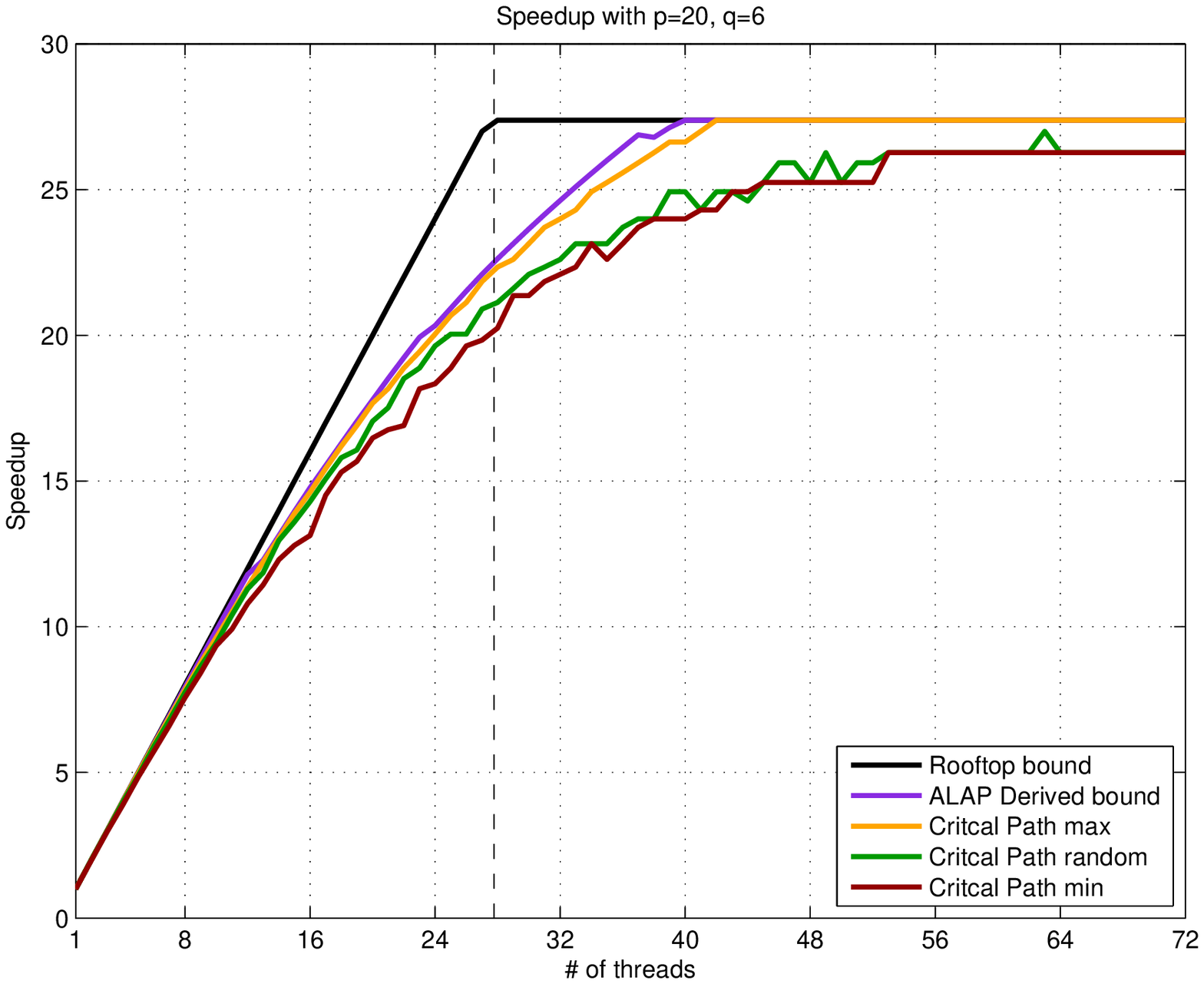}}%
}
\\
\subfloat[\Greedy]{%
    \label{fig:newboundGREEDY}%
    \hspace{0mm}\resizebox{.475\textwidth}{!}{\includegraphics{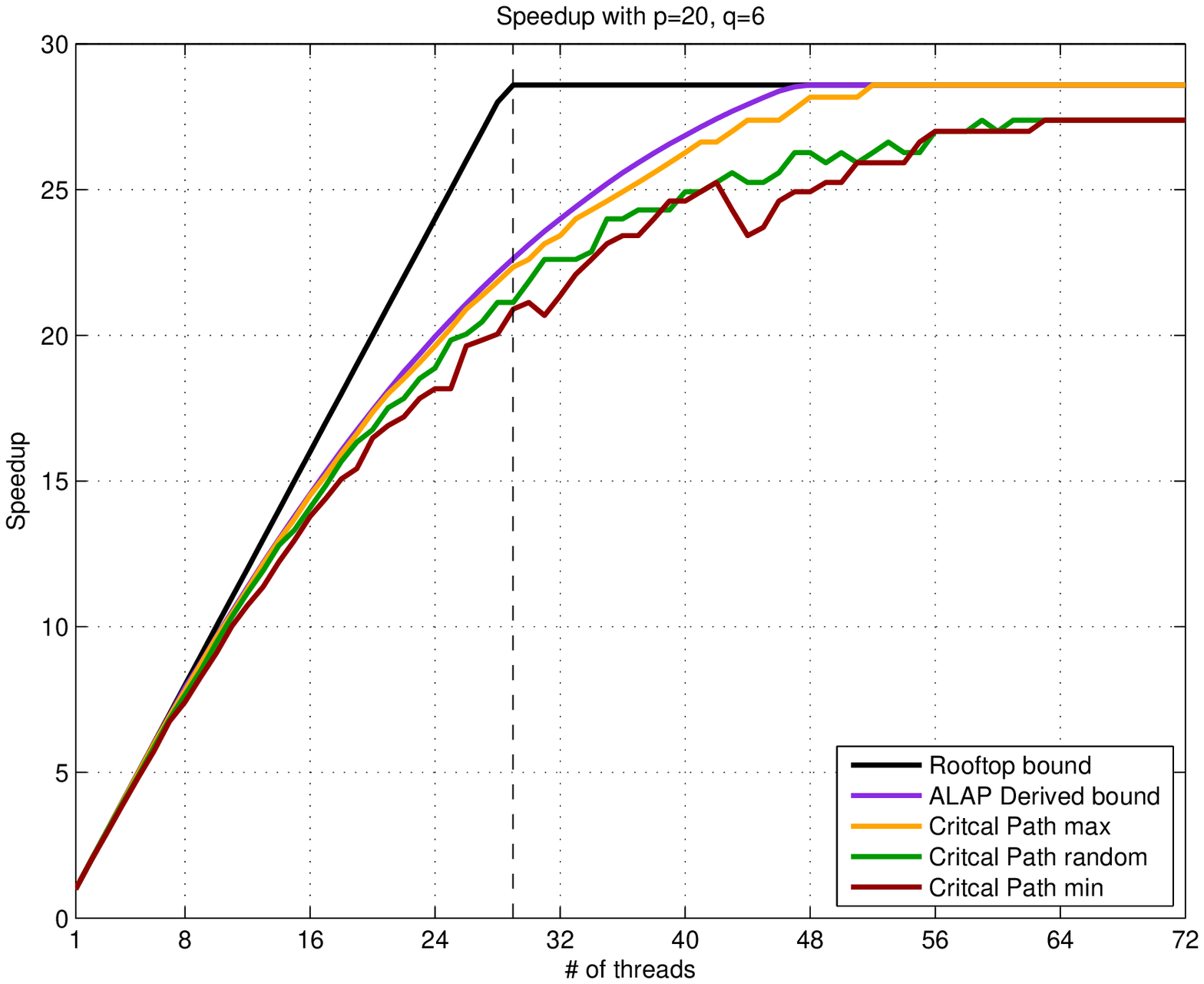}}%
}
\subfloat[\GA]{%
    \label{fig:newboundGrASAP}%
    \hspace{7mm}\resizebox{.475\textwidth}{!}{\includegraphics{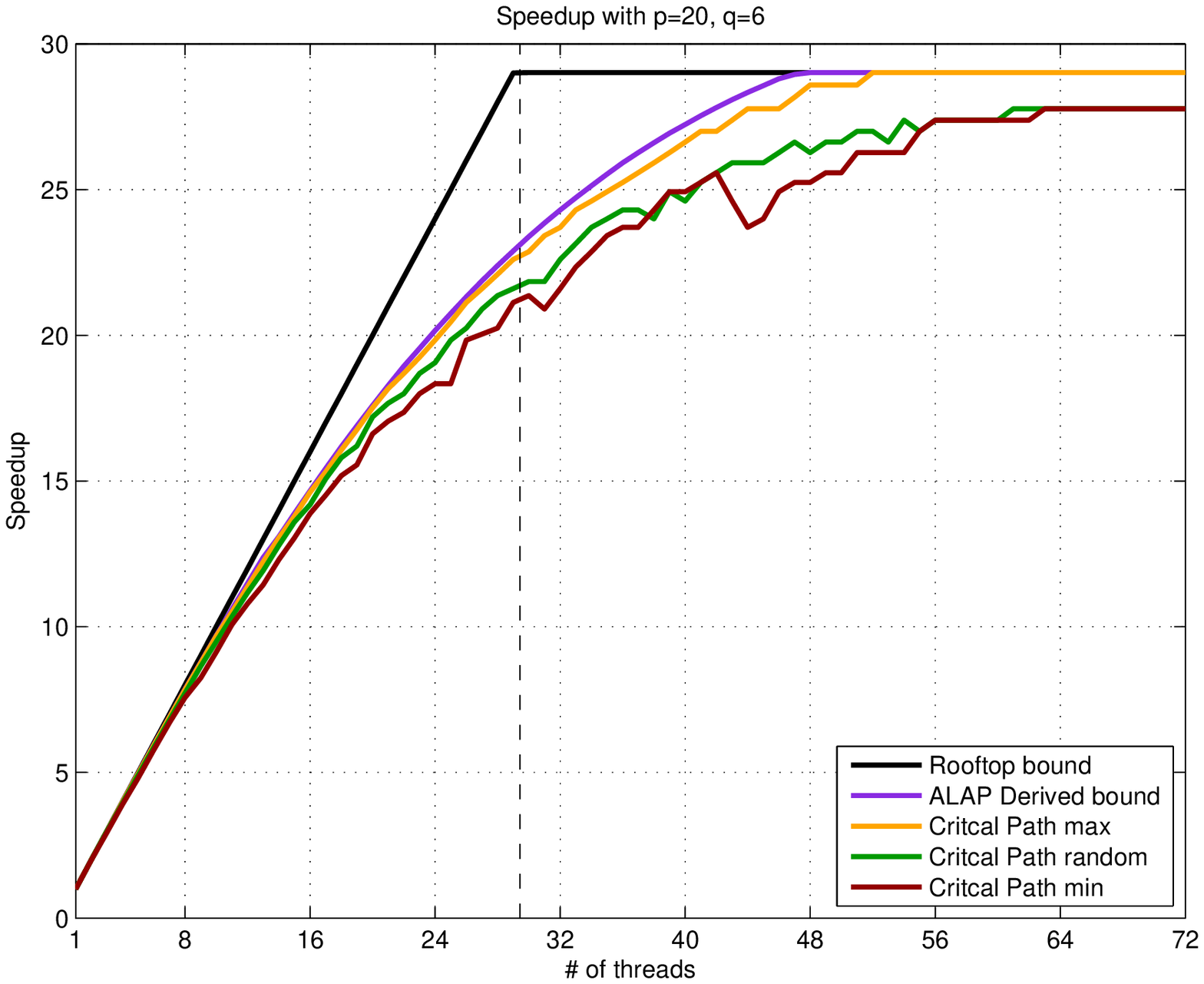}}%
}
\caption{\label{fig:newbounds} Scheduling comparisons for each of the
algorithms versus the ALAP Derived bounds on a matrix of $20 \times 6$ tiles.}
\end{figure*}
\renewcommand{\baselinestretch}{\normalspace}

The \GA algorithm for a tiled matrix is optimal in the scope of unbounded
resources. However by the manner in which the ALAP Derived bound is computed,
the bound created by using \GA cannot hold for all of the other algorithms.
Consider the ALAP execution of the \MC and \GA algorithms on a matrix of
$15\times 4$ tiles.  In Figure~\ref{fig:tail}, we show the execution of the last
tasks for \GA on the left and \MC on the right.  More of the tasks in the ALAP
execution for \MC are pushed towards the end of the execution which means the
ALAP Derived bound will be higher than that of \GA for a schedule that uses
fewer than 10 processors.  In other words, as we add more processors, the Lost
Area (LA) increases much faster for \GA than it does for \MC.  
\linespread{1.0}
\begin{figure*}[htb]
\centering
    \pgfmathsetlength{\imagewidth}{10cm}
    \pgfmathsetlength{\imagescale}{\imagewidth/600}
    \resizebox{0.95\linewidth}{!}{%
        \begin{tikzpicture}[x=\imagescale,y=-\imagescale]
            \node[anchor=south east,inner sep=0pt,outer sep=0pt] at (-100.5,-0.5)
            {\includegraphics[trim = 200mm 0mm 0mm 0mm, clip = true, width=\imagewidth]{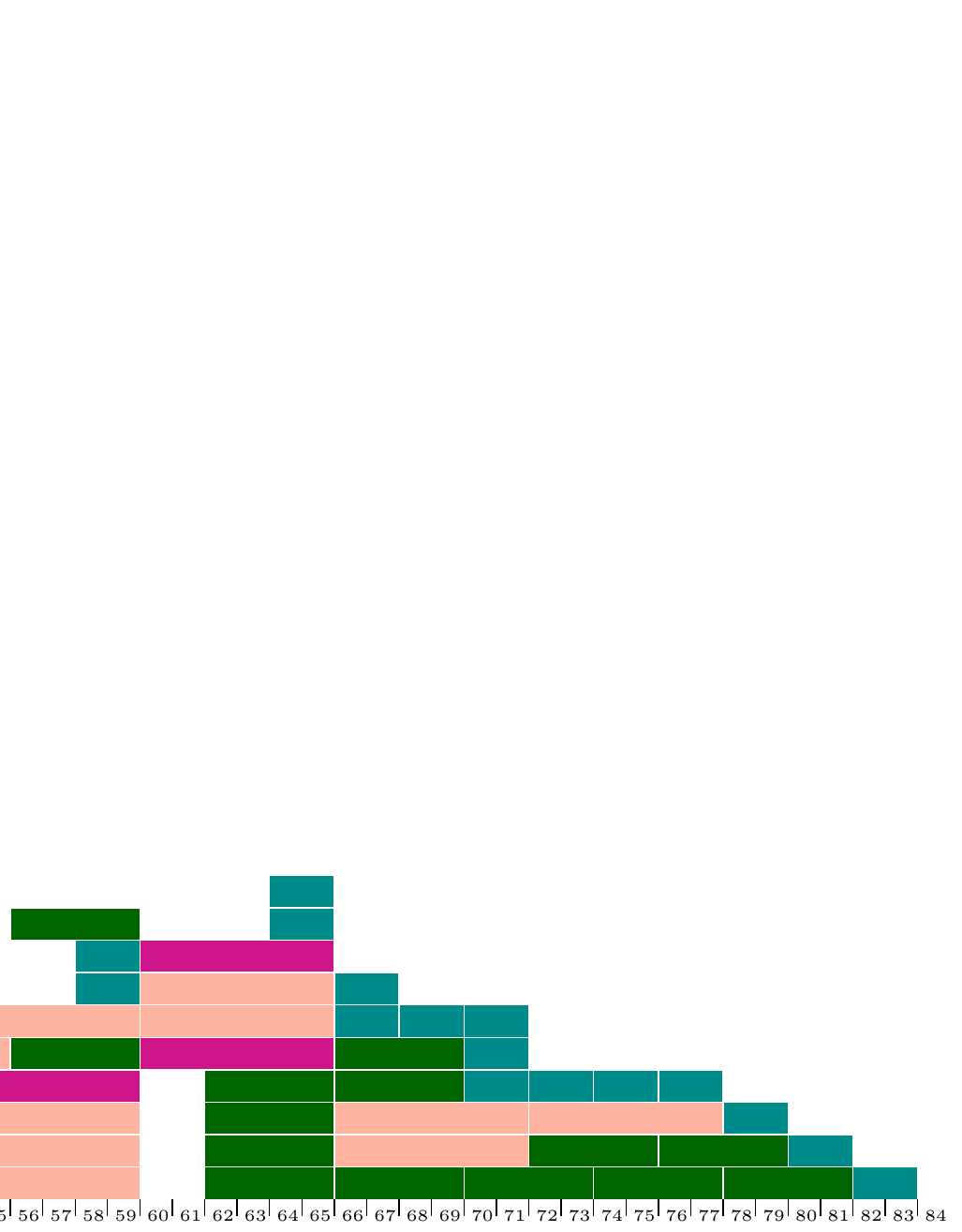}};
            \draw [fill=white] (-710,0) -- (-670,0) -- (-670,-130) --
            (-655,-140) -- (-685,-150) -- (-670,-160) -- (-670,-300) -- (-710,-300);
            \draw [fill=white,white] (-715,-295) rectangle (-665,-305);
            \draw [fill=white,white] (-715,2) rectangle (-665,-2);
            \node at (-400,25) {\GA};
            \node[anchor=south west,inner sep=0pt,outer sep=0pt] at (-10.5,-0.5)
            {\includegraphics[trim = 230mm 0mm 0mm 0mm, clip = true, width=\imagewidth]{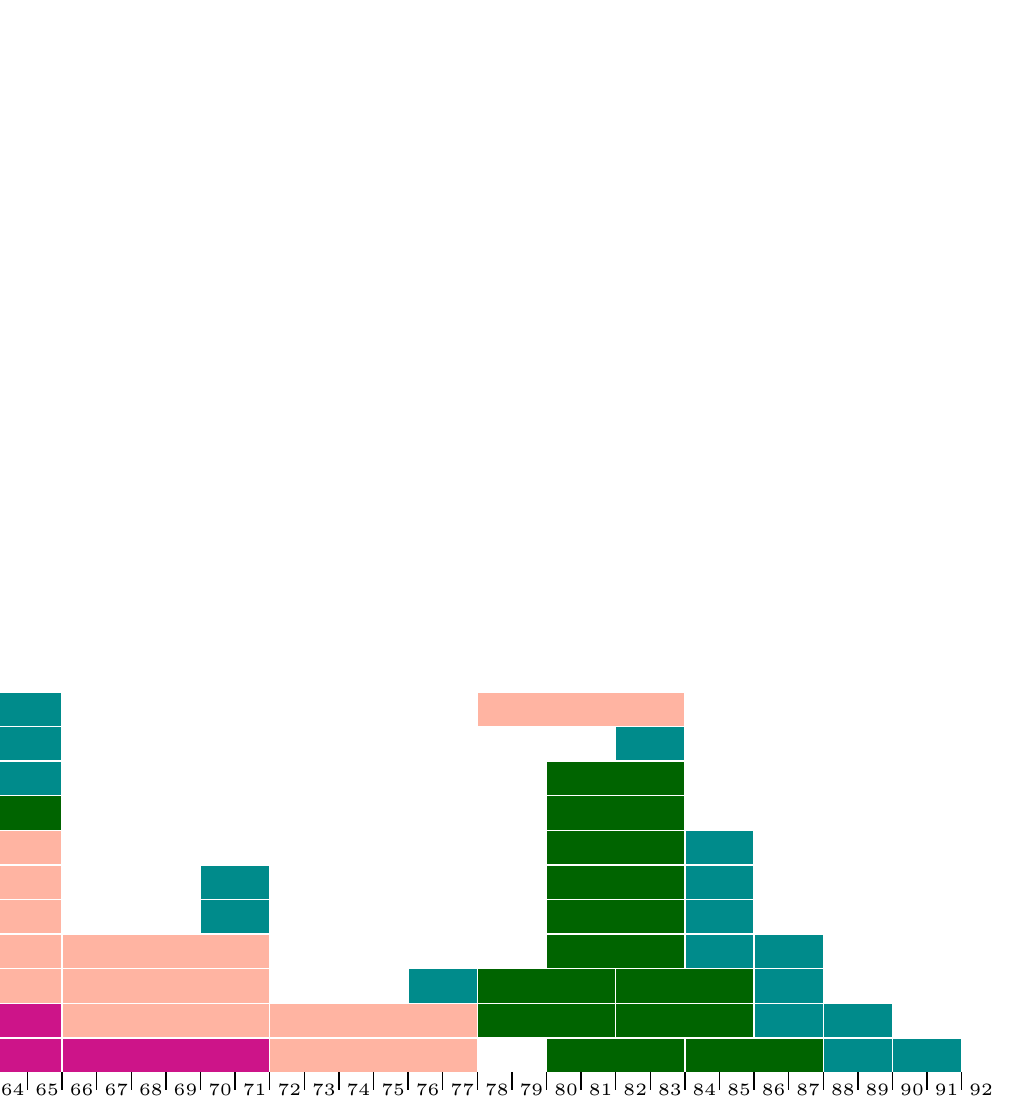}};
            \draw [fill=white] (-15,0) -- (5,0) -- (5,-130) --
            (20,-140) -- (-10,-150) -- (5,-160) -- (5,-300) -- (-15,-300);
            \draw [fill=white,white] (-21,-295) rectangle (10,-305);
            \draw [fill=white,white] (-21,2) rectangle (10,-2);
            \node at (275,25) {\MC};
            \definecolor{tempcolor}{RGB}{0,100,0}
            \draw[color=white,fill=tempcolor] (490,-415) -- (500,-415) -- (500,-405)
            -- (490,-405) -- cycle node[color=black,above=0.1cm, right=0.125cm]
            {\tiny{GEQRT}};
            \definecolor{tempcolor}{RGB}{0,139,139}
            \draw[color=white,fill=tempcolor] (490,-400) -- (500,-400) -- (500,-390) --
            (490,-390) -- cycle node[color=black,above=0.1cm, right=0.125cm]
            {\tiny{TTQRT}};
            \definecolor{tempcolor}{RGB}{205,20,137}
            \draw[color=white,fill=tempcolor] (490,-385) -- (500,-385) -- (500,-375)
            -- (490,-375) -- cycle node[color=black,above=0.1cm, right=0.125cm]
            {\tiny{UNMQR}};
            \definecolor{tempcolor}{RGB}{255,180,162}
            \draw[color=white,fill=tempcolor] (490,-370) -- (500,-370) -- (500,-360)
            -- (490,-360) -- cycle node[color=black,above=0.1cm, right=0.125cm]
            {\tiny{TTMQR}};
        \end{tikzpicture}
    }
    \caption{\label{fig:tail} Tail-end execution using ALAP on unbounded resources
for \GA and \MC on a matrix of $15\times 4$ tiles.}
\end{figure*}
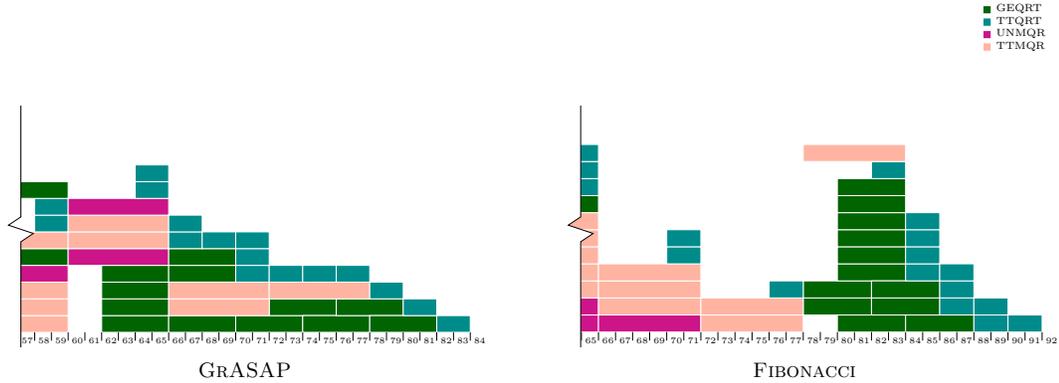
\renewcommand{\baselinestretch}{\normalspace}
Since the critical path length for \MC is greater than that of \GA, after a
certain number of processors, the ALAP Derived bound for \MC falls below that of
\GA.  These observations are evident in Figure~\ref{fig:boundcomps} where we
show a comparison of the bound for each algorithm.  Recall that the Rooftop
bound~\eqref{eqn:rooftop} only takes into account the critical path length of an
algorithm such that for \GA it can be considered a bound for all the algorithms
since \GA is optimal for unlimited resources and thus has the shortest critical
path length.

\linespread{1.0}
\begin{figure}[htb]
\centering
\subfloat[Speedup]{%
    \hspace{0mm}\resizebox{.475\textwidth}{!}{\includegraphics{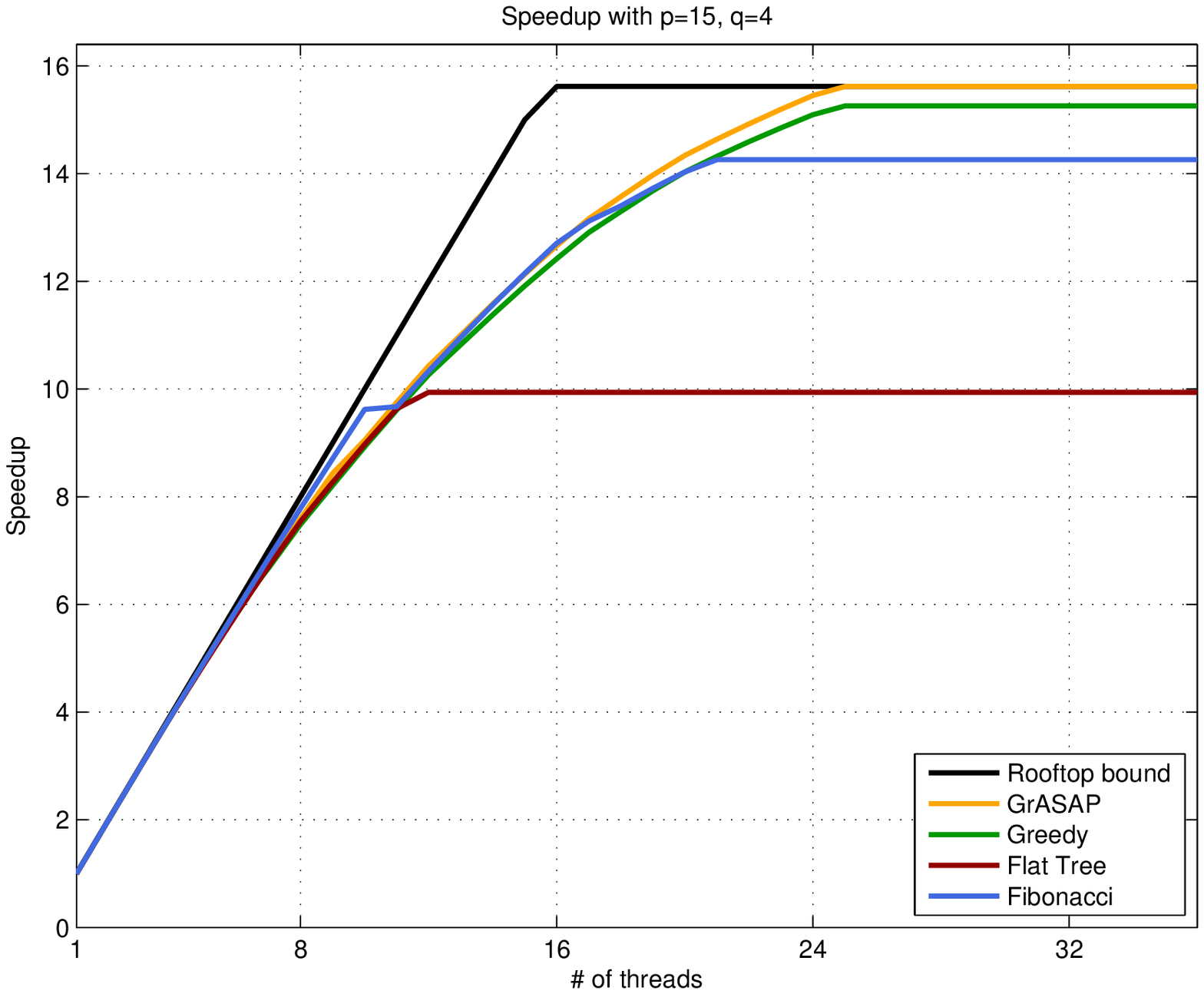}}
}
\subfloat[Efficiency]{%
    \hspace{2mm}\resizebox{.475\textwidth}{!}{\includegraphics{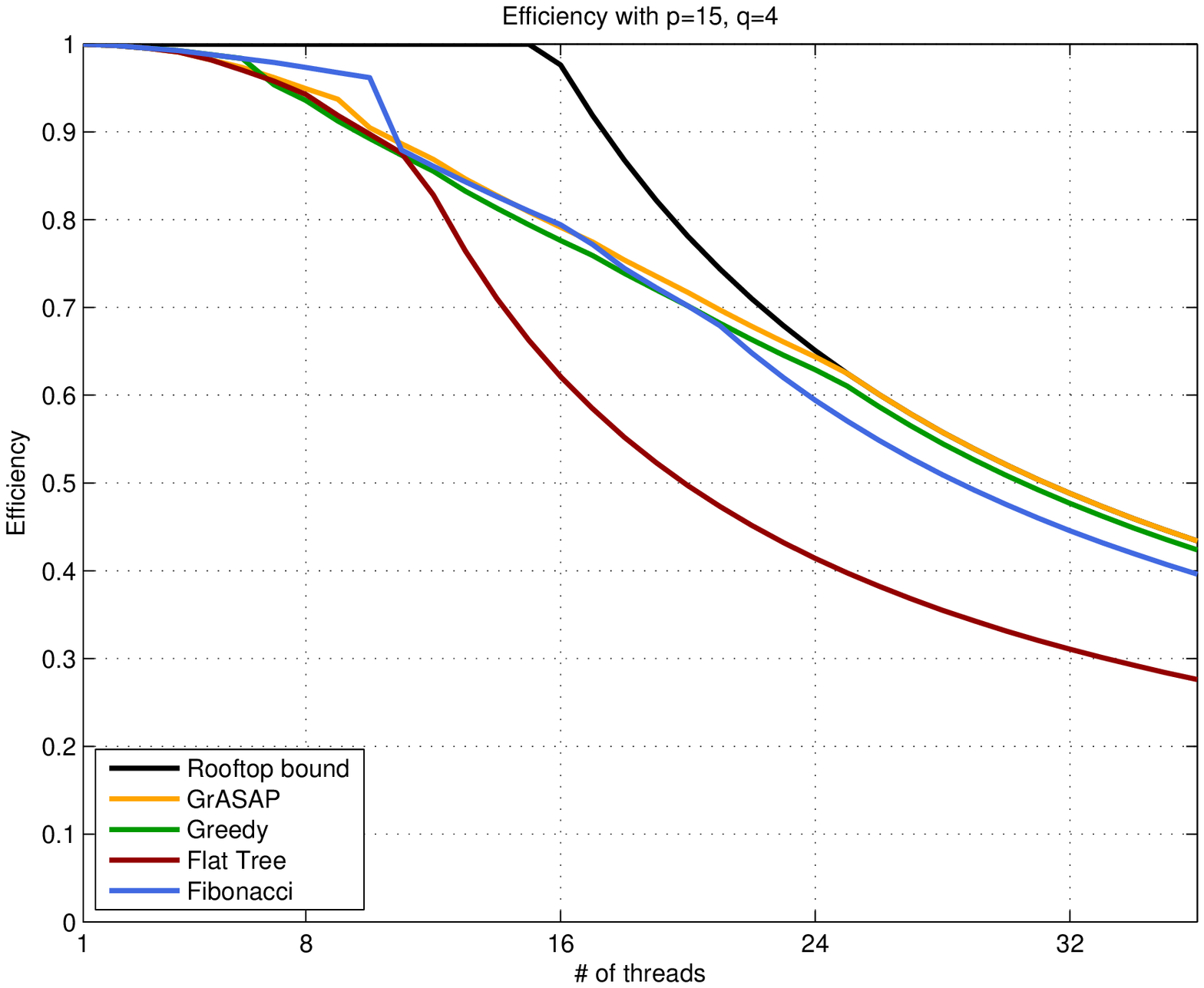}}
}
\caption{ALAP Derived bound comparison for all algorithms for a matrix of
$15\times 4$ tiles.}
\label{fig:boundcomps} 
\end{figure}
\renewcommand{\baselinestretch}{\normalspace}

\section{Optimality}
\label{sec:QRoptimal}
There is no reason for the tree found optimal in Chapter~\ref{chp:tiledqr} on an
unbounded number of resources to be optimal on a bounded number resources.  We
cast the problem of finding the optimal tree with the optimal schedule as an
integer programming problem.  (A complete description of the formulation can be
found in Appendix~\ref{appendix1}.)  Similarly, in~\cite{6012902} a
Mixed-Integer Linear Programming approach was used to provide an upper bound on
performance.  However, the integer programming problem size grows exponentially
as the matrix size increases, thus the largest feasible problem size was a
matrix of $5 \times 5$ tiles.  In Figure~\ref{fig:5x5_optimal} we show the
speedup of the \GA algorithm with its bound and make comparisons to an optimal
tree with an optimal schedule and Table~\ref{tab:5x5_optimal} provides the actual
schedule lengths for all of the algorithms using the CP Method for the matrix of
$5\times 5$ tiles.. 

\linespread{1.0}
\begin{figure}[htbp]
    \centering%
    \resizebox{.45\linewidth}{!}{ \includegraphics{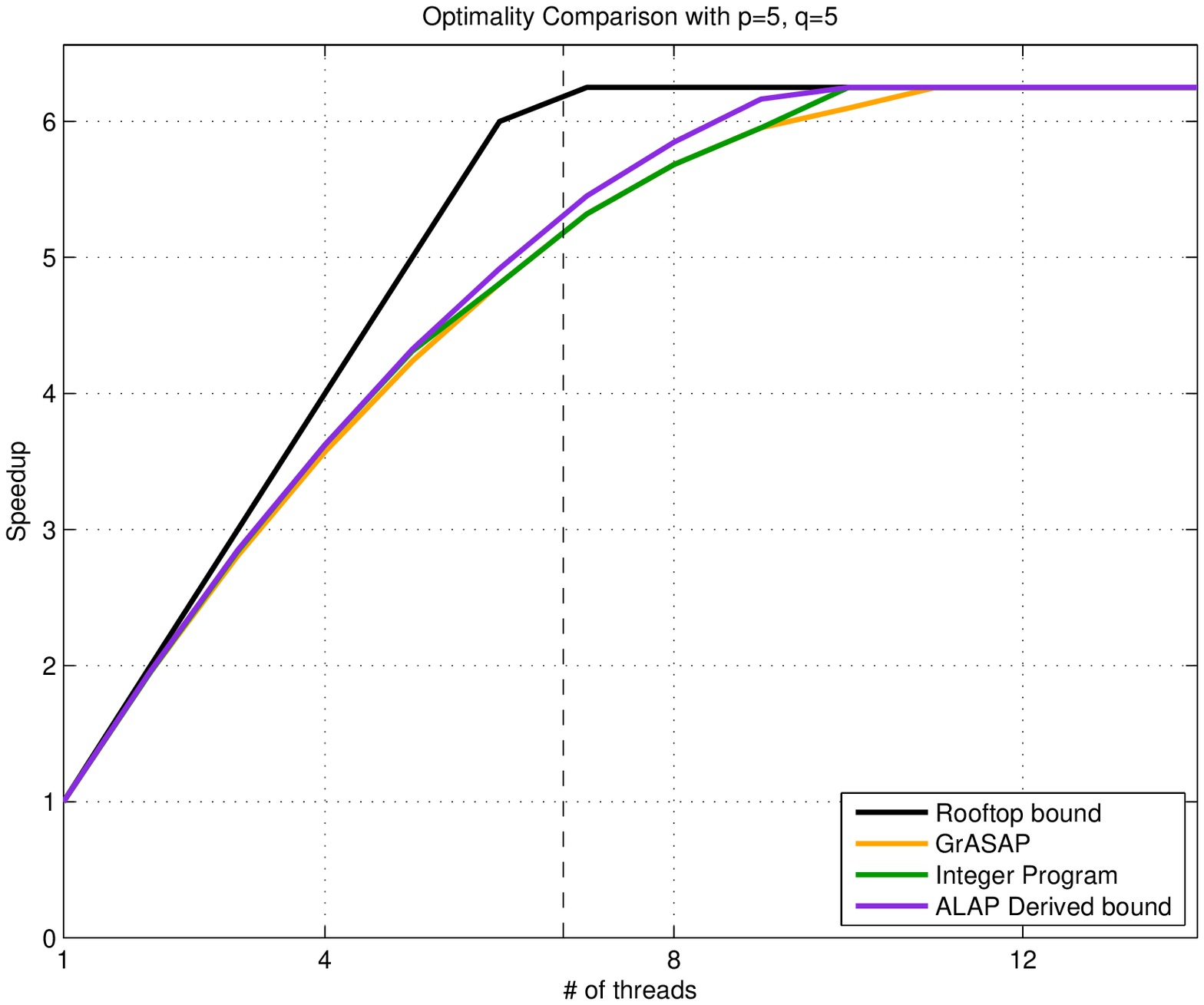}}
    \caption{Comparison of speedup for CP Method on \GA, ALAP Derived bound from \GA, and
    optimal schedules for a matrix of $5\times 5$ tiles on 1 to 14 processors}
    \label{fig:5x5_optimal}
\end{figure}
\begin{table}
            \resizebox{\linewidth}{!}{%
                \begin{tabular}{ccccccc}
                    \toprule
                           & ALAP Derived  & Optimal       & \multicolumn{4}{c}{CP Method}\\
                    Procs  & Bound(\GA)    & Tree/Schedule & \GA & \Greedy & \MC & \FT\\
                    \cmidrule(l ){1-1}
                    \cmidrule(lr){2-2}
                    \cmidrule(lr){3-3}
                    \cmidrule( r){4-7}
                        1  &   500 & 500 & 500 & 500 & 500 & 500\\ 
                        2  &   255 & 256 & 256 & 256 & 256 & 256\\ 
                        3  &   176 & 176 & 178 & 178 & 178 & 176\\ 
                        4  &   138 & 138 & 140 & 140 & 140 & 140\\ 
                        5  &   116 & 116 & 118 & 118 & 118 & 116\\ 
                        6  &   102 & 104 & 104 & 104 & 104 & 104\\ 
                        7  &    92 &  94 &  94 &  94 &  94 &  94\\ 
                        8  &    86 &  88 &  88 &  88 &  88 &  88\\ 
                        9  &    82 &  84 &  84 &  84 &  86 &  86\\ 
                       10  &    80 &  80 &  82 &  82 &  86 &  86\\ 
                       11  &    80 &  80 &  80 &  80 &  86 &  86\\ 
                       12  &    80 &  80 &  80 &  80 &  86 &  86\\ 
                       13  &    80 &  80 &  80 &  80 &  86 &  86\\ 
                       14  &    80 &  80 &  80 &  80 &  80 &  86\\ 
                    \bottomrule
                \end{tabular}
            }
            \caption{Schedule lengths for matrix of $5\times 5$ tiles} 
            \label{tab:5x5_optimal}
\end{table}
\renewcommand{\baselinestretch}{\normalspace}

\section{Elimination Tree Scheduling}
\label{sec:elimsched}
We pair up the choice of the elimination tree with a type of scheduling strategy
to obtain the following bounds:
\[ \resizebox{0.90\textwidth}{!}{ 
    $\left( \begin{array}{c}
        \mbox{\GA}\\ 
        \mbox{Rooftop bound}
    \end{array}\right) 
    \leq 
    \left( \begin{array}{c}
        \mbox{optimal tree}\\ 
        \mbox{optimal schedule}
    \end{array}\right) 
    \leq
    \left( \begin{array}{c}
        \mbox{\GA}\\ 
        \mbox{optimal schedule}
    \end{array}\right) 
    \leq
    \left( \begin{array}{c}
        \mbox{\GA}\\ 
        \mbox{CP schedule}
    \end{array}\right)$}
    \]
Moreover, we also have
\[ \resizebox{0.95\textwidth}{!}{ 
    $\left( \begin{array}{c}
        \mbox{\GA}\\ 
        \mbox{Rooftop bound}
    \end{array}\right) 
    \leq 
    \left( \begin{array}{c}
        \mbox{\GA}\\ 
        \mbox{ALAP Derived Bound}
    \end{array}\right) 
    \leq
    \left( \begin{array}{c}
        \mbox{\GA}\\ 
        \mbox{optimal schedule}
    \end{array}\right) 
    \leq
    \left( \begin{array}{c}
        \mbox{\GA}\\ 
        \mbox{CP schedule}
    \end{array}\right)$}
    \]
Combining these inequalities with Table~\ref{tab:5x5_optimal} gives rise to the following questions:
\begin{quote}
    \emph{Given an optimal elimination tree for the tiled QR factorization on an
    unbounded number of resources}
    \begin{enumerate}[(Q1)]
        \item \emph{does there always exist a scheduling strategy such that the
            schedule on limited resources is optimal?}
        \item \emph{ does the ALAP Derived bound for this elimination tree hold true
            for any scheduling strategy on any other elimination tree?}
    \end{enumerate}
\end{quote}
From Chapter~\ref{chp:tiledqr} we have that \GA is an optimal elimination tree
for the tiled QR factorization.  We know that the length of an optimal schedule
for \GA on $p$ processors will necessarily be greater or equal to the ALAP
Derived bound for \GA on $p$ processors by way of construction of the ALAP
Derived bound. Thus (Q1) implies (Q2).  We cannot address the first question
directly since the size of the matrix needed to produce a counter example is too
large for verification with the integer programming formulation.


Therefore we need to find a matrix size for which a schedule exists whose
execution length is smaller than the ALAP Derived bound from \GA on the same
matrix.  As we have seen in Figure~\ref{fig:boundcomps}, the \MC elimination
tree on a tall and skinny tiled matrix provides the best hope. 

Consider a matrix of $34 \times 4$ tiles on 10 processors.  The ALAP Derived
bound from \GA is 188.  Using the Critical Path method to schedule the \MC
elimination tree we obtain a schedule length of 184.  Therefore the ALAP Derived
bound from \GA does not hold for this schedule.  By implication, we have that
(Q1) is false.  However, the Rooftop bound from \GA is still a valid bound for
all of the schedules.

\section{Conclusion}
\label{sec:qrschedconclusion}
In this chapter we have applied the same tools used in
Chapter~\ref{chp:cholfact} to provide performance bounds for the tiled QR
factorization.  Further, we have shown that the ALAP Derived bound is algorithm dependent.
This leaves that only bound we have for all algorithms is the Rooftop bound as
computed using the \GA algorithm.

The analysis in this chapter has also shown that an optimal algorithm for an
unbounded number of resources does not imply that a scheduling strategy exists
such that it can be scheduled optimally.


\chapter{Strassen Matrix-Matrix Multiplication}\label{chp:strassen}

Matrix multiplication is the underlying operation in many if not most of the
applications in numerical linear algebra and as such, it has garnered much
attention. Algorithms such as the Cholesky factorization, LU decomposition and
more recently the QR-based Dynamically Weighted Halley iteration for polar
decomposition~\cite{NakatsukasaHigham2012}, spend a majority of their
computational cost in matrix-matrix multiplication. The conventional BLAS Level
3 subprogram for matrix-matrix multiplication is of $O(n^{\alpha})$, where
$\alpha = log_2{8} = 3$, computational cost but there exist subcubic
computational cost algorithms.   In 1969, Volker Strassen~\cite{Strassen1969}
presented an algorithm that computes the product of two square matrices of size
$n \times n$, where $n$ is even, using only 7 matrix multiplications at the cost
of needing 18 matrix additions/subtractions which then can be called recursively
for each of the 7 multiplications.  This compares to the standard cubic
algorithm which requires 8 matrix multiplications and only 4 matrix additions.
When Strassen's algorithm is applied recursively down to a constant size, the
computational cost is $O(n^{\alpha})$ where $\alpha = \log_2{7} \approx 2.807$.
Two years later, Shmuel Winograd proved that a minimal of 7 matrix
multiplications and 15 matrix additions/subtractions, which is less than the 18
of Strassen's, are required for the product of two $2 \times 2$ matrices, see
\cite{Winograd1971}.  These discoveries have spawned a flurry of research over
the years.  In 1978, Pan~\cite{4567976} showed that $\alpha < 2.796$.  In the
late 1970's and early 1980's, Bini~\cite{Bini1979} provided $\alpha < 2.78$ with
Sch\"{o}nage~\cite{doi:10.1137/0210032} following up by showing $\alpha < 2.522$
but was usurped the following year by Romani~\cite{doi:10.1137/0211020} who
discerned $\alpha < 2.517$.  In 1986, Strassen brought forth a new approach
which lead to $\alpha < 2.497$.  In 1990, Coppersmith and
Winograd~\cite{Coppersmith:1987:MMV:28395.28396} improved upon Strassen's result
providing the asymptotic exponent $\alpha < 2.376$.  This final result still
stands, but it is conjectured that $\alpha = 2+\varepsilon$ for any $\varepsilon
> 0$ where $\varepsilon$ can me made as small as possible.  Although the
Coppersmith-Winograd algorithm may be reasonable to implement, since the
constant of the algorithm is huge and will not provide an advantage except for
very large matrices, we will not consider it and instead focus on the
Strassen-Winograd Algorithm.

\section{Strassen-Winograd Algorithm}
\label{sec:SWalg}

Here we discuss the algorithm as it would be implemented to compute the product
of two matrices $A$ and $B$ where the result is stored into matrix $C$.  The
algorithm is recursive, thus we describe one step.  Given the input matrices A,
B, and C, divide them into four submatrices,
\[ A = \left[ \begin{array}{cc} A_{11} & A_{12} \\ A_{21} & A_{22} \end{array} \right], \qquad 
   B = \left[ \begin{array}{cc} B_{11} & B_{12} \\ B_{21} & B_{22} \end{array} \right], \qquad 
   C = \left[ \begin{array}{cc} C_{11} & C_{12} \\ C_{21} & C_{22} \end{array} \right] \]
then the 7 matrix multiplications and 15 matrix additions/subtractions are
computed as depicted in Table~\ref{tab:SWAlg} and 
Figure~\ref{fig:strassenwinogradDAG} shows the task graph of the
Strassen-Winograd algorithm for one level of recursion.
\linespread{1.2}
\begin{table}[htb]
   \centering
   \caption{Strassen-Winograd Algorithm}
   \begin{tabular}{l||ll}
      \toprule
      {\bf Phase 1} & $T_1 = A_{21} + A_{22}$ & $T_5 = B_{12} - B_{11}$\\
                    & $T_2 = T_1    - A_{11}$ & $T_6 = B_{22} - T_5$\\   
                    & $T_3 = A_{11} - A_{21}$ & $T_7 = B_{22} - B_{12}$\\
                    & $T_4 = A_{12} - T_2$    & $T_8 = T_6    - B_{21}$\\   
      \midrule
      {\bf Phase 2} & $Q_1 = T_2 \times T_6$        & $Q_5 = T_1 \times T_5$\\
                    & $Q_2 = A_{11} \times B_{11}$  & $Q_6 = T_4 \times B_{22}$\\
                    & $Q_3 = A_{12} \times B_{21}$  & $Q_7 = A_{22} \times T_8$\\
                    & $Q_4 = T_3 \times T_7$        &\\
      \midrule
      {\bf Phase 3} & $U_1 = Q_1 + Q_2$ & $C_{11} = Q_2 + Q_3$\\
                    & $U_2 = U_1 + Q_4$ & $C_{12} = U_1 + U_3$\\
                    & $U_3 = Q_5 + Q_6$ & $C_{21} = U_2 - Q_7$\\
                    &                   & $C_{22} = U_2 + Q_5$\\
      \bottomrule
   \end{tabular}
   \label{tab:SWAlg}
\end{table}
\renewcommand{\baselinestretch}{\normalspace}
\begin{figure}[htbp]
   \begin{center}
      \resizebox{0.75\linewidth}{!}{\includegraphics[trim = 0mm 12mm 0mm 12mm, clip=true]{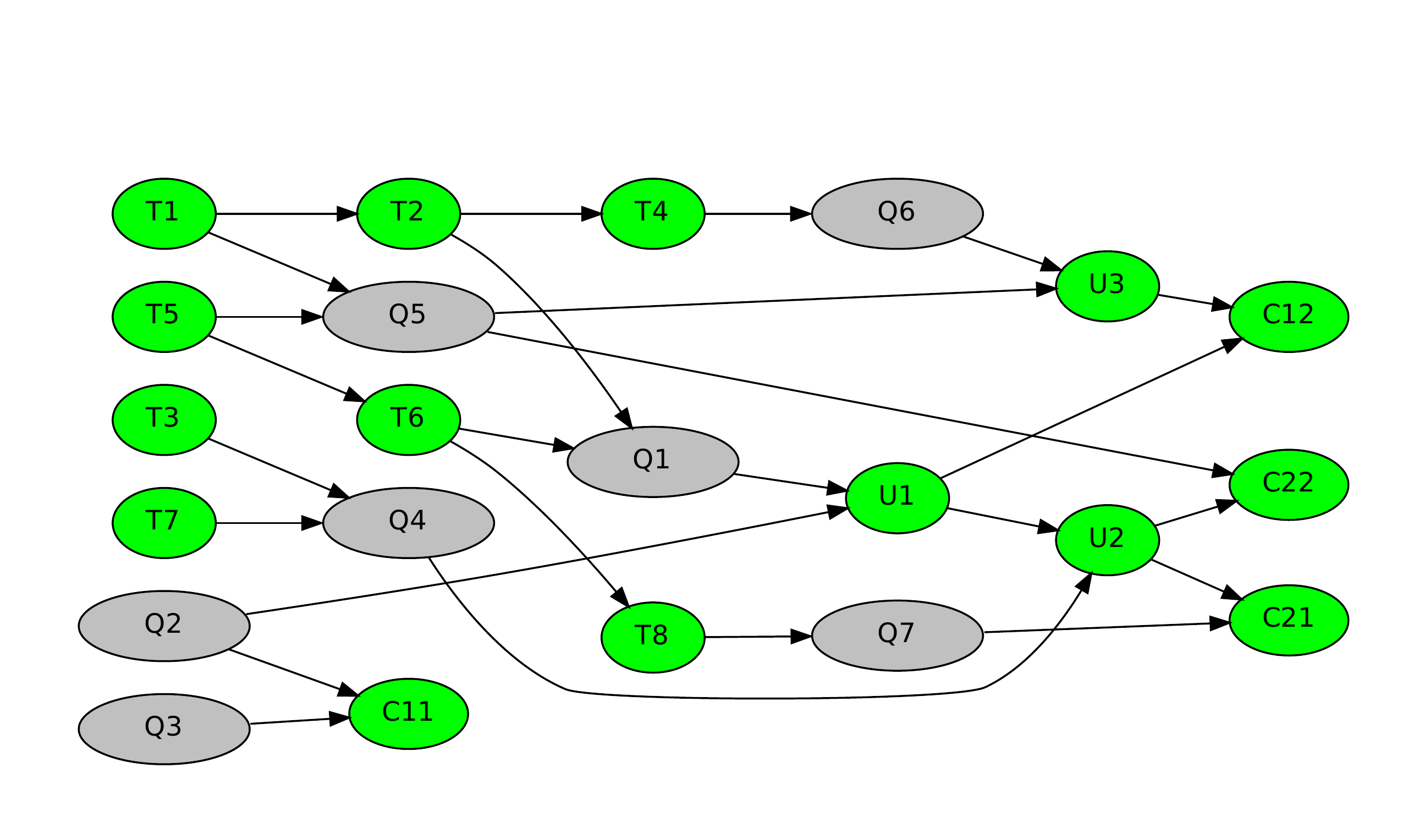}}
   \end{center}
   \caption{Task graph for the Strassen-Winograd Algorithm.  Execution time
   progresses from left to right. Large ovals depict multiplication and small
   ovals addition/subtraction.}
   \label{fig:strassenwinogradDAG}
\end{figure}

In essence, Strassen's approach is very similar to the observation that Gauss
made concerning the multiplication of two complex numbers.  The product of
$(a+bi)(c+di) = ac - bd + (bc + ad)i$ would naively take four multiplications,
but can actually be accomplished via three multiplications by discerning that
$bc + ad = (a + b)(c + d) - bc - ad$.

\section{Tiled Strassen-Winograd Algorithm}
\label{sec:tiledSW}

In our tiled version, the matrices are subdivided such that each submatrix is
of the form
\[ M_{ij} = \left[ \begin{array}{cccc} M_{ij11} & M_{ij12} & \dots  & M_{ij1n}\\
                                       M_{ij21} & M_{ij22} & \dots  & M_{ij2n}\\
                                       \vdots      & \vdots      & \ddots & \vdots\\
                                       M_{ijn1} & M_{ijn2} & \dots  & M_{ijnn}
                                    \end{array} \right] \]
where the matrices $M_{ijkl}$ are tiles of size $n_b \times n_b$.  As before
one can proceed with full recursion, unlike before this would not terminate
at the scalar level, but rather it would terminate with the multiplication of two
tiles using a sequential BLAS Level 3 matrix-matrix multiplication.  The recursion can
also be cutoff at a higher level at which point the tiled matrix multiplication
of Algorithm~\ref{alg:tiledMM} computes the resulting multiplication.  For the
addition/subtraction of the submatrices in Phase 1 and Phase 3 of the
Strassen-Winograd algorithm, a similar approach is used which is also executed
in parallel.

\linespread{1.2}
\begin{algorithm}[htbp]
  \DontPrintSemicolon
  \tcc{Input: $n \times n$ tiled matrices $A$ and $B$, Output: $n \times n$
  tiled matrix $C$ such that $C = A \times B$ }
  \For{$\textnormal{i} = 1$ to $n$}{
     \For{$\textnormal{j} = 1$ to $n$}{
        \For{$\textnormal{k} = 1$ to $n$}{
          $C_{ij} \leftarrow A_{ik} \times B_{kj} + C_{ij}$
       }
     }
  }
  \caption{Tiled Matrix Multiplication (tiled\_gemm)}
\label{alg:tiledMM}
\end{algorithm}
\renewcommand{\baselinestretch}{\normalspace}

If the cutoff for the recursion occurs before the tile level, the computation
for each $C_{ij}$ can be executed in parallel.  Therefore our tiled
implementation of the Strassen-Winograd algorithm exploits two levels of
parallelism.  Moreover, this allows some parts of the matrix multiplications to
occur early on as can be seen in Figure~\ref{fig:winbufp4r1} which shows the
DAG for a matrix of $4 \times 4$ tiles with one level of recursion.  Both
Figure~\ref{fig:winbufp4r1} and Figure~\ref{fig:strassenwinogradDAG} illustrate
one level of recursion but the tiled task graph of a $4 \times 4$ tiled matrix
clearly portrays the high degree of parallelism.

\linespread{1.2}
\begin{figure}[htpb]
   \begin{center}
      \includegraphics[trim = 76mm 51mm 69mm 30mm, clip=true]{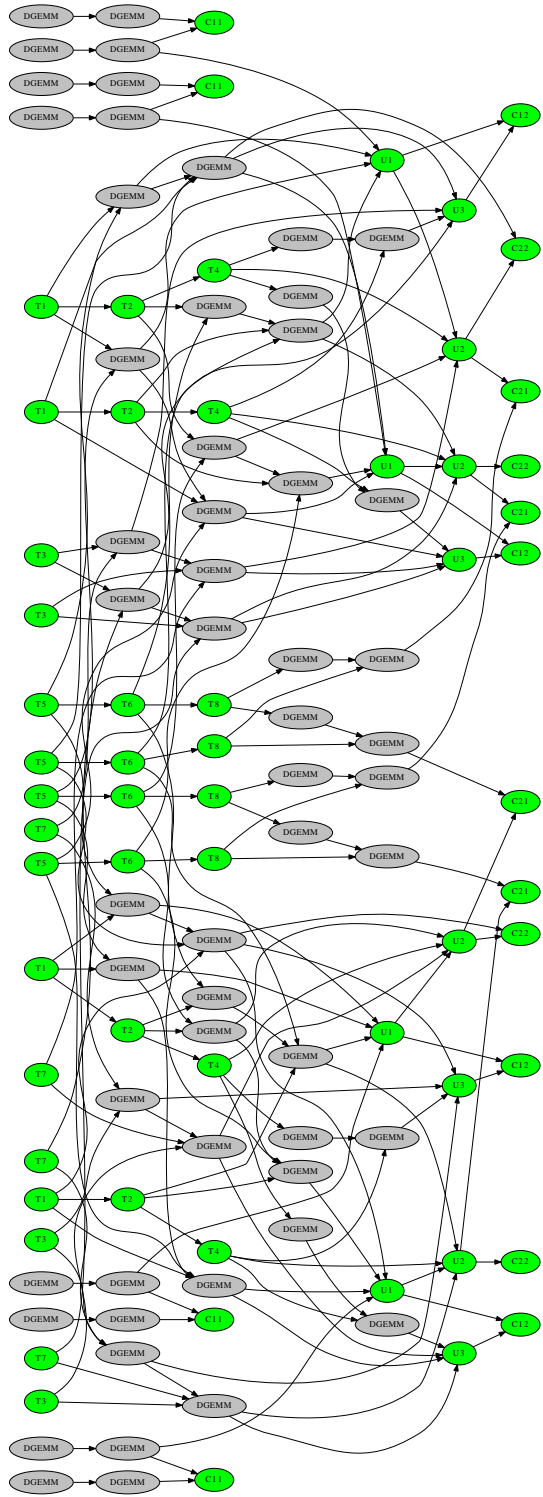}
   \end{center}
   \caption{Strassen-Winograd DAG for matrix of $4 \times 4$ tiles with one
   recursion. Execution time progresses from left to right. Large ovals depict
   multiplication and small ovals addition/subtraction.}
   \label{fig:winbufp4r1}
\end{figure}
\renewcommand{\baselinestretch}{\normalspace}

The conventional matrix-matrix multiplication algorithm requires 8
multiplications and 4 additions whereas the Strassen-Winograd algorithm requires
7 multiplications and 15 additions/subtractions for each level of recursion.
Therefore, there are more tasks for the Strassen-Winograd algorithm as compared
to the conventional matrix-matrix multiplication and it would
behoove us to reduce the number of tasks which would also reduce the algorithmic
complexity.  On the other hand, since we are reducing the number of
multiplications, the computational cost is also reduced since this requires a
cubic operation versus the quadratic operation of the addition/subtractions.  

The total number of tasks, $T$, of the Strassen-Winograd algorithm is given by
\[ T = 7^r\left( \frac{p}{2^r} \right)^3 + 15 \sum_{i=0}^{r-1} 7^{r-i-1}\left( \frac{p}{2^{r-i}} \right)^2 
   = p^3 \left( \frac{7}{8} \right)^r + 5p^2 \left( \left( \frac{7}{4} \right)^r - 1 \right)\]
which is minimized at recursion level $r_{min}$ when
\[ r_{min} = \left\lceil \left( \frac{\ln \left( \frac{p \ln \left( \frac{8}{7}
\right)}{5 \ln \left( \frac{7}{4} \right)} \right)}{\ln 2} \right) \right\rceil, \]
and the total number of flops, $F$, is given by
\[ F = m 7^r\left( \frac{p}{2^r} \right)^3 + 15a \sum_{i=0}^{r-1} 7^{r-i-1}\left( \frac{p}{2^{r-i}} \right)^2 
   = m p^3 \left( \frac{7}{8} \right)^r + 5 a p^2 \left( \left( \frac{7}{4}
   \right)^r - 1 \right),\]
where $m = 2 n_b^3 - n_b^2$ for the multiplications and $a = n_b^2$ for the
additions/subtractions.

As we increase the recursion, the number of tasks will decrease up to a certain
point, $r_{min}$.  The reason for this being at each recursion we reduce the
number of $p^3$ tasks by $\frac{1}{8}$ while increasing the $p^2$ tasks by 15.

\linespread{1.2}
\begin{table}
   \centering
   \begin{tabular}{@{}rccc@{}}
      \toprule
        $p$ & $r_{min}$ & Gflop(SW) & Gflop(GEMM)\\
      \cmidrule( r){1-1}
      \cmidrule(lr){2-2}
      \cmidrule(lr){3-3}
      \cmidrule(l ){4-4}
          4 &     1    & 8.96e-01 &  1.02e+00\\
          8 &     1    & 7.15e+00 &  8.18e+00\\
         16 &     1    & 5.72e+01 &  6.55e+01\\
         32 &     1    & 4.57e+02 &  5.24e+02\\
         64 &     2    & 3.20e+03 &  4.19e+03\\
        128 &     3    & 2.24e+04 &  3.35e+04\\
        256 &     4    & 1.57e+05 &  2.68e+05\\
        512 &     5    & 1.09e+06 &  2.14e+06\\
       1024 &     6    & 7.69e+06 &  1.71e+07\\
       \bottomrule
   \end{tabular}
   \caption{Recursion levels which minimize the number of tasks for a tiled matrix of size $p \times p$}
   \label{tab:rmin}
\end{table}
\renewcommand{\baselinestretch}{\normalspace}

\linespread{1.2}
\begin{table}
   \centering
   \begin{tabular}{@{}cccc@{}}
      \toprule
      algorithm & recursion & tasks & Gflop\\
      \cmidrule( r){1-1}
      \cmidrule(lr){2-2}
      \cmidrule(lr){3-3}
      \cmidrule(l ){4-4}
      \multirow{7}{*}{strassen\_winograd} &  1  &    1,896,448  &  2.92e+04\\
            & 2  & 1,774,592 & 2.56e+04\\
            & 3  & 1,762,048 & 2.24e+04\\
            & 4  & 1,915,712 & 1.96e+04\\
            & 5  & 2,338,288 & 1.72e+04\\
            & 6  & 3,212,252 & 1.51e+04\\
            & 7  & 4,859,338 & 1.33e+04\\
      \cmidrule( r){1-1}
      \cmidrule(lr){3-3}
      \cmidrule(l ){4-4}
      tiled\_DGEMM & & 4,177,920 &  3.36e+04\\
      \bottomrule
   \end{tabular}
   \caption{$128\times 128$ tiles of size $n_b = 200$}
   \label{tab:comp128sw}
\end{table}
\renewcommand{\baselinestretch}{\normalspace}

In our experiments, letting the tile size $n_b = 200$ provided the best
performance.  Thus Table~\ref{tab:rmin} shows the corresponding values for the
minimizing recursion level for various number of tiles. However, as can be seen
in Table~\ref{tab:comp128sw}, the $r_{min}$ which provides the minimum number of
tasks does not provide the least amount of computational cost.  The
computational costs will be minimized at the full recursion.

\linespread{1.2}
\begin{table}
   \centering
   \begin{tabular}{@{}rrrrr@{}}
      \toprule
      $p$ & $r$ & \# tasks & CP & Ratio\\
      \cmidrule( r){1-1}
      \cmidrule(lr){2-2}
      \cmidrule(lr){3-3}
      \cmidrule(lr){4-4}
      \cmidrule(l ){5-5}
      \multirow{2}{*}{4}   & 0 &      64 &     2 &     32.0\\
                           & 1 &     116 &     7 &     16.6\\
      \midrule                                     
      \multirow{3}{*}{8}   & 0 &     512 &     3 &    170.7\\
                           & 1 &     688 &     9 &     76.4\\
                           & 2 &   1,052 &    52 &     20.2\\
      \midrule                                     
      \multirow{4}{*}{16}  & 0 &   4,096 &     4 &  1,023.5\\
                           & 1 &   4,544 &    13 &    349.5\\
                           & 2 &   5,776 &    66 &     87.5\\
                           & 3 &   8,324 &   361 &     23.1\\
      \midrule                                     
      \multirow{5}{*}{32}  & 0 &  32,768 &     5 &  6,553.6\\
                           & 1 &  32,512 &    21 &  1,548.2\\
                           & 2 &  35,648 &    94 &    379.2\\
                           & 3 &  44,272 &   459 &     96.5\\
                           & 4 &  62,108 & 2,524 &     24.6\\
      \midrule                                     
      \multirow{5}{*}{64}  & 0 & 262,144 &     6 & 43,690.7\\
                           & 1 & 244,736 &    37 &  6,614.5\\
                           & 2 & 242,944 &   150 &  1,619.6\\
                           & 3 & 264,896 &   655 &    404.4\\
                           & 4 & 325,264 & 3,210 &    101.3\\
      \bottomrule
   \end{tabular}
   \caption{Comparison of the total number of tasks and critical path length
   for matrix of $p \times p$ tiles.}
   \label{tab:SWcptasks}
\end{table}
\renewcommand{\baselinestretch}{\normalspace}

Even though the number of tasks and thereby the computational complexity are
minimized for $r_{min}$, Table~\ref{tab:SWcptasks} shows that the critical path
length increases exponentially with the recursion levels.  Thus the amount of
parallelism is likewise reduced for each recursion level.

\linespread{1.2}
\begin{algorithm}[htbp]
  \DontPrintSemicolon
  \tcc{p is equal to the recursion cutoff, do the multiplication}
  \If{ $p = r$ }{
      tiled\_gemm( $p, A, B, C$ )\;
  }
  \Else{
      \tcc{p is greater than the recursion cutoff, so we split the problem in half}
      $p = p / 2$\;
      \;
      \tcc{Phase 1}
      tiled\_geadd( $p, A_{21}, A_{22}, T_1$ )\;
      tiled\_geadd( $p, T_1,    A_{11}, T_2$ )\;
      tiled\_geadd( $p, A_{11}, A_{21}, T_3$ )\;
      tiled\_geadd( $p, A_{12}, T_2   , T_4$ )\;
      tiled\_geadd( $p, B_{12}, B_{11}, T_5$ )\;
      tiled\_geadd( $p, B_{22}, T_5   , T_6$ )\;
      tiled\_geadd( $p, B_{22}, B_{12}, T_7$ )\;
      tiled\_geadd( $p, T_6   , B_{21}, T_8$ )\;
      \;
      \tcc{Phase 2}
      tiled\_gesw( $p, T_2   , T_6   , Q_1$ )\;
      tiled\_gesw( $p, A_{11}, B_{11}, Q_2$ )\;
      tiled\_gesw( $p, A_{12}, B_{21}, Q_3$ )\;
      tiled\_gesw( $p, T_3   , T_7   , Q_4$ )\;
      tiled\_gesw( $p, T_1   , T_5   , Q_5$ )\;
      tiled\_gesw( $p, T_4   , B_{22}, Q_6$ )\;
      tiled\_gesw( $p, A_{22}, T_8   , Q_7$ )\;
      \;
      \tcc{Phase 3}
      tiled\_geadd( $p, Q_1, Q_2, U_1$ )\;
      tiled\_geadd( $p, U_1, Q_4, U_1$ )\;
      tiled\_geadd( $p, Q_5, Q_6, U_1$ )\;
      tiled\_geadd( $p, Q_2, Q_3, C_{11}$ )\;
      tiled\_geadd( $p, U_1, U_3, C_{12}$ )\;
      tiled\_geadd( $p, U_2, Q_7, C_{21}$ )\;
      tiled\_geadd( $p, U_2, Q_5, C_{22}$ )\;
  }
  \caption{Tiled Strassen-Winograd (tiled\_gesw)}
\label{alg:tiledSW}
\end{algorithm}
\renewcommand{\baselinestretch}{\normalspace}

\section{Related Work}
\label{sec:SWRelated}

In practice the Strassen-Winograd algorithm imparts a large amount of overhead
for small matrices, thus it is customary to overcome this issue by using it in
conjunction with a conventional matrix multiplication operation.  The key idea
is to provide a recursion cutoff point such that once this is reached, the
algorithm switches from calling itself recursively to calling, e.g., the BLAS
Level 3 matrix multiplication.  The recursion cutoff point is a tuning
parameter which depends upon the machine architecture and can be either dynamic
or static.  In~\cite{Suter:2001:MPI:846235.849401}, one level of recursion is
used while Chou~\cite{Chou94parallelizingstrassen's} provides two levels of
recursion by explicitly coding the 49 matrix multiplications which then is
processed by either 7 or 49 processors.  Our approach is to provide the
recursion cutoff point as a parameter which can be set by the user.

There are various methods which can be used to deal with non square matrices
and/or matrices of odd order.  Methods such as static padding, dynamic padding,
and dynamic peeling all provide these mechanisms.  In this chapter, we only
study matrices of tile order $p = 2^k$, i.e., $n = 2^k \cdot n_b$. 

In \cite{Boyer:2009:MES:1576702.1576713}, Boyer et al.~propose schedules for
both in-place and out-of-place implementations without the need for extra computations.  Discovery of
these algorithms was accomplished by using an exhaustive search algorithm.
Their out-of-place algorithm makes use of the resultant matrix for temporary
storage for the intermediate computations.  This introduces unwanted data
dependencies and more data movement leading to a loss of parallelism.  Hence, we
do not consider their out-of-place algorithm and focus on the classical
Strassen-Winograd algorithm by using temporary storage allocations for all of
the intermediate computations. As an example, given a tiled matrix of $128
\times 128$ tiles of size $200 \times 200$, the input matrices and resultant
matrix require $1.96608$ GB of storage and allowing full recursion for our
algorithm would require an additional $3.2766$ GB of temporary storage (see
Figure~\ref{fig:memalloc}).
\linespread{1.2}
\begin{figure}[htpb]
   \begin{center}
       \resizebox{0.475\linewidth}{!}{\includegraphics[trim = 0mm 0mm 0mm 0mm, clip=true]{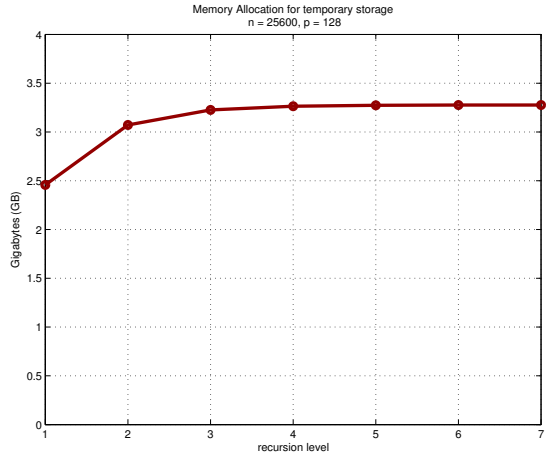}}
   \end{center}
   \caption{Required extra memory allocation for temporary storage for varying
   recursion levels.}
   \label{fig:memalloc}
\end{figure}
\renewcommand{\baselinestretch}{\normalspace}

In~\cite{Douglas94gemmw:a} a sequential implementation of Strassen-Winograd is studied by
Douglas et al.~which also provides means for a hybrid parallel implementation
where the lower level is the sequential Strassen-Winograd and the upper level
algorithm is the standard subdivided matrix multiplication.  Comparisons are
made to sequential implementations available in the IBM's ESSL and Cray's
Scientific and Math library.  Although this would have been an interesting
project in and of itself within a tiled framework, the emphasis in this chapter
was to provide the Strassen-Winograd at the upper level.  

After this work was completed, Ballard et al.~published work which places
Strassen-Winograd in a parallel distributed
environment~\cite{Ballard:2012:CPA:2312005.2312044}.  The paper is
well written and they do show improvement over the standard algorithm for
matrices of order over 94,000.  Their algorithm is communication optimal and
applies better to the distributed environment than the shared memory environment
seeing that we do not have as much control over the memory distribution.

\section{Experimental results}
\label{sec:SWexperiments}

All experiments were performed on a 48-core machine composed of eight hexa-core
AMD Opteron 8439 SE (codename Istanbul) processors running at 2.8 GHz. Each
core has a theoretical peak of 11.2 Gflop/s with a peak of 537.6 Gflop/s for
the whole machine. The Istanbul micro-architecture is a NUMA architecture where
each socket has 6 MB of level-3 cache and each processor has a 512 KB level-2
cache and a 128 KB level-1 cache.  After having benchmarked the AMD ACML and
Intel MKL BLAS libraries, we selected MKL (10.2) since it appeared to be
slightly faster in our experimental context.  Using MKL, for DGEMM each core
has a peak of 9.7 Gflop/s with a peak of 465.6 Gflops/s for the whole machine.
Linux 2.6.32 and Intel Compilers 11.1 were also used in conjunction with PLASMA
2.3.1.

The parameter for tile size has a direct effect on the amount of data movement
and the efficiency of the kernels.  Figure~\ref{fig:tilesize} presents the
performance comparison for varying tile sizes indicating $n_b = 200$ as the
most efficient.  As expected, it is also evidenced that the efficiency of the
Strassen-Winograd algorithm is dependent upon the efficiency of the tiled DGEMM. 
When running on 48 threads, increasing the recursion level decreases the
performance of the algorithm as depicted in Figure~\ref{fig:rectilecomps} ($r=0$
is the tiled matrix-matrix multiplication).

\begin{figure*}[htbp]
   \centering
   \subfloat[Tile size comparison]{
      \label{fig:tilesize}
      \resizebox{0.475\linewidth}{!}{\includegraphics[trim = 0mm 0mm 0mm 0mm, clip=true]{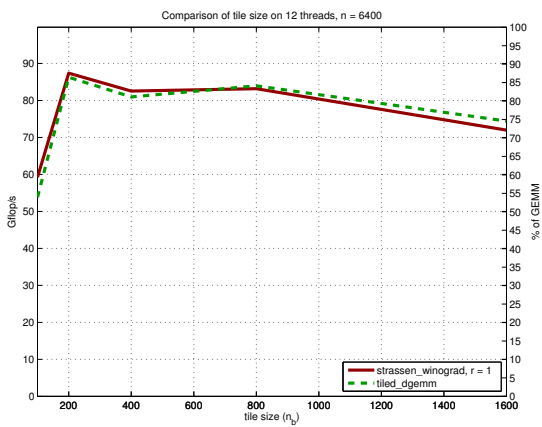}}
   }
   \subfloat[Recursion level comparison]{
      \label{fig:reclevel}
      \resizebox{0.475\linewidth}{!}{\includegraphics[trim = 0mm 0mm 0mm 0mm, clip=true]{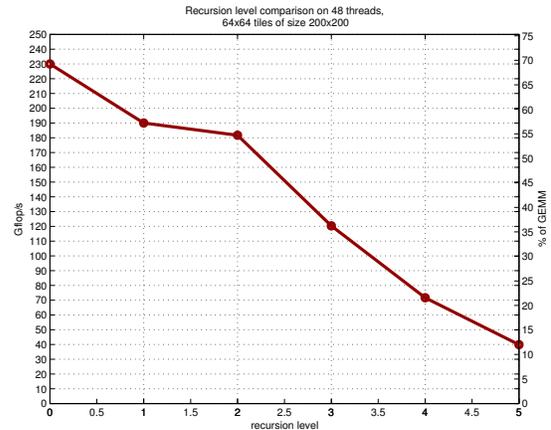}}
   }
   \caption{\label{fig:rectilecomps} Comparison of tuning parameters $n_b$ and $r$.}
\end{figure*}

Our implementation allows for the tuning of the recursion level and can range
from one recursion up to full recursion. In Figure~\ref{fig:reclevel}, matrices
of $64 \times 64$ tiles are used and the recursion level ranges from one to
five.  Although $r_{min} = 2$ for $64 \times 64$ tiles, the best performance
using 48 threads is seen at $r=1$.  This is due to the amount of parallelism
lost by having the critical path length increase as the recursion level
increases which offsets any gains of the reduction in tasks, and computational
cost and complexity.
\linespread{1.0}
\begin{figure*}[htbp]
   \centering
   \subfloat[Scalability comparison]{
      \label{fig:scal12800}
      \resizebox{0.5\linewidth}{!}{\includegraphics[trim = 1mm 0mm 1mm 0mm, clip=true]{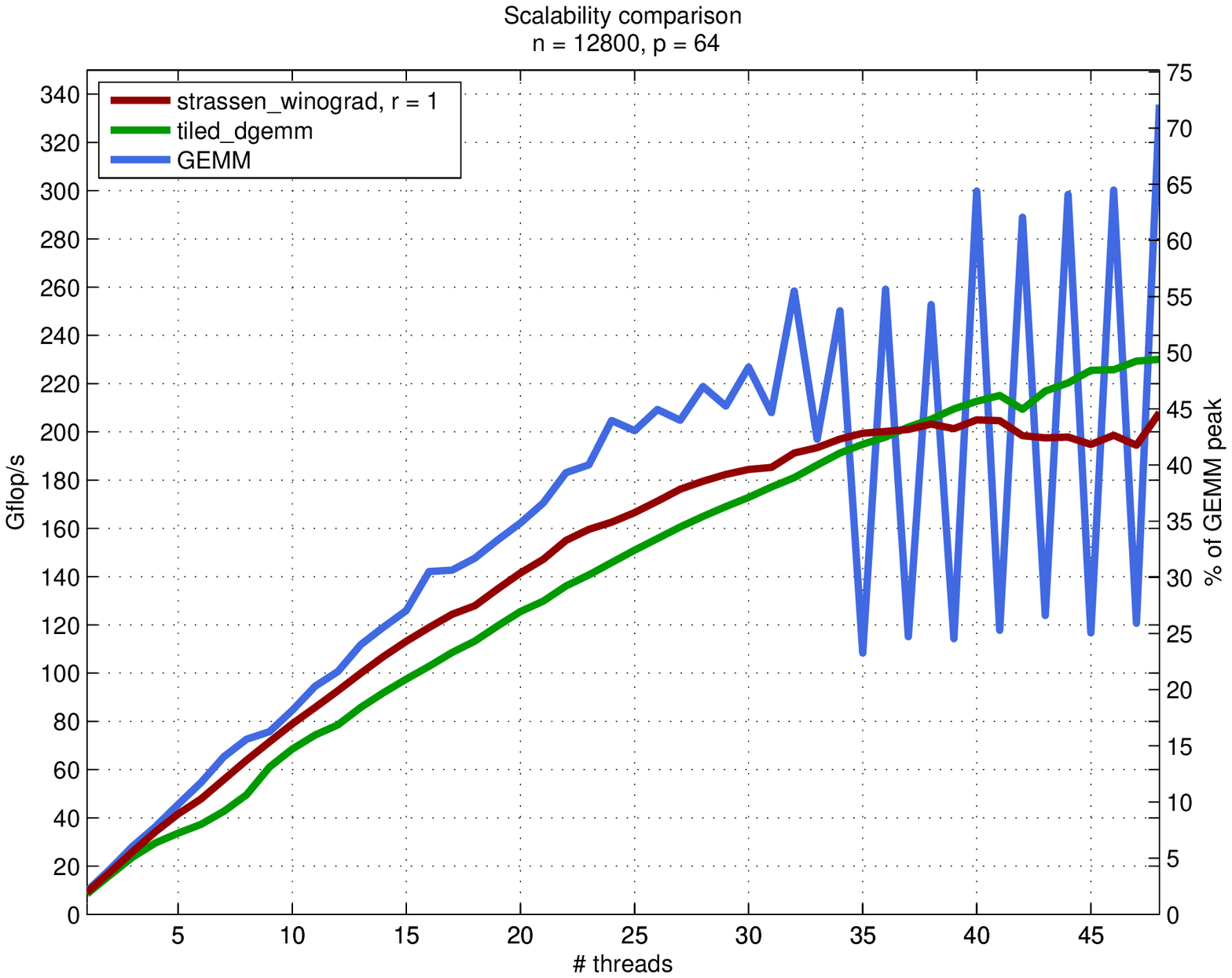}}
   }
   \subfloat[Efficiency comparison]{
      \label{fig:eff12800}
      \hspace{-2mm}\resizebox{0.475\linewidth}{!}{\includegraphics[trim = 0mm 0mm 1mm 0mm, clip=true]{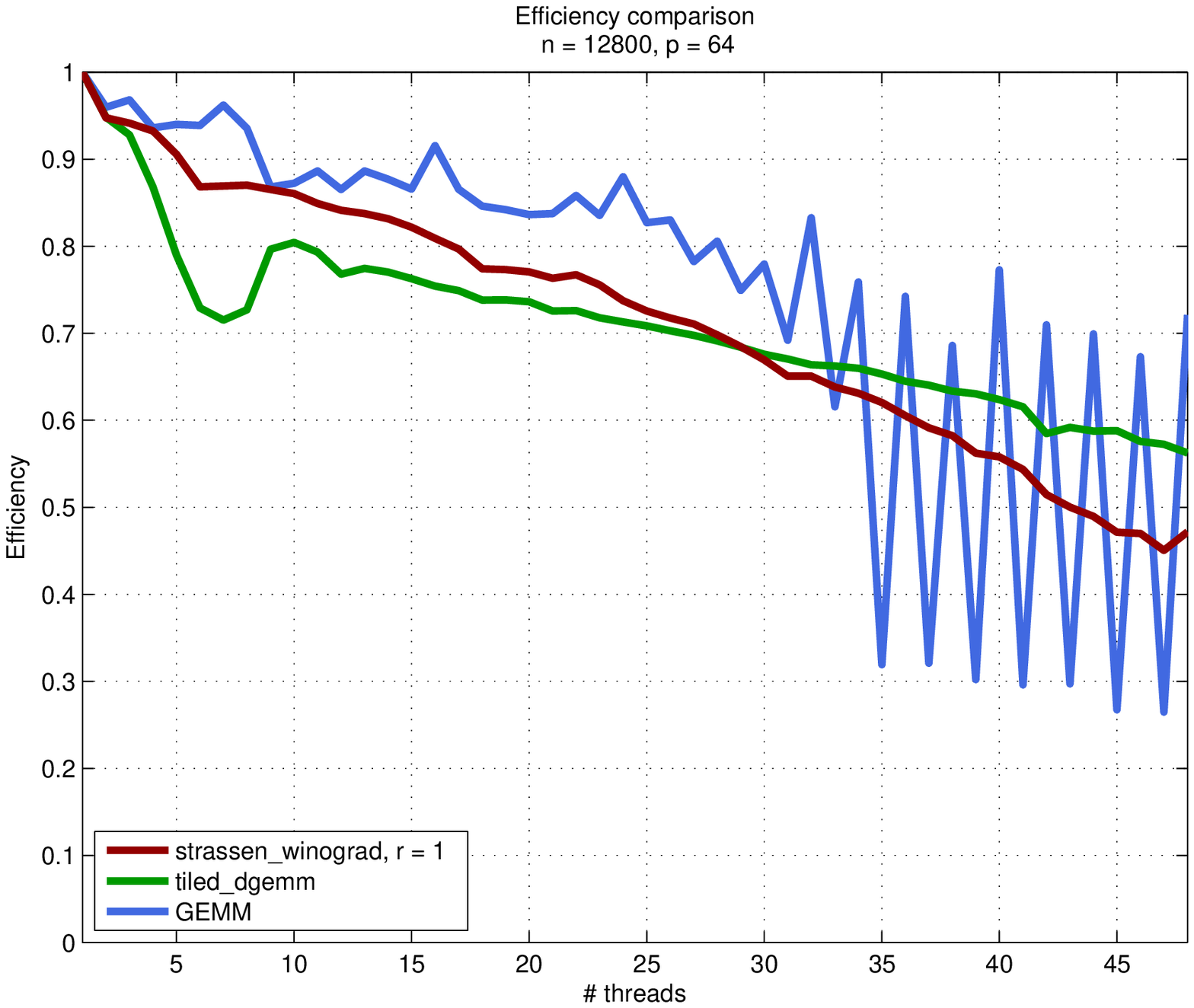}}
   }
   \caption{\label{fig:scaleff12800} Scalability and Efficiency comparisons on 48 threads with matrices of $64 \times 64$ tiles and $n_b=200$.}
\end{figure*}
\renewcommand{\baselinestretch}{\normalspace}

Figures~\ref{fig:scal12800} and \ref{fig:eff12800} illustrate the performance
and efficiency reached by the Strassen-Winograd implementation, with $r=1$, as
compared to the multithreaded MKL DGEMM and the tiled DGEMM implementation.
The Strassen-Winograd implementation outperforms the tiled DGEMM up to the
point where we loose parallelism.  However, both show sub-par performance when
compared to the multithreaded MKL implementation.

\linespread{1.0}
\begin{figure*}[htbp]
   \centering
   \subfloat[Scalabilty comparison]{
      \label{fig:12scal12800}
      \resizebox{0.5\linewidth}{!}{\includegraphics[trim = 1mm 0mm 1mm 0mm, clip=true]{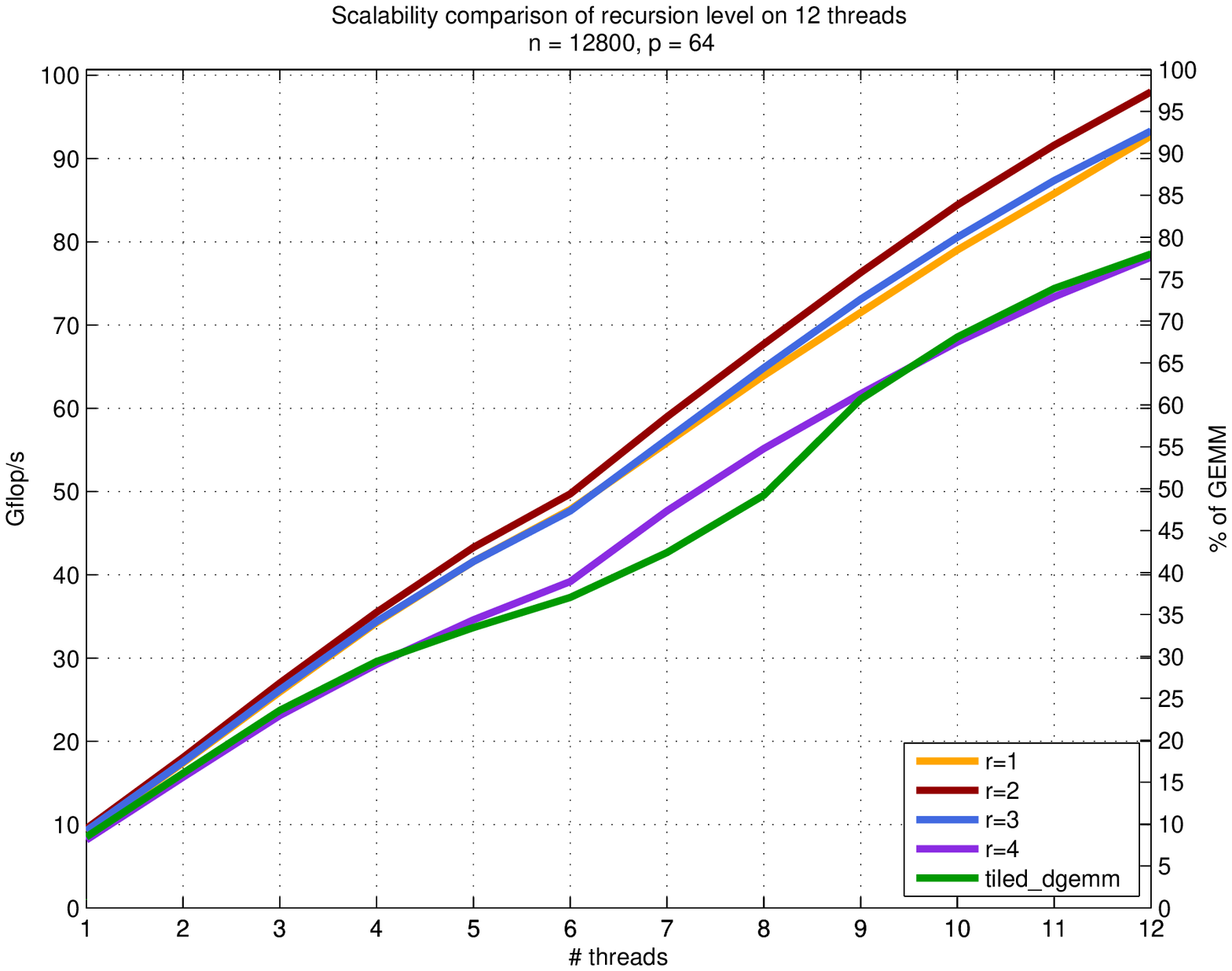}}
   }
   \subfloat[Efficiency comparison]{
      \label{fig:12eff12800}
      \hspace{-2mm}\resizebox{0.475\linewidth}{!}{\includegraphics[trim = 0mm 0mm 1mm 0mm, clip=true]{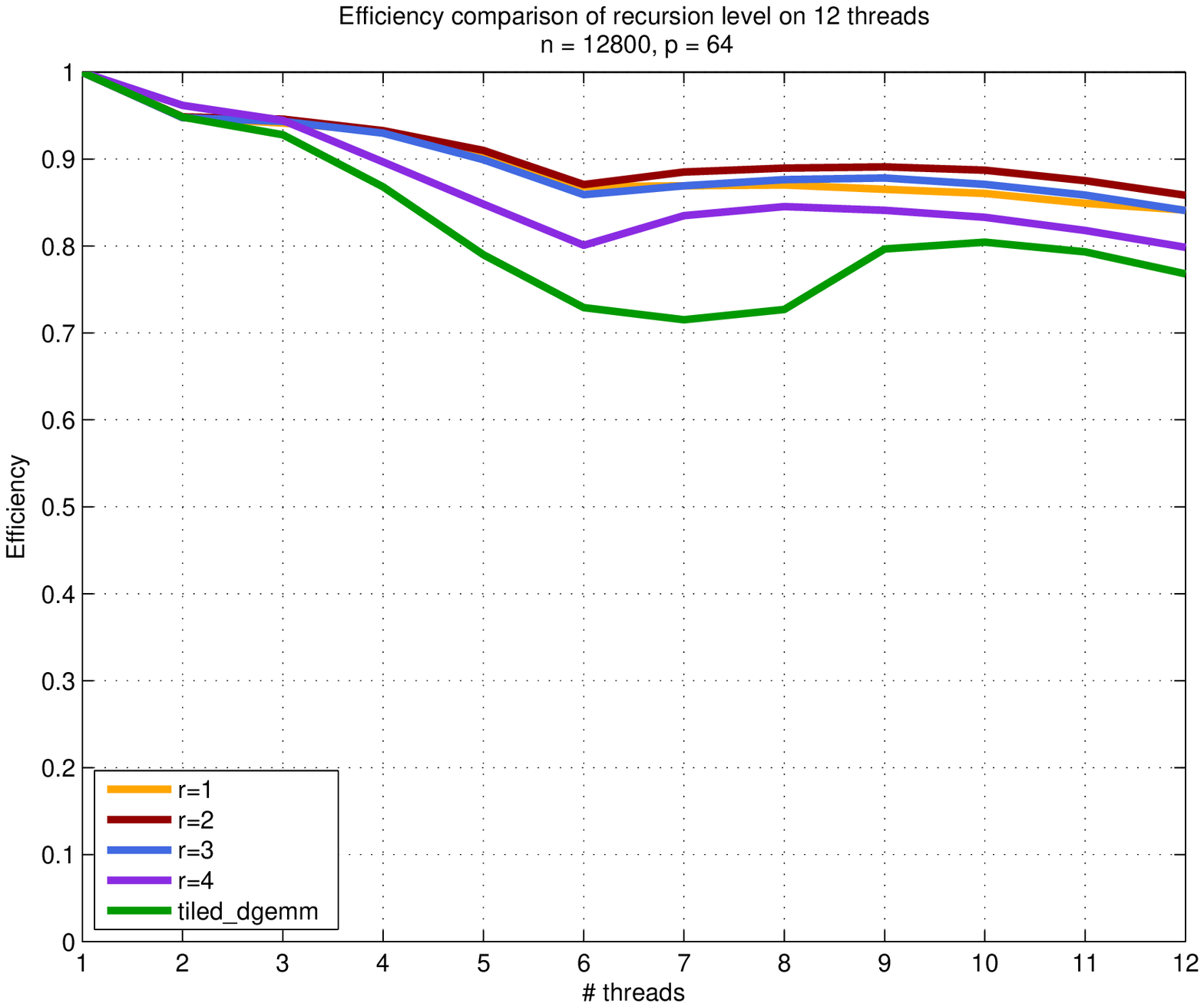}}
   }
   \caption{\label{fig:12scaleff12800} Scalabilty and efficiency comparisons executed on 12 threads with matrices of $64 \times 64$ tiles and $n_b=200$.}
\end{figure*}
\renewcommand{\baselinestretch}{\normalspace}

If we run on 12 cores (typical current architecture for a node) then we do
outperform tiled DGEMM (Algorithm~\ref{alg:tiledMM}) if a loss of parallelism is
not a factor and there is not too much data movement, i.e., keep $n_b$ small
enough so that more tiles fit into the cache but large enough to retain the
efficiency of the GEMM kernel.  Depicted in Figure~\ref{fig:12scaleff12800}, the
performance of the Strassen-Winograd algorithm is best when $r = r_{min}$ since
the number of tasks and computational complexity is minimized which reflects the
analysis of \Section~\ref{sec:tiledSW}.

\section{Conclusion}
\label{sec:SWconclusion}

In this chapter we have shown and analyzed an implementation of the
Strassen-Winograd algorithm in a tiled framework and provided comparisons to
the multithreaded MKL standard library as well as a tiled matrix
multiplication.  The interest in this implementation is that it can support any
level of recursion and any level of parallelism through the use of the recursion and tile
size parameters.

Albeit that our implementation did not perform as well as the highly tuned and
optimized multithreaded MKL library, on 12 cores with 2 levels of recursion,
its performance was only lower by about 2\%. Ultimately, it is a formidable
task to surpass the MKL implementation considering the computational complexity
of a recursive tiled algorithm.  

\chapter{Conclusion}\label{chp:conclusion}
%
%

In this thesis, we have studied the tiled algorithms both theoretically and
experimentally.  Our aim was to alleviate the constraints of memory boundedness,
task granularity, and synchronicity imposed by the LAPACK library as brought to
light in the Cholesky Decomposition example in the introduction.  Moreover, we
have also detailed that one may translate the LAPACK algorithms directly to tiled
algorithms while in other cases a new approach may provide better performance
gains.

In the study of the Cholesky Inversion, the tiled algorithms provided a unique
insight into the interaction of the three distinct steps involved in the
algorithm and how these may be intertwined.  We had observed that the choice of
the variant in the inversion of the triangular factor (Step 2) has a great
impact on the performance of the algorithm.  In the scope of unlimited number of
processors, this choice can lead to an algorithm which performs the inversion of
the matrix in almost the same amount of time it takes to do the Cholesky
factorization.  Moreover, we note that the combination of the variants with the
shortest critical path length does not translate into the best performing
pipelined algorithm.  Even though variant 3 of Step 2 (the triangular inversion)
did not provide us with the shortest critical path length within itself,
combined with the other two steps, it does provide a Cholesky Inversion
algorithm that performed the best overall.

We have also observed that a simple translation from the LAPACK routines may not
provide a tiled algorithm with the best performance.  We showed that
loop-reversals are needed to alleviate anti-dependencies which negatively affect
the performance of the tiled algorithms.  

In the case where a tiled algorithm already existed, namely the QR
Factorization, we made use of a new tiled algorithm to improve upon performance.
These algorithms exhibit more parallelism than state-of-the-art implementations
based on elimination trees.  Using ideas from the 1970/80's, we have
theoretically shown that the new algorithm, {\sc GrASAP}, is optimal in the
scope of unbounded number of processors.  We have provided accurate estimations
for the length of the critical paths for all of these new tiled algorithms and
have provided explicit formulas for some of the algorithms.

In the framework of bounded number of processors, our theoretical work has
afforded a new bound which more accurately reflects the performance expectations
of the tiled algorithms.  In the case of the Cholesky Factorization, the
schedule produced using the Critical Path Method proved to be within seven
percent of the ALAP Derived bound indicating that this scheduling strategy is
well suited for this application.  The ALAP Derived bound has also been used as
a tool to show that the optimality in the scope unbounded number of processors
does not translate to optimality in the scheduling on a bounded number of
processors.  

Overall, the theoretical and experimental portions of this thesis give credence
to the impact of tiled algorithms for multi/many-core architectures.

\begin{appendix}
\chapter{Integer Programming Formulation of Tiled QR} \label{appendix1}
%
%

\section{IP Formulation}
We formulate the problem in question using integer programming. An equivalent binary programming formulation was also constructed (with time as an additional index), but we found the integer formulations were more quickly solved.

Let $i$, $j$, and $h$ denote rows (ranging $1, \dots, p$); $k$, $l$, $l_1$ denote columns (ranging $1, \dots, q$); $r$, $s$, and $t$ denote time steps (ranging $1, \dots, T$). The upper bound on the number of time steps, ($T$), may come from any existing algorithm (greedy, ASAP, GRASP, etc.). Let the decision variables be defined as follows:
\subsection{Variables}
Let all $t$ be an integer ranging from zero to $T$ denoting the time we complete the following tasks (and $t = 0$ means the task is never performed).
\begin{itemize}
  \item \textbf{UNMQR:} Complete applying the reflectors from GEQRT across the row. (Requires 6 units of time.)
    \begin{multline}
        \begin{tabular}{r p{.55\textwidth}}
          $w_{ikl} = t \in [0,T]:$ & if we finish the update of tile $(i,k)$ at
          time  $t$. 
      \end{tabular}
    \end{multline}
\[(\mbox{This update was necessitated by} x_{il} = s: \; l < k, s < t).\]
  \item \textbf{GEQRT:} Factor a square tile into a triangle. (Requires 4 units of time.)
    \begin{multline}
        \begin{tabular}{r p{.55\textwidth}}
            $x_{ik} = t \in [0,T]:$ & if we complete triangularization of tile $(i,k)$
            at the end of time unit $t$. 
  \end{tabular}
    \end{multline}
  \item \textbf{TTMQR:} Update the entries in two rows after TTQRT. (Requires 6 units of time.)
    \begin{multline}
        \begin{tabular}{r p{.55\textwidth}}
        $y_{ijkl} = t \in [0,T]:$ & if we finish the update of tile $(i,k)$ and $(j,k)$ at the end of time unit $t$.
      \end{tabular}
    \end{multline}
    \[(\mbox{This update was necessitated by } z_{ijl} = s: l<k, s<t)\]
  \item \textbf{TTQRT:} Cancel one triangular tile using another triangle tile. (Requires 2 units of time.)
    \begin{multline}
        \begin{tabular}{r p{.55\textwidth}}
            $z_{ijk} = t \in [0,T]:$ & if we complete zeroing tile $(i,k)$ using
            tile $(j,k)$ at the end of time unit $t$.
      \end{tabular}
    \end{multline}
\end{itemize}
For the $y$ and $z$ actions, it is useful to have a binary variable which is $1$ if the action occurs. Explicitly,
\begin{equation}
  \hat{y}_{ijkl} = 
  \left\{
  \begin{array}{c l}
    1:& \mbox{ if } y_{ijkl} > 0 \\
    0:& \mbox{ otherwise }
  \end{array}
  \right.
  \;\;\;
  \hat{z}_{ijk} = 
  \left\{
  \begin{array}{c l}
    1:& \mbox{ if } z_{ijk} > 0 \\
    0:& \mbox{ otherwise }
  \end{array}
  \right.
\end{equation}

\subsection{Constraints}
\begin{enumerate}
  \item Time constraints for each of the four actions:
    \begin{enumerate}
      \item Time for $w_{ikl}$
        \begin{enumerate}
          \item $w_{ikl}$ must occur at least 3 time steps after earlier $w$ updates.
            \begin{equation}
              w_{ikl} \ge w_{ikl_1} + 3 \;\; \forall\; k \ge 2, i \ge l, l_1 < l < k
            \end{equation}
          \item $w_{ikl}$ must occur at least 3 time steps after earlier $y$ updates.
            \begin{equation}
              w_{ikl} \ge y_{ijkl_1} + y_{jikl_1} + 3 \;\; \forall\; k \ge 2, i \ge l,j,l_1 < l < k 
            \end{equation}
          \item $w_{ikl}$ must occur at least 3 time steps before $y$ updates in the current column.
            \begin{equation}
              w_{ikl} +3 \le y_{ijkl} + y_{jikl} + (1-\hat{y}_{ijkl} - \hat{y}_{jikl})T \;\; \forall\; i,j,l<k, k \ge 2
              \label{wygap}
            \end{equation}
          \item $w_{ikl}$ must occur at least 2 time steps before $x$.
            \begin{equation}
              w_{ikl} + 2 \le x_{ik} \;\; \forall\; i \ge k \ge 2, l < k 
              \label{wxgap}
            \end{equation}
          \item $w_{ikl}$ must occur at least 1 time step before $z$ actions.
            \begin{equation}
              w_{ikl} + 1 \le z_{ijk} + z_{jik} + (1-\hat{z}_{ijk} - \hat{z}_{jik})t \;\; \forall\; i,j, k \ge 2, l<k
              \label{wzgap}
            \end{equation}
        \end{enumerate}
      \item Time for $x_{ik}$
        \begin{enumerate}
          \item $x_{ik}$ must occur 2 time steps after any $w$ updates. (Identical to equation \eqref{wxgap}.)
          \item $x_{ik}$ must occur 2 time steps after any $y$ updates.
            \begin{equation}
              x_{ik} \ge y_{ijkl} + y_{jikl} + 2 \;\; \forall j,l,i \ge k, l<k
              \label{xygap}
            \end{equation}
          \item $x_{ik}$ must occur 1 time steps before any $z$ action.
            \begin{equation}
              x_{ik} + 1 \le z_{ijk} + z_{jik} + (1-\hat{z}_{ijk} - \hat{z}_{jik})T \;\; \forall\; i \ge k, j
              \label{xzgap}
            \end{equation}
        \end{enumerate}
      \item Time for $y_{ijkl}$
        \begin{enumerate}
          \item $y_{ijkl}$ must occur 3 time steps after $w$ updates originating from the same column. (Identical to equation \ref{wygap}.)
          \item $y_{ijkl}$ must occur 2 time steps before any $x_{ik}$ or $x_{jk}$. (Identical to equation \ref{xygap}.)
          \item $y_{ijkl}$ must be 3 time steps before or after any $y$ action involving rows $i$ or $j$. \\
            We must enforce 
            \begin{eqnarray*}
              y_{ijkl} + y_{jikl} + 3 &\le& y_{hikl} + y_{ihkl} + (1-\hat{y}_{hikl} -\hat{y}_{ihkl})T\\
              &\mbox{or}&\\
              y_{hikl} + y_{ihkl} + 3 &\le& y_{ijkl} + y_{jikl} + (1-\hat{y}_{ijkl} -\hat{y}_{jikl})T\\
            \end{eqnarray*}
            And 
            \begin{eqnarray*}
              y_{ijkl} + y_{jikl} + 3 &\le& y_{hjkl} + y_{jhkl} + (1-\hat{y}_{hjkl} -\hat{y}_{jhkl})T\\
              &\mbox{or}&\\
              y_{hjkl} + y_{jhkl} + 3 &\le& y_{ijkl} + y_{jikl} + (1-\hat{y}_{ijkl} -\hat{y}_{jikl})T\\
            \end{eqnarray*}
            We define the binary variables $\delta^1_{hijkl}$, $\delta^2_{hijkl}$, $\delta^3_{hijkl}$, and $\delta^4_{hijkl}$, and include the disjunctive constraints
            \begin{IEEEeqnarray}{rCl}
              y_{ijkl} + y_{jikl} + 3 &\le& y_{hikl} + y_{ihkl} +
              (1-\hat{y}_{hikl} -\hat{y}_{ihkl})T + \delta^1_{hijkl}T
              \IEEEeqnarraynumspace\\ 
              && \forall\; h,i,j, l<k \ge 2\nonumber\\
              y_{hikl} + y_{ihkl} + 3 &\le& y_{ijkl} + y_{jikl} +
              (1-\hat{y}_{ijkl} -\hat{y}_{jikl})T + \delta^2_{hijkl}T
              \IEEEeqnarraynumspace\\ 
              && \forall\; h,i,j, l<k \ge 2\nonumber\\
              \delta^1_{hijkl} + \delta^2_{hijkl} &\ge& 1 \;\; \forall\; h,i,j, l<k \ge 2
            \end{IEEEeqnarray}
            and 
            \begin{IEEEeqnarray}{rCl}
              y_{ijkl} + y_{jikl} + 3 &\le& y_{hjkl} + y_{jhkl} + (1-\hat{y}_{hjkl} -\hat{y}_{jhkl})T + \delta^3_{hijkl}T 
              \IEEEeqnarraynumspace\\ 
              && \forall\; h,i,j, l<k \ge 2\nonumber\\
              y_{hjkl} + y_{jhkl} + 3 &\le& y_{ijkl} + y_{jikl} + (1-\hat{y}_{ijkl} -\hat{y}_{jikl})T + \delta^4_{hijkl}T
              \IEEEeqnarraynumspace\\ 
              && \forall\; h,i,j, l<k \ge 2\nonumber\\
              \delta^3_{hijkl} + \delta^4_{hijkl} &\ge& 1 \;\; \forall\; h,i,j, l<k\ge 2
            \end{IEEEeqnarray}
          \item $y_{ijkl}$ must occur 3 time steps before any $z$ action involving rows $i$ or $j$.
            \begin{equation}
              \begin{array}{r l}
                y_{ijkl} + 3 \le & z_{hik} + z_{ihk} + (1-\hat{z}_{hik} - \hat{z}_{ihk})T \;\; \forall\; h,i,j,k>l, k\ge 2\\
                y_{ijkl} + 3 \le & z_{hjk} + z_{jhk} + (1-\hat{z}_{hjk} - \hat{z}_{jhk})T \;\; \forall\; h,i,j,k>l, k\ge 2
              \end{array}
              \label{yzgap}
            \end{equation}
        \end{enumerate}
      \item Time for $z_{ijk}$
        \begin{enumerate}
          \item $z_{ijk}$ must occur 1 time step after any $w$ action. (Identical to equation \ref{wzgap}.)
          \item $z_{ijk}$ must occur 1 time step after any $y$ action. (Identical to equation \ref{yzgap}.) 
          \item $z_{ijk}$ must occur 1 time step after any $x$ action. (Identical to equation \ref{xzgap}.)
          \item $z_{ijk}$ must occur 1 time step before or after any other $z$ action.
            \begin{enumerate}
              \item Case 1. We use $(i,k)$ to zero $(j,k)$ and $(i,k)$ to zero $(h,k)$ \\
                We need to enforce 
                \begin{eqnarray*}
                  z_{jik} + 1 &\le& z_{hik} + (1-\hat{z}_{hik})T \\
                  &\mbox{or}&\\
                  z_{hik} + 1 &\le& z_{jik} + (1-\hat{z}_{jik})T
                \end{eqnarray*}
                To do this, we define binary variables $\delta^5_{hijk}$ and $\delta^6_{hijk}$ and include the disjunctive constraints as follows.
                \begin{eqnarray}
                  z_{jik} + 1 &\le& z_{hik} + (1-\hat{z}_{hik})T + \delta^5_{hijk}T \;\; \forall\; h,i,j,k \\
                  z_{hik} + 1 &\le& z_{jik} + (1-\hat{z}_{jik})T + \delta^6_{hijk}T \;\; \forall\; h,i,j,k \\
                  \delta^5_{hijk} + \delta^6_{hijk} &\ge& 1 \;\; \forall\; h,i,j,k
                \end{eqnarray}
              \item Case 2. We use $(i,k)$ to zero $(j,k)$ and $(h,k)$ to zero $(i,k)$
                We need to enforce
                \begin{equation}
                  z_{jik} + 1 \le z_{ihk} + (1-\hat{z}_{ihk})T \;\; \forall\; h,i,j,k (i>j??) \\
                \end{equation}
            \end{enumerate}
        \end{enumerate}
    \end{enumerate}
  \item A tile can't zero itself, (and hence we can't update just a single row afterwards).
    \begin{equation}
      z_{iik} = 0 \;\; \forall \; i, k \qquad y_{iikl} = 0 \;\; \forall \; i,k,l
    \end{equation}
  \item Both tiles involved in TTQRT must be triangles before one can zero another
    \begin{equation}
      x_{ik} \le (1-\hat{z}_{ijk})T + z_{ijk} \;\; \forall\; i,j,k, \qquad x_{jk} \le (1-\hat{z}_{ijk})T + z_{ijk} \;\; \forall\; i,j,k
    \end{equation}
  \item Force updates after triangle and zeroing actions.
    \begin{enumerate}
      \item After a tile is triangularized, updates must occur in the next column.
        \begin{equation}
          x_{ik} \le w_{ilk} - 3 \qquad \forall\; i,k<q, i\ge k, l > k
        \end{equation}
      \item After a tile is zeroed, updates must occur in the next column.
        \begin{equation}
          z_{ijk} \le \left( y_{ijlk} + y_{jilk} \right) \qquad \forall\; i,j,k<l
        \end{equation}
    \end{enumerate}
  \item The updates of $(i,k)$ (arising from pivot $l$) from triangularizing must occur before updates from zeroing involving $(i,k)$ (also arising from pivot $l$).
    \begin{equation}
      w_{ikl} \le (1 - \hat{y}_{ijkl} - \hat{y}_{jikl})T + y_{ijkl} +  y_{jikl} \qquad \forall\;i,j,k>l
    \end{equation}
  \item No updates to a tile can occur after triangularization.
    \begin{equation}
      x_{ik} \ge w_{ikl} \qquad \forall\; i\ge k,k>l \qquad x_{ik} \ge y_{ijkl} + y_{jikl} \qquad \forall\; i\ge k,j,k>l
    \end{equation}
  \item After a tile $(i,k)$ is zeroed, we can't use it to zero.
    \begin{equation}
      z_{ijk} \ge z_{hik} \qquad \forall\; h,i,j,k
    \end{equation}
  \item Tiles on or below the diagonal must be triangularized at some point, (and we can't finish a triangularization until time step 2.)
    \begin{equation}
      x_{ik} \ge 2 \qquad \forall\; i \ge k
    \end{equation}
  \item Tiles strictly below the diagonal must be zeroed at some point.
    \begin{equation}
      \sum_j \hat{z}_{ijk} = 1 \qquad \forall\; i > k
    \end{equation}
  \item Can't triangularize above the diagonal. 
    \begin{equation}
      x_{ik} = 0 \;\; \forall\; i<k
    \end{equation}
  \item Force binary variables
    \begin{equation}
      \hat{y}_{ijkl} \le y_{ijkl} \qquad \hat{y}_{ijkl}*T \ge y_{ijkl} \qquad \forall i,j,k,l
    \end{equation}
    \begin{equation}
      \hat{z}_{ijk} \le z_{ijk} \qquad \hat{z}_{ijk}*T \ge z_{ijk} \qquad \forall i,j,k
    \end{equation}
\end{enumerate}

\subsubsection{Precedence constraints}
The most cumbersome constraints to formulate are those forcing the corresponding TTQRT and TTMQR operations to be executed in the same order. 

    For each tile $(i,k)$ involved in a zeroing process with $(j,k)$ and $(h,k)$, the order of the updates must follow the order of the zeroing processes. We want $z_{jik} < z_{ihk}$ (for some $h$) or $z_{jik} < z_{hik}$ (for some $h$), to force $y_{ij(k+1)k} < y_{ih(k+1)k}$. \\
\begin{enumerate}
  \item{$a^1$}
    \begin{equation}
      a^1_{hijk} =
      \left\{
          \begin{tabular}{r p{9cm}}
              1 : & if we [use $(i,k)$ to zero $(h,k)$] and also [use $(i,k)$ to zero
              $(j,k)$ ]\\
              0 : & otherwise  
          \end{tabular}
      \right.
    \end{equation}
    \begin{equation}
      \begin{array}[h!]{r c l}
        a^1_{hijk} &\le& \hat{z}_{hik}\\
        a^1_{hijk} &\le& \hat{z}_{jik}\\
        a^1_{hijk} + 1&\ge& \hat{z}_{hik} + \hat{z}_{jik}
      \end{array}
    \end{equation}

  \item{$a^2$}
    \begin{equation}
      a^2_{hijk} =
      \left\{
          \begin{tabular}{r p{9cm}}
              1 : & if we [use $(h,k)$ to zero $(i,k)$] after [using $(i,k)$ to zero
              $(j,k)$]\\
              0 : & otherwise  
          \end{tabular}
      \right.
    \end{equation}
    \begin{equation}
      \begin{array}[h!]{r c l}
        a^2_{hijk} &\le& \hat{z}_{ihk}\\
        a^2_{hijk} &\le& \hat{z}_{jik}\\
        a^2_{hijk} + 1&\ge& \hat{z}_{ihk} + \hat{z}_{jik}
      \end{array}
    \end{equation}

  \item{$b$}
    \begin{equation}
      b_{hijk} =
      \left\{
      \begin{array}{ll}
        1 & : \mbox{if }[z_{hik} > z_{jik} \mbox{] regardless of whether [} z_{hik} = 0 \mbox{]} \\
        0 & : \mbox{otherwise}  
      \end{array}
      \right.
    \end{equation}
    \begin{equation}
      \begin{array}[h!]{r c l}
        T b_{hijk} & \ge & z_{hik} - z_{jik} \\
        T(b_{hijk}-1) & \le & z_{hik} - z_{jik} 
      \end{array}
    \end{equation}

  \item{$c^1$}
    \begin{equation}
      c^1_{hijk} =
      \left\{
      \begin{array}{ll}
        1 & : \mbox{if }[z_{hik} > z_{jik}] \\
        0 & : \mbox{otherwise}  
      \end{array}
      \right.
    \end{equation}
    \begin{equation}
      \begin{array}[h!]{r c l}
        c^1_{hijk} &\le& a^1_{hijk} \\
        c^1_{hijk} &\le& b_{hijk} \\
        c^1_{hijk} +1 &\ge& a^1_{hijk} +  b_{hijk} \\
      \end{array}
    \end{equation}

  \item{$c$}
    \begin{equation}
      c_{hijk} =
      \left\{
          \begin{tabular}{r p{9cm}}
              1 : & if zeroing actions in column $k$ of rows $h$ and $i$ to occur after the zeroing actions of rows $i$ and $j$\\
              0 : & otherwise  
          \end{tabular}
      \right.
    \end{equation}
    \begin{equation}
      \begin{array}[h!]{r c l}
        c_{hijk} &\ge& c^1_{hijk} \\
        c_{hijk} &\ge& a^2_{hijk} \\
        c_{hijk} &\le& c^1_{hijk} + a^2_{hijk} 
      \end{array}
    \end{equation}

    So we now have a variable $c_{hijk}$ that is 1 when updates of row $h$ and $i$ must come before the updates of $i$ and $j$. We can define similar variables for the updates rather than the zeros.
  \item{$d$}
    \begin{equation}
      d_{hijk} =
      \left\{
          \begin{tabular}{r p{9cm}}
              1 : & if we [update $(i,k)$ and $(h,k)$ together] and
              also [update $(i,k)$ and  $(j,k)$ together]. (All updates occur because of zeroing in column $k-1$)\\
              0 : & otherwise  
          \end{tabular}
      \right.
    \end{equation}
    \begin{equation}
      \begin{array}[h!]{r c l}
        d_{hijk} &\le&  \hat{y}_{hik(k-1)} + \hat{y}_{ihk(k-1)} \\
        d_{hijk} &\le&  \hat{y}_{jik(k-1)} + \hat{y}_{ijk(k-1)} \\
        d_{hijk}+1 &\ge& \hat{y}_{hik(k-1)} + \hat{y}_{ihk(k-1)} + \hat{y}_{jik(k-1)} + \hat{y}_{ijk(k-1)}\\
      \end{array}
    \end{equation}

  \item{$e$}
    \begin{equation}
      e_{hijk} =
      \left\{
          \begin{tabular}{r p{9cm}}
              1 : & if updates in column $k$ of rows $h$ and $i$ to occur after updates of $i$ and $j$ or the updates in $i$ and $j$ never happen\\
              0 : & \mbox{otherwise}  
          \end{tabular}
      \right.
    \end{equation}
    \begin{equation}
      \begin{array}[h!]{r c l}
        T e_{hijk} & \ge & \left( y_{hik(k-1)} + y_{ihk(k-1)} \right) - \left( y_{jik(k-1)} + y_{ijk(k-1)} \right) \\
        T (e_{hijk}-1)  & \le & \left( y_{hik(k-1)} + y_{ihk(k-1)} \right) - \left( y_{jik(k-1)} + y_{ijk(k-1)} \right) \\
      \end{array}
    \end{equation}

  \item{$f$}
    \begin{equation}
      f_{hijk} =
      \left\{
      \begin{tabular}{r p{9cm}}
          1 : & if the updates in column $k$ of rows $h$ and $i$ to occur after
          updates of $i$ and $j$\\
          0 : & otherwise
      \end{tabular}
      \right.
    \end{equation}
    \begin{equation}
      \begin{array}[h!]{r c l}
        f_{hijk} &\ge& d_{hijk} \\
        f_{hijk} &\ge& e_{hijk} \\
        f_{hijk} &\le& d_{hijk} + e_{hijk} 
      \end{array}
    \end{equation}

    Lastly, to force the updates in order:
    \begin{equation}
      c_{hijk} \le f_{hijl} \qquad \forall h,i,j,k<l
    \end{equation}
\end{enumerate}
\subsection{Objective function}
Let total\_time be a variable such that 
\begin{eqnarray*}
  total\_time &\ge& w_{ikl} \qquad \forall i,k,l\\
  total\_time &\ge& x_{ik} \qquad \forall i,k\\
  total\_time &\ge& y_{ijkl} \qquad \forall i,j,k,l\\
  total\_time &\ge& z_{ijk} \qquad \forall i,j,k\\
\end{eqnarray*}
Then our objective function is:
\begin{equation}
  \min total\_time
\end{equation}

\end{appendix}

\bibliographystyle{plain}
\bibliography{thesis}

\glstoctrue
\glsaddall
\printglossaries

\end{document}